\documentclass[UTF-8,reqno]{amsart}
\usepackage{enumerate}
\usepackage{mhequ}
\usepackage[margin=1.2in]{geometry}
\usepackage{amssymb,url,color, booktabs,nccmath}
\usepackage{mathrsfs}
\usepackage{enumitem}
\usepackage{graphicx}
\usepackage{tikz}

\usepackage{color}
\usepackage[colorlinks=true]{hyperref}
\hypersetup{
    linkcolor=blue,          % color of internal links
    citecolor=red,        % color of links to bibliography
    filecolor=blue,      % color of file links
    urlcolor=cyan
}

\definecolor{darkergreen}{rgb}{0.0, 0.5, 0.0}

\long\def\zhu#1{{\color{red}\footnotesize Zhu:\ #1}}

\numberwithin{equation}{section}
\def\theequation{\arabic{section}.\arabic{equation}}
\newcommand{\be}{\begin{eqnarray}}
\newcommand{\ee}{\end{eqnarray}}
\newcommand{\ce}{\begin{eqnarray*}}
\newcommand{\de}{\end{eqnarray*}}
\newtheorem{theorem}{Theorem}[section]
\newtheorem{lemma}[theorem]{Lemma}
\newtheorem{remark}[theorem]{Remark}
\newtheorem{definition}[theorem]{Definition}
\newtheorem{proposition}[theorem]{Proposition}
\newtheorem{Examples}[theorem]{Example}
\newtheorem{corollary}[theorem]{Corollary}
\newtheorem{assumption}{Assumption}[section]

\newcommand{\rmk}[1]{\textcolor{red}{#1}}

\newcommand{\LL}{\mathscr{L}}
\newcommand{\UU}{\mathscr{U}}

\newcommand{\SZ}{\mathscr{Z}_N}

\def\Wick#1{\,\colon\!\! #1 \!\colon}

\def\PPhi{\mathbf{\Phi}}

\def\eps{\varepsilon}

\def\p{\partial}

\def\[{{\Big[}}
\def\]{{\Big]}}
\def\<{{\langle}}
\def\>{{\rangle}}
\def\({{\Big(}}
\def\){{\Big)}}

\def\bx{{\mathbf{x}}}

\def\dif{{\mathord{{\rm d}}}}

\def\no{\nonumber}
\def\={&\!\!=\!\!&}

 \newcommand{\eqdef}{\stackrel{\mbox{\tiny def}}{=}}

\def\bC{{\mathbf C}}

\def\mZ{{\mathbb Z}}

\def\1{{\mathbf{1}}}

\def\E{\mathbf E}

\def\geq{\geqslant}
\def\leq{\leqslant}
\def\ge{\geqslant}
\def\le{\leqslant}

\def\eps{\varepsilon}

\def\p{\partial}

\def\[{{\Big[}}
\def\]{{\Big]}}
\def\<{{\langle}}
\def\>{{\rangle}}
\def\({{\Big(}}
\def\){{\Big)}}

\def\bx{{\mathbf{x}}}

\def\dif{{\mathord{{\rm d}}}}

\def\no{\nonumber}
\def\={&\!\!=\!\!&}
\def\bt{\begin{theorem}}
\def\et{\end{theorem}}
\def\bl{\begin{lemma}}
\def\el{\end{lemma}}
\def\br{\begin{remark}}
\def\er{\end{remark}}
\def\bx{\begin{Examples}}
\def\ex{\end{Examples}}
\def\bd{\begin{definition}}
\def\ed{\end{definition}}
\def\bp{\begin{proposition}}
\def\ep{\end{proposition}}
\def\bc{\begin{corollary}}
\def\ec{\end{corollary}}

\def\geq{\geqslant}
\def\leq{\leqslant}
\def\ge{\geqslant}
\def\le{\leqslant}

 \def\R{\mathbb R}
 \def\R{\mathbb R}    
\def\N{\mathbb N}  
   
\def\<{\langle} \def\>{\rangle}

\allowdisplaybreaks

\begin{document}

\title{Large $N$ Limit of the $O(N)$ Linear Sigma Model  via Stochastic Quantization}

\author{Hao Shen}
\address[H. Shen]{Department of Mathematics, University of Wisconsin - Madison, USA}
\email{pkushenhao@gmail.com}

\author{Scott Smith}
\address[S. Smith]{Department of Mathematics, University of Wisconsin - Madison, USA
}
\email{ssmith74@wisc.edu}

\author{Rongchan Zhu}
\address[R. Zhu]{Department of Mathematics, Beijing Institute of Technology, Beijing 100081, China; Fakult\"at f\"ur Mathematik, Universit\"at Bielefeld, D-33501 Bielefeld, Germany}
\email{zhurongchan@126.com}

\author{Xiangchan Zhu}
\address[X. Zhu]{ Academy of Mathematics and Systems Science,
Chinese Academy of Sciences, Beijing 100190, China; Fakult\"at f\"ur Mathematik, Universit\"at Bielefeld, D-33501 Bielefeld, Germany}
\email{zhuxiangchan@126.com}

\begin{abstract}
This article studies large $N$ limits of a coupled system of $N$ interacting $\Phi^4$ equations posed over $\mathbb{T}^{d}$ for $d=2$, known as the $O(N)$ linear sigma model.  Uniform in $N$ bounds on the dynamics are established, allowing us to show convergence to a mean-field singular SPDE, also proved to be globally well-posed.  Moreover, we show tightness of the invariant measures in the large $N$ limit.
For large enough mass, they converge to the (massive) Gaussian free field, the unique invariant measure of the mean-field dynamics, at a rate of order $1/\sqrt{N}$ with respect to the Wasserstein distance.  We also consider fluctuations and obtain tightness results for certain $O(N)$ invariant observables, along with an exact description of the limiting correlations.
\end{abstract}


\date{\today}

\maketitle

\setcounter{tocdepth}{1}
\tableofcontents

\section{Introduction}

%\scott{Should the title reflect that we only work on 1d? or 2d?}
%\hao{I think it's fine to just say it in abstract}%\scott{For example `Large N Limit of the Linear Sigma Model on $\mathbb T^{2N}$ via Stochastic Quantization''; not sure if this is better, mostly just forward thinking in case we also want to do 3d later }

In this paper, we consider the following system of equations on the $d$-dimensional torus $\mathbb T^d$ for $d=2$  %\hao{I made the equation formal, and immediately mention Wick in 2d, because I feel a bit weird to say ``in 1d ''   $\Wick{\Phi_j^2\Phi_i}$ means the usual product - hope this is fine for you!}
\begin{equation}\label{eq:21}
\LL \Phi_i= -\frac{1}{N}\sum_{j=1}^N \Phi_j^2\Phi_i+\xi_i,\quad \Phi_i(0)=\phi_i,
\end{equation}
where $\LL=\p_t-\Delta+m$ with $m\geq0$, $N\in \N$, and $i\in\{1,\cdots,N\}$.
The collection $(\xi_i)_{i=1}^N$ consists of $N $ independent space-time white noises on a stochastic basis, i.e. $(\Omega,\mathcal{F},\mathbf{P})$ with a filtration,
and  $(\phi_i)_{i=1}^{N}$ are random initial datum independent of  $(\xi_i)_{i=1}^{N}$.
In $d=2$, the system \eqref{eq:21} requires renormalization, and the formal product $\Phi_j^2\Phi_i$ will be interpreted
as the Wick product
  $\Wick{\Phi_j^2\Phi_i}$ whose definition is postponed to Section \ref{sec:nonlinear}.
  % means the renormalization of $\Phi_j^2\Phi_i$. For $d=1$, $\Wick{\Phi_j^2\Phi_i}$ is the usual product.

This system arises as the stochastic quantization of the following $N$-component generalization of the $\Phi^4_d$ model, given by the  (formal) measure %\hao{not sure if there should be the $2$}\zhu{we are sure otherwise there's $\sqrt{2}$ before the noise}
\begin{equation}\label{e:Phi_i-measure}
\dif\nu^N(\Phi)\eqdef \frac{1}{C_N}\exp\bigg(-\int_{\mathbb T^d} \sum_{j=1}^N|\nabla \Phi_j|^2+m \sum_{j=1}^N\Phi_j^2
+\frac{1}{2N} \Big(\sum_{j=1}^N\Phi_j^2\Big)^2 \dif x\bigg)\mathcal D \Phi
\end{equation}
over $\R^N$ valued fields $\Phi=(\Phi_1,\Phi_2,...,\Phi_N)$ and $C_N$ is a normalization constant.
In $d=2$, the  interaction should be Wick renormalized   $\Wick{\big(\sum_{j=1}^N\Phi_j^2\big)^2}$
for the measure to make sense.
This is also referred to as the $O(N)$ linear sigma model, since this formal measure is invariant
under a rotation of the $N$ components of $\Phi$. \footnote{The word ``linear'' here only means that the target space $\mathbb R^N$ is a linear space. ``Nonlinear'' sigma models on the other hand refers to similar models where the target space is subject to certain nonlinear constraints, e.g. $\Phi$ takes value in a sphere in $\mathbb R^N$ or more generally in a manifold.} This symmetry will play an important role throughout the paper.

Our focus in this article is on the asymptotic behavior as $N\to \infty$ of the system \eqref{eq:21} and its invariant measures \eqref{e:Phi_i-measure} as well as observables which preserve the $O(N)$ symmetry.  Note that a factor $1/N$ has been introduced in front of the nonlinearity (resp. the quartic term in the measure),
and heuristically, this compensates the sum of $N$ terms so that one could hope to obtain an interesting limit as $N\to \infty$. The study of physically meaningful quantities associated with a quantum field theory model such as \eqref{e:Phi_i-measure} as $N\to \infty$ is generally referred to as a {\it large $N$ problem};  see Section~\ref{sec:LargeN} where we introduce more background, references in physics and mathematics, and different approaches to this problem. To the best of our knowledge, the present article provides the first rigorous results on large $N$ problems in the formulation of stochastic quantization.

In Theorem \ref{th:1} below, we study the $N \to \infty$ limit of each component in the Wick renormalized version of \eqref{eq:21} in $d=2$, c.f. \eqref{eq:Phi2d} below, and show that a suitable mean-field singular SPDE governs the limiting dynamics.  Before giving the statement, let us first comment on the notion of solution used.  Recall that the well-posedness of \eqref{eq:21} in the case $N=1$ and $d=2$ (i.e. the dynamical $\Phi^4_2$ model) is now well developed: two classical works being \cite{MR1113223} where martingale solutions are constructed and \cite{DD03} where strong solutions are addressed, as well as the more recent approach to global well-posedness in \cite{MW17}.  These results can be generalized to the vector case (with fixed $N>1$) without much extra effort. As in  \cite{DD03} and \cite{MW17} the solutions are defined by the decomposition $\Phi_i=Z_i+Y_i$, where
 \begin{align}\label{e:1:li}
\LL Z_i&=\xi_i,
\\
\LL Y_i& =-\frac{1}{N}\sum_{j=1}^N(Y_j^2Y_i+Y_j^2Z_i+2Y_jY_iZ_j+2Y_j\Wick{Z_iZ_j}
+ \Wick{Z_j^2}Y_i+\Wick{Z_i Z_j^2 })\label{e:1:Y}
\end{align}
and $\Wick{Z_iZ_j}$, $\Wick{Z_i Z_j^2 }$ are Wick renormalized products (see Section \ref{sec:nonlinear}).  For the uninitiated reader, note that \eqref{e:1:Y} arises by inserting the decomposition of $\Phi_{i}$ into \eqref{eq:21} and re-interpreting the ill-defined products $Z_{i}Z_{j}$ and $Z_{i}Z_{j}^{2}$ that appear.

The mean-field SPDE formally associated to \eqref{eq:21} takes the form
\begin{equation}
\LL \Psi_{i}=-\E[\Psi_{i}^{2}]\Psi_{i}+\xi_{i}, \quad \Psi_i(0)=\psi_i\label{eq:formPsi}.
\end{equation}
On the formal level this equation arises naturally: assuming the initial conditions $\{\phi_{i}\}_{i=1}^{N}$ are exchangeable, 
\footnote{This means that the sequence of random variables
$(\phi_1,\cdots,\phi_N)$ has the same joint probability distribution
as $(\phi_{\pi(1)},\cdots,\phi_{\pi(N)})$
for any permutation $\pi$.}
the components $\{\Phi_{i}\}_{i=1}^{N}$ will have identical laws, so that replacing the empirical average $\frac{1}{N}\sum_{j=1}^{N}\Phi_{j}^{2}$ in \eqref{eq:21} by its mean and re-labelling $\Phi$ as $\Psi$ leads us to \eqref{eq:formPsi}.  In two space dimensions, \eqref{eq:formPsi} is a singular SPDE where the ill-defined non-linearity depends on the law of the solution and similar to \eqref{eq:21} it also requires a renormalization.  Postponing for the moment a more complete discussion of this point, we now state our first main result.
 % One of our goals in the study of the present problem is to make some progress in establishing global bounds for singular SPDE's in which  strong damping is not at hand;
%\scott{I wonder if we can argue that it bears some resemblance to the general open problem of getting global bounds on singular SPDE's without using the strong damping; what we do is somehow of intermediary difficulty; not 100 percent sure how to sell this, but its not completely false.  In fact, one might be able to make this argument for the $\Psi$ equation. }

%For most of the paper, we will focus on two spatial dimensions,
%where  the system \eqref{e:Phi_i} needs to be Wick renormalized.
%We will postpone to Section~\ref{sec:nonlinear} for the precise definition of Wick renormalization
%in the vector setting.

%The first part of this paper (Section~\ref{sec:nonlinear}--Section~\ref{sec:dif}) studies the large $N$ limit of the dynamics in non-equilibrium setting.
%Let  $\Phi_i=Z_i+Y_i$ be the solution to the Wick renormalized equation  \eqref{eq:21}.
%Our main result states that in the large $N$ limit, we obtain a linear SPDE but with a {\it nontrivial} drift
%which is defined {\it self-consistently}, namely, the ``drift term'' depends on the law of the solution itself.
%\zhu{explain exchangeable}\hao{Added a footnote and updated letter}
\begin{theorem}[Large $N$ limit of the dynamics for $d=2$]\label{th:1}
Let $\{(\phi_{i}^{N},\psi_{i}) \}_{i=1}^{N}$ be  random initial datum with components in $\bC^{-\kappa}$ for some small $\kappa>0$ and all moments finite, where $\bC^{-\kappa}$ denotes the Besov space introduced in Section \ref{s:not}.  Assume that for each $i \in \N$,  $\phi_{i}^{N}$ converges to $\psi_{i}$ in $L^{p}(\Omega; \bC^{-\kappa})$ for all $p>1$, $\frac1N\sum_{i=1}^N\|\phi_i^N-\psi_i\|_{\bC^{-\kappa}}^p\to^{\mathbf{P}}0$ and $(\psi_{i})_{i}$ are iid. Here $\to^{\mathbf{P}}0$ means the convergence in probability. 

Then for each component $i$ and all $T>0$, the solution $\Phi_{i}^{N}$ defined by \eqref{e:1:li}-\eqref{e:1:Y} with initial datum $\phi_{i}^{N}$ converges in probability to $\Psi_i$ in $C([0,T], \bC^{-1}(\mathbb T^2))$ as $N\to \infty$, where $\Psi_i$ is the unique solution to the mean-field SPDE formally described by
\begin{equation}\label{eq:Psi2}
\LL \Psi_i= -\mathbf{E}[\Psi_i^2-Z_i^2] \Psi_i+\xi_i,\quad \Psi_i(0)=\psi_i,
\end{equation}
and $Z_i$ is the stationary solution to \eqref{e:1:li}.
Furthermore, under the additional hypothesis that $(\phi_{i}^{N},\psi_{i})_{i=1}^{N}$ are exchangeable, for each $t>0$ it holds that
\begin{equation}
\lim_{N \to \infty}\E\|\Phi_i^N(t)-\Psi_i(t)\|_{L^2(\mathbb{T}^2)}^2=0.
\end{equation}
 %\hao{The above two convergence statements are not in $L^2 (\mathbb T^2)$, right?}
\end{theorem}

In Section \ref{sec:dif} we actually prove this convergence result under more general conditions for initial data (see Assumption \ref{a:main}).
Along the way to Theorem \ref{th:1}, we prove new uniform in $N$ bounds through suitable energy estimates on the remainder equation \eqref{e:1:Y}.  We are inspired in part by the approach in \cite{MW17}, but  subtleties arise as we track carefully the dependence of the bounds on $N$.  Indeed, the natural  approach  (e.g.  \cite{MW17} for dynamical $\Phi^4_2$ model) to obtain global in time bounds for fixed $N$ is to exploit the damping effect from $Y_j^2Y_i$.
However, the extra factor $1/N$ before the nonlinear terms makes this effect weaker as $N$ becomes large.  In fact, the moral is that we cannot exploit the strong damping effect at the level of a fixed component $Y_{i}$, rather we're forced to consider aggregate quantities, and ultimately we focus on the empirical average of the $L^2$-norm (squared) instead of the $L^p$-norm, $p>2$, c.f. Lemma \ref{Y:L2} and Remark \ref{re:1}.  This is natural on one hand due to the coupling of the components, but also for the slightly more subtle point that we ought to respect the structure of the mean-field SPDE \eqref{eq:Psi2}, for which the damping effect seems to hold only in the mean square sense, not at the path-by-path level.

In this direction, we now discuss a bit more the solution theory for the mean-field SPDE \eqref{eq:Psi2}.   While the notion of solution we use is again via the Da-Prato/Debussche trick, the well-posedness theory for \eqref{eq:formPsi} requires more care than for $\Phi^4_2$  since we cannot proceed by pathwise arguments alone.  In fact, similar to \eqref{e:1:li}-\eqref{e:1:Y}, we understand \eqref{eq:Psi2} via the decomposition $\Psi_i=Z_i+X_i$ with $X_i$ satisfying
 \begin{align}
\LL X_i& =-(\E[X_j^2]X_i+\E [X_j^2]Z_i+2\E[X_jZ_j]X_i+2\E [X_jZ_j]Z_i
).\label{e:1:X}
\end{align}
Here we actually introduce an independent copy $(X_j,Z_j)$ of $(X_i,Z_i)$, which turns out to be useful for both the local and global well-posedness of \eqref{eq:Psi2}.  Indeed, one point is that the term $\E [X_jZ_j]Z_i$ in \eqref{e:1:X} cannot be understood in a classical sense; however  we can view it as a conditional expectation $\E [X_jZ_jZ_i|Z_i]$ and use properties of the Wick product $Z_{i}Z_{j}$ to give a meaning to this, c.f. Lemma \ref{lem:Z1}.  Furthermore, to obtain global bounds, using this independent copy allows us to approach the a priori estimates for \eqref{e:1:X} much like the uniform in $N$ bounds for \eqref{e:1:Y}.   Indeed, after taking expectation, $\E[X_j^2]X_i$ in \eqref{e:1:X} also plays the role of the damping mechanism, which helps us to obtain uniform bounds on the mean-squared $L^2$-norm of $X_i$ c.f. Lemma \ref{lem:l2}.

%Note that it is a priori not even clear whether \eqref{eq:Psi2} would be a meaningful equation (even it looks linear), since it is not immediately clear that the deterministic $\mu$  is well-defined, and it is expected to be {\it distributional} that is multiplied with the distributional solution $\Psi_i$. The fact that  \eqref{eq:Psi2} does make sense as well as the construction of its long time solution are discussed in Section~\ref{sec:limit}.
%Theorem~\ref{th:1} will then be proved in Section \ref{sec:dif} and it will follow from Theorem~\ref{th:conv-v}.

Theorem~\ref{th:1} can be viewed as a {\it mean field} limit result in the context of singular SPDE systems. Our proof is  indeed inspired by certain mean field limit techniques, and we combine them with a priori estimates that are specific to our model - see the discussion above Theorem~\ref{th:conv-v} % Remark~\ref{rem:mf-phil}
 for a more detailed discussion on this strategy.
We will provide more background discussion on mean field limits below in Section~\ref{sec:mean}.   By a classical coupling argument, this result also yields a propagation of chaos type statement: if the initial condition is asymptotically chaotic (i.e. independent components as $N\to \infty$) then, although the $\Phi$-system is interacting, as  $N\to \infty$ the limiting system becomes decoupled (\cite[Def.3,Def.5]{MR3317577}).

The second part of this paper (Section~\ref{sec:inv}) is concerned with equilibrium theories, namely stationary solutions, invariant measures, and large $N$ convergence.  For $N=1$, the long-time behavior of the solutions was investigated in \cite{RZZ17CMP} and \cite{TW18}.  In the vector valued setting, by lattice approximation (see \cite{HM18a, ZZ18, GH18a}), strong Feller property in \cite{HM18} and irreducibility in \cite{HS19} it can be shown that $\nu^N$ is the unique invariant measure to \eqref{eq:21} and the law of $\Phi_i(t)$ converges to $\nu^N$ as $t\to \infty$.
Our goal then is to study the large $N$ limit of the $O(N)$ linear sigma model $\nu^N$.  Our second main result yields the convergence of the unique invariant measure $\nu^N$ of \eqref{eq:21} to the invariant measure of \eqref{eq:Psi2}, provided the mass is sufficiently large.

To state the result, consider the projection onto the $i^{th}$ component,
\begin{equ}[e:Pr-Pi]
\Pi_i:\mathcal{S}'(\mathbb{T}^d)^N\rightarrow \mathcal{S}'(\mathbb{T}^d),
\qquad \Pi_i(\Phi)\eqdef \Phi_i.
\end{equ}
 Noting that $\nu^N$ is a measure on $\mathcal{S}'(\mathbb{T}^d)^N$, we define the marginal law $\nu^{N,i}\eqdef \nu^N\circ \Pi_i^{-1}$. Furthermore, consider
\begin{equ}[e:Pr-Pik]
 \Pi^{(k)}:\mathcal{S}'(\mathbb{T}^d)^N\rightarrow \mathcal{S}'(\mathbb{T}^d)^k,
 \qquad
 \Pi^{(k)}(\Phi)=(\Phi_i)_{1\leq i\leq k}
 \end{equ}
  and define the marginal law of the first $k$ components by $\nu^{N}_k \eqdef \nu^N\circ (\Pi^{(k)})^{-1}$.

\begin{theorem}[Large N limit of the invariant measures]\label{th:1.2}
There exists $m_0>0$ such that the following results hold:
\begin{itemize}[leftmargin=.2in]
  \item The Gaussian free field $\nu\eqdef \mathcal{N}(0,\frac12(m-\Delta)^{-1})$ is an invariant measure for \eqref{eq:Psi2}.
  \item  The sequence of probability measures $(\nu^{N,i})_{N \geq 1}$ are tight on $\bC^{-\kappa}$ for $\kappa>0$.
  \item For $m\geq m_0$, the Gaussian free field $\mathcal{N}(0,\frac12(m-\Delta)^{-1})$ is the unique invariant measure to equation \eqref{eq:Psi2}.
  \item
For $m\geq m_0$,   $\nu^{N,i}$ converges to $\nu$ and $\nu^N_k$ converges to $\nu\times\cdots\times \nu$, as $N\to \infty$. Furthermore, $\mathbb{W}_2(\nu^{N,i},\nu)\lesssim N^{-\frac{1}{2}}$.
\end{itemize}
\end{theorem}

These statements will follow from Theorem~\ref{th:con}, Theorem~\ref{th:u} and Theorem~\ref{th:m1}.
Here $\mathbb{W}_2(\nu_1,\nu_2)$ is  the $\bC^{-\kappa}$-Wasserstein distance
defined in \eqref{def:wa} before  Theorem~\ref{th:m1}.
%We note that the invariant measure to \eqref{eq:Psi2} might be different by choosing the different renormalization constant (see Remark \ref{rem:1} and Remark \ref{rem:2}).
%\hao{I would suggest maybe delete the last sentence: it's a subtlety and might confuse the reader when reading an introduction. If we keep it here we can say "might have different mass" instead of "might be different"}
The Gaussian free field limit is expected (at a heuristic level) by physicists e.g. \cite{wilson1973quantum} and also in mathematical physics \cite{MR578040}.
\footnote{In \cite{wilson1973quantum} it was written that ``If one now looks at vacuum expectation values of individual $\Phi$ fields, all diagrams vanish like $1/N$ (at least), except for the free-field terms.''
In the Introduction of \cite{MR578040}  it was mentioned that
 ``the $1/N$ expansion predicts that the theory is close to Gaussian as $N$ becomes large enough'',
 but this reference did not intend to prove this statement (see Section~\ref{sec:LargeN} below).
}
Our result Theorem~\ref{th:1.2} provides a precise justification provided $m\ge m_0$,
 with the  convergence rate $N^{-\frac{1}{2}}$ (which is expected to be optimal, see for instance \cite[Remark~4]{MR3858403}) in terms of Wasserstein distance. The large $m$ assumption could also be formulated as a small nonlinearity assumption - see Remark~\ref{rem:lambda}.

%$$\mathbb{W}_2(\nu_1,\nu_2):=\inf_{\pi\in\mathscr{C}(\nu_1,\nu_2)}\left(\int\|\phi-\psi\|_{\bC^{-\kappa}}^2\pi(\dif \phi,\dif \psi)\right)^{1/2},$$
%where $\mathscr{C}(\nu_1,\nu_2)$ denotes all the coupling of $\nu_1, \nu_2$ satisfying $\int\|\phi\|_{\bC^{-\kappa}}^2\nu_i(\dif \phi)<\infty$ for $i=1,2$.

Note that  the study of ergodicity properties of the dynamic  \eqref{eq:Psi2} is nontrivial.
% (see \cite{Fengyu} and the reference therein.\hao{maybe cite this somewhere later. Not very clear why we cite it here}\zhu{you could put it later})
In fact,  the dynamic for $\Psi$ depends on the law of $\Psi$ itself, so the associated semigroup
 is generally nonlinear (see Section~\ref{Uniqueness  of invariant measure}).
  As a result, the general ergodic theory for Markov process (see e.g. \cite{DZ96}, \cite{HMS11}, \cite{HM18}) could not be directly applied here. Instead, we prove the solutions to \eqref{eq:Psi2} converge to the limit  directly as time goes to infinity, which requires $m\geq m_0$. %We also use the so-called coming down from infinity property for \eqref{eq:Psi2} (see \eqref{bd2:uniX}), which has been deduced for the dynamical $\Phi^4_d$ model in \cite{MW18, TW18}.

We now comment on our approach to the fourth part of Theorem \ref{th:1.2}.  It would be natural to try and use Theorem \ref{th:1} together with the tightness result from the second part of Theorem \ref{th:1.2} to derive the convergence of $\nu^{N,i}$ to $\nu$ directly (see e.g. \cite{HM18a}). However, it is not clear to the authors how to implement this strategy in the present setting.  Indeed, to apply Theorem \ref{th:1}, it is important that each component $\psi_i$ of the initial data is independent of each other. However, we are not able to deduce that an arbitrary limit point $\nu^*$ has this property. If we use $P_t^*\nu^*$ to denote the marginal distribution of the solution to \eqref{eq:Psi2} starting from the initial distribution $\nu^*$, we cannot write $P_t^*\nu^*$ as $\int (P_t^*\delta_\psi)\nu^*(\dif \psi)$ due to the lack of linearity, which makes it difficult to overcome the assumption of independence.  Alternatively, we follow the idea in \cite{GH18a} and construct a jointly stationary process $(\Phi, \Psi)$ whose components satisfy \eqref{eq:21} and \eqref{eq:Psi2}, respectively. In this case $\Psi=Z$, since the Gaussian free field gives the unique invariant measure to \eqref{eq:Psi2}. We then establish the convergence of $\nu^{N,i}$ to $\nu$ by deriving suitable uniform estimates on the stationary process.

Our next result is concerned with observables, in the stationary setting.
In QFT models with continuous symmetries, physically interesting quantities involve more than just a component of the field itself but also
quantities composed by the fields which preserve the symmetries, called invariant observables. These acquire the same interest in SPDE (a natural example being the gauge invariant observables, e.g. \cite[Section 2.4]{Shen2018Abelian}). In the present setting of  \eqref{eq:21},
a natural quantity  that is invariant under $O(N)$-rotation is the ``length'' of $\Phi$;
another being the quartic interaction in \eqref{e:Phi_i-measure}.
We thus consider the following two $O(N)$ invariant observables:
for $\Phi\backsimeq\nu^N$
\begin{equ}[e:twoObs]
\frac{1}{\sqrt{N}} \sum_{i=1}^N\Wick{\Phi^2_i},\qquad  \frac{1}{N} \Wick{\Big(\sum_{i=1}^N\Phi_i^2\Big)^2}.
\end{equ}
Here %the Wick product is with respect to the Gaussian free field measure;
the precise definition is given in Section \ref{sec:label}.
One could consider more general renormalized polynomials of $\sum_{i} \Phi_i^2$
but we choose to focus on the above two in this article.
We establish the large $N$ tightness of these observables as random fields in suitable Besov spaces by using  iteration to derive  improved uniform estimates in the stationary case.

Note that the physics literature usually considers integrated quantities, i.e. partition function of correlations of these observables. Our SPDE approach allows us to study these observables as random fields with precise regularity as $N\to \infty$ which is new.  %The above result can also be shown in dimension one, and in fact we can obtain more refined information about the limit in this case.  This is the focus of the final part of our paper.  In $d=1$ the equations are less singular and uniform estimates are simpler;
%we provide in Section~\ref{sec:d1} and Appendix \ref{sec:A.3} the arguments in $d=1$ that
%simplify in comparison to $d=2$ (while skipping details that follow essentially the same way as $d=2$).
%A reader might also use Section~\ref{sec:d1} with the proof in Appendix \ref{sec:A.3} to grasp the main ideas of our analysis before embarking on the more challenging case of $d=2$.

Moreover,  we investigate the {\it nontrivial statistics} of the large $N$ limit of the $O(N)$ invariant  observables. %To this end, we suitably recenter
%the system  \eqref{eq:21}, that is, the  Wick ordered system
%defined by subtracting the finite  variance of the stationary solution $Z_i$ (see \eqref{eq:1m} for the precise definition), and also Wick order the observables (see \eqref{e:-Cw}).
We show that although for large enough $m$
the invariant measure of $\Phi_i$ converges
as $N\to \infty$ to the invariant measure of $Z_i$ i.e. Gaussian free field,
the limits of the observables \eqref{e:twoObs} have different laws than those if $\Phi_i$ in \eqref{e:twoObs}  were replaced by
$Z_i$:
\begin{equ}[e:two-OZ]
	\frac{1}{\sqrt{N}} \sum_{i=1}^N\Wick{Z^2_i},
	\qquad  \frac{1}{N} \Wick{\Big(\sum_{i=1}^NZ_i^2\Big)^2}.
\end{equ}

\begin{theorem}\label{th:1.3}
Suppose that $\Phi\backsimeq\nu^N$.  For $m$ large enough, the following result holds for any $\kappa>0$:
\begin{itemize}[leftmargin=.4in]
  \item  $\frac{1}{\sqrt{N}} \sum_{i=1}^N\Wick{\Phi^2_i}$ is tight in $B^{-2\kappa}_{2,2}$.
  \item  $\frac{1}{N} \Wick{(\sum_{i=1}^N\Phi_i^2)^2}$ is tight in $B^{-3\kappa}_{1,1}$.
  \item  
%  The limiting laws of the observables \eqref{e:twoObs} are different from those of \eqref{e:two-OZ};
%  in fact, we have an explicit formula for 
 The Fourier transform of the two point correlation function of $\frac{1}{\sqrt{N}} \sum_{i=1}^N\Wick{\Phi^2_i}$
  in the limit as $N\to \infty$ is
  given by the explicit formula $2\widehat{C^2} /(1+2\widehat{C^2})$, where $C=\frac12(m-\Delta)^{-1}$ and $\widehat{C}$ is the Fourier transform;
  moreover
  %$ \E \frac{1}{N} \Wick{ (\PPhi^2)^2} (x)$
  $\E \frac{1}{N} \Wick{\big(\sum_{i=1}^N\Phi_i^2\big)^2}$
  converges as $N\to \infty$ to the explicit formula
  given by
  $-4 \sum_{k\in \mathbb Z^2} \widehat{C^2}(k)^2 /(1+2\widehat{C^2}(k))$.
 (In particular the limiting laws of the observables \eqref{e:twoObs} are different from those of \eqref{e:two-OZ}).
%  \item For $d=1$, $$
%\lim_{N\to \infty}
%\frac{1}{\sqrt{N}} \sum_{i=1}^N\Wick{\Phi^2_i} \neq
%\lim_{N\to \infty}
%\frac{1}{\sqrt{N}} \sum_{i=1}^N\Wick{Z_i^2}
%\quad
%\lim_{N\to \infty}
%\frac{1}{N} \Wick{(\sum_{i=1}^N\Phi^2_i)^2} \neq
%\lim_{N\to \infty}
%\frac{1}{N} \Wick{\sum_{i=1}^N{Z}_i^2)^2}
%$$
\end{itemize}
\end{theorem}

These results are proved in Theorem~\ref{th:o1} and Theorem \ref{theo:nontrivial2}.
The last statement on correlation formulas of the observables
are known -- first heuristically by  physicists who expressed these formulas in terms of the
sum of ``bubble'' diagrams, and then derived in \cite[Eq.~(15)]{MR578040} using constructive field theory techniques such as ``chessboard estimates''.
Our new proofs of these correlation formulas
using PDE methods  are quite simple and straightforward once all the a priori estimates are available.
We expect that these methods can be applied to study more $O(N)$ invariant observables
and higher order correlations; we will pursue these in future work.

Let's also mention the three dimensional construction of  local solutions \cite{Hairer14,GIP15,MR3846835}, global solutions \cite{MW18,GH18,AK17,moinat2020space}, as well as a priori bounds in fractional dimension $d<4$ by \cite{CMW19}, though we  focus on $d=2$ in this paper.  It would also be interesting to see if  our methodology could be used to  study limits of other singular SPDE systems as dimensionality of the target space tends to infinity, such as coupled dynamical $\Phi^4_3$, \footnote{In fact, we have obtained some partial results for coupled dynamical $\Phi^4_3$, such as convergence of invariant measures to the Gaussian free field.} coupled KPZ systems \cite{MR3653951}, random loops in $N$ dimensional manifolds \cite{hairer2019geo,hairer2016motion, RWZZ17, CWZZ18} and
Yang-Mills model \cite{CCHS20} with $N$ dimensional Lie groups (or abelian case \cite{Shen2018Abelian} with Higgs field generalized to value in $\mathbb C^N$). These are of course left to further work.

\subsection{Large N problem in QFT: background and motivation}
\label{sec:LargeN}

Large $N$ methods (or ``$1/N$ expansions'') in theoretical physics are ubiquitous and
are generally applied to models where dimensionality of the target space is large.
It was first used in
\cite{stanley1968spherical} for spin models,
and then developed in quantum field theories (QFT) which was
pioneered by \cite{wilson1973quantum} ($\Phi^4$ type and Fermionic models), \cite{gross1974dynamical} (Fermionic models),
\cite{tHooft1974planar} (Yang-Mills model), and the idea was soon popularized and extended to many other systems - see  \cite{LargeN1993} for an edited comprehensive collection of articles on large
$N$ as applied to a wide spectrum of problems in quantum field theory and statistical
mechanics; see also the review articles \cite{witten1980}, \cite[Chapter~8]{coleman1988} and \cite{moshe2003} for summaries of the progress. Loosely speaking, in terms of our model \eqref{e:Phi_i-measure}, the ordinary QFT perturbative calculation of for instance a two-point correlation of $\Phi_i$ is given by sum of Feynman graphs with two external legs and degree-$4$ internal vertices, each vertex carrying two distinct summation variables
and a factor $1/N$ that represents the interaction $\frac1N \sum_{i,j}\Phi_i^2 \Phi_j^2$, such as
(a) (b) below

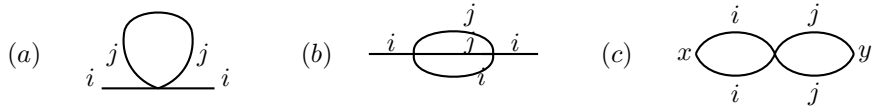
\begin{center}
$(a)\quad$
\begin{tikzpicture}[scale=1.5,baseline=10]
\draw[thick] (0,0) to  (1,0);
\draw[thick] (0.5,0)  [bend left=40] to (0.2,0.5);
\draw[thick]  (0.2,0.5)  [bend left=80] to (0.8,0.5);
\draw[thick]  (0.8,0.5) [bend left=40] to (0.5,0);
\node at (-0.1,0.1) {$i$}; \node at (1.1,0.1) {$i$};
\node at (0.1,0.3) {$j$}; \node at (0.9,0.3) {$j$};
\end{tikzpicture}
$\qquad (b)\quad$
\begin{tikzpicture}[scale=1.5,baseline=-3]
\draw[thick] (0,0) to  (0.4,0);
\draw[thick] (0.4,0)  [bend left=100] to (1.1,0);
\draw[thick] (0.4,0) to (1.1,0);
\draw[thick] (0.4,0)  [bend right=100] to (1.1,0);
\draw[thick] (1,0) to  (1.5,0);
\node at (0.2,0.1) {$i$}; \node at (1.3,0.1) {$i$};
\node at (0.9,0.1) {$j$}; \node at (0.9,0.35) {$j$}; \node at (1,-0.2) {$i$};
\end{tikzpicture}
$\qquad (c)\quad$
%\begin{tikzpicture}[scale=1.5,baseline=-3]
%\draw[thick] (0.2,0.2) to  (0.4,0);\draw[thick] (0.2,-0.2) to  (0.4,0);
%\draw[thick] (0.4,0)  [bend left=70] to (1.1,0);
%\draw[thick] (0.4,0)  [bend right=70] to (1.1,0);
%\draw[thick] (1.1,0) to  (1.3,0.2);\draw[thick] (1.1,0) to  (1.3,-0.2);
%\node at (0.1,0.1) {$i$}; \node at (0.1,-0.1) {$i$};
%\node at (1.3,0.1) {$j$}; \node at (1.4,-0.1) {$j$};
%\node at (0.9,0.35) {$k$}; \node at (1,-0.2) {$k$};
%\end{tikzpicture}
\begin{tikzpicture}[scale=1.5,baseline=-3]
\draw[thick] (0,0)  [bend left=70] to (0.7,0);
\draw[thick] (0,0)  [bend right=70] to (0.7,0);
\node at (0.35,0.35) {$i$}; \node at (0.35,-0.35) {$i$};
\draw[thick] (0.7,0)  [bend left=70] to (1.4,0);
\draw[thick] (0.7,0)  [bend right=70] to (1.4,0);
\node at (1.05,0.35) {$j$}; \node at (1.05,-0.35) {$j$};
\node at (-0.1,0) {$x$};\node at (1.5,0) {$y$};
\end{tikzpicture}
\end{center}
Heuristically,
 graph (a) is of order $\frac1N \sum_j \approx O(1)$ and
  graph (b) is of order $\frac{1}{N^2} \sum_j \approx O(\frac1N)$.
 The philosophy of  \cite{wilson1973quantum} is that graphs with ``self-loops'' such as (a) get {\it cancelled} by Wick renormalization,
 and all other graphs with internal vertices  including (b) are at least of order $O(1/N)$ and thus vanish,
so the theory would be asymptotically Gaussian free field - which is what we prove in Theorem~\ref{th:1.2}.
On the other hand  for  observables
 such as $ \frac{1}{\sqrt{N}} \sum_{i=1}^N\Wick{\Phi^2_i} $,
 two-point correlation at $x,y$
 may have $O(1) $ contributions as shown in graph (c), \footnote{but there are infinitely many $O(1)$ graphs.} which is the heuristic behind the existence of a nontrivial correlation structure for such  observables as in Theorem~\ref{th:1.3}.
The ``$1/N$ expansion'' is a re-organization of the series in the parameter $1/N$,
 with each term typically being a (formal) sum of infinitely many orders of the ordinary perturbation theory. Besides directly examining  the perturbation theory, alternative (and more systematic) methodologies  of analyzing such expansion were discovered in physics, for instance a method via ``dual'' field  \cite{coleman1974}, \cite[Section~2]{moshe2003}, via Schwinger-Dyson equations \cite{symanzik1977}, or via stochastic quantization (with references below).

Rigorous  study of large $N$ in mathematical physics was initiated  by Kupiainen \cite{MR574175,MR578040,MR582622}.
The literature most related to the present article is \cite{MR578040}, which studied the QFT in continuum in $d=2$ given by \eqref{e:Phi_i-measure}, and proved that the $1/N $ expansion of the  pressure (i.e. vacuum energy, or log of partition per area)  is asymptotic,
 and each order in this expansion
 % $p=\sum_{j=0}^{m-1} p_j N^{-j} + R_m$ where $R_m$ has explicit upper bound of order $N^{-m}$ and $p_j$
  can be described by sums of infinitely many Feynman diagrams of certain types.
 Borel summability of $1/N $ expansion of Schwinger functions for this model was discussed in \cite{MR661137}.

In \cite{MR574175} Kupiainen also proved that  on the lattice with fixed lattice spacing, the large $N$ expansion of correlation functions of the N-component
nonlinear sigma model (which simplifies to ``spherical model'' as $N\to \infty$) is asymptotic above the spherical model criticality; % and mass gap was established for these temperatures and $N$ sufficiently large;
asymptoticity was later extended to Borel summability by \cite{MR678004}.
Large $N$ limit and expansion for Yang-Mills model
has also been rigorously studied:
see \cite{Levy2011} (also \cite{MR2864481}) for convergence of Wilson loop observables to master field in the continuum plane,
and \cite{MR3919447} (resp. \cite{chatterjee2016}) for
computation of correlations of the Wilson loops % in strongly coupled $SO(N)$ lattice Yang-Mills
 in the large $N$ limit (resp. $1/N$ expansion) which relates to  string theory.

 Large $N$ problems in the stochastic quantization formalism
 have also been discussed in the physics literature, for instance  \cite{Alfaro1983,AlfaroSakita},
 \cite[Section~8]{damgaard1987}.
 \cite[Section~5.1]{moshe2003}  is close to our setting; it makes an ``ansatz'' that $\frac{1}{N}\sum_{j=1}^N\Phi_j^2$ in \eqref{eq:21} would self-average in the large $N$ limit to a constant; our present paper justifies this ansatz and in the non-equilibrium setting generalizes it.

In summary, the study of large $N$ problems in QFT is motivated by the following properties (among others).
The first property is simplification or solvability as $N\to\infty$. This is the motivation ever since the earliest literature \cite{stanley1968spherical}  as aforementioned: the model studied therein becomes a simplified, solvable model as $N\to\infty$ known as the Berlin-Kac spherical model. In our setting this simplification or solvability heuristics are reflected by the Gaussian free field asymptotic as well as the rigorous derivation of exact formula (which would not be possible for finite $N$) for certain correlation of observables in Theorem~\ref{th:1.2} and Theorem~\ref{th:1.3}.
Another property is that when $N$ is large, $1/N$ serves as a natural perturbation parameter in QFT models, as already discussed above. Of course this went much farther than just simplifying things later when applied to more sophisticated models like gauge theory, for which $1/N$ expansions led to the discovery of so called gauge-string duality as mentioned above.

\subsection{Mean field limits.}
\label{sec:mean}

As mentioned above the proof of our main theorems borrows some ingredients from
mean field limit theory (MFT).
To the best of our knowledge, the study of mean field problems originated from McKean \cite{MR0233437}.
Typically, a mean field problem is concerned with a system of $N$ particles interacting with each other,
which is often modeled by a system of stochastic {\it ordinary} differential equations, for instance, driven by independent Brownian motions. A prototype of such systems has the form $\dif X_i = \frac{1}{N} \sum_j f(X_i,X_j) \dif t + \dif B_i$, see for instance the classical reference by Sznitman \cite[Sec~I(1)]{MR1108185},
 and in the $N\to \infty$ limit one could obtain decoupled SDEs each interacting with the law of itself:
 $\dif Y_i = \int f(Y_i,y) \mu(\dif y) \dif t + \dif B_i$ where $\mu(\dif y) $ is the law of $Y_i$.
 So just as in QFT the motivation of MFT is also a simplification of an  $N$-body system to  a one-body equation which interacts with itself, i.e. the system is factorized.

In simple situations the interaction $f$ is assumed to be ``nice'', for instance globally Lipschitz (\cite{MR0233437}); much of the literature aims to prove such limits under more general assumptions on the interaction, see  \cite{MR1108185} for a survey.
%one instance  that is to some extent relevant to our problem is \cite{MR2860672}  where the interaction is assumed to be non-Lipschitz.
\footnote{In the context of SDE systems, one also considers the empirical measures of the particle configurations, and aims to show their convergence as  $N\to \infty$ to the McKean-Vlasov PDEs, which are typically deterministic. %(however, in  situations such as the particles are driven by space-time random fields the McKean-Vlasov PDEs could becomes SPDEs, see for instance Kotelenez-Kurtz, Dawson-Vaillancourt\cite{}).
Note that in this paper we do not consider the ``analogue'' of McKean-Vlasov PDE (which would be infinite dimensional) in the context of our model.}
Our Theorem~\ref{th:1} can be viewed as a result of this flavor, in an SPDE setting,
and in fact the starting point of our proof is indeed close in spirit to  \cite[Sec~I(1)]{MR1108185}
where one subtracts $X_i$ from $Y_i$ to cancel the noise and then bound a suitable norm of the difference.

We note that mean field limits are studied under much broader frameworks or scopes of applications,
such as mean field limit in the context of rough paths  (e.g. \cite{MR3299600,BCD18,CDFM18}),
mean field games (e.g. survey \cite{MR2295621}), quantum dynamics (e.g. \cite{MR2680421} and references therein).
We do not intend to have a comprehensive list, but rather refer to survey articles \cite{MR3468297,MR3317577} and the book \cite[Chapter~8]{spohn2012large}
besides \cite{MR1108185}.

The study of mean field limit for SPDE systems also has precursors, see for instance the book \cite[Chapter~9]{MR1465436} or \cite{MR3160067}. However these results make strong assumptions on the interactions of the SPDE systems such as linear growth and globally Lipschitz drift, and certainly do not cover the singular regime where renormalization is required as in our case.

\subsection{Structure of the paper}
This paper is organized as follows. 
Sections \ref{sec:nonlinear}-\ref{sec:dif} are devoted to proof of Theorem \ref{th:1}. First in Section \ref{sec:re} we recall the definition of the renormalization for $Z_i$, which satisfies the linear equation \eqref{e:1:li}. Then a uniform in $N$ estimate for the average of the $L^2$-norm of $Y_i$, the solutions to equation  \eqref{e:1:Y}, is derived in Section \ref{sec:Y}. Local well-posedness to equation \eqref{eq:Psi2} is proved in Section \ref{sec:4.1}. Global well-posedness to equation \eqref{eq:Psi2} is proved in Section \ref{sec:4.2} by combining a uniform $L^p$-estimate with Schauder theory. The difference estimate for $\Phi_i-\Psi_i$ is given in Section \ref{sec:dif}, which gives the  proof of Theorem \ref{th:1}.

Section \ref{sec:inv} is concerned with the proof of Theorem \ref{th:1.2}. % and Theorem \ref{th:1.3}. 
In Section \ref{Uniqueness of invariant measure}, uniqueness of invariant measures to \eqref{eq:Psi2} for large $m$ is proved. The convergence of invariant measures from $\nu^{N,i}$ to the Gaussian free field $\nu$ is shown in Section \ref{sec:con} by comparing the stationary solutions $(\Phi_i, Z_i)$.

 Section \ref{sec:non} mainly concentrates on the observables and  the nontriviality of the statistics of the observables.
 Section \ref{sec:label} is devoted to the study of the observables and the proof of first two part of Theorem \ref{th:1.3}.
 We derive an $L^p$-estimate of $Y_i$ in Section \ref{se:lp} and prove the third part of Theorem \ref{th:1.3} in Section \ref{se:non}. 

 Finally %in Section \ref{sec:d1}  all the results have been stated in the one dimensional case. In Section \ref{sec:7.1} the convergence of the dynamics is stated.  Section \ref{sec:7.2} is the corresponding part of the uniqueness of the invariant measure and the convergence of the invariant measures. Section \ref{sec:7.3} mainly concentrates on the proof of the nontriviality of the statistics of the observables. 
 in Appendix \ref{s:not}, we collect the  notations and useful lemmas used throughout the paper. In Appendix \ref{sec:A1} we give the proof of global well-posedness of equation \eqref{e:1:Y}. %In Appendix \ref{sec:A.3} we give the proof of the result in Sections \ref{sec:7.1}-\ref{sec:7.2}. 
 In Appendix \ref{sec:D2} the application of Dyson-Schwinger equations has been derived, which is useful in studying the limiting law of the observables. In Appendix \ref{sec:E2} we give the proof of  Step \ref{diffEst5} in the proof of Theorem \ref{th:conv-v}.
%\zhu{We have modified this part. Please double check.}\hao{Looks good.}

\subsection*{Acknowledgments}
 We would like to thank Antti Kupiainen for helpful discussion on large $N$ problems and are grateful to Fengyu Wang and Xicheng Zhang for helpful discussion on distributional dependent SDE. H.S.
gratefully acknowledges financial support from NSF grant DMS-1712684 / DMS-1909525 and DMS-1954091. R.Z. and X.Z. are grateful to
the financial supports of the NSFC (No. 11771037,
11922103) and the financial support by the DFG through the CRC 1283 “Taming uncertainty
and profiting from randomness and low regularity in analysis, stochastics and their applications” and the support by key Lab of Random Complex Structures and Data Science, Chinese Academy of Science.

%%%%%%%%%%%%%%%%%%%%%%%%%%%%%%%%%%%%%%%%%%%%%%%%%%%%%%%%%%%%%%%%  Preliminary
%%%%%%%%%%%%%%%%%%%%%%%%%%%%%%%%%%%%%%%%%%%%%%%%%%%%%%%%%%%%%%%%
%%%%%%%%%%%%%%%%%%%%%%%%%%%%%%%%%%%%%%%%%%%%%%%%%%%%%%%%%%%%%%%%
%%%%%%%%%%%%%%%%%%%%%%%%%%%%%%%%%%%%%%%%%%%%%%%%%%%%%%%%%%%%%%%%
%%%%%%%%%%%%%%%%%%%%%%%%%%%%%%%%%%%%%%%%%%%%%%%%%%%%%%%%%%%%%%%%
%%%%%%%%%%%%%%%%%%%%%%%%%%%%%%%%%%%%%%%%%%%%%%%%%%%%%%%%%%%%%%%%
%%%%%%%%%%%%%%%%%%%%%%%%%%%%%%%%%%%%%%%%%%%%%%%%%%%%%%%%%%%%%%%%
%%%%%%%%%%%%%%%%%%%%%%%%%%%%%%%%%%%%%%%%%%%%%%%%%%%%%%%%%%%%%%%%

%\medskip

%%%%%%%%%%%%%%%%%%%%%%%%%%%%%%%%%%%%%%%%%%%%%%%%%%%%%%%%%%%%%%%%  Uniform in N Bounds
%%%%%%%%%%%%%%%%%%%%%%%%%%%%%%%%%%%%%%%%%%%%%%%%%%%%%%%%%%%%%%%%
%%%%%%%%%%%%%%%%%%%%%%%%%%%%%%%%%%%%%%%%%%%%%%%%%%%%%%%%%%%%%%%%
%%%%%%%%%%%%%%%%%%%%%%%%%%%%%%%%%%%%%%%%%%%%%%%%%%%%%%%%%%%%%%%%
%%%%%%%%%%%%%%%%%%%%%%%%%%%%%%%%%%%%%%%%%%%%%%%%%%%%%%%%%%%%%%%%
%%%%%%%%%%%%%%%%%%%%%%%%%%%%%%%%%%%%%%%%%%%%%%%%%%%%%%%%%%%%%%%%
%%%%%%%%%%%%%%%%%%%%%%%%%%%%%%%%%%%%%%%%%%%%%%%%%%%%%%%%%%%%%%%%
%%%%%%%%%%%%%%%%%%%%%%%%%%%%%%%%%%%%%%%%%%%%%%%%%%%%%%%%%%%%%%%%

\section{Uniform in $N$ bounds on the dynamical linear sigma model}\label{sec:nonlinear}

In this section, we obtain new estimates on the Wick renormalized version of \eqref{eq:21}, given by
\begin{equation}\label{eq:Phi2d}
\LL \Phi_i=-\frac{1}{N}\sum_{j=1}^N \Wick{\Phi_j^2\Phi_i}+\xi_i,\quad \Phi_i(0)=\phi_i.
\end{equation}
The notion of solution to \eqref{eq:Phi2d} is the same as in \cite{DD03}  and \cite{MW17}, where the case $N=1$ is treated.  For a fixed $N$, these well-posedness results are easy to generalize to the present setting, so we only give the statement here and refer the reader to Appendix \ref{sec:A1} for the proof.  Our primary goal in this section is rather to obtain bounds which are stable with respect to the number of components $N$, which we will send to infinity in Section \ref{sec:dif}.

\medskip

As is well known, it's natural to consider initial datum to \eqref{eq:Phi2d} belonging to a negative H\"{o}lder space with exponent just below zero.  We will be slightly more general and consider random initial datum of the form $\phi_i=z_i+y_i$ % defined on the same stochastic basis $(\Omega,\mathcal{F},\mathbf{P})$  and
 satisfying  $\mathbf{E}\|z_i\|_{\bC^{-\kappa}}^p\lesssim1$ for $\kappa>0$ small enough and every $p>1$, and $\mathbf{E}\|y_i\|_{L^2}^2\lesssim1$, where the implicit constants are independent of $i, N$.
 %Furthermore, we impose the qualitative hypothesis that $(\phi_{i})_{i}$ is independent of the space-time white noises $(\xi_{i})_{i}$.
 %
  %and that $z_{i}$ and $z_{j}$ are independent for $i \neq j$.
%\scott{I geuss there is not enough time to do this now, but In the version we submit to a journal, I would prefer to remove all assumptions of the form $\mathbf{E}\|z_i\|_{\bC^{-\kappa}}^p\lesssim1$ and replace them with just $z_{i} \in L^{p}(\Omega; \bC^{-\kappa} )$ since my opinion is that the symbol $\lesssim$ should be reserved for a universal constant, not a constant depending on the initial datum.  Also, my impression is that the independence of $z_{i}$ is not used in this section.}

\medskip

The notion of solution to \eqref{eq:Phi2d} is based on the now classical trick of Da-Prato and Debussche, c.f. \cite{DD03}.  Namely, we say that $\Phi_{i}$ is a solution to \eqref{eq:Phi2d} provided the decomposition $\Phi_i = Z_i + Y_i$ holds, where $Z_{i}$ is a solution to the linear SPDE
\begin{equation}\label{eq:li1}
\LL Z_i=\xi_i,\quad Z_i(0)=z_{i},
\end{equation}
and $Y_{i}$ is a weak solution to the remainder equation
\begin{equation}\label{eq:22}\aligned
\LL Y_i& =-\frac{1}{N}\sum_{j=1}^N(Y_j^2Y_i+Y_j^2Z_i+2Y_jY_iZ_j+2Y_j\Wick{Z_iZ_j}
+ \Wick{Z_j^2}Y_i+\Wick{Z_i Z_j^2 })
,\\ Y_i(0)& =y_i.
\endaligned\end{equation}
The notation $\Wick{Z_iZ_j} $, $ \Wick{Z_j^2}$ and $\Wick{Z_i Z_j^2}$ denotes a renormalized product of Wick type which will be defined in Section \ref{sec:re} below.

\subsection{Renormalization}\label{sec:re}

To define the renormalized products appearing in \eqref{eq:22}, it is convenient to make a further splitting of $Z_{i}$ relative to the corresponding stationary solution to \eqref{eq:li1}, which we will denote by $\tilde{Z}_i$.  For $\tilde{Z}_i$, these products have a canonical definition that we now recall.  Namely, let $\xi_{i,\varepsilon}$ be a space-time mollification of $\xi_i$ defined on $\R\times \mathbb{T}^2$
 and let $\tilde{Z}_{i,\varepsilon}$ be the stationary solution to $\LL \tilde{Z}_{i,\varepsilon}=\xi_{i,\varepsilon}$.
For convenience, we assume that all the noises are mollified with a common bump function. In particular, $\tilde Z_{i,\eps}$ are i.i.d. mean zero Gaussian.  For $k\ge 1$ and $i_1,\dots,i_k \in \{1,\cdots,N\}$ we then write
$\Wick{ \tilde Z_{i_1}\cdots \tilde Z_{i_k}} $
as the limit of
$\Wick{ \tilde Z_{i_1,\eps}\cdots \tilde Z_{i_k,\eps}} $
as $\eps \to 0$. Here $\Wick{ \tilde Z_{i_1,\eps}\cdots \tilde Z_{i_k,\eps}} $ is the canonical
Wick product, which in particular is mean zero. % (see for instance \cite[Def.~4.3]{HaiShen}).
More precisely,
\begin{equ}[e:wick-tilde]
\Wick{\tilde{Z}_i\tilde{Z}_j}=
\begin{cases}
\lim\limits_{\varepsilon\to0}(\tilde{Z}_{i,\varepsilon}^2-a_\varepsilon)  &  (i=j)\\
 \lim\limits_{\varepsilon\to0}\tilde{Z}_{i,\varepsilon}\tilde{Z}_{j,\varepsilon} & (i\neq j)
\end{cases}
\quad
\Wick{\tilde{Z}_i\tilde{Z}_j^2} =
\begin{cases}
 \lim\limits_{\varepsilon\to0}(\tilde{Z}_{i,\varepsilon}^3-3a_\varepsilon \tilde{Z}_{i,\varepsilon})   & (i=j)\\
 \lim\limits_{\varepsilon\to0}(\tilde{Z}_{i,\varepsilon}\tilde{Z}_{j,\varepsilon}^2-a_\varepsilon \tilde{Z}_{i,\varepsilon}) & (i\neq j)
\end{cases}
\end{equ}
 where $a_\varepsilon=\mathbf{E}[\tilde Z_{i,\varepsilon}^2(0,0)]$ is a diverging constant independent of $i$ and the limits are understood in $C_T\bC^{-\kappa}$ for $\kappa>0$. (see \cite[Section 5]{MW17} for more details).

%We use $\bar{Z}_i$ to denote the solution to \eqref{eq:li1} with initial data $0$. In [MW17]
 % $:\bar{Z}_i^2:, :\bar{Z}_i^3:$ have been defined by using renormalization  and we use $\bar{Z}_i\bar{Z}_j$ and $:\bar{Z}_j^2:\bar{Z}_j$ to denote $:\bar{Z}_i^2:$ and $:\bar{Z}_j^3:$ for simplicity. The corresponding term could also be defined for $\tilde{Z}_i$ as in [MW17].
We now define the Wick products for $Z_{i}$ by combining the above with the smoothing properties of the heat semi-group $S_t$ associated with $\LL$.
Defining $\tilde{z}_{i} \eqdef z_i-\tilde{Z}_i(0)$, we have the decomposition
$$
Z_i=\tilde{Z}_i+S_t\tilde{z}_{i}.
$$
We then overload notation and define
 the Wick products of $Z_i$
by the binomial formula  \footnote{This definition is in line with \cite[(5.42)]{MW17},
which first considers a linear solution with $0$ initial condition rather than a stationary solution as here.}
 namely
\begin{align*}
\Wick{Z_j^2}
&=\Wick{\tilde{Z}_j^2}
	+2S_t\tilde{z}_{j}\tilde{Z}_j+(S_t\tilde{z}_{j})^2,
\\
\Wick{Z_j^3}
&=\Wick{\tilde{Z}_j^3}
	+3S_t\tilde{z}_{j}\Wick{\tilde{Z}_j^2}
	+3(S_t\tilde{z}_{j})^2\tilde{Z}_j
	+(S_t\tilde{z}_{j})^3,
\end{align*}
and for $i\neq j$
\begin{align*}
\Wick{Z_iZ_j}
&=\Wick{\tilde{Z}_i \tilde{Z}_j}
	+ S_t\tilde{z}_{i} \tilde{Z}_j
	+ S_t\tilde{z}_{j} \tilde{Z}_i
	+ S_t\tilde{z}_{i}S_t\tilde{z}_{j},
%=(\tilde{Z}_i+S_t\tilde{z}_{i})(\tilde{Z}_j+S_t\tilde{z}_{j}),
\\
\Wick{Z_iZ_j^2}
&= \Wick{\tilde{Z}_i\tilde{Z}_j^2}
+ S_t\tilde{z}_{i}  \Wick{\tilde{Z}_j^2 }
+ 2S_t\tilde{z}_{j} \Wick{ \tilde{Z}_i \tilde{Z}_j}
+ 2S_t\tilde{z}_{i} S_t\tilde{z}_{j}  \tilde{Z}_j
+ (S_t\tilde{z}_{j})^2  \tilde{Z}_i
+ S_t\tilde{z}_{i}(S_t\tilde{z}_{j})^2.
\end{align*}
We caution the reader that  this definition is non-canonical, in the sense that these renormalized products are not necessarily mean zero.  By the calculation in \cite[Corollary~3]{MW17} (see also \cite[Lemma 3.5]{RZZ17}) we have the following estimate:
\bl\label{le:ex}
For each $\kappa'>\kappa>0$ and all $p \geq 1$, we have the following bounds
\begin{align}
\mathbf{E}\|\tilde{Z}_i\|_{C_T\bC^{-\kappa}}^p+\mathbf{E}\|Z_i\|_{C_T\bC^{-\kappa}}^p \lesssim 1. \nonumber \\
\mathbf{E}\| \Wick{\tilde{Z}_i\tilde{Z}_j} \|_{C_T\bC^{-\kappa}}^p+\mathbf{E}\| \Wick{\tilde{Z}_i \tilde{Z}_j^2}\|_{C_T\bC^{-\kappa}}^p\lesssim 1. \nonumber \\
\mathbf{E}(\sup_{t\in[0,T]}t^{\kappa'}\|\Wick{Z_iZ_j}\|_{\bC^{-\kappa}})^p+\mathbf{E}(\sup_{t\in[0,T]}t^{2\kappa'}\|\Wick{Z_iZ_j^2}\|_{\bC^{-\kappa}})^p\lesssim 1. \nonumber
\end{align}
Furthermore, the proportional constants in the inequalities are independent of $i, j,N$.
\el

By Lemma \ref{le:ex}, there exists a measurable  $\Omega_0\subset \Omega$ with $\mathbf{P}(\Omega_0)=1$ such that for $\omega\in \Omega_0$ and every $i,j$
$$
\|Z_i\|_{C_T\bC^{-\kappa}}
	+\sup_{t\in[0,T]}t^{\kappa'}\|\Wick{Z_iZ_j}\|_{\bC^{-\kappa}}
	%+\sup_{t\in[0,T]}t^{\kappa'}\|:Z_j^2:\|_{\bC^{-\kappa}}
	+\sup_{t\in[0,T]}t^{2\kappa'}\|\Wick{Z_iZ_j^2}\|_{\bC^{-\kappa}}
<\infty.
$$
%$$\sup_{t\in[0,T]}t^{2\kappa'}\|Z_i:Z_j^2:\|_{\bC^{-\kappa}}
%+\sup_{t\in[0,T]}t^{2\kappa'}\|:Z_j^3:\|_{\bC^{-\kappa}}<\infty.$$
In the following we always consider $\omega\in \Omega_0$.  With the above choice of renormalization, classical arguments from \cite{DD03} can be used to obtain local existence and uniqueness to equation \eqref{eq:22}  by a pathwise fixed point argument.  This solution can also be shown to be global, as a simple consequence of a much stronger result, Lemma \ref{Y:L2}, which will be established in detail below.  Since the well-posedness arguments for solving equation \eqref{eq:22} with a fixed number of components is essentially known, we relegate the proof to Appendix \ref{sec:A1} and only state the result here.

\begin{lemma}\label{lem:Yglobal}
For each $N$, there exist unique global solutions $(Y_i)$ to equation  \eqref{eq:22} such that for $1\leq i\leq N$,  $Y_i\in C_TL^2\cap L^4_TL^4\cap L^2_TH^1$.
\end{lemma}

\subsection{Uniform in $N$ estimate}\label{sec:Y}
%\scott{Maybe we should keep the statement here but move the proof to an appendix: for two reasons, 1) the result is expected and 2) we refer to estimates which appear only later in the article}.
%Since $m\geq0$, we assume $m=0$ in this section for simplicity. \hao{think again if we want to do so - also double check if what follows below are consistently assuming this}\scott{I think we should not set m to be zero here}\hao{I agree; otherwise the paper is not very consistent}
We now turn to our uniform in $N$ bounds on equation \eqref{eq:22} and note that $Y_{i}$ itself depends on $N$, but we omit this throughout.  In the following lemma, we show that the empirical averages of the $L^{2}$ norms of $Y_{i}$ can be controlled pathwise in terms of averages of the $C_{T}\bC^{-\kappa}$ norms of $Z_{i}$, $\Wick{Z_{i}Z_{j}}$ and $\Wick{Z_{i}^{2}Z_{j}}$ discussed in Lemma \ref{le:ex} .
%We will aim to establish  an estimate for \eqref{eq:22} which is {\it uniform in} $N$. To this end, taking $L^2$-inner product with $Y_i$, we have
%\begin{equation}\label{uni21}
%\aligned&\frac{1}{2}\frac{\dif}{\dif t} \|Y_i\|_{L^2}^2+\|\nabla  Y_i\|_{L^2}^2
%	+\int\frac{1}{N}\sum_{j=1}^NY_j^2Y_i^2\dif x+m\|Y_i\|_{L^2}^2
%\\&=-\int\frac{1}{N}\sum_{j=1}^N(Y_iY_j^2Z_i+2Y_jY_i^2Z_j+2Y_jY_i \Wick{Z_iZ_j}+\Wick{Z_j^2}Y_i^2+Y_i \Wick{Z_j^2Z_i})\dif x.
%\endaligned\end{equation}
%Here the integration is over the torus.

\bl\label{Y:L2}
Let $s \in [2\kappa,\frac{1}{4})$.  There exists a universal constant $C$ such that
\begin{align}
\frac{1}{N}\sup_{t\in[0,T]}\sum_{j=1}^{N}\|Y_{j}\|_{L^{2} }^2+\frac{1}{N}\sum_{j=1}^{N}\| \nabla  Y_{j}\|_{L^2_T L^2 }^2+\bigg \| \frac{1}{N}\sum_{i=1}^{N}Y_{i}^{2}\bigg \|_{L^2_T L^2 }^{2}
\leq C\int_{0}^{T} \!\!\! R_{N}\dif t+\frac{1}{N}\sum_{j=1}^{N}\|y_{j}\|_{L^{2} }^2,\label{s15}
\end{align}
where
\begin{equation}\label{eq:RN}
\aligned
R_{N}
&:=1+ \bigg (\frac{1}{N}\sum_{j=1}^{N} \|Z_j\|_{\bC^{-s}}^{2}  \bigg )^{\frac{2}{1-s}}
+\bigg (\frac{1}{N}\sum_{j=1}^{N}\| \Wick{Z_{j}^{2}} \|_{\bC^{-s}} \bigg )^{\frac{4}{2-s}}
\\
&+\bigg (\frac{1}{N^{2}}  \sum_{i,j=1}^{N}\|\Wick{Z_{j}Z_{i}}\|_{\bC^{-s}}^{2} \bigg )^{\frac{2}{2-s}}
+\bigg (\frac{1}{N^{2}}  \sum_{i,j=1}^{N}\| \Wick{Z_{j}^{2}Z_{i}}\|_{\bC^{-s}}^{2} \bigg ).
\endaligned
\end{equation}
\el

\begin{proof}
\newcounter{UnifN} % proofstep = 0
\refstepcounter{UnifN} % increases value by 1
The proof is based on an energy estimate.  In Step \ref{UnifN1} we establish the energy identity \eqref{s10} which identifies the coercive quantities and involves three types of terms on the RHS.  These are labelled $I_{N}^{1}$, $I_{N}^{2}$, and $I_{N}^{3}$, which are respectively linear, quadratic, and cubic in $Y$.  In Steps \ref{UnifN4}-\ref{UnifN3}, we estimate each of these quantities, proceeding in order of difficulty, in terms of the coercive terms and the quantities $R_{N}^{i}$ for $i=1,2,3$ defined below.  The main ingredient is Lemma~\ref{lem:dual+MW}, restated here: for $s \in (0,1)$
\begin{equation}
|\langle g, f \rangle |\lesssim \big (\|\nabla g \|_{L^{1}}^{s}\|g\|_{L^{1}}^{1-s}+ \|g\|_{L^{1}} \big )\|f\|_{\bC^{-s}}  \label{s1}.
\end{equation}
The final output of Steps \ref{UnifN1}-\ref{UnifN3} is that for some universal constant $C$ it holds
\begin{align}
&\frac{1}{N}\sum_{i=1}^{N}\frac{\dif}{\dif t}\|Y_{i}\|_{L^{2}}^{2}+ \frac{1}{N}\sum_{i=1}^{N}\|\nabla Y_{i}\|_{L^{2}}^{2} +\frac{1}{N^{2}}\bigg \|\sum_{i=1}^{N}Y_{i}^{2} \bigg \|_{L^{2}}^{2}+ \frac{m}{N}\sum_{j=1}^{N}\|  Y_{j}\|_{L^2 }^2 \nonumber \\
&\leq C\frac{R_{N}^{1}}{N}+C \big (R_{N}^{2}+R_{N}^{3} \big )\frac{1}{N}\sum_{i=1}^{N} \|Y_{i}\|_{L^{2}}^{2} ,\label{eq:A}
\end{align}
where  $R_{N}^{i}$ for $i=1, 2, 3$ are defined in \eqref{eq:R1}, \eqref{eq:R2}, \eqref{e:def-SZ} below.
Noting that by H\"{o}lder's inequality
\begin{align}
\frac{1}{N}\sum_{i=1}^{N} \|Y_{i}\|_{L^{2}}^{2}=\frac{1}{N}\bigg \| \sum_{i=1}^{N}Y_{i}^{2} \bigg \|_{L^{1}} \leq \frac{1}{N}\bigg \| \sum_{i=1}^{N}Y_{i}^{2} \bigg \|_{L^{2}},  \nonumber
\end{align}
the estimate \eqref{s15} follows from Young's inequality with exponents $(2,2)$ and an integration over $[0,T]$.  The condition $s \in [2\kappa,\frac{1}{4})$ ensures that $R_{N}$ is integrable  near the origin, c.f. Lemma \ref{le:ex}.

%This follows from Lemma~\ref{lem:multi} together with \cite[Proposition~8]{MW17}.

%In Steps \ref{UnifN2} and \ref{UnifN3} we use \eqref{s1} to establish the inequalities \eqref{s6} and \eqref{s2}, which can then be integrated over $[0,T]$ to obtain the bound \eqref{s15}. I didn't account for the lower order term in {\color{blue} blue} yet.

{\sc Step} \arabic{UnifN} \label{UnifN1} \refstepcounter{UnifN} (Energy balance)

In this step, we justify the energy identity
%\footnote{The quantity $\frac{1}{N^2}\|\sum_i Y_{i}^{2}  \|^2_{L^{2}}$ on the LHS could be thought of as playing the same role as the $L^4$-norm in the $L^2$-estimate of $\Phi^4$ model.}
\begin{align}
\frac{1}{2} \sum_{i=1}^{N}\frac{\dif}{\dif t}\|Y_{i}\|_{L^{2}}^{2}+ \sum_{i=1}^{N}\|\nabla Y_{i}\|_{L^{2}}^{2}+m\sum_{i=1}^{N} \|Y_{i}\|_{L^{2}}^{2} +\bigg \| \frac{1}{\sqrt{N}}\sum_{i=1}^{N}Y_{i}^{2} \bigg \|_{L^{2}}^{2}
=I_{N}^{1}+I_{N}^{2}+I_{N}^{3},\label{s10}
\end{align}
where the quantities $I_{N}^{i}$ for $i=1,2,3$ are defined by
\begin{align*}
I_{N}^{1}&\eqdef -\frac{1}{N}\sum_{i,j=1}^{N} \langle Y_{i}, \Wick{Z_{j}^{2}Z_{i}}\rangle\\
I_{N}^{2}&\eqdef -\frac{1}{N}\sum_{i,j=1}^{N} 2 \langle Y_{i}Y_{j},\Wick{Z_{j}Z_{i}} \rangle +\langle Y_{i}^{2} ,\Wick{Z_{j}^{2}} \rangle  \\
I_{N}^{3}&\eqdef-\frac{1}{N}\sum_{i,j=1}^{N} 3 \langle Y_{i}^{2}Y_{j},Z_{j} \rangle
\end{align*}
%the old version
%\begin{align}
%I_{N}& \eqdef -\frac{1}{N}\sum_{i,j=1}^{N} \int 3Y_{i}^{2}Y_{j}Z_{j} \dif x. \nonumber \\
%I_{N}'& \eqdef -\frac{1}{N}\sum_{i,j=1}^{N} \int\big (2Y_{i}Y_{j}\Wick{Z_{j}Z_{i}}+Y_{i}^{2}\Wick{Z_{j}^{2}}+Y_{i} \Wick{Z_{j}^{2}Z_{i}} \big ) \dif x. \nonumber
%\end{align}
Notice that $I_{N}^{1}$, $I_{N}^{2}$, and $I_{N}^{3}$ are linear, quadratic, and cubic in $Y$, respectively.  Formally, the identity \eqref{s10} follows from testing \eqref{eq:22} by $Y_{i}$, integrating by parts, summing over $i=1,\dots,N$, and using symmetry with respect $i$ and $j$.  Since $Y_{i}$ is not sufficiently smooth in the time variable, some care is required to make this fully rigorous, and we direct the reader to  \cite[Proposition 6.8]{MW17} for more details.
%A rigorous justification can be carried out as in \cite{MW17}.
\medskip

{\sc Step} \arabic{UnifN} \label{UnifN4} \refstepcounter{UnifN} (Estimates for $I_{N}^{1}$)

In this step, we show there is a universal constant $C$ such that
\begin{align}
I_{N}^{1}\leq \frac{1}{4}  \sum_{i=1}^{N}\|\nabla Y_{i}\|_{L^{2}}^{2}+\sum_{i=1}^{N}\|Y_{i}\|_{L^{2}}^{2}
+CR_{N}^{1},\label{se1}
\end{align}
where
\begin{equation}
R_{N}^{1}\eqdef \sum_{i=1}^{N}\bigg \|\frac{1}{N}\sum_{j=1}^{N} \Wick{Z_{j}^{2}Z_{i}} \bigg \|_{\bC^{-s}}^{2}.\label{eq:R1}
\end{equation}
To establish \eqref{se1}, we apply \eqref{s1} with $Y_{i}$ playing the role of $g$  and $\frac{1}{N}\sum_{j=1}^{N}\Wick{Z_{j}^{2}Z_{i}}$ playing the role of $f$ to find
\begin{align}
I_{N}^{1}&\lesssim  \sum_{i=1}^{N}\big (  \|Y_{i}\|_{L^{1}}^{1-s}\|\nabla Y_{i}\|_{L^{1}}^{s}+\|Y_{i}\|_{L^{1}} \big )\bigg \|\frac{1}{N}\sum_{j=1}^{N} \Wick{Z_{j}^{2}Z_{i}} \bigg \|_{\bC^{-s}} . %\nonumber  \\
%&\lesssim \bigg ( \sum_{i=1}^{N}\|Y_{i}\|_{L^{1}}^{2}   \bigg )^{\frac{1-s}{2}} \bigg ( \sum_{i=1}^{N}\|\nabla Y_{i}\|_{L^{1}}^{2}   \bigg )^{\frac{s}{2} } \bigg ( \sum_{i=1}^{N}\bigg \|\frac{1}{N}\sum_{j=1}^{N} \Wick{Z_{j}^{2}Z_{i}} \bigg \|_{\bC^{-s}}^{2}      \bigg )^{\frac12} \nonumber\\
%&\qquad \qquad+\bigg ( \sum_{i=1}^{N}\|Y_{i}\|_{L^{1}}^{2}   \bigg )^{\frac{1}{2}}\bigg ( \sum_{i=1}^{N}\bigg \|\frac{1}{N}\sum_{j=1}^{N} \Wick{Z_{j}^{2}Z_{i}} \bigg \|_{\bC^{-s}}^{2}\bigg )^{1/2},
\label{s7}
\end{align}
%where we applied H\"older's inequality for the summation in $i$ using exponents $(\frac{2}{1-s},\frac{2}{s},2)$ for the first term and $(2,2)$ for the second term.
We now use Young's inequality with
exponents $(\frac{2}{1-s},\frac{2}{s},2)$ for the first term and $(2,2)$ for the second term
and the embedding of $L^{2}$ into $L^{1}$ to obtain \eqref{se1}.
%considering separately the cases $m=0$ and $m>0$.

\medskip

{\sc Step} \arabic{UnifN} \label{UnifN2} \refstepcounter{UnifN} (Estimates for $I_{N}^{2}$ )

In this step, we show there is a universal constant C such that
\begin{align}
I_{N}^{2} \leq \frac{1}{4} \sum_{i=1}^{N}\|\nabla Y_{i}\|_{L^{2}}^{2}
+C(1+R_{N}^{2})\bigg (  \sum_{i=1}^{N} \|Y_{i}\|_{L^{2}}^{2}\bigg ),\label{s6}
\end{align}
where
\begin{align}\label{eq:R2}
R_{N}^{2}&\eqdef \bigg (\frac{1}{N^{2}}\sum_{i,j=1}^{N}\|\Wick{Z_{j}Z_{i}}\|_{\bC^{-s}}^{2}  \bigg )^{\frac{1}{2-s}}+\bigg \|\frac{1}{N}\sum_{j=1}^{N} \Wick{Z_{j}^{2}} \bigg \|_{\bC^{-s}}^{\frac{2}{2-s} }
\quad \eqdef \Big(\widetilde{R_{N}^{2}}\Big)^{\frac{1}{2-s}}+ \Big(\overline{R_{N}^{2}}\Big)^{\frac{2}{2-s}}.
\end{align}
Applying \eqref{s1} with $Y_{i}Y_{j}$ playing the role of $g$   and $\Wick{Z_{j}Z_{i}}$ playing the role of $f$ followed by H\"older's inequality in $L^{2}$, the product rule and symmetry with respect to $i,j$ we find
\begin{align}
\frac{1}{N }\sum_{i,j=1}^{N} \langle Y_{i}Y_{j},& \Wick{Z_{j}Z_{i}} \rangle
\lesssim
\frac{1}{N }\sum_{i,j=1}^{N}\big (  \|Y_{i}Y_{j}\|_{L^{1}}^{1-s}\|\nabla(Y_{i}Y_{j})\|_{L^{1}}^{s}+\|Y_{i}Y_{j}\|_{L^{1}} \big )\|\Wick{Z_{j}Z_{i}}\|_{\bC^{-s}}\nonumber
\\
&\lesssim
\frac{1}{N }\sum_{i,j=1}^{N}\big ( \|Y_{j}\|_{L^{2}}\|Y_{i}\|_{L^{2}}^{1-s}\|\nabla Y_{i}\|_{L^{2}}^{s}+\|Y_{i}\|_{L^{2}}\|Y_{j}\|_{L^{2}} \big )\|\Wick{Z_{j}Z_{i}}\|_{\bC^{-s}} \nonumber
 \\
&\lesssim
 \bigg ( \sum_{j=1}^{N}\|Y_{j}\|_{L^{2}}^{2}   \bigg )^{\frac12} \bigg ( \sum_{i=1}^{N}\|Y_{i}\|_{L^{2}}^{2(1-s)}\|\nabla Y_{i}\|_{L^{2}}^{2s}   \bigg )^{\frac12}
  \Big(\widetilde{R_{N}^{2}}\Big)^{\frac12}
 %\bigg (\frac{1}{N^{2}}\sum_{i,j=1}^{N}\|\Wick{Z_{j}Z_{i}}\|_{\bC^{-s}}^{2}  \bigg )^{\frac12}
+\bigg (\sum_{i=1}^{N} \|Y_{i}\|_{L^{2}}^{2}\bigg )
   \Big(\widetilde{R_{N}^{2}}\Big)^{\frac12}
 %\bigg ( \frac{1}{N^{2}}\sum_{i,j=1}^N \|\Wick{Z_{j}Z_{i}}\|_{\bC^{-s}}^2 \bigg )^{1/2}
 \nonumber \\
 &\lesssim \bigg ( \sum_{j=1}^{N}\|Y_{j}\|_{L^{2}}^{2} \bigg )^{1-\frac{s}{2}}\bigg ( \sum_{i=1}^{N}\|\nabla Y_{i}\|_{L^{2}}^{2}   \bigg )^{\frac{s}{2}}
   \Big(\widetilde{R_{N}^{2}}\Big)^{\frac12}
%\bigg (\frac{1}{N^{2}}\sum_{i,j=1}^{N}\|\Wick{Z_{j}Z_{i}}\|_{\bC^{-s}}^{2}  \bigg )^{\frac12} \nonumber \\
+\bigg ( \sum_{i=1}^{N} \|Y_{i}\|_{L^{2}}^{2}\bigg )
  \Big(\widetilde{R_{N}^{2}}\Big)^{\frac12}
% \bigg ( \frac{1}{N^{2}}\sum_{i,j=1}^N \|\Wick{Z_{j}Z_{i}}\|_{\bC^{-s}}^2 \bigg )^{1/2}
\label{s8},
\end{align}
where we used H\"older's inequality for the summation in $i,j$ with exponents $(2,2)$ followed by H\"older's inequality for the summation in $i$ with exponents $(\frac{1}{1-s},\frac{1}{s})$.  Finally, applying \eqref{s1} with $Y_{i}^{2}$ playing the role of $g$ and $\frac{1}{N}\sum_{j=1}^{N} \Wick{Z_{j}^{2}}$ playing the role of $f$ we find
\begin{align}
\frac{1}{N} \sum_{i,j=1}^{N} \langle Y_{i}^{2}, & \Wick{Z_{j}^{2}} \rangle
\lesssim
\sum_{i=1}^{N}\big ( \|Y_{i}^{2}\|_{L^{1}}^{1-s}\|Y_{i}\nabla Y_{i}\|_{L^{1}}^{s}+\|Y_{i}^{2}\|_{L^{1}} \big )
\overline{R_{N}^{2}}
% \bigg \|\frac{1}{N}\sum_{j=1}^{N} \Wick{Z_{j}^{2}} \bigg \|_{\bC^{-s}}
\lesssim
\sum_{i=1}^{N}\big (  \|Y_{i} \|_{L^{2}}^{2-s}\|\nabla Y_{i}\|_{L^{2}}^{s}+\|Y_{i} \|_{L^{2}}^{2} \big )
\overline{R_{N}^{2}}
%\bigg \|\frac{1}{N}\sum_{j=1}^{N} \Wick{Z_{j}^{2}} \bigg \|_{\bC^{-s}}
\nonumber \\
%&\lesssim \bigg (\frac{1}{N}\sum_{i=1}^{N}\|Y_{i}\|_{L^{2}}^{2-s}\|\nabla Y_{i}\|_{L^{2}}^{s}  \bigg )\bigg (\frac{1}{N} \sum_{j=1}^{N}\| \Wick{Z_{j}^{2}} \|_{B_{\infty,\infty}^{-s}}   \bigg ) \nonumber \\
&\lesssim
 \bigg [ \bigg (\sum_{i=1}^{N} \|Y_{i}\|_{L^{2}}^{2} \bigg )^{1-\frac{s}{2}} \bigg ( \sum_{i=1}^{N}\|\nabla Y_{i}\|_{L^{2}}^{2}  \bigg )^{\frac{s}{2}}+\sum_{i=1}^{N}\|Y_{i}\|_{L^{2}}^{2} \bigg ]
 \overline{R_{N}^{2}}
 % \bigg \|\frac{1}{N}\sum_{j=1}^{N} \Wick{Z_{j}^{2}} \bigg \|_{\bC^{-s}}
  \label{s9},
\end{align}
where we used H\"older's inequality for the summation in $i$ with exponents $(\frac{2}{2-s},\frac{2}{s})$.  The inequality \eqref{s6} now follows from \eqref{s8}-\eqref{s9} by Young's inequality with exponents $(\frac{2}{2-s},\frac{2}{s})$.
%\scott{Maybe add remark somewhere about why we don't use the nonlinear dissipation at this stage; related to the factors of $N$.}\zhu{we add a remark later}

\medskip

{\sc Step} \arabic{UnifN} \label{UnifN3} \refstepcounter{UnifN} (Estimates for $I_{N}^3$: cubic terms in $Y$)

In this step, we show there exists a universal constant $C$ such that
\begin{equation}
I_{N}^{3} \leq \frac{1}{4} \bigg (\sum_{i=1}^{N}\|\nabla Y_{i}\|_{L^{2}}^{2}+\bigg \| \frac{1}{\sqrt{N}}\sum_{i=1}^{N}Y_{i}^{2}\bigg \|_{L^{2}}^{2}  \bigg )+C(1+R_{N}^{3})\bigg (  \sum_{i=1}^{N} \|Y_{i}\|_{L^{2}}^{2}\bigg ),\label{s2}
\end{equation}
where
\begin{equ}[e:def-SZ]
R_{N}^{3}\eqdef \bigg ( \frac{1}{N} \sum_{j=1}^{N}\|Z_{j}\|_{\bC^{-s}}^{2}  \bigg )^{\frac1{1-s}}
=(\frac1N \SZ)^{\frac1{1-s}}
\qquad
\mbox{with}
\quad
\SZ \eqdef \sum_{j=1}^{N}\|Z_{j}\|_{\bC^{-s}}^{2}.
\end{equ}
Appealing again to \eqref{s1}, we find
\begin{align} \label{q1}
I_{N}^{3}
&\lesssim \frac{1}{N}\sum_{j=1}^{N} \Big ( \Big \|  \sum_{i=1}^{N}Y_{i}^{2}Y_{j} \Big \|_{L^{1}}^{1-s} \Big \|\nabla \Big ( \sum_{i=1}^{N}Y_{i}^{2}Y_{j} \Big ) \Big \|_{L^{1}}^{s}+\Big \|  \sum_{i=1}^{N}Y_{i}^{2}Y_{j} \Big \|_{L^{1}} \Big ) \|Z_{j}\|_{\bC^{-s}} \\
& \lesssim\frac{1}{N} \Big (\sum_{j=1}^{N}\Big \|  \sum_{i=1}^{N}Y_{i}^{2}Y_{j} \Big \|_{L^{1}}^{2(1-s) }\Big \|\nabla \Big ( \sum_{i=1}^{N}Y_{i}^{2}Y_{j} \Big ) \Big \|_{L^{1}}^{2s}  \Big )^{\frac12}
\SZ^{\frac12}
%\Big (  \sum_{j=1}^{N}\|Z_{j}\|_{\bC^{-s}}^{2}   \Big )^{1/2}
+\frac{1}{N}\Big ( \sum_{j=1}^{N} \Big \|  \sum_{i=1}^{N}Y_{i}^{2}Y_{j} \Big \|_{L^{1}}^{2} \Big )^{\frac12}
\SZ^{\frac12}
%\Big (  \sum_{j=1}^{N}\|Z_{j}\|_{\bC^{-s}}^{2}   \Big )^{1/2}
\nonumber.
\end{align}
By H\"older's inequality, it holds that
\begin{equation}
\Big \|  \sum_{i=1}^{N}Y_{i}^{2}Y_{j} \Big \|_{L^{1}} \lesssim \Big \|\sum_{i=1}^{N}Y_{i}^{2} \Big \|_{L^{2}} \|Y_{j}\|_{L^{2}}. \label{q2}
\end{equation}
Furthermore, we find that
\begin{align}
&\Big \| \nabla \Big ( \sum_{i=1}^{N}Y_{i}^{2}Y_{j} \Big )  \Big \|_{L^{1}}
\lesssim \Big \| \sum_{i=1}^{N}Y_{i}^{2}\nabla Y_{j}  \Big \|_{L^{1}}+\Big \| \sum_{i=1}^{N} \nabla Y_{i}Y_{i}Y_{j}  \Big \|_{L^{1}} \nonumber \\
&\lesssim \Big \| \sum_{i=1}^{N}Y_{i}^{2} \Big \|_{L^{2}} \|\nabla Y_{j}\|_{L^{2}}+\Big ( \sum_{i=1}^{N}\|\nabla Y_{i}\|_{L^{2}}^{2} \Big )^{1/2} \Big \| \sum_{i=1}^{N}Y_{i}^{2}Y_{j}^{2} \Big \|_{L^{1}}^{1/2}.
\end{align}
Hence, we find that
\begin{align}
&\sum_{j=1}^{N}\Big \|  \sum_{i=1}^{N}Y_{i}^{2}Y_{j} \Big \|_{L^{1}}^{2(1-s) }\Big \|\nabla \Big (\sum_{i=1}^{N}Y_{i}^{2}Y_{j} \Big ) \Big \|_{L^{1}}^{2s}
\;\; \lesssim \;\; \Big \| \sum_{i=1}^{N}Y_{i}^{2} \Big \|_{L^{2}}^{2} \Big ( \sum_{j=1}^{N}\|Y_{j}\|^{2(1-s) }_{L^{2}}\|\nabla Y_{j}\|_{L^{2}}^{2s}   \Big )  \nonumber\\
&\qquad\qquad\qquad\qquad+\Big \| \sum_{i=1}^{N}Y_{i}^{2} \Big \|_{L^{2}}^{2(1-s)}\Big ( \sum_{i=1}^{N}\|\nabla Y_{i}\|_{L^{2}}^{2} \Big )^{s}
 \Big ( \sum_{j=1}^{N}\Big \| \sum_{i=1}^{N}Y_{i}^{2}Y_{j}^{2} \Big \|_{L^{1}}^{s}\|Y_{j}\|_{L^{2}}^{2(1-s)}  \Big )  \nonumber\\
& \lesssim \Big \| \sum_{i=1}^{N}Y_{i}^{2} \Big \|_{L^{2}}^{2} \Big ( \sum_{j=1}^{N}\|Y_{j}\|^{2 }_{L^{2}}   \Big )^{1-s}\Big ( \sum_{j=1}^{N} \|\nabla Y_{j}\|_{L^{2}}^{2} \Big )^{s}  \nonumber \\
&\qquad\qquad\qquad\qquad+\Big \| \sum_{i=1}^{N}Y_{i}^{2} \Big \|_{L^{2}}^{2(1-s)}\Big ( \sum_{i=1}^{N}\|\nabla Y_{i}\|_{L^{2}}^{2} \Big )^{s}
\Big ( \sum_{j=1}^{N}\Big \| \sum_{i=1}^{N}Y_{i}^{2}Y_{j}^{2} \Big \|_{L^{1}}  \Big )^{s}\Big ( \sum_{j=1}^{N} \|Y_{j}\|_{L^{2}}^{2}  \Big )^{1-s} \nonumber \\
& \lesssim \Big \| \sum_{i=1}^{N}Y_{i}^{2} \Big \|_{L^{2}}^{2} \Big ( \sum_{j=1}^{N}\|Y_{j}\|^{2 }_{L^{2}}   \Big )^{1-s}\Big ( \sum_{j=1}^{N} \|\nabla Y_{j}\|_{L^{2}}^{2} \Big )^{s} .  \nonumber
 \end{align}
 Inserting this into \eqref{q1}, taking the square root, and using \eqref{q2} we find
 \begin{align}
 I^3_{N} &\lesssim \frac{1}{N}\Big \| \sum_{i=1}^{N}Y_{i}^{2} \Big \|_{L^{2}} \!\!
 \Big ( \sum_{j=1}^{N}\|Y_{j}\|^{2 }_{L^{2}}   \Big )^{\frac{1-s}{2} }\!\!
 \Big ( \sum_{j=1}^{N} \|\nabla Y_{j}\|_{L^{2}}^{2} \Big )^{\frac{s}{2}}
  \SZ^{\frac12}
  %\Big (  \sum_{j=1}^{N}\|Z_{j}\|_{\bC^{-s}}^{2}   \Big )^{1/2}
 + \frac{1}{N}\Big \| \sum_{i=1}^{N}Y_{i}^{2} \Big \|_{L^{2}}\!\!
 \Big ( \sum_{j=1}^{N}\|Y_{j}\|^{2 }_{L^{2}}   \Big )^{\frac{1}{2} }
 \SZ^{\frac12}
  %\Big (  \sum_{j=1}^{N}\|Z_{j}\|_{\bC^{-s}}^{2}   \Big )^{1/2}
 \label{ssz1}.
 \end{align}
 Applying Young's inequality with exponent $(2,\frac{2}{1-s},\frac{2}{s})$ we arrive at \eqref{s2}.
\end{proof}

\bc\label{co:Y} Let  $q\geq1$, $s \in [2\kappa,\frac{2}{q+1})$.  There exists a universal constant $C$ such that
\begin{align*}
&\sup_{t\in[0,T]}\bigg(\frac{1}{N}\sum_{j=1}^{N}\|Y_{j}\|_{L^{2} }^2\bigg)^q+\int_0^T \bigg(\frac{1}{N}\sum_{j=1}^{N}\|Y_{j}\|_{L^{2} }^2\bigg)^{q-1}\bigg[\frac{1}{N}\sum_{j=1}^{N}\| \nabla  Y_{j}\|_{L^2}^2+\bigg \| \frac{1}{N}\sum_{i=1}^{N}Y_{i}^{2}\bigg \|_{L^2}^{2}\bigg]\dif t  \nonumber\\
&\leq C\int_{0}^{T}R_{N}^{\frac{q+1}{2}}\dif t+\bigg(\frac{1}{N}\sum_{j=1}^{N}\|y_{j}\|_{L^{2} }^2\bigg)^q,
\end{align*}
with $R_N$ introduced in Lemma \ref{Y:L2}.
\ec

\begin{proof}
Set $V=\frac{1}{N}\sum_{i=1}^N\|Y_i\|_{L^{2}}^{2}$
and $G=\frac{1}{N}\sum_{j=1}^{N}\| \nabla  Y_{j}\|_{L^2 }^2
+\big \| \frac{1}{N}\sum_{i=1}^{N}Y_{i}^{2}\big \|_{L^2 }^{2}$. By \eqref{eq:A} in the proof of Lemma \ref{Y:L2} we deduce for $q\geq1$
\begin{equ}
\frac{\dif}{\dif t}V^q+GV^{q-1}\leq CR_{N}V^{q-1}\leq CR_N^{\frac{q+1}{2}}+\frac{1}{2}V^{q+1} .
\end{equ}
Note that  $G \geq V^{2}$ since
$\|\sum_{i=1}^{N}Y_{i}^{2}\|_{L^{1}}=\sum_{i=1}^{N}\|Y_{i}\|_{L^{2}}^{2}$,
which implies the result.
\end{proof}

%\scott{On a first reading the final estimate looks bad for large times, but as we know, this is because there are different options on how to conclude after doing Steps 1-4.  This should be emphasized somewhere.  We could also use the m when we do Gronwall.  On a related note, I would like to know the following: for a non-negative solution to the ODE $\dot{E}+\lambda E^{2}=gE$ where $\lambda \in \R$ and $g \in L^{1}_{t}$, what is the optimal bound on E? Roughly speaking, at times where $\lambda E >g$, we should absorb to the left.  How to see that in an estimate? }\zhu{Please take a look at \cite[Lemma 3.8]{TW18}.} \scott{I was asking about $g \in L^{1}_{T}$ not $g \in L^{\infty}_{T}$, but I geuss this is not urgent.}

\bl\label{Y:L3} Let $s \in [2\kappa,1/4)$.  There exists a universal constant $C$ such that
\begin{align}
&\sup_{t\in[0,T]}\sum_{j=1}^{N}\|Y_{j}\|_{L^{2} }^2+\sum_{j=1}^{N}\| \nabla  Y_{j}\|_{L^2_T L^2 }^2+\frac{1}{N}\bigg \|\sum_{i=1}^{N}Y_{i}^{2}\bigg \|_{L^2_T L^2 }^{2}  \nonumber\\
&\leq C(\|R_N^0\|_{L^1_T}+\sum_{j=1}^{N}\|y_{j}\|_{L^{2} }^2)\exp\Big\{\int_{0}^{T}(1+R_N^2+R_{N}^3)\dif t\Big\},\label{bs15}
\end{align}
where $R_N^2$, $R_N^3$ given in the proof of Lemma \ref{Y:L2} and
\begin{equation}\label{eq:R0}R_N^0=\frac{1}{N^2} \sum_{i=1}^{N}\Big\|\sum_{j=1}^{N}\Lambda^{-s}(\Wick{Z_{j}^{2}Z_{i}})\Big\|^2_{L^2}. \end{equation}
%Furthermore, if $\E\Wick{Z_j^2}=0$ for  every $j$, $\E R_N^0\lesssim1$.
\el
\begin{proof}
The proof is almost the same as Lemma~\ref{Y:L2}.
We appeal to Steps 1,3, and 4 of Lemma \ref{Y:L2} and only modify Step 2.  To estimate $I_{N}^{1}$ we write
\begin{align}
I_{N}^{1}
&\leq \frac{1}{N} \sum_{i=1}^{N}\|\Lambda^sY_{i}\|_{L^{2}}
	  \Big\|\sum_{j=1}^{N}\Lambda^{-s}(\Wick{Z_{j}^{2}Z_{i}})\Big\|_{L^2} \nonumber\\
%&\lesssim
%\varepsilon\sum_{i=1}^{N}\|\Lambda^sY_{i}\|^2_{L^{2}}
%	+  \frac{1}{N^2} \sum_{i=1}^{N}\Big\|\sum_{j=1}^{N}\Lambda^{-s}(\Wick{Z_{j}^{2}Z_{i}})\Big\|^2_{L^2} \nonumber\\
&\leq   \frac{1}{8}\sum_{i=1}^{N}\|Y_{i}\|^2_{L^{2}}
+ \frac{1}{8}\sum_{i=1}^{N}\|\nabla Y_{i}\|^2_{L^{2}}
+ \frac{C}{N^2} \sum_{i=1}^{N}\Big\|\sum_{j=1}^{N}\Lambda^{-s}(\Wick{Z_{j}^{2}Z_{i}})\Big\|^2_{L^2},\label{ns7}
\end{align}
where, in the last step, we applied Young's inequality for products, and then interpolation.  Combining \eqref{s6}, and \eqref{s2} \eqref{s9} with \eqref{ns7} and inserting these inequalities into the energy identity \eqref{s10}, we obtain
\begin{align}
&\sum_{j=1}^{N}\frac{\dif }{\dif t}\|Y_{j}\|_{L^{2}}^2+\sum_{j=1}^{N}\| \nabla  Y_{j}\|_{L^2 }^2+\frac{1}{N}\bigg \|\sum_{i=1}^{N}Y_{i}^{2}\bigg \|_{L^2}^{2}
+m\sum_{j=1}^{N}\|  Y_{j}\|_{L^2 }^2 \nonumber\\
&\leq CR_N^0+\sum_{j=1}^{N}\|Y_{j}\|_{L^{2}}^2C(1+R_N^2+R_{N}^3).\label{bs16}
\end{align}
The estimate \eqref{bs15} now follows from Gronwall's inequality. %\hao{$\|Y_{j}\|_{L^{2}}$ here looks strange, should it be $\|Y_{j}\|_{L^{2}}^2$?}
\end{proof}

\begin{remark}\label{re:1} For the estimate of $\frac{1}{N}\sum_{i=1}^N\|Y_i\|_{L^2}^2$ in Lemma \ref{Y:L2}, the dissipation term $\big\|\frac1N\sum_{i=1}^NY_i^2\big\|_{L^2}^2$ could be used to avoid Gronwall's Lemma. However, for  $\sum_{i=1}^N\|Y_i\|_{L^2}^2$ or $\frac{1}{N}\sum_{i=1}^N\|Y_i\|_{L^p}^p$ for $p>2$
it is  less clear how to exploit
 the corresponding dissipation term and we need to use Gronwall's inequality to derive a uniform estimate. Since the $R_N^2, R_N^3$ appear in the exponential, this makes it unclear how to obtain moment estimates directly.
\end{remark}

%%%%%%%%%%%%%%%%%%%%%%%%%%%%%%%%%%%%%%%%%%%%%%%%%%%%%%%%%%%%%%%%  Limiiting equation
%%%%%%%%%%%%%%%%%%%%%%%%%%%%%%%%%%%%%%%%%%%%%%%%%%%%%%%%%%%%%%%%
%%%%%%%%%%%%%%%%%%%%%%%%%%%%%%%%%%%%%%%%%%%%%%%%%%%%%%%%%%%%%%%%
%%%%%%%%%%%%%%%%%%%%%%%%%%%%%%%%%%%%%%%%%%%%%%%%%%%%%%%%%%%%%%%%
%%%%%%%%%%%%%%%%%%%%%%%%%%%%%%%%%%%%%%%%%%%%%%%%%%%%%%%%%%%%%%%%
%%%%%%%%%%%%%%%%%%%%%%%%%%%%%%%%%%%%%%%%%%%%%%%%%%%%%%%%%%%%%%%%
%%%%%%%%%%%%%%%%%%%%%%%%%%%%%%%%%%%%%%%%%%%%%%%%%%%%%%%%%%%%%%%%
%%%%%%%%%%%%%%%%%%%%%%%%%%%%%%%%%%%%%%%%%%%%%%%%%%%%%%%%%%%%%%%%

\section{Global solvability of the mean-field SPDE}\label{sec:limit}
%\begin{equation}\label{eq:Psi2}
%\LL\Psi_i=-\mu\Psi_i+\xi_i,\quad \Psi_i(0)=\psi_i
%\end{equation}
%where (formally at this stage) $\mu$ is given by
%$$
%\mu\eqdef \mathbf{E}[\Psi_i^2-\tilde{Z}_i^2].
%$$
% Here $(\xi_i)_{i\in \N}$ are independent space-time white noises as in  \eqref{eq:21} and $\tilde{Z}_i$ is the stationary solution to \eqref{eq:li1}.
In this section, we develop a solution theory for the mean-field SPDE \eqref{eq:Psi2}, the renormalized version of the formal equation \eqref{eq:formPsi}.
%\begin{equation}
%\LL \Psi_{i}=-\E[\Psi_{i}^{2}]\Psi_{i}+\xi_{i}.\nonumber
%\end{equation}
In two dimensions this is a singular SPDE where the ill-defined non-linearity depends on the law of the solution.  As a result, we cannot proceed via path-wise arguments alone as in \cite{DD03}  and \cite{MW17} and we need to develop a few new tricks for both the local and global well-posedness.

We begin by explaining our assumptions on the initial data and our notion of solution to \eqref{eq:Psi2}.  The initial datum $\psi_{i}$ decompose as $\psi_i=z_i+\eta_i$, where $\mathbf{E}\|z_i\|_{\bC^{-\kappa}}^p\lesssim 1$ for $\kappa>0$ and every $p>1$, and $\mathbf{E}\|\eta_i\|_{L^4}^4<\infty$ (except for Lemma~\ref{lem:Lp} which is an $L^p$ estimate).
%Both parts of the initial condition are defined on the same stochastic basis $(\Omega,\mathcal{F},\mathbf{P})$ and independent of the space-time white noises $(\xi_i)_{i\in \N}$.
We define $\Psi_{i}$ to be a solution to the renormalized, mean-field SPDE \eqref{eq:Psi2} starting from $\psi_{i}$ provided that $\Psi_i=Z_i+X_i$ holds, where $Z_i$ is the solution to \eqref{eq:li1} with $Z_i(0)=z_{i}$ as in Section \ref{sec:nonlinear} and $X_{i}$ is a random process satisfying
\begin{equation}\label{eq:X2}
\LL X_i= -\mu_{X_i}(X_i+Z_i),\quad X_i(0)=\eta_i.
\end{equation}
Here, $\mu_{X_i}$ depends on the law of $X_{i}$ and is defined as
$$\mu_{X_i} \eqdef \mathbf{E}[X_i^2]+2\mathbf{E}[X_iZ_i]+\mathbf{E}[\Wick{Z_i^2}].$$
In the following we write $\mu$ for $\mu_{X_i}$ for simplicity.

We now comment on the meaning of the non-linearity in equation \eqref{eq:X2}.
Recall from Section \ref{sec:re} that
 $Z_i\in \bC^{-\kappa}$ (Lemma \ref{le:ex}), while $\mathbf{E}[\Wick{Z_i^2}]=\E[(S_t\tilde{z}_{i})^2]$ with $\tilde{z}_{i}=z_i-\tilde{Z}_i(0)$, so by Schauder theory we expect that $X_{i}$ is H\"{o}lder continuous.  Hence, we anticipate that $\mathbf{E}[X_i^2]$ is a well-defined function, while $\mathbf{E}[X_iZ_i]$ is a distribution satisfying for $t>0$ and $\beta>\kappa$
$$
\|\mathbf{E}[X_iZ_i](t)\|_{\bC^{-\kappa}}\lesssim \mathbf{E}[\|X_i(t)\|_{\bC^{\beta}}\|Z_i(t)\|_{\bC^{-\kappa}}].
$$
We immediately find that all terms in $\mu(X_i+Z_i)$ are classically defined in the sense of distributions except for $\mathbf{E}[X_iZ_i]Z_i$, which requires more care and a suitable probabilistic argument.  The idea used to overcome this difficulty, which is repeated in different ways throughout the section, is to view the expectation $\mu$ as coming from a suitable independent copy of $(X_{i}, Z_{i})$.
To avoid notational confusion, we now comment further on our convention throughout this section. We consider equation \eqref{eq:Psi2} for a {\it fixed} $i$ and when we write  $(\eta_j,z_j) $ for  $j\neq i$ we  mean an independent copy of $(\eta_i,z_i)$, % and when we write $Z_j$ (resp. $X_j$) it means an  independent copy of $Z_i$ (resp. $X_i$).
and we then write $(Z_j,X_j)$ for the solution driven by  white noise $\xi_j$, which is independent of $\xi_i$, from initial data $(z_j,\eta_j)$.

%Here, $\mu_{1,2,3}$ depend on $i$ (since initial conditions depend on $i$), but we omit this dependence in the notation for simplicity.

%\scott{I would prefer not to introduce $\mu_{1},\mu_{2},\mu_{3}$ since it does not really shorten many expressions and it is harder for me to remember which is which.}

\subsection{Local well-posedness}\label{sec:4.1}

\bl\label{lem:Z1} For $p\in [1,\infty]$ and $0<\kappa<s$ it holds
\begin{equation}
\|Z_{i}\E[Z_{i}X_{i}] \|_{B^{-\kappa}_{p,\infty} }  \lesssim \big ( \mathbf{E}  \|X_{i}\|_{B^{s}_{p,\infty}}^{2} \big )^{\frac{1}{2} } \big ( \E[\|\Wick{Z_{i}Z_{j}}\|_{\bC^{-\kappa} }^{2}  \mid Z_{i} ] \big )^{\frac{1}{2}} .
\end{equation}
Here the conditional expectation is on the $\sigma$-algebra generated by the stochastic process $Z_i$.
\el
\begin{proof}
	Letting $(X_{j},Z_{j})$ be an independent copy of $(X_{i},Z_{i})$ we have
	$$Z_{i}\E[Z_{i}X_{i}] =Z_{i}\E[Z_{j}X_{j}] =\E[ \Wick{Z_{i}Z_{j}}X_{j} \mid Z_{i}].$$
	We then use  Jensen's inequality to find
	\begin{align}
	\|Z_{i}\E[Z_{j}X_{j}] \|_{B^{-\kappa}_{p,\infty} } &\leq \E \big [ \|\Wick{Z_{i}Z_{j}}X_{j} \|_{B^{-\kappa}_{p,\infty}}  \mid Z_{i}\big ]\leq \E \big [ \|\Wick{Z_{i}Z_{j}}\|_{\bC^{-\kappa} } \| X_{j} \|_{B^{s}_{p,\infty}}  \mid Z_{i}\big ], \nonumber
	\end{align}
	where we used Lemma \ref{lem:multi} in the last line.  The claim now follows from conditional H\"{o}lder's inequality and the independence of $X_{j}$ from $Z_{i}$.
\end{proof}
\iffalse
\bl\label{lem:ZL2} It holds that
$$\|\Lambda^{-2\kappa}(\mu_2 Z_i)\|_{L^2}\lesssim\mathbf{E}[\|\Lambda^{2\kappa}\bar{X}_i\|_{L^2}\|Z_i(\omega)\circ \bar{Z}_i\|_{\bC^{-\kappa}}]
+\mathbf{E}[\|\Lambda^{3\kappa}\bar{X}_i\|_{L^2}\|\bar{Z}_i\|_{\bC^{-\kappa}}]\|Z_i\|_{\bC^{-\kappa}}.$$
\el
\begin{proof}
	We use  the same decomposition as Lemma \ref{lem:Z1} and obtain
	$$\|Z_i\circ\mathbf{E}[\bar{X}_i\prec Z_j]\|_{B^{-\kappa}_{2,2}}\lesssim \mathbf{E}[\|\bar{X}_i\|_{B^{2\kappa}_{2,2}}\|Z_i(\omega)\circ Z_j\|_{\bC^{-\kappa}}]+\mathbf{E}[\|\bar{X}_i\|_{B^{3\kappa}_{2,2}}\|Z_j\|_{\bC^{-\kappa}}]\|Z_i\|_{\bC^{-\kappa}}.$$
	For the other term we have
	$$\|Z_i\prec\mathbf{E}[\bar{X}_i\prec Z_j]\|_{B^{-2\kappa}_{2,2}}+\|Z_i\succ\mathbf{E}[\bar{X}_i\prec Z_j]\|_{B^{-2\kappa}_{2,2}}\lesssim \|Z_i\|_{\bC^{-\kappa}}\mathbf{E}[\|\bar{X}_i\|_{B^{\kappa}_{2,2}}\|Z_j\|_{\bC^{-\kappa}}],$$
	$$\|\mathbf{E}[\bar{X}_i\succcurlyeq Z_j] Z_i\|_{B^{-\kappa}_{2,2}}\lesssim \|Z_i\|_{\bC^{-\kappa}}\mathbf{E}[\|\bar{X}_i\|_{B^{3\kappa}_{2,2}}\|Z_j\|_{\bC^{-\kappa}}].$$

\end{proof}\fi

We now apply the above result to obtain a local well-posedness result for \eqref{eq:X2}, which yields in turn a local well-posedness result for \eqref{eq:Psi2}.

\bl
There exists $T^*>0$ small enough such that \eqref{eq:X2} has a unique mild solution $X_i\in L^2(\Omega;C_{T^*}L^4\cap C((0,T^*];\bC^\beta))$  and for $\beta>3\kappa$ small enough, $\gamma=\beta+\frac{1}{2}$, one has
$$\mathbf{E}[\sup_{t\in[0,T^*]} t^\gamma\|X_i\|_{\bC^\beta}^2]\leq1.$$
\el

\begin{proof}
For $T>0$ define the ball
$$
\mathcal{B}_{T} \eqdef \big\{X_{i}\in L^2(\Omega,C((0,T];\bC^{\beta})) \mid  \mathbf{E}[\sup_{t\in[0,T]} t^\gamma\|X_i\|_{\bC^\beta}^2]\leq 1, \quad X_i(0)=\eta_i \big \}.
$$
Here we endow the space $C((0,T];\bC^{\beta})$
with  norm  $(\sup_{t\in[0,T]} t^\gamma\|f(t)\|_{\bC^\beta}^2)^{1/2}$.

	For $X_{i}\in \mathcal{B}_{T}$, define $\mathcal{M}_{T}X_{i}: (0,T] \mapsto \bC^{\beta}$ via
	$$\mathcal{M}_{T}X_{i}(t):=\int_0^tS_{t-s}\E[X_{i}^{2}+2X_{i}Z_{i}+\Wick{Z_{i}^{2}} ](X_{i}+Z_{i} )\dif s+S_t\eta_i.$$
	Using Lemma \ref{lem:heat} and Lemma \ref{lem:multi} noting $\beta>\kappa$, we find that
	\begin{align}
	&\|\mathcal{M}_{T}X_{i}(t)-S_{t}\eta_i\|_{\bC^\beta}\lesssim \int_0^t(t-s)^{-\frac{\beta+\kappa}{2}} \|\E [X_{i}Z_{i}] Z_{i}\|_{\bC^{-\kappa}}\dif s \nonumber\\
	&+\int_0^t \big (\|\E X_{i}^{2} \|_{\bC^\beta}+\|\E\Wick{Z_{i}^{2}}  \|_{\bC^\beta}+(t-s)^{-\frac{\beta+\kappa}{2}}\|\E [X_{i}Z_{i}] \|_{\bC^{-\kappa}}  \big ) \|X_{i}\|_{\bC^\beta}\dif s \nonumber\\
	&+\int_0^t(t-s)^{-\frac{\beta+\kappa}{2}}  \big(\|\E X_{i}^{2}\|_{\bC^\beta}+\|\E\Wick{Z_{i}^{2}} \|_{\bC^\beta} \big) \|Z_{i}\|_{\bC^{-\kappa}} \dif s
	\quad \eqdef \quad \sum_{i=1}^3J_i(t).\nonumber
	\end{align}
	We start by applying Lemma \ref{lem:Z1} to obtain the pathwise bound
	\begin{align}
	J_{1}(t)& \lesssim \int_{0}^{t}(t-s)^{-\frac{\beta+\kappa}{2}} \big (\E\|X_{i}\|_{\bC^{\beta}}^{2} \big )^{\frac{1}{2}}\big (\E[\|\Wick{Z_{i}Z_{j}}\|_{\bC^{-\kappa}}^{2} \mid Z_{i}  ] \big )^{\frac{1}{2} } \dif s \nonumber \\
	& \lesssim \big (\E[\sup_{r\in[0,t]}(r^{\kappa'}\|\Wick{Z_{i}Z_{j}}(r)\|_{\bC^{-\kappa}})^{2} \mid Z_{i}  ] \big )^{\frac{1}{2} }  \int_{0}^{t}(t-s)^{-\frac{\beta+\kappa}{2}}s^{-\frac{\gamma+2\kappa'}{2}}\dif s \nonumber,
	\end{align}
	for $\kappa'>\kappa>0$,
which is integrable provided $\beta<2-\kappa$ and $\gamma<2-2\kappa'$.  We may now apply Lemma \ref{le:ex} to find that
	\begin{equation}
	\E|J_{1}(t)|^{2} \lesssim \E[\sup_{r\in[0,t]}(r^{\kappa'}\|\Wick{Z_{i}Z_{j}}(r)\|_{\bC^{-\kappa} })^{2}]t^{2-(\beta+\kappa+2\kappa'+\gamma )}\lesssim t^{2-(\beta+\kappa+2\kappa'+\gamma) } \nonumber.
	\end{equation}
	Before estimating $J_{2}(t)$ and $J_{3}(t)$, we make three observations.  First note that by Jensen's inequality and Lemma \ref{lem:multi} it holds
	\begin{equation}
	\|\E X_{i}^{2} \|_{\bC^\beta} \lesssim \mathbf{E}[\|X_i\|_{\bC^\beta}^2] \lesssim s^{-\gamma}\nonumber.
	\end{equation}
	Furthermore, using again Lemma \ref{lem:heat} we find
	\begin{equation}
	\|\E\Wick{Z_{i}^{2}}  \|_{\bC^\beta}
	\leq\E\big \| (S_{t}\tilde{z}_{i} )^{2} \big \|_{\bC^{\beta}} \leq \E\big \| S_{t}\tilde{z}_{i} \big \|_{\bC^{\beta}}^{2}\lesssim s^{-(\beta+\kappa) }.\nonumber
	\end{equation}
Finally, note that by Lemma \ref{lem:multi}
\begin{equation}
	\|\E [X_{i}Z_{i}] \|_{\bC^{-\kappa}}
\lesssim \mathbf{E}[\|X_i\|_{\bC^{\beta}}\|Z_i\|_{\bC^{-\kappa}}]
\lesssim \big ( \mathbf{E}[\|X_i\|_{\bC^{\beta}}^{2} ] \big )^{\frac{1}{2}}\big ( \mathbf{E}[ \|Z_i\|_{\bC^{-\kappa}}^{2} ] \big )^{\frac{1}{2}} \lesssim s^{-\frac{\gamma}{2} }.\nonumber
\end{equation}
	Inserting these three bounds, we find the inequalities
	\begin{align}
	J_2(t) &\lesssim \int_0^t \big (s^{-\gamma}+s^{-(\beta+\kappa)} +(t-s)^{-\frac{\beta+\kappa}{2}}s^{-\frac{\gamma}{2}} \big ) \|X_{i}\|_{\bC^\beta}\dif s \nonumber. \\
	J_{3}(t) &\lesssim \int_0^t(t-s)^{-\frac{\beta+\kappa}{2}}  \big(s^{-\gamma}  +s^{-(\beta+\kappa)} \big) \|Z_{i}\|_{\bC^{-\kappa}} \dif s \nonumber.
	\end{align}
	Squaring and taking expectation, we find
	\begin{align}
	\E|J_{2}(t)|^{2}& \lesssim t^{-\gamma}\big ( t^{2-2\gamma}+t^{2-2(\beta+\kappa )}+t^{2-(\beta+\kappa+\gamma)}  \big ) \nonumber.\\
	\E|J_{3}(t)|^{2}& \lesssim \big ( t^{2-(\beta+\kappa+2\gamma )}+t^{2-3(\beta+\kappa)} \big ) \nonumber.
	\end{align}
Under our assumption on $\beta$ these
are all bounded by $t^{-\gamma}$.
	Finally, note that by Lemma \ref{lem:heat} and the embedding $L^4\subset \bC^{-\frac{1}{2}}$, we obtain
	\begin{align}\label{eq:S}\|S_t\eta\|_{\bC^{\beta}}\lesssim t^{-\frac{1+2\beta}{4}}\|\eta\|_{L^4}.\end{align}
	
	Combining the above estimates we can find $T^*$ small enough to have
	\begin{align*}
	\mathbf{E}[\sup_{t\in[0,T^*]}t^\gamma\|\mathcal{M}_{T^*}X(t)\|_{\bC^\beta}^2]\leq 1,
	\end{align*}
	which implies that for $T^*$ small enough $\mathcal{M}_{T^*}$  maps $\mathcal{B}_{T^*}$ into itself. The contraction property follows similarly.
	Now the local existence and uniqueness in $L^2(\Omega;C((0,T],\bC^\beta))$  follows. Furthermore, we know $\int_0^tS_{t-s}\mu(X(s)+Z(s))\dif s$ is continuous in $\bC^\beta$ and $S_t\eta_i\in C_TL^4$. The result follows.
\end{proof}

\subsection{Global well-posedness}\label{sec:4.2}

We now extend our local solution to a global solution through a series of a priori bounds, starting with a uniform in time on the $L^{2}(\Omega; L^{2})$ norm of $X_{i}$ together with an $L^{2}(\Omega ; L^{2}_{T}H^{1})$ bound.

\bl\label{lem:l2} There exists a universal constant $C$ such that

\begin{align}
\sup_{t\in[0,T]}\E \|X_{i}\|_{L^{2}}^{2}+\E \|\nabla X_{i}\|_{L^{2}_TL^{2} }^{2}+\| \E X_{i}^{2} \|_{L^{2}_T L^{2} }^{2}+m\E \| X_{i}\|_{L^{2}_T L^{2}   }^{2}
 \leq C \int_{0}^{T}R \dif t+\E\|\eta_i\|_{L^2}^2, \label{s30}
\end{align}
where, for $i \neq j$ we define
\begin{align}
R\eqdef 1 +\big ( \E\|Z_i\|_{\bC^{-s}}^{2}  \big )^{\frac{2}{1-s}} +\E\|\Wick{Z_{j}^{2}Z_{i}}\|_{\bC^{-s}}^{2}
 +C\big ( \E \|\Wick{Z_{j}Z_{i}}\|_{\bC^{-s}}^{2} \big )^{2} +C\big (\E \|\Wick{Z_{i}^{2}} \|_{\bC^{-s}} \big )^{4}.\nonumber
\end{align}
\el

\begin{proof}
	\newcounter{GlobPsi} % proofstep = 0
	\refstepcounter{GlobPsi} % increases value by 1
	The proof is similar in spirit to the proof of Lemma \ref{Y:L2}, proceeding by energy estimates.
	%It turns out to be convenient to introduce an independent copy of $(X_{i},Z_{i},\Wick{Z_{i}^{2}})$ which we denote by  $(X_{j},Z_{j},\Wick{Z_{j}^{2}})$.
	
	{\sc Step} \arabic{GlobPsi} \label{GlobPse1} \refstepcounter{GlobPsi} (Expected energy balance)
	
	In this step, we establish the following identity
	\begin{align}
	\frac{1}{2}\frac{\dif}{\dif t}\E \|X_{i}\|_{L^{2}}^{2}+\E \|\nabla X_{i}\|_{L^{2}}^{2}+ \|\E X_{i}^{2}\|_{L^{2}}^{2}+m\E\|X_i\|_{L^2}^2 =I^{1}+I^{2}+I^{3}, \label{s20}
	\end{align}
where
	\begin{align}\label{eq:I}
	I^{1}\eqdef \E\langle X_{i},\Wick{ Z_{i}Z_{j}^{2}} \rangle, \qquad
	I^{2}\eqdef \E\langle X_{i}^{2},  \Wick{Z_{j}^{2}}\rangle+2\E\langle X_{i}X_{j}, Z_{i}Z_{j}\rangle, \qquad
	I^{3}\eqdef 3\E\langle X_{i}^{2}X_{j},Z_{j} \rangle .
	\end{align}
	Testing \eqref{eq:X2} with $X_{i}$, integrating by parts and using that $X_{i}$, $X_{i}Z_{i}$, and $\Wick{Z_{i}^{2}}$ are respectively equal in law to $X_{j}$, $X_{j}Z_{j}$, and $\Wick{Z_{j}^{2}}$ we find
	\begin{equs}
		\frac{1}{2}\frac{\dif}{\dif t} &\|X_{i}\|_{L^{2}}^{2} +\|\nabla X_{i}\|_{L^{2}}^{2}+m\|X_i\|_{L^2}^2+\|X_{i}^{2}\E X_{i}^{2}\|_{L^{1}}
		\\
		&=-\langle X_{i},Z_{i}\E(\Wick{Z_{j}^{2}} ) \rangle
		-\langle X_{i}^{2},\E(\Wick{Z_{j}^{2}} )\rangle-2 \langle X_{i},Z_{i}\E(X_{j}Z_{j}) \rangle
		 -2 \langle X_{i}^{2},\E(X_{j}Z_{j}) \rangle-\langle X_{i}\E(X_{j}^{2}),Z_{i} \rangle.
	\end{equs}
	Taking expectation on both sides, using independence, and the fact that $X_{i}^{2}X_{j}Z_{j}$ has the same law as $X_{j}^{2}X_{i}Z_{i}$ we obtain \eqref{s20}. %This argument can be made rigorous with an approximation argument.

	\medskip
	
	{\sc Step} \arabic{GlobPsi} \label{GlobPsi4} \refstepcounter{GlobPsi} (Estimates for $I^{1}$)
	
	In this step, we show there is a universal constant $C$ such that
	\begin{equation}
	I^{1} \leq \frac{1}{4} \big ( \|\E X_{i}^{2}\|_{L^{2}}^{2} +\E\|\nabla X_{i}\|_{L^{2}}^{2} \big )+C\big ( 1+ \E\|\Wick{ Z_{i}Z_{j}^{2}}\|_{\bC^{-s}}^{2} \big ).\label{se2}
	\end{equation}
	To prove the claim, we apply \eqref{s1} to have % followed by two applications of H\"{o}lder's inequality to obtain
	\begin{align}
	I^{1}&\lesssim  \E\big [ \big ( \|X_{i}\|_{L^{1}}^{1-s}\|\nabla X_{i}\|_{L^{1}}^{s}+\|X_{i}\|_{L^{1}}  \big )\|\Wick{Z_{j}^{2}Z_{i}}\|_{\bC^{-s}}\big ]. \nonumber
%	&\lesssim \big (\E\|X_{i}\|_{L^{1}}^{2} \big )^{\frac{1-s}{2} }\big (\E\|\nabla X_{i}\|_{L^{2}}^{2}  \big )^{\frac{s}{2}}  \big ( \E\|\Wick{Z_{j}^{2}Z_{i}}\|_{\bC^{-s}}^{2} \big )^{1/2}
%	 +\big (\E\|X_{i}\|_{L^{1}}^{2}  \big )^{1/2} \big ( \E\|\Wick{Z_{j}^{2}Z_{i}}\|_{\bC^{-s}}^{2} \big )^{1/2} \label{s27}.
	\end{align}
	Hence, \eqref{se2} follows from the inequality $\E\|X_{i}\|_{L^{1}}^{2} \leq \|\E X_{i}^{2}\|_{L^{2}}$ and Young's inequality with exponents $(\frac{2}{1-s},\frac{2}{s},2)$ and $(2,2)$.
	
	\medskip
	
	{\sc Step} \arabic{GlobPsi} \label{GlobPsi2} \refstepcounter{GlobPsi} (Estimates for $I^{2}$ )
	
	In this step, we show that  there is a universal constant $C$ such that
	\begin{align}
	I^{2} \leq \frac{1}{4} \big (\E \|\nabla X_{i}\|_{L^{2}}^{2}+ \|\E X_{i}^{2}\|_{L^{2}}^{2} \big )
	 +C+C\big (\E \|\Wick{Z_{j}Z_{i}}\|_{\bC^{-s}}^{2} \big )^{2}
	+C\big (\E \|\Wick{Z_{j}^{2}} \|_{\bC^{-s}} \big )^{4}.  \label{s26}
	\end{align}
	Using again \eqref{s1}, Young's inequality, H\"{o}lder's inequality and the independence of $X_{i}$ and $X_{j}$ we obtain
	\begin{align}
	\E \langle X_{i}X_{j}, \Wick{Z_{j}Z_{i}} \rangle
	&\lesssim  \E\big (\|X_{i}X_{j}\|_{L^{1}}+\|\nabla X_{i}X_{j} \|_{L^{1}}+\|X_{i}\nabla X_{j} \|_{L^{1}} \big ) \|\Wick{Z_{j}Z_{i}}\|_{\bC^{-s}} \nonumber \\
	&\lesssim  \E \big (\|X_{i}\|_{L^{2}} \|X_{j}\|_{L^{2}}+\|\nabla X_{i}\|_{L^{2}} \| X_{j} \|_{L^{2}} \big ) \|\Wick{Z_{j}Z_{i}}\|_{\bC^{-s}} \nonumber \\
	&\lesssim  \big (\E\|X_{i}\|_{L^{2}}^{2} \E\|X_{j}\|_{L^{2}}^{2}+\E\|\nabla X_{i}\|_{L^{2}}^{2} \E\| X_{j} \|_{L^{2}}^{2} \big )^{1/2} \big ( \E \|\Wick{Z_{j}Z_{i}}\|_{\bC^{-s}}^{2}\big )^{1/2}\nonumber \\
	&\lesssim \big (\|\E X_{i}^{2} \|_{L^{1}}^{2}+\E\|\nabla X_{i}\|_{L^{2}}^{2} \| \E X_{j}^{2}  \|_{L^{1}} \big )^{1/2} \big ( \E \|\Wick{Z_{j}Z_{i}}\|_{\bC^{-s}}^{2}\big )^{1/2}
	\label{s28}.
	\end{align}
Similarly, using this time independence of $X_{i}^{2}$ and $\Wick{Z_{j}^{2}} $ we obtain
\begin{align}
	\E \langle  X_{i}^{2}, \Wick{Z_{j}^{2}} \rangle
	&\lesssim  \E\big (\|X_{i}^{2}\|_{L^{1}}+\|\nabla X_{i}\|_{L^{2}} \|X_{i} \|_{L^{2}}\big ) \|\Wick{Z_{j}^{2}} \|_{\bC^{-s}} \nonumber \\
	&= \big (\|\E X_{i}^{2}\|_{L^{1}}+\E\|\nabla X_{i}\|_{L^{2}} \|X_{i} \|_{L^{2}}\big ) \,\E\|\Wick{Z_{j}^{2}} \|_{\bC^{-s}}\nonumber \\
	&\lesssim \big (\|\E X_{i}^{2}\|_{L^{1}}+\big (\E\|\nabla X_{i}\|_{L^{2}}^{2} \big )^{1/2} \|\E X_{i}^{2}\|_{L^{1}}^{1/2}  \big )\, \E\|\Wick{Z_{j}^{2}} \|_{\bC^{-s}}. \label{s29}
	\end{align}
	To obtain \eqref{s26} we use Young's inequality with exponents  $(2,2)$ and $(2,4,4)$ for both \eqref{s28} \eqref{s29}.
	
	\medskip
	
	{\sc Step} \arabic{GlobPsi} \label{GlobPsi3} \refstepcounter{GlobPsi} (Estimates for $I^{3}$)
	
	In this step, we show there is a universal constant $C$ such that
	\begin{align}
	I^{3} \leq \frac{1}{4} \big (\E \|\nabla X_{i}\|_{L^{2}}^{2}+ \|\E X_{i}^{2}\|_{L^{2}}^{2} \big )
	+C\big(\big ( \E\|Z_j\|_{\bC^{-s}}^{2}  \big )^{\frac{2}{1-s}}+1\big). \label{s21}
	\end{align}
To this end, we write
	\begin{align}
	I^{3}
	&\lesssim  \E\big (  \|  X_{i}^{2}X_{j}  \|_{L^{1}}^{1-s}  \|\nabla ( X_{i}^{2}X_{j})\|_{L^{1}}^{s}+\|  X_{i}^{2}X_{j}  \|_{L^{1}} \big) \|Z_{j}\|_{\bC^{-s}}\nonumber  \\
	& \lesssim (\E \|  X_{i}^{2}X_{j}  \|_{L^{1}}\|Z_{j}\|_{\bC^{-s}})^{1-s}  (\E\|\nabla ( X_{i}^{2}X_{j})\|_{L^{1}}\|Z_{j}\|_{\bC^{-s}})^s +\E\|  X_{i}^{2}X_{j}  \|_{L^{1}}\|Z_{j}\|_{\bC^{-s}}\nonumber .
	\end{align}
	By independence and H\"older's inequality, it holds that
	\begin{align}
	\E \|X_{i}^{2}X_{j}\|_{L^{1}}\|Z_j \|_{\bC^{-s}}
	&\lesssim \big \| \E X_{i}^{2} \E [|X_{j}| \|Z_j\|_{\bC^{-s}}] \big\|_{L^{1}}
	\lesssim \|\E X_{i}^{2}\|_{L^{2}}\big \| (\E X_{j}^{2})^{1/2}  (\E\|Z_j\|_{\bC^{-s}}^{2})^{1/2} \big \|_{L^{2}} \nonumber \\
	&\lesssim \|\E X_{i}^{2}\|_{L^{2}} \big ( \E\|X_{j}\|_{L^{2}}^{2} \big )^{\frac{1}{2}} \big(\E\|Z_j\|_{\bC^{-s}}^{2} \big)^{1/2} \label{s23},
	\end{align}
	where we used that $\|(\E X_{j}^{2})^{1/2}\|_{L^{2}}=\|\E X_{j}^{2}\|_{L^{1}}^{\frac{1}{2}}=\big ( \E\|X_{j}\|_{L^{2}}^{2} \big )^{\frac{1}{2}}$.  Furthermore,
\begin{align}
	{}&\E  \|X_{i}^{2}\nabla X_{j}\|_{L^{1}}    \|Z_j  \|_{\bC^{-s}}
	=\big \| \E X_{i}^{2}\E\big (|\nabla X_{j}|\| Z_j  \|_{\bC^{-s}} \big )  \big \|_{L^{1}} \nonumber \\
	&\leq \| \E X_{i}^{2}\|_{L^{2}} \big \|\E\big (|\nabla X_{j}|\| Z_j  \|_{\bC^{-s}} \big )  \big \|_{L^{2}}
	\leq  \| \E X_{i}^{2}\|_{L^{2}} \big \| (\E |\nabla X_{j}|^{2} )^{\frac{1}{2}} (\E \| Z_j  \|_{\bC^{-s}}^{2} )^{\frac{1}{2}}  \big \|_{L^{2}} \nonumber \\
	&\lesssim\| \E X_{i}^{2}\|_{L^{2}} \big ( \E \|\nabla X_{j}\|_{L^{2}}^{2} \big )^{1/2} \big ( \E \|Z_j  \|_{\bC^{-s}}^{2} \big )^{\frac{1}{2}}\label{s24} .
	\end{align}
	Similarly, note that
\begin{align}
	 \E  \|X_{i}X_{j}\nabla X_{i}\|_{L^{1}} & \|Z_j\|_{\bC^{-s}}
	\lesssim \E  \|X_{i}X_{j}\|_{L^{2}}  \big ( \|\nabla X_{i}\|_{L^{2}}  \|Z_j\|_{\bC^{-s}} \big )\nonumber  \\
	&\lesssim \big (\E  \|X_{i}X_{j}\|_{L^{2}}^{2}  \big )^{1/2}\big ( \E\|\nabla X_{i}\|_{L^{2}}^{2}\E \|Z_j\|_{\bC^{-s}}^{2}  \big )^{1/2} \nonumber \\
	&\lesssim  \|\E X_{i}^{2}\|_{L^{2} }   \big ( \E\|\nabla X_{i}\|_{L^{2}}^{2} \big )^{1/2} \big ( \E \|Z_j  \|_{\bC^{-s}}^{2} \big )^{1/2}\label{s25} .
	\end{align}
	Combining the above estimate we arrive at
	\begin{align}
	I^{3}
	&\lesssim \|\E X_{i}^{2}\|_{L^{2} }\big ( \E\|\nabla X_{i}\|_{L^{2}}^{2} \big )^{\frac{s}{2}}\big (\E\|X_{i}\|_{L^{2}}^{2} \big )^{\frac{1-s}{2}}\big ( \E \|Z_j  \|_{\bC^{-s}}^{2} \big )^{\frac12}
	+\|\E X_{i}^{2}\|_{L^{2}} \big ( \E\|X_{j}\|_{L^{2}}^{2} \big )^{\frac{1}{2}} \big(\E\|Z_j\|_{\bC^{-s}}^{2} \big)^{\frac12} \nonumber
	 \\
	&\lesssim  \|\E X_{i}^{2}\|_{L^{2} }^{\frac{3-s}{2}}   \big ( \E\|\nabla X_{i}\|_{L^{2}}^{2} \big )^{\frac{s}{2}} \big ( \E \|Z_j  \|_{\bC^{-s}}^{2} \big )^{\frac12}+\|\E X_{i}^{2}\|_{L^{2} }^{\frac{3}{2}}    \big ( \E \|Z_j  \|_{\bC^{-s}}^{2} \big )^{\frac12}.
	\label{se55}
	\end{align}
	Applying Young's inequality with exponents $(\frac{4}{3-s},\frac{2}{s},\frac{4}{1-s})$ we arrive at \eqref{s21}.
\end{proof}

In Section ~\ref{sec:dif} we will study the large N limit of  \eqref{eq:Phi2d} by comparing the dynamics of each component to the corresponding mean-field evolution.  To control the equation for the difference, we will need a stronger control on $X_{i}$ than the $L^2$ type bound obtained above.  In the following lemma, we show that $L^p$ bounds can be propagated in time, which will turn out to be a necessary ingredient in Section~\ref{sec:dif}.

\bl\label{lem:Lp}
Let $p> 2$ and assume that $\E\|\eta_i\|_{L^p}^p\lesssim1$.
%It holds that for every $p>1$ satisfying $sp<1$ and $\frac{2}{p}+s<1$
Then, we have
$$\sup_{t\in[0,T]}\E \|X_{i}\|_{L^{p}}^{p}+\E\| |X_{i}|^{\frac{p-2}{2}}\nabla X_{i}\|_{L^{2}_{T}L^{2}}^{2} +\|\E |X_{i}|^{p}\E X_{i}^{2} \|_{L^{1}_{T}L^{1}}  \lesssim1,$$
where the implicit constant is independent of $i$.
\el
\begin{proof}
	
	\newcounter{GlobPsi1} % proofstep = 0
	\refstepcounter{GlobPsi1} % increases value by 1
	Given $p> 2$, we fix  $s>0$ sufficiently small such that
	$sp<\frac{1}{2}$ and $\frac{2}{p}+s<1$.
	%We omit the subscripts $i$ for simplicity of notation,
	%namely we write $X$ for $X_i$, $Z$ for $Z_i$, and $\bX$, $\bZ$ for their independent copies.
	We will perform an $L^p$ estimate: integrating \eqref{eq:X2} against $|X_i|^{p-2}X_i$ we get
	\begin{align}\label{eq:Lp}
	&\frac{1}{p}\frac{\dif}{\dif t}\E \|X_i\|_{L^{p}}^{p}
	+(p-1)\E \||X_i|^{p-2}|\nabla X_i|^2\|_{L^1}+\E\||X_i|^p X_j^2\|_{L^1}+m\E\|X_i\|_{L^p}^p\nonumber
	\\=&-2\E \Big\< \E [X_jZ_j],|X_i|^p\Big\>
	-\E\Big\<\E[X_j^2]|X_i|^{p-2}X_i,Z_i\Big\>
	-2\E \Big\< \E[X_jZ_j]Z_i,|X_i|^{p-2}X_i \Big\> \nonumber
	\\
	&\qquad +\E\Big\<\E[\Wick{Z_j^2}],|X_i|^p \Big\>+\E\Big\< \E[\Wick{Z_j^2}]X_i|X_i|^{p-2},Z_i \Big\>
	=: \sum_{k=1}^5I_k.
	\end{align}
Set $D\eqdef \|X_i^{p-2}|\nabla X_i|^2\|_{L^1}$ and $A\eqdef \|X_i^pX_j^2\|_{L^1}$.
We claim that there is some $R$ so that % \hao{I added $+C$}
	\begin{align}\label{e:EAED-R}
	I_k\le \frac1{10} \E A+ \frac1{10} \E D+ (\E[\|X_i^{p}\|_{L^1}]+C) R,
	\qquad
	\mbox{with }\int_0^TR \lesssim1.
	\end{align}

	{\sc Step} \arabic{GlobPsi1} \label{GlobPse1z} \refstepcounter{GlobPsi1} (Estimate of $I_1$)
	
Using Lemma~\ref{lem:dual+MW}, we have
	\begin{align*}
	I_1  %=-2\E\int X_jZ_jX_i^p\, \dif x
	% \lesssim \E \Big[ \|X_jX_i^p\|_{B^{s}_{1,1}}\|Z_j\|_{\bC^{-s}} \Big] \notag
	\lesssim \E \Big[ \|X_j X_i^p\|_{L^1}^{1-s}\|\nabla (X_j |X_i|^p)\|_{L^1}^{s}\|Z_j\|_{\bC^{-s}}\Big]
	+\E\Big[\|X_j X_i^p\|_{L^1}\|Z_j\|_{\bC^{-s}}\Big]
	=: I_1^{(1)}+I_1^{(2)}.% \label{e:5.4I1}
	\end{align*}
Using
\begin{equ}[e:A12Xp12]
\|X_j X_i^p\|_{L^1}\lesssim A^{1/2}\|X_i^p\|_{L^1}^{1/2}
\end{equ}
 and independence, one has %H\"older and Young's inequality,
	\begin{align*}
	I_1^{(2)} %\le \E \Big[\|\bX X^p\|_{L^1} & \|\bZ\|_{\bC^{-s}}\Big]
	\leq\E\left[ A^{1/2} \|X_i^{p}\|_{L^1}^{1/2}\|Z_j\|_{\bC^{-s}}\right]
	\le \frac{1}{10}\E A+ C \E\|X_i^{p}\|_{L^1}\E\|Z_j\|_{\bC^{-s}}^2.
	\end{align*}

Regarding $I_1^{(1)}$,  using  H\"older inequality
and then Gagliardo-Nirenberg with $(s,q,r,\alpha)=(0,4,2,\frac12)$,
\begin{align*}
	\|\nabla (X_j |X_i|^p)\|_{L^1}
	& \le
	 \|\nabla  X_j\|_{L^{2}}\||X_i|^{\frac{p}{2}}\|_{L^4}^2
	+2\||X_i|^{\frac{p}{2}}X_j\|_{L^{2}}\|\nabla |X_i|^{\frac{p}{2}}\|_{L^{2}} \\
	& \lesssim
	\||X_i|^{\frac{p}{2}}\|_{H^1} \||X_i|^{\frac{p}{2}}\|_{L^2} \|X_j\|_{H^1}
	+ \sqrt{AD}
\end{align*}
Since $\||X_i|^{\frac{p}{2}}\|_{H^1} \lesssim D^{\frac12}+ \||X_i|^{\frac{p}{2}}\|_{L^2}$,	
together with \eqref{e:A12Xp12}
one has
\begin{equs}
I_1^{(1)}
&	\lesssim
\E \Big[A^{\frac{1-s}{2}}\|X_i^p\|_{L^1}^{\frac{1-s}{2}}
\Big(
D^{\frac{s}{2}} \|X_i^p\|_{L^1}^{\frac{s}{2}} \|X_j\|_{H^1}^s
+ \|X_i^p\|_{L^1}^{s}  \|X_j\|_{H^1}^s
 +A^{\frac{s}{2}}D^{\frac{s}{2}}\Big)
 \|Z_j\|_{\bC^{-s}}\Big]
\\
 	&\le
	\frac{1}{10}\E A+ \frac{1}{10}\E D
	+C\E\|X_i^p\|_{L^1}
	 \Big(
	 \E\|X_j\|_{H^1}^{2s} \|Z_j\|_{\bC^{-s}}^{2}
	+\E\|X_j\|_{H^1}^{\frac{2s}{1+s}}\|Z_j\|_{\bC^{-s}}^{\frac{2}{1+s}}
	+\E\|Z_j\|_{\bC^{-s}}^{2/(1-s)}\Big),
 \end{equs}
 where  in the last inequality we used independence and
	Young's inequality for products with exponents $(\frac{2}{1-s},2,\frac{2}{s})$ for the first and third term,
and exponents $(\frac{2}{1-s},\frac{2}{1+s})$ for the second term. 	Therefore
invoking Lemma~\ref{lem:l2} to deduce
$$\int_0^T \Big(
\E\|X_j\|_{H^1}^{2s} \|Z_j\|_{\bC^{-s}}^{2}
+\E\|X_j\|_{H^1}^{\frac{2s}{1+s}}\|Z_j\|_{\bC^{-s}}^{\frac{2}{1+s}}\Big)\dif s\lesssim 1,$$
which implies a bound of the form \eqref{e:EAED-R}.

	\medskip
	
	{\sc Step} \arabic{GlobPsi1} \label{GlobPsi4z} \refstepcounter{GlobPsi1} (Estimates for $I_{2}$)

	For the second term on the right hand side of \eqref{eq:Lp} we use Lemma \ref{lem:interpolation} to have
	\begin{align*}
	&I_2
	=\E\< \Lambda^s(X_j^2X_i|X_i|^{p-2}),\Lambda^{-s}Z_i \>
	\\
	&\lesssim\E\Big[\|\Lambda^s(X_j^2)\|_{L^p}\|X_i^{p-1}\|_{L^{\frac{p}{p-1}}}\|\Lambda^{-s}Z_i\|_{L^\infty}\Big]
	+\E\Big[\|X_j^2\|_{L^{p}}\|\Lambda^s(X_i|X_i|^{p-2})\|_{L^{\frac{p}{p-1}}}\|\Lambda^{-s}Z_i\|_{L^\infty}\Big]
	\\
	&=:I_2^{(1)}+I_2^{(2)}.
	\end{align*}
	%where $\frac{1}{p_1}+\frac{1}{p_2}=1$.
	%
	%For $I_2$ we use \cite[Proposition 3.25]{MW17} to have
	%\begin{align*}&I_2=\E\int[X_j^2X_i^{p-1}Z_i]\\
	%\lesssim&\E\|X_j^2X_i^{p-1}\|_{L^1}^{1-s}\|\nabla (X_j^2X_i^{p-1})\|_{L^1}^s\|Z_i\|_{\bC^{-s}}+\E\|X_j^2X_i^{p-1}\|_{L^1}\|Z_i\|_{\bC^{-s}}
	%\end{align*}
	%For the second term we have
	%\begin{align*}
	%&\E\|X_j^2X_i^{p-1}\|_{L^1}\|Z_i\|_{\bC^{-s}}\lesssim \E\|X_j^2\|_{L^p}\E\|X_i^{p-1}\|_{L^{\frac{p}{p-1}}}\|Z_i\|_{\bC^{-s}}
	%\\\lesssim &\E\|X_j\|_{H^1}^2\E\|X_i^p\|_{L^{1}}^{\frac{p-1}{p}}\|Z_i\|_{\bC^{-s}}\lesssim \E\|X_j\|_{H^1}^2\left(\E\|X_i^p\|_{L^{1}}+\E\|Z_i\|_{\bC^{-s}}^p\right).
	%\end{align*}
	%
	%\rmk{If we use \cite[Proposition 3.25]{MW17} we  meet the following term
	%$\|\nabla (X_j^2X_i^{p-1})\|_{L^1}$, which seems not easy to handle. }
	%
	%
	%
Using  independence, $\|\Lambda^s(X_j^2)\|_{L^p}
	\lesssim \|\Lambda^s X_j\|_{L^{2p}} \| X_j\|_{L^{2p}}$ by \eqref{e:Lambda-prod},
and Lemma~\ref{lem:emb}(iii) with $sp<1$,
$$
I_2^{(1)}\lesssim  \E\Big[\|X_i\|_{L^p}^{p-1}\|\Lambda^{-s}Z_i\|_{L^\infty}\Big]\E\Big[\|\Lambda^s(X_j^2)\|_{L^p}\Big]
\lesssim (\E[\|X_i\|_{L^p}^p]+1)\E[\| X_j\|_{H^1}^2].
$$

	%
	%For the second term we use independence  to have
	%\begin{align*}
	%I_2^{(2)}=\E[\|X_j^2\|_{L^{p}}]\E[\|\Lambda^s(X_i^{p-1})\|_{L^{\frac{p}{p-1}}}\|\Lambda^{-s}Z_i\|_{L^\infty}].
	%\end{align*}

	Regarding $I_2^{(2)}$, by the interpolation Lemma \ref{lem:interpolation} followed by H\"older's inequality,
	\begin{align}\label{z1}
	\|\Lambda^s(X_i|X_i|^{p-2})\|_{L^{\frac{p}{p-1}}} & \lesssim\|\nabla(X_i|X_i|^{p-2})\|_{L^{\frac{p}{p-1}}}^s\|X_i^{p-1}\|_{L^{\frac{p}{p-1}}}^{1-s}
	+\|X_i^{p-1}\|_{L^{\frac{p}{p-1}}}\nonumber
	\\ & \lesssim D^{s/2}\|X_i^p\|_{L^{1}}^{(1-s)\frac{p-1}{p}+\frac{s(p-2)}{2p}}+\|X_i^p\|_{L^{1}}^{\frac{p-1}{p}}.
	\end{align}
	%Also, by Sobolev embedding $H^{\beta}\subset L^{2p}$, $\beta=1-\frac{1}{p}<1$ and interpolation Lemma \ref{lem:interpolation}, followed by H\"older inequality with exponents $(\frac{1}{\beta},\frac{1}{1-\beta})$, we have
By \eqref{e:Gagliardo} with $(q,s,\alpha,r)=(2p,0,\beta,2) $ with $\beta\eqdef 1-\frac{1}{p}$, and then H\"older inequality
	\begin{align}\label{z2}
	\E \|X_j^2\|_{L^{p}}
	\lesssim
	\E \Big[ \|X_j\|_{H^1}^{2\beta}\| X_j\|_{L^{2}}^{2(1-\beta)} \Big]
	\lesssim
	%  \|\bX\|_{H^1}^{2} + \| \bX\|_{L^{2}}^{2} .
	(\E\|X_j\|_{H^1}^{2} )^\beta (\E\| X_j\|_{L^{2}}^{2} )^{1-\beta}.
	\end{align}
Recall from Lemma \ref{lem:l2} that $\E[\| X_j\|_{L^{2}}^{2}]\lesssim1$.
With  \eqref{z1}-\eqref{z2}, using again independence,
	and H\"{o}lder's inequality with exponents $(\frac{2}{s},\frac{2}{2-s})$,
	%Young's inequality for products,
	together with Lemma \ref{le:ex}, we obtain
\begin{align*}
I_2^{(2)}
\lesssim
	\big(\E\|X_j\|_{H^1}^{2} \big)^\beta
	\big(\E D\big)^{\frac{s}{2}}
	\E\Big[\|X_i^p\|_{L^{1}}^{\eta}\|\Lambda^{-s}Z_i\|_{L^\infty}^{\frac{2}{2-s}}\Big]^{1-\frac{s}{2}}
	+\E[\|X_j\|_{H^1}^{2}]^\beta (\E\|X_i^p\|_{L^{1}}\E\| X_j\|_{L^{2}}^{2} +1)
\end{align*}
	where $\eta\eqdef (1-\frac{1}{p}-\frac{s}{2})\frac{2}{2-s}$ and clearly $\eta<1$.
	The first term on the RHS can be bounded by, using Young's inequality with exponents $(\frac{2}{s},\frac{2}{2-s})$ and then with exponents $(\frac{1}{\eta},\frac{1}{1-\eta})$,
	\begin{align*}
	\frac{1}{10}\E[D] & + C \E \Big[\|X_i^p\|_{L^{1}}^{\eta}\|\Lambda^{-s}Z_i\|_{L^\infty}^{\frac{2}{2-s}}\Big]\,
	\E\Big[\| X_j\|_{H^1}^{2}\Big]^{\frac{2\beta}{2-s}} %+(\E[\|X_i^p\|_{L^{1}}]+1)\E[\|X_j\|_{H^1}^{2}]^\beta
	\\ & \le \frac{1}{10}\E[D]+C \Big(\E \|X_i^p\|_{L^{1}}+\E \|\Lambda^{-s}Z_i\|_{L^\infty}^{\frac{2}{2-s}\frac{1}{1-\eta}}\Big)\E\Big[\| X_j\|_{H^1}^{2}\Big]^{\frac{2\beta}{2-s}} %+(\E[\|X_i^p\|_{L^{1}}]+1)\E[\|X_j\|_{H^1}^{2}]^\beta
%	\\ & \le \varepsilon\E[D]+ C (\E \|X_i^p\|_{L^{1}}+1)\E[\| X_j\|_{H^1}^{2}]^{\frac{2\beta}{2-s}} %+(\E[\|X_i^p\|_{L^{1}}]+1)\E[\|X_j\|_{H^1}^{2}]^\beta
	\end{align*}
	By Lemma \ref{lem:emb} and Lemma \ref{le:ex} we easily find $\E \|\Lambda^{-s}Z_i\|_{L^\infty}^{q}\lesssim1$ for every $q\geq1$.
Using Lemma \ref{le:ex}, and Lemma~\ref{lem:l2} noting that
 $2\beta/(2-s)<1$ by our smallness assumption on $s>0$,
we obtain a bound of the form \eqref{e:EAED-R} for $I_2$.

	\medskip
	
	{\sc Step} \arabic{GlobPsi1} \label{GlobPsi5z} \refstepcounter{GlobPsi1} (Estimate of $I_3$-$I_5$)

	Using Lemma \ref{lem:interpolation} we obtain
	\begin{align*}
	&I_3=\E \< \Lambda^s(X_j X_i|X_i|^{p-2}),\Lambda^{-s}(\Wick{Z_i Z_j})\>
	\\
	& \lesssim\E\Big[\|\Lambda^{-s}(\Wick{Z_i Z_j})\|_{L^\infty}\|\Lambda^s X_j\|_{L^p}\|X_i^p\|_{L^1}^{\frac{p-1}{p}}\Big]
	\\&\qquad+\E\Big[\|\Lambda^{-s}(\Wick{Z_i Z_j})\|_{L^\infty}\| X_j\|_{L^{p}}\|\Lambda^s(X_i|X_i|^{p-2})\|_{L^{\frac{p}{p-1}}}\Big]
	\quad =:I_3^{(1)}+I_3^{(2)}.
	\end{align*}
	For $I_3^{(1)}$ we use Sobolev embedding $H^1\subset H^s_p$ and Young's inequality and independence to have
	\begin{align*}
	I_3^{(1)}\lesssim \E[\|\Lambda^{-s}(\Wick{Z_i Z_j})\|_{L^\infty}^p]+\E[\| X_j\|_{H^1}^{\frac{p}{p-1}}]\,\E[\|X_i^p\|_{L^1}].
	\end{align*}
For $I_3^{(2)}$ we plug in \eqref{z1}: %for $p_1<\frac{p}{p-1}$, $p_2>p$
\begin{align*}
I_3^{(2)} &\lesssim
	\E\Big[\Big(D^{s/2}\|X_i^p\|_{L^{1}}^{(2p-2-sp)/(2p)}      %^{(1-s)\frac{p-1}{p}+\frac{s(p-2)}{2p}}
	+\|X_i^p\|_{L^{1}}^{\frac{p-1}{p}}\Big)\|X_j\|_{L^{p}}\|\Lambda^{-s}(\Wick{Z_i Z_j})\|_{L^\infty}\Big]
\end{align*}
Using Young's inequality with $(\frac{2}{s},\frac{2p}{2p-2-sp},p)$ and $(\frac{p}{p-1},p)$,
and Sobolev embedding,
	\begin{align*}
I_3^{(2)} \le	\frac{1}{10}\E D+ C \E\|X_i^p\|_{L^{1}}\E\| X_j\|_{H^1}^{\frac{2p}{2p-2-sp}}
+\E \| X_j\|_{H^1}^{\frac{p}{p-1}} \E \|X_i^p\|_{L^1}
+ C \E\|\Lambda^{-s}(\Wick{Z_i Z_j})\|_{L^\infty}^{p}.
	\end{align*}
For $s>0$ small enough $\frac{2}{p}+s<1$ so that $\frac{2p}{2p-2-sp}<2$,
so  Lemma~\ref{lem:l2}  applies.

	By \eqref{z1} we have for $\epsilon>0$ small enough
	\begin{align*}
	&I_4+I_5
	\\
	& \lesssim \|\E[\Wick{Z_j^2}]\|_{L^\infty}\E[\|X_i\|_{L^p}^p]+\|\E[\Wick{Z_j^2}]\|_{\bC^{s+\epsilon}}\E[\|\Lambda^s (X_i|X_i|^{p-2})\|_{L^1}\|\Lambda^{-s}Z_i\|_{L^\infty}]
	\\ & \lesssim \|\E[\Wick{Z_j^2}]\|_{L^\infty}\E[\|X_i\|_{L^p}^p]+\|\E[\Wick{Z_j^2}]\|_{\bC^{s+\epsilon}}
	\E[D^{s/2}\|X_i^p\|_{L^{1}}^{(1-s)\frac{p-1}{p}+\frac{s(p-2)}{2p}}\|\Lambda^{-s}Z_i\|_{L^\infty}]\\
	&\qquad\qquad+\|\E[\Wick{Z_j^2}]\|_{\bC^{s+\epsilon}}
	\E[\|X_i^p\|_{L^{1}}^{\frac{p-1}{p}}\|\Lambda^{-s}Z_i\|_{L^\infty}]
	\\&\lesssim\E[D]^{s/2}\E[\|X_i^p\|_{L^{1}}^{\eta}\|\Lambda^{-s}Z_i\|_{L^\infty}^{\frac{2}{2-s}}]^{1-s/2}\|\E[\Wick{Z_j^2}]\|_{\bC^{s+\epsilon}}
	+\|\E[\Wick{Z_j^2}]\|_{\bC^{s+\epsilon}}(\E[\|X_i\|_{L^p}^p]+1)
	\\&\le \frac{1}{10}\E[D]+ C \E[\|X_i^p\|_{L^{1}}]+ C \|\E[\Wick{Z_j^2}]\|_{\bC^{s+\epsilon}}^{\frac{2q}{2-s}}\E[\|\Lambda^{-s}Z_i\|_{L^\infty}^{\frac{2q}{2-s}}]
	+ C \|\E[\Wick{Z_j^2}]\|_{\bC^{s+\epsilon}}(\E[\|X_i\|_{L^p}^p]+1),
	\end{align*}
	for $q=1/(1-\eta)$.
Combining all the above estimates and using Gronwall's inequality, we obtain the claimed bound.	
\end{proof}

We now conclude this section by combining our energy estimates with Schauder theory to obtain a global H\"{o}lder bound on $X_{i}$.

\bl\label{lem:priori} Assume that $\E\|\eta_i\|_{L^4}^4\lesssim1$. For $\beta>\kappa$ sufficiently small,
$\gamma=\beta+\frac12$, we have
$$\mathbf{E}[\sup_{t\in[0,T]}t^\gamma\|X_i\|_{\bC^\beta}^2]\lesssim1.$$
\el
\begin{proof}
	Recall that $X_{i}$ satisfies the mild formulation of \eqref{eq:X2}, which we write using our independent copy $(X_{j},Z_{j})$ as
	$$X_{i}(t)=S_t\eta_i+\int_0^tS_{t-s}\E[X_{j}^{2}+2X_{j}Z_{j}+\Wick{Z_{j}^{2}} ](X_{i}+Z_{i} )\dif s.$$
	We start by applying the Schauder estimate, Lemma \ref{lem:heat}, with $\delta$ playing the role $\beta+1$, $\beta+\kappa$, and $\beta+\kappa+\frac{2}{3}$ respectively to bound
\begin{align*}
\|& X_i(t)-S_{t}\eta_{i} \|_{\bC^\beta}
\lesssim
\int_0^t(t-s)^{-\frac{\beta+1}{2}}\|\E[X_j^2]X_i\|_{\bC^{-1}}\dif s \\
&+\int_0^t(t-s)^{-\frac{\beta+\kappa}{2}}\|\E[\Wick{Z_j^2}](X_i+Z_i)\|_{\bC^{-\kappa}}\dif s
	+\int_0^t(t-s)^{-\frac{\beta+2/3+\kappa}{2}}\|\E[X_jZ_j] Z_i\|_{\bC^{-\kappa-\frac{2}{3}}}\dif s
	\\&+\int_0^t(t-s)^{-\frac{\beta+2/3+\kappa}{2}}\Big(\|\E[X_jZ_j]X_i\|_{\bC^{-\kappa-\frac{2}{3}} }+\|\E[X_j^2]Z_i\|_{\bC^{-\kappa-\frac{2}{3}} } \Big)\dif s
	\quad \eqdef \sum_{i=1}^4J_i.
\end{align*}
	To estimate $J_{1}$, first recall Lemma \ref{lem:Lp} implies that
	\begin{equation}\label{bd:mu1}
	\sup_{t\in[0,T]}\|\E X_j^2\|_{L^2}^2=\sup_{t\in[0,T]}\|\E X_i^2\|_{L^2}^2 \lesssim \sup_{t\in[0,T]}\E \|X_i\|_{L^4}^4\lesssim1,
	\end{equation}
	which can be combined with the Sobolev embedding $L^{2} \hookrightarrow \bC^{-1}$ in $d=2$ corresponding to Lemma \ref{lem:emb} with $\alpha=0$, $p_{1}=q_{1}=2$ and $p_{2}=q_{2}=\infty$ to find
	\begin{align*}
	J_1\lesssim \int_0^t(t-s)^{-\frac{\beta+1}{2}}\|\E[X_j^2]\|_{L^2}\|X_i\|_{\bC^\beta}\dif s \lesssim \int_0^t(t-s)^{-\frac{\beta+1}{2}}\|X_i\|_{\bC^\beta}\dif s.
	\end{align*}
	We now turn to $J_{2}$ and apply Lemma \ref{lem:multi} to find for $\beta>\kappa$ and $\kappa'>\kappa$
	\begin{align*}
	J_2\lesssim&\int_0^t(t-s)^{-\frac{\beta+\kappa}{2}}[\|\E[\Wick{Z_j^2}]\|_{\bC^{-\kappa} }\|X_i\|_{\bC^\beta}+\|\E[\Wick{Z_j^2}]\|_{\bC^{2\kappa}}\|Z_i\|_{\bC^{-\kappa}}\dif s\\
	\lesssim& \int_0^t(t-s)^{-\frac{\beta+\kappa}{2}}s^{-\frac{\kappa'}{2}}\|X_i\|_{\bC^\beta}\dif s+\|Z_i\|_{C_T\bC^{-\kappa}}.
	\end{align*}
	We now turn to $J_{3}$ and $J_{4}$ and  use the Besov embedding $B^{-\kappa}_{3,\infty} \hookrightarrow \bC^{-\kappa-\frac{2}{3}}$ in $d=2$ in  Lemma \ref{lem:emb}. %with $\alpha=-\kappa$, $p_{1}=3$ and $q_{1}=p_{2}=q_{2}=\infty$.
	Let's begin with $J_{3}$ which is simpler. Using  Lemma \ref{lem:Lp} and Lemma \ref{lem:emb} and Lemma \ref{lem:interpolation} we have
	\begin{align} &\int_0^T\|\E[X_i^2]\|_{B_{3,\infty}^{2\kappa}}^2\dif s\lesssim\int_0^T[\mathbf{E}\|X_i^2\|_{B_{3,\infty}^{2\kappa}}]^2\dif s\lesssim\int_0^T[\mathbf{E}\|\Lambda^{2\kappa}(X_i^2)\|_{L^3}]^2\dif s\lesssim\int_0^T[\mathbf{E}\|\Lambda(X_i^2)\|_{L^{4/3}}]^2\dif s\no
	\\\lesssim&\int_0^T[\mathbf{E}\|X_i\|_{H^1}\|X_i\|_{L^4}]^2\dif s
	\lesssim  \int_0^T\mathbf{E}[\| X_i\|_{H^1}^2]\mathbf{E}[\| X_i\|_{L^4}^2]\dif s\lesssim1,\label{sz1} \end{align}
	where we used \eqref{e:Lambda-prod} in the fourth inequality and H\"{o}lder inequality in the fifth inequality.
	Note that by H\"{o}lder's inequality in time which exponents $(\frac{3}{2},3)$ and taking into account that $\frac{3}{4}(\beta+\frac{2}{3}+\kappa)<1$ for $\beta$ small enough and using Lemma \ref{lem:multi} we find
	\begin{align}\label{bd:mu2Z}
	J_{3} &\lesssim \int_0^T\|\E[X_jZ_j] Z_i\|_{B^{-\kappa}_{3,\infty}}^3\dif s \lesssim\int_0^T(\mathbf{E}[\|X_j\|_{B^{2\kappa}_{3,\infty}}^2])^{\frac32}(\E[\|\Wick{Z_i Z_j}\|_{\bC^{-\kappa}}^2\mid Z_{i}])^{\frac32}\dif s\nonumber\\
	&\lesssim 1+\int_0^T(\E[\|\Wick{Z_i Z_j}\|_{\bC^{-\kappa}}^2\mid Z_{i}])^{6}\dif s.
	\end{align}
	Here we used Lemma \ref{lem:Z1} and \eqref{sz1}.
	Finally, we turn to $J_{4}$.  By Lemma \ref{lem:multi} and Lemma \ref{lem:interpolation} we deduce
	\begin{align*}
	\|X_i\|_{B^{2\kappa}_{3,\infty}}\lesssim \|X_i\|_{B^{2\kappa}_{\frac{4}{1+2\kappa},\infty}}\lesssim \|X_i\|_{B^{1}_{2,\infty}}^{2\kappa}\|X_i\|_{B^{0}_{4,\infty}}^{1-2\kappa}\lesssim \|X_i\|_{H^1}^{2\kappa}\|X_i\|_{L^4}^{1-2\kappa},
	\end{align*}
	which implies that
	\begin{align}\label{bd:mu2}&\int_0^T\|\E(X_iZ_i)\|_{B^{-\kappa}_{3,\infty}}^3\dif s\lesssim \int_0^T\mathbf{E}[\|X_i\|_{B^{2\kappa}_{3,\infty}}^3\|Z_i\|_{\bC^{-\kappa}}^3]
	\lesssim \int_0^T\mathbf{E}[\| X_i\|_{H^1}^{6\kappa }\|X_i\|_{L^4}^{3(1-2\kappa)}\|Z_i\|_{\bC^{-\kappa}}^3]\nonumber
	\\\lesssim&\int_0^T\mathbf{E}[\| X_i\|_{H^1}^{2}]+\int_0^T\mathbf{E}[\|X_i\|_{L^4}^{4}]+\int_0^T\mathbf{E}[\|Z_i\|_{\bC^{-\kappa}}^{l}]\lesssim1,
	\end{align}
	for some $l>1$.  This combined with \eqref{sz1} and Lemma \ref{lem:multi} implies that
	\begin{align*}
J_4\lesssim&\int_0^t(t-s)^{-\frac{\beta+2/3+\kappa}{2}} \Big(\|\E[X_jZ_j]\|_{B^{-\kappa}_{3,\infty}}\|X_i\|_{\bC^\beta}
	+\|\E[X_j^2]\|_{B_{3,\infty}^{2\kappa}}\|Z_i\|_{\bC^{-\kappa}}\Big)\dif s
	\\\lesssim&\int_0^t(t-s)^{-\frac{\beta+2/3+\kappa}{2}}\|\E[X_jZ_j]\|_{B^{-\kappa}_{3,\infty}}\|X_i\|_{\bC^\beta}\dif s
	+\|Z_i\|_{C_T\bC^{-\kappa}}\end{align*}
	%Using \eqref{bd:mu2} we have for some $\eta>1$
	%$$\int_0^t(t-s)^{-\frac{\eta(\beta+1/2+\kappa)}{2}}\|\mu_2\|_{B^{-\kappa}_{4,\infty}}^\eta\dif s\lesssim1.$$
	For $S_t\eta_i$ we  use \eqref{eq:S} to have the desired bound.
	Combining the above estimates, using \eqref{bd:mu2},  H\"{o}lder's inequality and Gronwall's inequality, 	 the result follows.
\end{proof}

Combining the local well-posedness result and the uniform estimate Lemma \ref{lem:priori} we conclude the following result:

\bt\label{Th:global}
For given $Z_i$ as the solution to \eqref{eq:li1} and $\E\|\eta_i\|_{L^4}^4\lesssim1$, there exists a unique solution $X_i\in L^2(\Omega;C((0,T];\bC^\beta)\cap C_TL^4)$ to \eqref{eq:X2} such that
$$\mathbf{E}[\sup_{t\in[0,T]}t^\gamma\|X_i\|_{\bC^\beta}^2]+\sup_{t\in[0,T]}\E\|X_i\|_{L^4}^4+\E\|X_i\|_{L_T^2H^1}^2\lesssim1.$$
In particular, for every $\psi_i\in \bC^{-\kappa}$ with $\E\|\psi_i\|_{\bC^{-\kappa}}^p\lesssim1$, $p>1$,   there exists a unique solution $\Psi_i\in L^2(\Omega;C_T\bC^{-\kappa})$ to \eqref{eq:Psi2} such that
$$\mathbf{E}[\sup_{t\in[0,T]}t^\gamma\|\Psi_i-Z_i\|_{\bC^\beta}^2]\lesssim1,$$
for $\beta>3\kappa>0$ small enough and  $\gamma=\frac{1}{2}+\beta$.
\et

\section{Large N limit of the dynamics}\label{sec:dif}
In this section we study the large N behavior of a fixed component $\Phi_{i}^{N}$ satisfying \eqref{eq:Phi2d} with initial condition $\phi_i^N=y_i^N+z_i^N$.  Namely, under suitable assumptions on the initial conditions, we show that as $N \to \infty$, the component converges to the corresponding solution $\Psi_i$ to \eqref{eq:Psi2} with initial condition $\psi_i=\eta_i+z_i$.  Recall that by definition, $\Phi_{i}^{N}=Y_{i}^{N}+Z_{i}^{N}$, where $Y_{i}^{N}$ satisfies \eqref{eq:22} and $Z_{i}^{N}$ satisfies \eqref{eq:li1} with initial conditions $y_{i}^{N}$ and $z_{i}^{N}$ respectively.  Similarly,  $\Psi_{i}=X_{i}+Z_{i}$, where $X_{i}$ satisfies \eqref{eq:X2} and $Z_{i}$ satisfies \eqref{eq:li1} with initial conditions $y_{i}$ and $\eta_{i}$ respectively.  We now define
\begin{equation}
v_{i}^{N}\eqdef Y_{i}^{N}-X_{i} \nonumber.
\end{equation}
%which satisfies
%\begin{align*}
%\LL v_i^{N} & =-\frac{1}{N}\sum_{j=1}^N \Wick{\Phi_j^2\Phi_i}+\mu \Psi_i, \quad v_{i}^{N} (0)=y_{i}^{N}-\eta_{i}.
%\\
%& =-\frac{1}{N}\sum_{j=1}^N\Phi_j^2(\Phi_i-\Psi_i)
%	- \frac{1}{N}\sum_{j=1}^N(\Phi_j^2-\Psi_j^2)\Psi_i
%	-(\frac{1}{N}\sum_{j=1}^N\Psi_j^2-\mu)\Psi_i
%\\
%&=-\frac{1}{N}\sum_{j=1}^N(Y_j^2+2Y_jZ_j^N+\Wick{ (Z_j^{N})^2})(v_i+u_i)
%\\
%&\qquad-\frac{1}{N}\sum_{j=1}^N(v_j+u_j)(\Phi_j+\Psi_j)\Psi_i
%-(\frac{1}{N}\sum_{j=1}^N\Psi_j^2-\mu)\Psi_i \;.
%\end{align*}
For future reference, we note that in light of the decomposition, c.f. Section \ref{sec:re} for the definition of $\tilde{Z}_{i}$,
$$Z_i^N=\tilde{Z}_i+S_t(z_i^N-\tilde{Z}_i(0)),\quad Z_i=\tilde{Z}_i+S_t(z_i-\tilde{Z}_i(0)),$$
it follows that
\begin{equation}
\Phi_i^N-\Psi_i=Y_i^N-X_i+Z_i^N-Z_i=v_i^{N}+S_t(z_i^N-z_i). \nonumber
\end{equation}
Hence, our main task is to study $v_{i}^{N}$ and this will occupy the bulk of the proof.  We now give our assumptions on the initial conditions.
\begin{assumption}
	\label{a:main}Suppose the following assumptions:
	\begin{itemize}
		\item The random variables $\{( z_{i},\eta_{i}) \}_{i=1}^{N}$ are iid.%exchangeable and
		%$\{(z_i, \eta_i) \}_{i=1}^{N}$ are independent.
		%\item $z_i^N=z_i$ for $i=1,\dots , N$. 
		%\zhu{We add this condition. }\hao{OK.}
		%\zhu{Here we add the same distribution condition to use $v_i=^dv_j$ and $\E\|v_i\|_{L^2}^2=\frac1N\sum_{j=1}^N\E\|v_j\|_{L^2}^2$. but it's not necessary.}\hao{OK. We will need to decide whether we should change $X_j$ to $\bar X_i$ like what we did in the previous section.}\zhu{in this section we keep the notation as $X_i, X_j$}
		\item  For every $p>1$, and 
		every $i$,
		$$\mathbf{E}[\|z_i^N-z_i\|_{\bC^{-\kappa}}^p]\rightarrow0, \quad
		\mathbf{E}[\|y_i^N-\eta_i\|_{L^2}^2]\to 0,%\qquad\frac1N\sum_{i=1}^N\|y_i^N-\eta_i\|_{L^2}^2\to^{\mathbf{P}} 0 
		\qquad \mbox{as } N\to \infty ,$$
		$$\frac1N\sum_{i=1}^N\|z_i^N-z_i\|_{\bC^{-\kappa}}^p\rightarrow^{\mathbf{P}}0, \quad
		\frac1N\sum_{i=1}^N\|y_i^N-\eta_i\|_{L^2}^2\to^{\mathbf{P}} 0, \qquad \mbox{as } N\to \infty ,$$
		where $\to^{\mathbf{P}}$ means the convergence in probability.
		\item For some $q>1$, $p_0>4/(1-4\kappa)$, and every $p>1$
		$$
		\mathbf{E}[\|z_i^N\|_{\bC^{-\kappa}}^p+
		\|z_i\|_{\bC^{-\kappa}}^p]\lesssim1,
		\quad\mathbf{E}\|\eta_i\|_{L^{p_0}}^{p_0}\lesssim1,\quad \mathbf{E}[\frac1N\sum_{i=1}^N\|y_i^N\|_{L^2}^2]^q\lesssim1,$$
		where the implicit constant is independent of $i, N$.
		
	\end{itemize}
\end{assumption}

%\begin{remark}\label{rem:mf-phil}
The following theorem is our main convergence result - which in particular implies Theorem~\ref{th:1}. The proof is inspired by mean field theory for {\it SDE} systems such as Sznitman's article \cite{MR1108185}, which as the general philosophy starts by directly subtracting the two dynamics and thereby cancelling the white noises, and then controls the difference. To this end we establish energy estimates for the difference $v_i^{N}$ below, c.f. \eqref{diff11} from Step \ref{diffEst1}. The key to the proof is that for the terms collected in $I_1^N$, $I_2^N$ below we interpolate with $\bC^{-s}$ and $B^s_{1,1}$ spaces and leverage various a-priori estimates obtained in the previous sections; but for terms collected in $I_3^N$ which are suitably {\it centered}, we interpolate with {\it Hilbert spaces} and invoke the following  fact  \eqref{eq:Ui}, which in certain sense gives us a crucial ``factor of $1/N$'':
%\end{remark}

Recall that %the following fact will be important for us:
for mean-zero independent random variables $U_{1},\dots,U_{N}$ taking values in a Hilbert space $H$, we have
\begin{equation}
\E \Big \| \sum_{i=1}^{N}U_{i} \Big \|_{H}^{2} =\E \sum_{i=1}^{N}\|U_{i}\|_{H}^{2}.\label{eq:Ui}
\end{equation}
This simple fact is important for us since the square of the sum on
the l.h.s. of \eqref{eq:Ui} appears to have ``$N^2$ terms'' but under expectation it's only a sum of $N$ terms, in certain sense giving us a ``factor of $1/N$''.
%\hao{I moved these here, because it's not about Besov space and shouldn't be in that appendix.}

\bt\label{th:conv-v}
If the initial datum $(z_i^N,y_i^N, z_i, \eta_i)_{i}$ satisfy Assumption \ref{a:main},
then for every $i$ and every $T>0$, $\|v_i^N\|_{C_TL^2}$ converges to zero in probability, as $N\to \infty$.
Moreover, under the additional hypothesis that $(z_i^N,y_i^N, z_{i},\eta_{i})_{i=1}^{N}$ are exchangeable,  for all $t>0$ it holds
\begin{equation}\label{e:diff-t-0}
\lim_{N \to \infty} \E\|\Phi_{i}^{N}(t)-\Psi_{i}(t) \|_{L^{2}}^{2}=0.
\end{equation}
\et

\begin{proof}
	\newcounter{diffEst} % proofstep = 0
	\refstepcounter{diffEst} % increases value by 1
	The proof has a similar flavor to the Lemma \ref{Y:L2}, and in fact we will continue to use the notation $R_{N}^{i}$ for $i=1,2,3$ for the same quantities.  One additional ingredient required is the following instance of the Gagliardo-Nirenberg inequality (a special case of Lemma~\ref{lem:interpolation}),
	\begin{equation}
	\|g \|_{L^{4}} \leq C \|g \|_{H^{1}}^{1/2}\|g \|_{L^{2}}^{1/2} \label {diff2}.
	\end{equation}
	In the proof we omit the superscript $N$ and simply write $v_{i}$ for $v_{i}^{N}$ throughout. %By assumption we know $z_{i}^{N}=z_{i}$, so that also $Z_{i}^{N}=Z_{i}$.
%	\zhu{We modify some words here.} \hao{OK}
	Furthermore, in Steps \ref{diffEst1}-\ref{diffEst4} we work under the simplifying assumption that $z_{i}^{N}=z_{i}$, so that also $Z_{i}^{N}=Z_{i}$.  In Step \ref{diffEst5}, we sketch the argument in the more general case.

	{\sc Step} \arabic{diffEst} \label{diffEst1} \refstepcounter{diffEst} (Energy balance)
	
	In this step, we justify the following energy identity
	%\begin{align}
	%\frac{d}{dt} \|v \|_{L^{2}_{N}}^{2}+  \|\nabla v \|_{L^{2}_{N} }^{2}+\frac{1}{N}\|Y\otimes v \|_{L^{2}_{N \times N}}^{2}+\frac{1}{N}\ \|\langle X,v \rangle_{N} \|_{L^{2}}
	%\end{align}
	\begin{align}
	& \frac12\frac{\dif}{\dif t}\sum_{i=1}^{N} \|v_{i}\|_{L^{2}}^{2}+ \sum_{i=1}^{N}\|\nabla v_{i}\|_{L^{2}}^{2}+m\sum_{i=1}^{N} \|v_{i}\|_{L^{2}}^{2}+\frac{1}{N}\sum_{i,j=1}^{N} \|Y_{j}v_{i}\|_{L^{2}}^{2}+\frac{1}{N}\bigg \|\sum_{j=1}^{N}X_{j}v_{j}  \bigg \|_{L^{2}}^{2} = \sum_{k=1}^3 I_k^N \label{diff11}
	\end{align}
	where
	\begin{align}
	I_{1}^{N} & \eqdef
	-\frac{1}{N}\sum_{i,j=1}^{N}  \big ( 2\<v_{i}v_{j},\Wick{ Z_{j}Z_{i}}\>+\<v_{i}^{2},\Wick{Z_{j}^{2}}\>+2\<v_{i}^{2}Y_{j},Z_{j}\> \big )\;,
	\nonumber \\
	I_{2}^{N}& \eqdef
	-\frac{1}{N}\sum_{i,j=1}^{N} \< v_{i}v_{j}, \big (X_{i}Y_{j}+(3X_{j}+Y_{j} )Z_{i}  \big )\>\;,
	\nonumber \\
	I_{3}^{N}&\eqdef
	-\frac{1}{N}\sum_{i,j=1}^N \Big\< \big [\Wick{Z_{j}^{2} }-\E\Wick{Z_{j}^{2} }+ X_{j}(X_{j}+2Z_{j} )-\E  X_{j}(X_{j}+2Z_{j})  \big ] (X_{i}+Z_{i})\;,\;v_{i} \Big\>\;,
	%I_{4}^{N}&\eqdef
	%-\frac{1}{N}\sum_{i=1}^{N} \< \big [ (X_{i}^{2}-\E(X_{i}^{2}) )(X_{i}+Z_{i})+2\big (X_{i}Z_{i}-\E(X_{i}Z_{i})  \big )X_{i}  \big ], v_{i}\>   \nonumber \\
	%&\quad-\frac{1}{N}\sum_{i=1}^{N} \< \big [ 2(X_{i}\Wick{Z_{i}^{2}}-\E(X_{i}Z_{i})Z_{i}) +X_{i}(\Wick{Z_{i}^{2}}-\E\Wick{Z_{i}^{2}}  ) \big ], v_{i} \> \nonumber
	%\\&\quad-\frac{1}{N}\sum_{i=1}^{N} \< \big [ \Wick{Z_{i}^{3}}-\E(\Wick{Z_{i}^2})Z_{i}], v_i\>\;.
	\end{align}
	%\begin{align}
	%I_{1,1}^{N}&:=-\frac{1}{N}\sum_{i,j=1}^{N}\int u_{i}^{2}\Wick{Z_{j}^{2}}. \nonumber \\
	%I_{1,2}^{N}&:=-\frac{2}{N}\sum_{i,j=1}^{N}\int u_{i}^{2}Y_{j}Z_{j}. \nonumber \\
	%I_{2,1}^{N}&:=-\frac{2}{N}\sum_{i,j=1}^{N}\int u_{i}u_{j}\Wick{ Z_{j}Z_{i}}. \nonumber \\
	%I_{2,2}^{N}&:=-\frac{1}{N}\sum_{i,j=1}^{N}\int u_{i}u_{j}X_{i}Y_{j}. \nonumber \\
	%I_{2,3}^{N}&:=-\frac{3}{N}\sum_{i,j=1}^{N}\int u_{i}u_{j}X_{j}Z_{i}. \nonumber \\
	%I_{2,4}^{N}&:=-\frac{1}{N}\sum_{i,j=1}^{N}\int u_{i}u_{j}Y_{j}Z_{i}. \nonumber
	%\end{align}
	In the definition of $I_{3}^{N}$, to have a compact formula, we slightly abuse notation for the contribution of the diagonal part $i=j$, where we understand $Z_{i}Z_{j}$ to be $\Wick{Z_{i}^{2}}$ and $\Wick{Z_{j}^{2}}Z_{i}$ to be $\Wick{Z_{i}^{3}}$.
	
	We now turn the justification of this identity, and for the convenience of the reader, we write the equations for $Y_{i}$ and $X_{i}$ side by side as
	\begin{equs}
		\LL Y_{i}
		&=-\frac{1}{N}\sum_{j=1 }^{N}\Big ( Y_{j}^{2}Y_{i}+Y_{j}^{2}Z_{i}+2Y_{j}Z_{j}Y_{i}+2Y_{j}\Wick{Z_{i}Z_{j}}+Y_{i}\Wick{Z_{j}^{2}}+\Wick{Z_{i}Z_{j}^{2} } \Big ), \label{diff8}
		\\
		\LL X_{i}
		&=-\frac{1}{N}\sum_{j=1}^{N} \Big( \E(X_{j}^{2})X_{i}+\E(X_{j}^{2})Z_{i}+2\E(X_{j}Z_{j})X_{i}+2\E(X_{j}Z_{j})Z_{i}   +X_i\E\Wick{Z_{j}^{2}}\\&\qquad\qquad\qquad+Z_i\E\Wick{Z_{j}^{2}}\Big), \label{diff9}
	\end{equs}
	where we used that $X_{j}$ and $X_{i}$ are equal in law.  We now compare each of the first 4 terms in \eqref{diff8} to the corresponding terms in \eqref{diff9}.  Note first that
	\begin{align}
	Y_{j}^{2}Y_{i}-\E(X_{j}^{2})X_{i}&=Y_{j}^{2}Y_{i}-X_{j}^{2}X_{i}+\big (X_{j}^{2}-\E(X_{j}^{2}) \big )X_{i} \nonumber \\
	&=Y_{j}^{2}v_{i}+v_{j}(Y_{j}+X_{j})X_{i}+\big (X_{j}^{2}-\E(X_{j}^{2}) \big )X_{i} \nonumber.
	\end{align}
	Similarly, we find
	\begin{align}
	\big (Y_{j}^{2}-\E(X_{j}^{2} ) \big )Z_{i}&=v_{j}(Y_{j}+X_{j})Z_{i}+\big (X_{j}^{2}-\E(X_{j}^{2})  \big )Z_{i} \nonumber. \\
	2Y_{j}Z_{j}Y_{i}-2\E(X_{j}Z_{j})X_{i}&=2\big (v_{i}Y_{j}+v_{j}X_{i} \big )Z_{j}+2\big (X_{j}Z_{j}-\E(X_{j}Z_{j}) \big )X_{i} \nonumber.\\
	2Y_{j}\Wick{Z_{i}Z_{j}}-2\E(X_{j}Z_{j})Z_{i}&=2v_{j}\Wick{Z_{i}Z_{j}}+2\big (  X_{j}\Wick{Z_{i}Z_{j}}-\E(X_{j}Z_{j})Z_{i} \big ) \nonumber. \\
	Y_{i}\Wick{Z_{j}^{2}}&=v_{i}\Wick{Z_{j}^{2}}+X_{i}\Wick{Z_{j}^{2}}\nonumber.
	\end{align}
	Taking the difference of \eqref{diff8} and \eqref{diff9}, using the identities above, multiplying by $v_{i}$, integrating by parts, and summing over $i$ leads to \eqref{diff11}.  Indeed, notice that each equality gives a sum of two pieces, one with a factor of $v$ and one without any factor of $v$, but with a re-centering.  The terms which have a factor of $v$ lead to $I_{1}^{N}$ and $I_{2}^{N}$, except for $Y_{j}^{2}v_{i}$ and $v_{j}X_{j}X_{i}$, which lead to the two coercive quantities on the LHS of \eqref{diff11}.  The terms which have been re-centered lead to $I_{3}^{N}$. % and $I_{4}^{N}$. With independence in mind, we have divided into $i \neq j$ which leads to $I_{3}^{N}$ and $i=j$ with gives $I_{4}^{N}$.  Here, we used that for $i \neq j$, it holds $\Wick{Z_{i}Z_{j}}=Z_{i}Z_{j}$ and $\Wick{Z_{i}Z_{j}^{2}}=Z_{i}\Wick{Z_{j}^{2}}$.
	
	%\begin{align}
	%I_{1,1}^{N}&:=-\frac{1}{N}\sum_{i,j=1}^{N}\int u_{i}^{2}\Wick{Z_{j}^{2}}. \nonumber \\
	%I_{1,2}^{N}&:=-\frac{2}{N}\sum_{i,j=1}^{N}\int u_{i}^{2}Y_{j}Z_{j}. \nonumber \\
	%I_{2,1}^{N}&:=-\frac{2}{N}\sum_{i,j=1}^{N}\int u_{i}u_{j}\Wick{ Z_{j}Z_{i}}. \nonumber \\
	%I_{2,2}^{N}&:=-\frac{1}{N}\sum_{i,j=1}^{N}\int u_{i}u_{j}X_{i}Y_{j}. \nonumber \\
	%I_{2,3}^{N}&:=-\frac{3}{N}\sum_{i,j=1}^{N}\int u_{i}u_{j}X_{j}Z_{i}. \nonumber \\
	%I_{2,4}^{N}&:=-\frac{1}{N}\sum_{i,j=1}^{N}\int u_{i}u_{j}Y_{j}Z_{i}. \nonumber
	%\end{align}
	%The choice to divide terms in the manner above is not entirely obvious, so we will comment on this.  The integrands in $I^{N}_{1}$ and $I^{2}_{N}$ are each quadratic in $v$ and can be estimated pathwise by a small peice of the coercive quantities on the LHS of \eqref{dif11} together with a term which will require a Gronwall.  To estimate $I^{1,i}_{N}$ we can refer back to identical arguments in Lemma ? with $v$ playing the role of $Y$, but for $I^{2,i}_{N}$ we need to modify the arguments slightly.  The terms $I^{3,i}_{N}$ are linear in $v$ and have all been re-centered in order to obtain convergence in probability to zero.  These terms will need to be estimated first at a pathwise level and then in expectation in order to generate cancellations.  Here we use arguments in the flavor of the law of large numbers in a Hilbert space.  For future reference, we also note that since $Y_{i}=X_{i}+v_{i}$, it holds
	%\begin{align}
	%\|Y_{i}v_{j}\|_{L^{2}}^{2} \geq \frac{1}{2}\|X_{i}v_{j}\|_{L^{2}}^{2}+\frac{1}{2}\|v_{i}v_{j}\|_{L^{2}}^{2}\nonumber.
	%\end{align}

	{\sc Step} \arabic{diffEst} \label{diffEst7} \refstepcounter{diffEst} (Estimates for $I_{1}^{N}$ )
	
	In this step, we show there is a universal constant $C$ such that
	\begin{align}
	I_{1}^{N} \leq \frac{1}{8} \bigg ( \sum_{i=1}^{N}\|\nabla v_{i}\|_{L^{2}}^{2}+\frac{1}{N}\sum_{i,j=1}^{N} \|Y_{j}v_{i}\|_{L^{2}}^{2} \bigg )
	 +C(1+R_{N}^{2}+R_{N}^{3}+R_N^5 )\sum_{i=1}^{N} \|v_{i}\|_{L^{2}}^{2} \label{diff1},
	\end{align}
	where $R_{N}^{2}$ and $R_{N}^{3}$ are defined in terms of $Z$ in the same way as in \eqref{eq:R2} and \eqref{e:def-SZ} and
	\begin{align}
	R_{N}^{5}&\eqdef \bigg (1+\frac{1}{N} \sum_{j=1}^{N} \|\nabla Y_{j} \|_{L^{2}}^{2}   \bigg )^{s}\bigg (  \frac{1}{N}\sum_{i=1}^{N}    \|Z_{i}\|_{\bC^{-s}}^{2}\bigg )  \nonumber,
	\end{align}
	with $1>s\geq 2\kappa$ and $s$ small enough.
	Indeed, \eqref{diff1} follows from arguments identical to the ones leading to \eqref{s6} and \eqref{s2} in Lemma \ref{Y:L2}, but with a different labelling of the integrands which we now explain.  There are three contributions to $I_{1}^{N}$ and each can be treated separately.  For the contribution of $v_{i}v_{j}\Wick{ Z_{j}Z_{i}}$ we argue exactly as for \eqref{s8} but with $v_{i}v_{j}$ in place of $Y_{i}Y_{j}$.   For the contribution of $v_{i}^{2}\Wick{Z_{j}^{2}}$, we argue exactly as in \eqref{s9}, but with $v_{i}^{2}$ in place of $Y_{i}^{2}$.  This leads to the inequality
	\begin{align}
	\frac{1}{N}\sum_{i,j=1}^{N} \big ( 2\<v_{i}v_{j}, \Wick{Z_{i}Z_{j}}\>+\<v_{i}^{2}, \Wick{Z_{j}^{2} }\> \big )
	&\leq \frac{1}{16} \bigg ( \sum_{i=1}^{N}\|\nabla v_{i}\|_{L^{2}}^{2} \bigg )
	+C(1+R_{N}^{2})\sum_{i=1}^{N} \|v_{i}\|_{L^{2}}^{2} \label{diff12}.
	\end{align}
	Finally, for the contribution of  $v_{i}^{2}Y_{j}Z_{j}$ the argument is similar as for \eqref{s2}.  This leads to the estimate
	\begin{align}
	-\frac{1}{N} \sum_{j=1}^{N} & \Big\< \sum_{i=1}^{N} v_{i}^{2}Y_{j}, Z_{j}\Big\>
	\lesssim
	\frac{1}{N}\sum_{j=1}^{N} \Big ( \Big \|  \sum_{i=1}^{N}v_{i}^{2}Y_{j} \Big \|_{L^{1}}^{1-s} \Big \|\nabla \Big( \sum_{i=1}^{N}v_{i}^{2}Y_{j} \Big) \Big \|_{L^{1}}^{s}+\Big \|  \sum_{i=1}^{N}v_{i}^{2}Y_{j} \Big \|_{L^{1}} \Big ) \|Z_{j}\|_{\bC^{-s}}  \label{q1d}  \\
	& \lesssim\frac{1}{N} \Big (\sum_{j=1}^{N}\Big \|  \sum_{i=1}^{N}v_{i}^{2}Y_{j} \Big \|_{L^{1}}^{2(1-s) }\Big \|\nabla \Big ( \sum_{i=1}^{N}v_{i}^{2}Y_{j} \Big ) \Big \|_{L^{1}}^{2s}  \Big )^{1/2}
\SZ^{\frac12}
%\Big (  \sum_{j=1}^{N}\|Z_{j}\|_{\bC^{-s}}^{2}   \Big )^{1/2}
 +\frac{1}{N}\Big ( \sum_{j=1}^{N} \Big \|  \sum_{i=1}^{N}v_{i}^{2}Y_{j} \Big \|_{L^{1}}^{2} \Big )^{1/2}
\SZ^{\frac12},
%\Big (  \sum_{j=1}^{N}\|Z_{j}\|_{\bC^{-s}}^{2}   \Big )^{1/2}
	\nonumber
	\end{align}
where $\SZ \eqdef \sum_{j=1}^{N}\|Z_{j}\|_{\bC^{-s}}^{2} $ as in \eqref{e:def-SZ}.
	By H\"older's inequality, it holds that
	\begin{equation}
	\Big \|  \sum_{i=1}^{N}v_{i}^{2}Y_{j} \Big \|_{L^{1}} \lesssim \Big(\sum_{i=1}^{N}\|v_iY_j\|_{L^{2}}^2\Big)^{1/2}\Big(\sum_{i=1}^N\|v_{i}\|_{L^{2}}^2\Big)^{1/2}. \label{qd2}
	\end{equation}
	Furthermore, we find that
	\begin{align}
	&\Big \| \nabla \Big ( \sum_{i=1}^{N}v_{i}^{2}Y_{j} \Big )  \Big \|_{L^{1}}
	\lesssim \Big \| \sum_{i=1}^{N}v_{i}^{2}\nabla Y_{j}  \Big \|_{L^{1}}+\Big \| \sum_{i=1}^{N} \nabla v_{i}v_{i}Y_{j}  \Big \|_{L^{1}} \nonumber \\
	&\lesssim \sum_{i=1}^{N}\|v_{i} \|_{L^{4}}^2 \|\nabla Y_{j}\|_{L^{2}}+\Big ( \sum_{i=1}^{N}\|\nabla v_{i}\|_{L^{2}}^{2} \Big )^{1/2} \Big(\sum_{i=1}^{N}\| v_{i}Y_{j}  \|_{L^{2}}^2\Big)^{1/2}\nonumber\\
	&\lesssim \Big(\sum_{i=1}^{N}\|v_{i} \|_{H^1}^2\Big)^{1/2}\Big( \sum_{i=1}^{N}\|v_{i} \|_{L^2}^2\Big)^{1/2} \|\nabla Y_{j}\|_{L^{2}}+\Big ( \sum_{i=1}^{N}\|\nabla v_{i}\|_{L^{2}}^{2} \Big )^{1/2} \Big(\sum_{i=1}^{N}\| v_{i}Y_{j}  \|_{L^{2}}^2\Big)^{1/2},\nonumber
	\end{align}
	where we used \eqref{diff2} in the last step.
	Hence, we find that
	\begin{align}
	&\sum_{j=1}^{N}\Big \|  \sum_{i=1}^{N}v_{i}^{2}Y_{j} \Big \|_{L^{1}}^{2(1-s) }\Big \|\nabla \Big (\sum_{i=1}^{N}v_{i}^{2}Y_{j} \Big ) \Big \|_{L^{1}}^{2s} \nonumber \\
	& \lesssim \Big(\sum_{i,j=1}^{N}\|v_iY_j\|_{L^{2}}^2\Big)^{{1-s}}\Big(\sum_{i=1}^{N}\|v_{i} \|_{H^1}^2\Big)^{s}\Big( \sum_{i=1}^{N}\|v_{i} \|_{L^2}^2\Big) \Big(\sum_{j=1}^N\|\nabla Y_{j}\|_{L^{2}}^2\Big)^{s}\nonumber
	\\&+\Big ( \sum_{i=1}^{N}\|\nabla v_{i}\|_{L^{2}}^{2} \Big )^{s}\Big ( \sum_{i=1}^{N}\| v_{i}\|_{L^{2}}^{2} \Big )^{1-s} \Big(\sum_{i,j=1}^{N}\| v_{i}Y_{j}  \|_{L^{2}}^2\Big).  \nonumber
	\end{align}
	Inserting this into \eqref{q1d}, taking the square root, and using \eqref{qd2} and
	Young's inequality with exponent $(\frac{2}{1-s},\frac{2}{s},2)$ we arrive at

	\begin{align}
	-\frac{1}{N}\sum_{i,j=1}^{N} \< 2 v_{i}^{2}Y_{j}, Z_{j} \>& \leq \frac{1}{16} \bigg ( \sum_{i=1}^{N}\|\nabla v_{i}\|_{L^{2}}^{2}+\frac{1}{N}\sum_{i,j=1}^{N} \|Y_{j}v_{i}\|_{L^{2}}^{2} \bigg )
	+C(1+R_{N}^{3}+R_N^5)\sum_{i=1}^{N} \|v_{i}\|_{L^{2}}^{2} \label{diff13} .
	\end{align}
	Combining \eqref{diff12} and \eqref{diff13} and recalling the definition of $I_{1}^{N}$ we obtain \eqref{diff1}.

	{\sc Step} \arabic{diffEst} \label{diffEst2} \refstepcounter{diffEst} (Estimates for $I_{2}^{N}$ )
	
	In this step, we show there is a universal constant $C$ such that
	\begin{align}
	&I_{2}^{N} \leq \frac{1}{4} \Big (\sum_{i=1}^{N}\|\nabla v_{i}\|_{L^{2}}^{2} +\frac{1}{N}\sum_{i,j=1}^{N}\|v_{i}Y_{j}\|_{L^{2}}^{2}+\Big \|\frac{1}{\sqrt{N}} \sum_{j=1}^{N}v_{j}X_{j} \Big \|_{L^{2}}^2  \Big ) \nonumber \\
	&+C\Big (1+R_{N}^{3}+R_{N}^{4}+R_{N}^{5}+R_N^6+\Big(\frac{1}{N}\sum_{i=1}^{N}\|X_{i}\|_{L^{4}}^{2}\Big)^2\Big ) \Big ( \sum_{i=1}^{N}\|v_{i} \|_{L^{2}}^{2}  \Big ),\label{est55}
	\end{align}
	where $R_{N}^{3}$ is defined as in \eqref{e:def-SZ} and $R_{N}^{4}$ and $R_N^6$ are defined by
	\begin{align}
	R_{N}^{4}&\eqdef \Big (1+\frac{1}{N} \sum_{j=1}^{N} \|X_{j}\|_{L^{4}}^{2}   \Big )^{\frac{2s}{2-s}}\Big (  \frac{1}{N}\sum_{i=1}^{N}    \|Z_{i}\|_{\bC^{-s}}^{2}\Big )^{\frac{2}{2-s}}
	+ \Big (\frac{1}{N}\sum_{j=1}^{N}\|\nabla X_{j}\|_{L^{2}}^{2}    \Big )^{s }\Big (  \frac{1}{N}\sum_{i=1}^{N}    \|Z_{i}\|_{\bC^{-s}}^{2}\Big ).  \nonumber \\
	R_N^6&\eqdef \Big (\frac{1}{N}\sum_{j=1}^{N}\|Y_{j} \|_{L^{4}}^{2}    \Big )^{\frac{s}{1-s} }R_{N}^{3}.\nonumber
	\end{align}
	
	We break $I_{2}^{N}$ into the separate contributions where $v_{i}v_{j}$ multiplies $X_{i}Y_{j}$, $X_{j}Z_{i}$, and $Y_{j}Z_{i}$ respectively.  For the first contribution,  Cauchy-Schwarz's inequality yields
	\begin{align}
	&\frac{1}{N}\sum_{i,j=1}^{N} \int v_{i}v_{j} X_{i}Y_{j} \dif x \leq \frac{1}{8} \frac{1}{N }\sum_{i,j=1}^{N} \int v_{i}^{2}Y_{j}^{2}\dif x+C \frac{1}{N}\sum_{i,j=1}^{N} \int v_{j}^{2}X_{i}^{2}\dif x \nonumber \\
	&\leq \frac{1}{8} \frac{1}{N }\sum_{i,j=1}^{N} \int v_{i}^{2}Y_{j}^{2}\dif x+C \Big ( \sum_{i=1}^{N}\|v_{i} \|_{L^{4}}^{2}  \Big ) \Big (\frac{1}{N}\sum_{i=1}^{N}\|X_{i}\|_{L^{4}}^{2}  \Big ) \nonumber
	\\
	&\leq \frac{1}{8} \frac{1}{N}\sum_{i,j=1}^{N} \|v_{i}Y_{j}\|_{L^{2}}^{2}+C \Big ( \sum_{i=1}^{N}\|v_{i} \|_{H^{1}}^{2}  \Big )^{1/2}\Big ( \sum_{i=1}^{N}\|v_{i} \|_{L^{2}}^{2}  \Big )^{1/2} \Big (\frac{1}{N}\sum_{i=1}^{N}\|X_{i}\|_{L^{4}}^{2}  \Big ), \nonumber
	\end{align}
	where we used \eqref{diff2}.  Using Young's inequality with exponents $(2,2)$ leads to the last contribution to \eqref{est55}.
	The remaining contributions to $I_{2}^{N}$ are more involved to estimate.  Our next claim is that
	\begin{align}
	\frac{3}{N} \sum_{i=1}^{N} \Big\< v_{i}  \sum_{j=1}^{N}v_{j}X_{j}   , Z_{i}\Big\> &\leq \frac{1}{8} \Big ( \sum_{i=1}^{N}\|\nabla v_{i}\|_{L^{2}}^{2}+\Big \|\frac{1}{\sqrt{N}} \sum_{j=1}^{N}v_{j}X_{j} \Big \|_{L^{2}}^{2} \Big )+C(1+R_{N}^{3}+R_{N}^{4} )\sum_{i=1}^{N} \|v_{i}\|_{L^{2}}^{2}. \label{diff3}
	\end{align}
	
The basic setup is the same as the bound leading to \eqref{q1} via the inequality \eqref{s1} with $v_{i}\sum_{j=1}^{N}v_{j}X_{j}$ playing the role of $g$ and $Z_{i}$ playing the role of $f$, followed by an application of the Cauchy-Schwarz inequality for the summation in $i$. The LHS of \eqref{diff3} is then bounded by
	\begin{align}
%	&\frac{1}{N} \sum_{i=1}^{N} \Big\<v_{i}  \sum_{j=1}^{N}v_{j}X_{j}   , Z_{i}\Big\>  \nonumber  \\
	\sum_{s'\in\{s,0\}}
	\frac{1}{\sqrt{N} } \Big ( \sum_{i=1}^{N} \Big \|v_{i}  \sum_{j=1}^{N}v_{j}X_{j}  \Big \|_{L^{1}}^{2(1-s') } \Big \|\nabla \Big ( v_{i}  \sum_{j=1}^{N}v_{j}X_{j} \Big )  \Big \|_{L^{1}}^{2s' } \Big )^{\frac12}
	\Big (\frac{1}{N}\SZ  \Big )^{\frac12} .
	%+\frac{1}{\sqrt{N} } \Big ( \sum_{i=1}^{N}\Big \|v_{i}  \sum_{j=1}^{N}v_{j}X_{j}  \Big \|_{L^{1}}^{2} \Big )^{\frac12} \Big (\frac{1}{N}\SZ \Big )^{\frac12}
	\label{est11}
	\end{align}
	Using H\"older's inequality in the form $\big \| v_{i}\sum_{j=1}^{N}v_{j}X_{j} \big \|_{L^{1}} \leq \|v_{i}\|_{L^{2}}\big \|\sum_{j=1}^{N}v_{j}X_{j}\big \|_{L^{2}}$ together with
\begin{align}
	\Big \|\nabla \Big (  v_{i}\sum_{j=1}^{N}v_{j}X_{j}\Big ) \Big \|_{L^{1}}
	&\leq \|\nabla v_{i}\|_{L^{2}} \Big \|\sum_{j=1}^{N}v_{j}X_{j}\Big \|_{L^{2}}
	+\|v_{i}\|_{L^{2}}^{1/2}\|v_i\|_{H^1}^{1/2}\Big (\sum_{j=1}^{N}\|\nabla v_{j}\|_{L^{2}}^{2}  \Big )^{\frac12} \Big (\sum_{j=1}^{N} \|X_{j}\|_{L^{4}}^{2}   \Big )^{\frac12}\nonumber \\
	&+\|v_{i}\|_{L^{2}}^{1/2}\|v_i\|_{H^1}^{1/2}\Big (\sum_{j=1}^{N}\|v_{j}\|_{L^{2}}^{2}  \Big )^{1/4} \Big (\sum_{j=1}^{N}\|v_{j}\|_{H^1}^{2}  \Big )^{1/4}\Big (\sum_{j=1}^{N}\|\nabla X_{j}\|_{L^{2}}^{2}   \Big )^{1/2},\nonumber
\end{align}
	where we used \eqref{diff2},
	and inserting this into \eqref{est11} and applying H\"older's inequality for the summation in $i$ together with
	\begin{align}
	&\bigg ( \sum_{i=1}^{N}\|v_{i}\|_{L^{2}}^{2(1-s)}\|v_{i}\|_{L^{2}}^{s}\|v_{i}\|_{H^{1}}^{s} \bigg )^{\frac{1}{2}}  \leq \bigg ( \sum_{i=1}^{N}\|v_{i}\|_{L^{2}}^{2}   \bigg )^{\frac{2-s}{4} }\bigg ( \sum_{i=1}^{N}\|v_{i}\|_{H^{1}}^{2} \bigg )^{\frac{s}{4}},    \nonumber
	\end{align}
	we obtain a majorization by
	\begin{align}
	&\Big \|\frac{1}{\sqrt{N}} \sum_{j=1}^{N}v_{j}X_{j} \Big \|_{L^{2}} \Big (\sum_{i=1}^{N}\|v_{i}\|_{L^{2}}^{2}  \Big )^{\frac{1-s}{2}}\Big (\sum_{i=1}^{N} \|v_{i}\|_{H^{1} }^{2}  \Big )^{\frac{s}{2}}  \Big (  \frac{1}{N}\SZ\Big )^{1/2} \nonumber \\
	&+\Big \|\frac{1}{\sqrt{N}} \sum_{j=1}^{N}v_{j}X_{j} \Big \|_{L^{2}}^{1-s} \Big (\sum_{i=1}^{N}\|v_{i}\|_{L^{2}}^{2}  \Big )^{\frac{1}{2}-\frac{s}{4}}\Big (\sum_{i=1}^{N} \|  v_{i}\|_{H^1}^{2}  \Big )^{\frac{3s}{4}}
	 \Big (\frac{1}{N} \sum_{j=1}^{N} \|X_{j}\|_{L^{4}}^{2}   \Big )^{\frac{s}{2}} \Big (  \frac{1}{N}\SZ\Big )^{1/2} \nonumber \\
	&+\Big \|\frac{1}{\sqrt{N}} \sum_{j=1}^{N}v_{j}X_{j} \Big \|_{L^{2}}^{1-s} \Big (\sum_{i=1}^{N}\|v_{i}\|_{L^{2}}^{2}  \Big )^{\frac{1}{2}} \Big (\sum_{i=1}^{N}\|v_{i}\|_{H^1}^{2}  \Big )^{\frac{s}{2}}  \Big (\frac{1}{N}\sum_{j=1}^{N} \|\nabla X_{j}\|_{L^{2}}^{2}   \Big )^{\frac{s}{2}} \Big (  \frac{1}{N}\SZ\Big )^{1/2} \nonumber.
	\end{align}
	Finally, we apply Young's inequality with exponents $(2,\frac{2}{1-s},\frac{2}{s})$ for the first term, $(\frac{2}{1-s},\frac{4}{2-s},\frac{4}{3s})$ for the second term, and $(\frac{2}{1-s},2,\frac{2}{s})$ for the third term which leads to \eqref{diff3}.

	Similar to \eqref{diff3}, we now claim that
	\begin{align}
	\frac{1}{N} \sum_{i=1}^{N} \Big\<  \sum_{j=1}^{N}v_{i}v_{j}Y_{j}   , Z_{i}\Big\> &\leq \frac{1}{8} \Big ( \sum_{i=1}^{N}\|\nabla v_{i}\|_{L^{2}}^{2}+\frac{1}{N}\sum_{i,j=1}^{N}\|v_{i}Y_{j}\|_{L^{2}}^{2}   \Big )
	+C\Big(1+ \!\!\!\!\!\! \sum_{k\in\{3,5,6\}}\!\!\!\! \!\!R_{N}^{k} \Big)\sum_{i=1}^{N} \|v_{i}\|_{L^{2}}^{2}.\label{diff4}
	\end{align}
	
	The basic setup is again similar to the bound leading to \eqref{q1} via the inequality \eqref{s1} with $\sum_{j=1}^{N}v_{i}v_{j}Y_{j}$ playing the role of $g$ and $Z_{i}$ playing the role of $f$, followed by an application of the Cauchy-Schwarz inequality for the summation in $i$. The l.h.s. of \eqref{diff4} is then bounded by
	\begin{align}
	%&\frac{1}{N} \sum_{i=1}^{N} \Big\<  \sum_{j=1}^{N}v_{i}v_{j}Y_{j}   , Z_{i}\Big\>  \nonumber  \\
	\sum_{s'\in\{s,0\}}
	 \frac{1}{\sqrt{N} } \Big ( \sum_{i=1}^{N} \Big \|  \sum_{j=1}^{N}v_{i}v_{j}Y_{j}  \Big \|_{L^{1}}^{2(1-s') } \Big \| \sum_{j=1}^{N}\nabla (v_{i}v_{j}Y_{j})  \Big \|_{L^{1}}^{2s' } \Big )^{\frac12}
	\Big (\frac{1}{N}\SZ \Big )^{\frac12}.
	 %+\frac{1}{\sqrt{N} } \Big ( \sum_{i=1}^{N}\Big \| \sum_{j=1}^{N}v_{i}v_{j}Y_{j}  \Big \|_{L^{1}}^{2} \Big )^{\frac12} \Big (\frac{1}{N}\SZ \Big )^{\frac12}.
	 \label{diff5}
	\end{align}
	By H\"older's inequality and the Cauchy-Schwarz inequality we find
	\begin{equation}
	\Big \| \sum_{j=1}^{N}v_{i}v_{j}Y_{j} \Big \|_{L^{1}} \leq \Big ( \sum_{j=1}^{N}\|v_{i}Y_{j}\|_{L^{2}}^{2}\Big )^{1/2}\Big (\sum_{j=1}^{N}\|v_{j}\|_{L^{2}}^{2} \Big )^{1/2}, \nonumber
	\end{equation}
	together with \eqref{diff2} to have
\begin{align}
	\Big \| \sum_{j=1}^{N} & \nabla (v_{i}v_{j}Y_{j})  \Big \|_{L^{1}} \leq \Big (\sum_{j=1}^{N}\|v_{i}Y_{j}\|_{L^{2}}^{2}   \Big )^{1/2} \Big (\sum_{j=1}^{N}\|\nabla v_{j}\|_{L^{2}}^{2}  \Big )^{1/2} \nonumber \\
	&+\|\nabla v_{i}\|_{L^{2}} \Big (\sum_{j=1}^{N}\| v_{j}\|_{L^{4}}^{2}  \Big )^{1/2} \Big (\sum_{j=1}^{N} \| Y_{j}\|_{L^{4}}^{2}   \Big )^{1/2}
	+ \|v_{i}\|_{L^{4}} \Big (\sum_{j=1}^{N}\| v_{j}\|_{L^{4}}^{2}  \Big )^{1/2} \Big (\sum_{j=1}^{N} \|\nabla Y_{j}\|_{L^{2}}^{2}   \Big )^{1/2}. \nonumber
\end{align}
	Inserting this into \eqref{diff5} and applying H\"older's inequality for the summation in $i$ with exponents $(\frac{1}{1-s},\frac{1}{s})$ leads to a majorization by
	\begin{align}
	&\Big (\frac{1}{N}\sum_{i,j=1}^{N}\|v_{i}Y_{j}\|_{L^{2}}^{2}\Big )^{\frac{1}{2}} \Big (\sum_{i=1}^{N}\|v_{i}\|_{L^{2}}^{2}  \Big )^{\frac{1-s}{2}}\Big (\sum_{i=1}^{N} \|\nabla v_{i}\|_{L^{2} }^{2}  \Big )^{\frac{s}{2}}
	 \Big (  \frac{1}{N}\SZ\Big )^{1/2} \nonumber \\
	&+\Big (\frac{1}{N}\sum_{i,j=1}^{N}\|v_{i}Y_{j}\|_{L^{2}}^{2}\Big )^{\frac{1-s}{2}} \Big (\sum_{j=1}^{N}\|v_{j}\|_{L^{2}}^{2}  \Big )^{\frac{1-s}{2}}\Big (\sum_{j=1}^{N} \|v_{j}\|_{L^{4}}^{2}  \Big )^{\frac{s}{2}}
	  \Big (\sum_{i=1}^{N} \|\nabla v_{i}\|_{L^{2}}^{2}  \Big )^{\frac{s}{2}}  \Big (\frac{1}{N}\sum_{j=1}^{N} \|Y_{j}\|_{L^{4}}^{2}   \Big )^{\frac{s}{2}} \Big (  \frac{1}{N}\SZ\Big )^{\frac12} \nonumber\\
	& +\Big (\frac{1}{N}\sum_{i,j=1}^{N}\|v_{i}Y_{j}\|_{L^{2}}^{2}\Big )^{\frac{1-s}{2}} \Big (\sum_{i=1}^{N}\|v_{i}\|_{H^{1}}^{2}  \Big )^{\frac{s}{2}}\Big (\sum_{j=1}^{N} \|v_{j}\|_{L^{2}}^{2}  \Big )^{\frac{1}{2}}
	  \Big (\frac{1}{N}\sum_{j=1}^{N} \|\nabla Y_{j}\|_{L^{2}}^{2}   \Big )^{\frac{s}{2}} \Big (  \frac{1}{N}\SZ\Big )^{1/2}\nonumber .
	\end{align}
	Note that for the third term, we also took advantage of \eqref{diff2}.  We now apply Young's inequality with exponents $(2,\frac{2}{1-s},\frac{2}{s})$ for the first term, $(\frac{2}{1-s},\frac{2}{1-s},\frac{2}{s},\frac{2}{s} )$ for the second term, and $(\frac{2}{1-s},\frac{2}{s},2)$ for the third term, which leads to \eqref{diff4}.  Finally, combining \eqref{diff3} and \eqref{diff4} we obtain \eqref{est55}.

	{\sc Step} \arabic{diffEst} \label{diffEst3} \refstepcounter{diffEst} (Law of large numbers type bounds: estimates for $I_{3}^{N}$)
	
	For $I_{3}^{N}$ we obtain a bound in expectation in the spirit of the law of large numbers in a Hilbert space to generate cancellations.
	To this end, we define
	\begin{equation}
	G_{j}\eqdef  (X_{j}^2-\E X_j^2)+2(X_jZ_{j}-\E X_{j}Z_{j})+(\Wick{Z_{j}^{2} }-\E\Wick{Z_{j}^{2} })\eqdef G_j^{(1)}+G_j^{(2)}+G_j^{(3)}.\nonumber
	\end{equation}
	We show there is a universal constant $C$ such that
	\begin{align}\label{eq:zz}
	I_3^N & \leq C(\bar{R}_N+\bar{R}_N^\prime)+\frac18\frac1N \Big\|\sum_{i=1}^NX_iv_i\Big\|_{L^2}^2+\frac{1}{4}\sum_{i=1}^N\| v_i\|_{H^1}^2
	\\&+C\Big(\sum_{i=1}^N\|v_i\|_{L^2}^2\Big)
	\Big[R_N^7+1
	+\frac{1}{N}\sum_{i=1}^N\Big(\|X_i\|_{L^4}^{4/(1-2s)}+\|\Lambda^sX_i\|_{L^4}^4\Big)\Big],\nonumber
	\end{align}
	with
	\begin{align*}
	\bar{R}_N & \eqdef \frac{1}{N}\Big\|\sum_{j=1}^NG_j^{(1)}\Big\|_{H^{s}}^2
	+\!\!\! \sum_{k\in\{2,3\}}\frac{1}{N}\Big\|\sum_{j=1}^N G_j^{(k)}\Big\|_{H^{-s}}^2 ,
	%+\frac{1}{N}\Big\|\sum_{j=1}^NG_j^3\Big\|_{H^{-s}}^2,
	\qquad
	\bar{R}_N^\prime  \eqdef \sum_{k\in\{2,3\}} \frac{1}{N^2}\sum_{i=1}^N \| \sum_{j=1}^N G_j^{(k)} Z_i\|_{H^{-s}}^2,
	%+\frac{1}{N^2}\sum_{i=1}^N \| \sum_{j=1}^NG_j^3Z_i\|_{H^{-s}}^2,
	\\
	R_N^7 & \eqdef \Big(\frac{1}{N}\sum_{i=1}^N\|\Lambda^{-s}Z_i\|_{L^\infty}^2\Big)^{\frac{1}{1-s}}.
	\end{align*}
	
	We write $I_3^N=\sum_{k=1}^3(I_{3,k}^N+J_{3,k}^N)$ with
	$$I_{3,k}^N\eqdef\frac{1}{N}\sum_{i=1}^N\Big\<\sum_{j=1}^N G_j^{(k)} X_i,v_i \Big\>,\qquad J_{3,k}^N\eqdef\frac{1}{N}\sum_{i=1}^N\Big\<\sum_{j=1}^N G_j^{(k)} Z_i,v_i \Big\>.$$
	We consider each term separately: For $I_{3,1}^N$ we have the following
	\begin{align*}I_{3,1}^N
	\leq \frac{1}{N}\Big\|\sum_{j=1}^N G_j^{(1)}\Big\|_{L^2} \Big\|\sum_{i=1}^NX_iv_i\Big\|_{L^2}\leq
	C\frac{1}{N}\Big\|\sum_{j=1}^N G_j^{(1)}\Big\|_{L^2}^2+\frac18\frac1N \Big\|\sum_{i=1}^NX_iv_i\Big\|_{L^2}^2
	.
	\end{align*}
	For $J_{3,1}^N$ we use  \eqref{e:Lambda-prod}, the interpolation Lemma \ref{lem:interpolation} and Young's inequality to obtain
\begin{align*}
	J_{3,1}^N&=\frac{1}{N}\sum_{i=1}^N \Big\langle \Lambda^s \Big(\sum_{j=1}^NG_j^{(1)}v_i \Big),\Lambda^{-s}Z_i\Big\rangle
	\\&\le \frac{1}{N}\sum_{i=1}^N \Big[\Big\|\Lambda^{s} \sum_{j=1}^NG_j^{(1)}\Big\|_{L^2}\|v_i\|_{L^2}+\Big\| \sum_{j=1}^NG_j^{(1)}\Big\|_{L^2}\|\Lambda^{s}v_i\|_{L^2}\Big]\|\Lambda^{-s}Z_i\|_{L^\infty}
	\\&
	\le \frac{C}{N}\Big\|\sum_{j=1}^NG_j^{(1)}\Big\|_{H^{s}}^2
	+\frac{1}{20}\sum_{i=1}^N\| v_i\|_{H^1}^2
	+C \!\!\!  \sum_{s'\in\{0,s\}} \!\!\!\! \Big(\sum_{i=1}^N\|v_i\|_{L^2}^2\Big)
		\Big(\frac{1}{N}\sum_{i=1}^N\|\Lambda^{-s}Z_i\|_{L^\infty}^2\Big)^{\frac{1}{1-s'}}.
%	+\Big(\sum_{i=1}^N\|v_i\|_{L^2}^2\Big)
%	\Big(\frac{1}{N}\sum_{i=1}^N\|\Lambda^{-s}Z_i\|_{L^\infty}^2\Big)^{\frac{1}{1-s}} .
\end{align*}
	For $I_{3,2}^N$ we have
	\begin{align*}
	I_{3,2}^N
	&=\frac{1}{N}\sum_{i=1}^N \Big\langle\Lambda^{-s} \sum_{j=1}^NG_j^{(2)},\Lambda^s(X_iv_i)\Big\rangle
	 \le  \frac{C}{N}\Big\|\sum_{j=1}^NG_j^{(2)}\Big\|_{H^{-s}}^2+\frac{C}{N}\Big(\sum_{i=1}^N\|\Lambda^s(X_iv_i)\|_{L^2}\Big)^2
	\\& \le  \frac{C}{N}\Big\|\sum_{j=1}^NG_j^{(2)}\Big\|_{H^{-s}}^2
	+\frac{1}{20}\sum_{i=1}^N\| v_i\|_{H^1}^2+
	C\Big(\sum_{i=1}^N\|v_i\|_{L^2}^2\Big)\Big(\frac{1}{N}\sum_{i=1}^N(\|X_i\|_{L^4}^{4/(1-2s)}+\|\Lambda^sX_i\|_{L^4}^4)\Big),
	\end{align*}
where we used \eqref{e:Lambda-prod}, \eqref{diff2}  to have
	\begin{equation}\label{bd:sob}\aligned
	&\frac{1}{N}\Big(\sum_{i=1}^N\|\Lambda^s(X_iv_i)\|_{L^2}\Big)^2\lesssim \frac{1}{N}\Big(\sum_{i=1}^N\|\Lambda^sX_i\|_{L^4}\|v_i\|_{L^4}+\|\Lambda^sv_i\|_{L^4}\|X_i\|_{L^4}\Big)^2
	\\ & \lesssim \frac{1}{N}\Big(\sum_{i=1}^N\|v_i\|_{H^1}^{1/2}\|v_i\|_{L^2}^{1/2}\|\Lambda^sX_i\|_{L^4}+\sum_{i=1}^N\| v_i\|_{H^1}^{\frac{1}{2}+s}\|v_i\|_{L^2}^{\frac12-s}\|X_i\|_{L^4}\Big)^2
	%\\ & \lesssim\frac{1}{20}\sum_{i=1}^N\| v_i\|_{H^1}^2+\Big(\sum_{i=1}^N\|v_i\|_{L^2}^2\Big)\Big(\frac{1}{N}\sum_{i=1}^N(\|X_i\|_{L^4}^{4/(1-2s)}+\|\Lambda^sX_i\|_{L^4}^4)\Big)
	,\endaligned\end{equation}
followed by H\"older inequality with exponents $(4,4,2)$, $(\frac4{1+2s},\frac4{1-2s},2)$, Young's inequality  and finally  Jensen's inequalities for the terms with $ X_i$
	in the last inequality.
For $J_{3,2}^N$ we have
	\begin{align*}
	J_{3,2}^N\lesssim\frac{1}{N^2}\sum_{i=1}^N \| \sum_{j=1}^NG_j^{(2)}Z_i\|_{H^{-1}}^2+\frac{1}{20}\sum_{i=1}^N\| v_i\|_{H^1}^2.
\end{align*}
For $I_{3,3}^N$ we have
\begin{align*}
	I_{3,3}^N
	&\lesssim \frac{1}{N}\Big\|\sum_{j=1}^NG_j^{(3)}\Big\|_{H^{-s}}^2+\frac{1}{N}\Big(\sum_{i=1}^N\|\Lambda^s (X_i v_i)\|_{L^2}\Big)^2
	\\ & \lesssim  \frac{1}{N}\Big\|\sum_{j=1}^NG_j^{(3)}\Big\|_{H^{-s}}^2+\frac{1}{20}\sum_{i=1}^N\| v_i\|_{H^1}^2 +
	\Big(\sum_{i=1}^N\|v_i\|_{L^2}^2\Big)\Big(\frac{1}{N}\sum_{i=1}^N(\|X_i\|_{L^4}^{4/(1-2s)}+\|\Lambda^sX_i\|_{L^4}^4)\Big),
	\end{align*}
	where we used \eqref{bd:sob} in the last inequality.
	For the last term we have
	\begin{align*}J_{3,3}^N
	\lesssim& \frac{1}{N^2}\sum_{i=1}^N\|\sum_{j=1}^NG_j^{(3)}Z_i\|_{H^{-1}}^2+\frac{1}{20}\sum_{i=1}^N\|v_i\|_{H^1}^2
	.\end{align*}
	Combining all the estimates for $I_{3,k}^N$ and $J_{3,k}^N$ we arrive at \eqref{eq:zz}. In the following we calculate
	$\E\|\bar{R}_N\|_{L^1_T}+\E\|\bar{R}_N^\prime\|_{L^1_T}$.
	To this end, we recall the  general fact \eqref{eq:Ui} for centered  independent Hilbert space-valued random variables.
%	 for mean-zero independent random variables $U_{1},\dots,U_{N}$ taking values in a Hilbert space $H$, it holds that
%	\begin{equation}
%	\E \Big \| \sum_{i=1}^{N}U_{i} \Big \|_{H}^{2} =\E \sum_{i=1}^{N}\|U_{i}\|_{H}^{2}.\label{eq:Ui}
%	\end{equation}
	Applying \eqref{eq:Ui} we obtain
	\begin{align*}\E\|\bar{R}_N\|_{L^1_T}\lesssim \E\|G_1^{(1)}\|_{L_T^2H^s}^2+\E\|G_1^{(2)}\|_{L_T^2H^{-s}}^2+\E\|G_1^{(3)}\|_{L_T^2H^{-s}}^2.\end{align*}
	It is obvious that $\E\|G_1^{(3)}\|_{L_T^2H^{-s}}^2\lesssim1$.
	By Lemma \ref{lem:Lp} we know
	\begin{align}
	\E\|G_1^{(1)}\|_{L_T^2H^s}^2 & \lesssim \int_0^T\E(\|X_1\nabla X_1\|_{L^2}^2+\|X_1\|_{L^4}^4)\dif t\lesssim1,	\nonumber
\\
	\E\|G_1^{(2)}\|_{L_T^2H^{-s}}^2 &\lesssim\int_0^T \E[\|Z_1\|_{\bC^{-s/2}}^2\|X_1\|_{H^s}^2]\dif t
	\lesssim\int_0^T (\E[\|X_1\|_{H^1}^2+\|X_1\|_{L^4}^4]+1)\dif t\lesssim1,	\label{eq:zz2}
	\end{align}
	where we used Lemma \ref{lem:interpolation} and  Lemma \ref{lem:multi}. Therefore $\E\|\bar{R}_N\|_{L^1_T}\lesssim1$.
	For $\bar{R}_N^\prime$ we have
	\begin{align*}
	&\E\frac{1}{N^2}\sum_{i=1}^N \| \sum_{j=1}^NG_j^{(2)}Z_i\|_{H^{-s}}^2=\frac{1}{N^2}\sum_{i,j,\ell=1}^N\E \<G_j^{(2)}Z_i,G_\ell^{(2)} Z_i\>_{H^{-s}}
	\\
	&=\frac{1}{N^2}\Big[\sum_{i=j=\ell}+2\sum_{i=j\neq \ell}+\sum_{\ell=j\neq i}\Big]\lesssim\E\|X_1\Wick{Z_1^2}\|_{H^{-s}}^2+\E\|X_1\Wick{Z_1Z_2}\|_{H^{-s}}^2,
	\end{align*}
	where we used independence to have $\sum_{i\neq j\neq \ell}=0$.
	Similarly, we have
	\begin{align*}
	&\E\frac{1}{N^2}\sum_{i=1}^N \| \sum_{j=1}^NG_j^{(3)}Z_i\|_{H^{-s}}^2\lesssim\E\|\Wick{Z_1^3}\|_{H^{-s}}^2+\E\|\Wick{Z_1^2Z_2}\|_{H^{-s}}^2
	.\end{align*}
	Combining the above two estimates and using Lemma \ref{le:ex} and the same argument as in \eqref{eq:zz2} with $Z_1$ replaced by $\Wick{Z_1Z_2}$ and $\Wick{Z_1^2}$, we obtain  $\E\|\bar{R}_N^\prime\|_{L^1_T}\lesssim1$.

	{\sc Step} \arabic{diffEst} \label{diffEst4} \refstepcounter{diffEst} (Convergence of $v_{i}$ to zero in $L^{2}(\Omega)$)
	
	We now combine our estimates and conclude the proof of \eqref{e:diff-t-0}.  Namely, we insert the estimates \eqref{diff1} and \eqref{est55} into \eqref{diff11} and also appeal to our bounds from Step \ref{diffEst3} to obtain
	\begin{align}
	\frac{\dif}{\dif t}\sum_{i=1}^{N } & \|v_{i}\|_{L^{2}}^{2} \leq C(\bar{R}_{N}+\bar{R}_N^\prime)
	+C\Big(1+\sum_{i=2}^{7}R_{N}^{i}+  \frac{1}{N}\sum_{i=1}^{N}\big(\|X_{i}\|_{L^{4}}^{4/(1-2s)}+\|\Lambda^sX_i\|_{L^4}^4\big)  \Big)\sum_{i=1}^{N}\|v_{i}\|_{L^{2}}^{2}, \label{finEs1}
	\end{align}
	where $\bar{R}_{N}+\bar{R}_{N}^\prime$ is uniformly bounded in $L^{1}(\Omega \times [0,T])$.  Furthermore, by  Lemma \ref{le:ex},  Lemma \ref{Y:L2}, Lemma \ref{lem:Lp} and \eqref{diff2} we deduce also that $R_{N}^{i}$ is uniformly bounded in $L^{1}(\Omega \times [0,T])$ for each $i=2,\dots,7$.
	%\begin{align}
	%\sum_{i=2}^{7}\E \int_{0}^{T}R_{N}^{i} \dif t  <\infty \nonumber.
	%\end{align}
	By the Gagliardo-Nirenberg inequality in Lemma \ref{lem:interpolation} we have for $s\geq 2\kappa$,  $r>4$, $\frac{1}{4}=\frac{s}{2}+\frac{1-s}{r}$, $\|\Lambda^sX_i\|_{L^4}^4\lesssim\|X_i\|_{H^1}^{4s}\|X_i\|_{L^r}^{4(1-s)}$,  which combined with  Lemma \ref{lem:Lp}  implies that
	\begin{align*}
	\E \int_{0}^{T}  \frac{1}{N}\sum_{i=1}^{N} \big(\|X_{i}\|_{L^{4}}^{\frac4{1-2s}}+\|\Lambda^sX_i\|_{L^4}^4  \big )\dif t
	 \lesssim
	 \E \int_{0}^{T} \frac{1}{N}\sum_{i=1}^{N} \big(\|X_i\|_{H^1}^2+\|X_i\|_{L^{\frac4{1-2s}}}^{\frac4{1-2s}}+1 \big)  \dif t<\infty .
	\end{align*}
	We now divide \eqref{finEs1} by $N$ and use the above observations together with Gronwall's inequality.  Note that in light of Assumption \ref{a:main}, it holds that
	\begin{equation}
	\frac{1}{N}\sum_{i=1}^{N}\|v_{i}(0)\|_{L^{2}}^{2}\to^{\mathbf{P}} 0.
	\end{equation}
	%so in particular, $\frac{1}{N} \sum_{i=1}^{N}\|v_{i}(0)\|_{L^{2}}^{2} \to 0$ in probability.
	It now follows that
	\begin{align}\label{eq:zz3}
	\sup_{t\in[0,T]}\frac{1}{N}\sum_{i=1}^{N}\|v_{i}\|_{L^{2}}^{2}+\frac{1}{N}\sum_{i=1}^{N}\|v_{i}\|_{L_T^2H^1}^{2}+\frac{1}{N^2}\sum_{i,j=1}^N\|Y_jv_i\|_{L^2_TL^2}^2+\frac{1}{N^2}\|\sum_j X_jv_j\|_{L_T^2L^2}^2
	\end{align}
	 converges to zero in probability by Lemma \ref{lem:pro} below.  We now upgrade this from convergence in probability to convergence in $L^{1}(\Omega)$ by bounding higher moments and applying Vitali's      convergence theorem. Only in this part we use the condition that the initial conditions $(z_i^N,y_i^N,z_i,\eta_i)_{i=1}^N$ are exchangeable, which implies that the law of $v_i(t)$ and $v_j(t)$, $i\neq j$ are the same.

	  Indeed, first note that $\sup_{t\in[0,T]}\frac{1}{N}\sum_{i=1}^{N} \|Y_{i}\|_{L^{2}  }^{2}$ is uniformly bounded in $L^{q}(\Omega)$ for $q$ in Assumption \ref{a:main} by Lemma \ref{Y:L2}.  Additionally, by Lemma \ref{lem:Lp} , Jensen's inequality, and the fact that $X_{i}$ and $X_{j}$ are identically distributed (which follows from the i.i.d. hypothesis in Assumption \ref{a:main} ) it holds
	\begin{align}
	\sup_{N \geq 1} \sup_{t \in [0,T]} \E \bigg (\frac{1}{N}\sum_{i=1}^{N}\|X_{i}(t) \|_{L^{2} }^{2}  \bigg )^{2} \leq \sup_{t \in [0,T]}\E\|X_{1}(t)\|_{L^{2}}^{4}<\infty. \nonumber
	\end{align}
	Notice that at this stage we are appealing to the assumption $\E \|\eta_{i}\|_{L^{p_{0}}}^{p_{0}} \lesssim 1$ in order to meet the hypotheses of Lemma \ref{lem:Lp} and deduce the final step above.  Hence, by the triangle inequality we find that
	\begin{align}
	\sup_{N \geq 1} \sup_{t \in [0,T]}\E \bigg (\frac{1}{N}\sum_{i=1}^{N}  \|v_{i}(t) \|_{L^{2} }^{2}  \bigg )^{q}<\infty,\nonumber
	\end{align}
	which implies the following convergence upgrade: $\frac{1}{N}\sum_{i=1}^{N}\|v_{i}(t)\|_{L^{2}}^{2}$ converges to zero in $L^{1}(\Omega)$ for each $t \in [0,T]$.  Finally, we appeal once more to the first bullet point in Assumption \ref{a:main} which is designed to ensure that $v_{i}$ and $v_{j}$ have the same law.  As a consequence we can now pass from empirical averages to components in light of
	\begin{equation}
	\E \|v_{i}(t)\|_{L^{2}}^{2}=\frac{1}{N}\sum_{i=1}^{N}\E \|v_{i}(t)\|_{L^{2}}^{2} \to 0.\label{e:Evi-sum-Evi}
	\end{equation}
	%completing the proof of the first result.
	
	{\sc Step} \arabic{diffEst} \label{diffEst6} \refstepcounter{diffEst} (Convergence as a stochastic process) %\zhu{how to modify this step?}

	The proof is largely the same as above, except that we do not estimate $v_i$ by an average over $i $ as in \eqref{e:Evi-sum-Evi}, since a supremum over time would not commute with a sum over $i$. Instead we deduce the following  bound
\begin{align}
	&\frac{\dif}{\dif t}\|v_{i}\|_{L^{2}}^{2}  +\frac12 \|v_{i}\|_{H^{1}}^{2}
	\leq
	C\Big(\frac{\bar{R}_{N}}{N}+\tilde{\bar{R}}_N^\prime\Big)
	+\frac{1}{N}\sum_{i=1}^N\|v_{i}\|_{H^1}^{2}
	+\frac{1}{N^2}\Big\|\sum_j X_jv_j\Big\|_{L^2}^2\label{e:nest}
\\&
	+C\Big(1+\tilde{R}_N^{21}%+\tilde{R}_N
	+\!\!\! \sum_{k\in\{3,5\}}\!\!\! R_{N}^{k}+ \!\!\!\! \sum_{k\in\{3,4,7\}}\!\!\!\! \tilde{R}_{N}^{k}
	+\frac{1}{N}\sum_{j=1}^{N}\|X_{j}\|_{L^{4}}^{4}
	+ \Big(\|X_{i}\|_{L^{4}}^{4/(1-2s)}+\|\Lambda^sX_i\|_{L^4}^4 \Big)   \Big)\|v_{i}\|_{L^{2}}^{2}\nonumber\\&
	+C\Big(1+\tilde{R}_N^{22}+\tilde{R}_{N}^{3}+\tilde{R}_{N}^{5}+\tilde{R}_{N}^{6}+\|X_i\|_{L^4}^4 \Big)\frac{1}{N}\sum_{j=1}^N\|v_{j}\|_{L^{2}}^{2},\label{nest2}
\end{align}
	where all the ``tilde $R$-terms'' are defined analogously to their ``un-tilde'' counterparts with slight tweaks:
\begin{equs}
\tilde{\bar{R}}_N^\prime &\eqdef\frac{1}{N^2} \Big\| \sum_{j=1}^NG_j^2Z_i\Big\|_{H^{-s}}^2
	+\frac{1}{N^2} \Big\| \sum_{j=1}^NG_j^3Z_i\Big\|_{H^{-s}}^2,
\\
	%$$\tilde{R}_N\eqdef R_{N}^{3}+R_{N}^5+\sum_{i=3}^{4}\tilde{R}_{N}^{i}+\tilde{R}_N^7.$$
\tilde{R}_N^{21}&\eqdef\frac{1}{N}\sum_{j=1}^N\|\Wick{Z_j^2}\|^{2/(2-s)}_{\bC^{-s}},
\qquad\tilde{R}_N^{22}\eqdef\frac{1}{N}\sum_{j=1}^N\|\Wick{Z_iZ_j}\|^{2}_{\bC^{-s}},
\\
\tilde{R}_N^{3}&\eqdef\|Z_i\|^{2/(1-s)}_{\bC^{-s}},\qquad
\tilde{R}_{N}^{5}\eqdef \Big (1+\frac{1}{N} \sum_{j=1}^{N} \|\nabla Y_{j} \|_{L^{2}}^{2}   \Big )^{\frac{s}{1-s}}   \|Z_{i}\|_{\bC^{-s}}^{\frac2{1-s}},
\\
	\tilde{R}_{N}^{4}&\eqdef \Big (1+\frac{1}{N} \sum_{j=1}^{N} \|X_{j}\|_{H^1}^{2}   \Big )^{\frac{2s}{2-s}}  (\|Z_{i}\|_{\bC^{-s}}^{\frac{4}{2-s}}+1)
	 +\Big(\frac{1}{N} \SZ\Big)^{\frac{2}{2-s}}\|X_{i}\|_{H^1}^{\frac{4s}{2-s}} +\Big(\frac{1}{N} \SZ\Big)^{2}\|X_{i}\|_{L^4}^{2},
\\
\tilde{R}_N^6&\eqdef \Big (\frac{1}{N}\sum_{j=1}^{N}\|Y_{j} \|_{L^{4}}^{2}    \Big )^{\frac{s}{1-s} }\|Z_{i}\|_{\bC^{-s}}^{\frac{2}{1-s}},\qquad
	\tilde{R}_N^{7}\eqdef\|\Lambda^{-s}Z_i\|^{2/(1-s)}_{L^\infty},
\end{equs}
	%+ \Big (\frac{1}{N}\sum_{j=1}^{N}\|\nabla X_{j}\|_{L^{2}}^{2}    \Big )^{\frac{2s}{2-s} } \|Z_{i}\|_{B^{-s}_{\infty,\infty}}^{\frac{4}{2-s}}.  \nonumber \\
where $\SZ$ is as in \eqref{e:def-SZ}.
In fact all the terms are similar as above except the following two terms:
	$$-\frac{1}{N}\sum_{j=1}^N\int X_iX_jv_jv_i \dif x-\frac{2}{N}\sum_{j=1}^N\< X_iv_jv_i, Z_j\>:=J_1+J_2.$$
The term $J_1$ is treated differently than above, since without the sum over $i$ we could not move it to the l.h.s. as a coercive quantity. 	We have
	\begin{align*}
	J_1 & \leq \frac{1}{N}\sum_{j=1}^N\int v_j^2X_i^2 \dif x+\frac{1}{N}\sum_{j=1}^N\int v_i^2X_j^2 \dif x
	\\& \leq C \Big (\frac{1}{N} \sum_{j=1}^{N}\|v_{j} \|_{L^{4}}^{2}  \Big ) \|X_{i}\|_{L^{4}}^{2} +C\|v_i\|_{L^4}^2\frac{1}{N}\sum_{j=1}^N\|X_j\|_{L^4}^2
	\\& \leq C \Big (\frac{1}{N}  \sum_{j=1}^{N}\|v_{j} \|_{H^{1}}^{2}  \Big )^{1/2}\Big ( \frac{1}{N} \sum_{j=1}^{N}\|v_{j} \|_{L^{2}}^{2}  \Big )^{1/2}\|X_{i}\|_{L^{4}}^{2}+\frac{1}{8}\|v_i\|_{H^1}^2+C\|v_i\|_{L^2}^2\Big(\frac{1}{N}\sum_{j=1}^N\|X_j\|_{L^4}^2\Big)^2,
	\end{align*}
	which by  Young's inequality deduce one contribution to \eqref{nest2}.
	For the second term we have
	\begin{align}
	J_2
	&\lesssim
	\sum_{s'\in \{0,s\}}
	\Big (\frac{1}{N } \sum_{j=1}^{N} \|v_{i}v_{j}X_{i} \|_{L^{1}}^{2(1-s') } \|\nabla ( v_{i} v_{j}X_{i}  )  \|_{L^{1}}^{2s' } \Big )^{\frac12} \Big (\frac{1}{N} \SZ \Big )^{\frac12} .
	%+\Big (\frac{1}{N } \sum_{j=1}^{N} \|v_{i}  v_{j}X_{i}  \|_{L^{1}}^{2} \Big )^{\frac12} \Big (\frac{1}{N}\SZ \Big )^{\frac12}. 	
	\label{nest11}
	\end{align}
	Using H\"older's inequality in the form $ \| v_{i}v_{j}X_{i} \|_{L^{1}} \leq \|v_{i}\|_{L^{2}}\|v_{j}X_{i}\|_{L^{2}}$ together with the bound for the first term in $J_1$ we obtain the estimate for $s'=0$ in $J_2$, which corresponds to the last term in $\tilde{R}_N^4$. Moreover, we have
	\begin{align}
	\|\nabla (v_{i}v_{j}X_{i} )  \|_{L^{1}} &\leq \|\nabla v_{i}\|_{L^{2}} \|v_{j}X_{i} \|_{L^{2}} +\|v_{i}\|_{L^{2}}^{1/2}\|v_i\|_{H^1}^{1/2}\|\nabla v_{j}\|_{L^{2}} \|X_{i}\|_{L^{4}}\nonumber \\
	&\qquad +\|v_{i}\|_{L^{2}}^{1/2}\|v_i\|_{H^1}^{1/2}\|v_{j}\|_{L^{2}}^{1/2} \|v_{j}\|_{H^1}^{1/2}\|\nabla X_{i}\|_{L^{2}},\nonumber
	\end{align}
	and inserting this into the term $s'=s$ in \eqref{nest11} and applying H\"older's inequality for the summation in $j$ leads to the following
	\begin{align}
	&\Big(\frac{1}{N} \sum_{j=1}^{N}\|v_{j}X_{i}  \|^2_{L^{2}} \Big )^{1/2}\|v_{i}\|_{L^{2}}^{1-s}\|v_{i}\|_{H^{1} }^{s} \Big (  \frac{1}{N}  \SZ \Big )^{1/2} \nonumber \\
	&+\Big(\frac{1}{N} \sum_{j=1}^{N}\|v_{j}X_{i}  \|^{2(1-s)}_{L^{2}}\|v_{j}\|_{H^{1} }^{2s} \Big )^{1/2} \|v_{i}\|_{L^{2}}^{1-s/2}\|v_{i}\|_{H^{1} }^{s/2} \|X_{i}\|_{L^{4}}^s \Big (  \frac{1}{N} \SZ\Big )^{1/2} \nonumber \\
	&+\Big(\frac{1}{N}\sum_{j=1}^{N} \| v_{j}X_{i} \|_{L^{2}}^{2(1-s)}\|v_j\|_{H^1}^s\|v_j\|_{L^2}^s\Big)^{1/2}\|\nabla X_i\|_{L^2}^s\|v_{i}\|_{L^{2}}^{1-\frac{s}2} \|v_{i}\|_{H^1}^{\frac{s}2}
	\Big (  \frac{1}{N} \SZ \Big )^{1/2} \nonumber.
	\end{align}
	Finally, we apply Young's inequality and obtain the contribution of $\tilde{R}_N^4$ in the estimate \eqref{nest2}.
	
	Using the fact that \eqref{eq:zz3} converges to zero in probability, we deduce the $L^1(0,T)$ norm of \eqref{nest2} and the right hand side of \eqref{e:nest} converges to zero in probability.
	Then by Gronwall's inequality and Lemma
	\ref{lem:pro} imply $\sup_{t\in[0,T]}\|v_i(t)\|_{L^2}^2\to0$ in probability, as $N\to \infty$.
In this step we see that we don't use the condition that the initial conditions $(z_i^N,y_i^N,z_i,\eta_i)_{i=1}^N$ are exchangeable.

	{\sc Step} \arabic{diffEst} \label{diffEst5} \refstepcounter{diffEst} (General initial datum)
	% \zhu{how to modify this step?}
	To this end, define $u_{i}\eqdef S_t(z_i^N-z_i)$ and note that we have the following extra terms:
	\begin{align*}
	\bar{I}^N&:=-\frac{1}{N}\sum_{i,j=1}^N \bigg[\<Y_j^2,v_iu_i\>+2\<Y_jY_iu_j,v_i \>+2\<Y_jv_i,\Wick{Z_i^NZ_j^N}-\Wick{Z_iZ_j}\>\nonumber\\
	&\qquad\qquad\qquad+\<Y_iv_i,\Wick{Z_j^{N,2}}-\Wick{Z_j^2}\>+\<v_i,\Wick{Z_i^NZ_j^{N,2}}-\Wick{Z_iZ_j^2}\>\bigg]
	\end{align*}
	These  terms could also be estimated similarly as that for $I_1^N$ and $I_2^N$ by using $\|u_i\|_{\bC^s}\lesssim t^{-(s+\kappa)/2}\|z_i^N-z_i\|_{\bC^{-\kappa}}$.
Since the proof follows a similar line of argument as in Steps \ref{diffEst7}--\ref{diffEst2}, we place the details in appendix.
\end{proof}

%\br
%From the proof in Step 5 and Step 6, the first assumption in Assumption \ref{a:main} could be weakened to
%the condition that the random variables
%$\{(z_i, \eta_i) \}_{i=1}^{N}$ are i.i.d.. In this case,
%for every $i$ and every $T>0$, $\|v_i^N\|_{C_TL^2}$ converges to zero in probability, as $N\to \infty$.
%\er

%\zhu{add remark here}\hao{Now the supplementary file is titled as ``Estimates in Step 7 of Theorem 5.1 in main file.'' It should be ``Detailed estimates for Remark 4.2 in main file.''} \zhu{I made changes.}\hao{Is the supplementary file in our shared folder ``main tex''? I can't see it.}

%\br
%Theorem \ref{th:conv-v} also holds under the more general assumptions  that $z_{i}^{N} \neq z_{i} $ and for every $p\geq1$ $$\E[\|z_i^N-z_i\|_{\bC^{-\kappa}}^p]\to0,\quad \frac1N\sum_{i=1}^N\|z_i^N-z_i\|_{\bC^{-\kappa}}^p\to^{\mathbf{P}}0\quad \mbox{ as }N\to\infty,$$
%	$$\E[\|z_i^N\|_{\bC^{-\kappa}}^p+\|z_i\|_{\bC^{-\kappa}}^p]\lesssim1.$$
	   %To this end, define $u_{i}\eqdef S_t(z_i^N-z_i)$ and note that we have the following extra terms:
%\begin{align*}
%\bar{I}^N&:=-\frac{1}{N^2}\sum_{i,j=1}^N \bigg[\<Y_j^2,v_iu_i\>+2\<Y_jY_iu_j,v_i \>+2\<Y_jv_i,\Wick{Z_i^NZ_j^N}-\Wick{Z_iZ_j}\>\nonumber\\
%&\qquad\qquad\qquad+\<Y_iv_i,\Wick{Z_j^{N,2}}-\Wick{Z_j^2}\>+\<v_i,\Wick{Z_i^NZ_j^{N,2}}-\Wick{Z_iZ_j^2}\>\bigg]
%\end{align*}
%\begin{align*}
%\bar{I}_{1}^N&:=-\frac{1}{N^2}\sum_{i,j=1}^N [\<Y_j^2,v_iu_i\>+2\<Y_jv_iu_i, Z_j^N\>+\<\Wick{Z_j^{N,2}},v_iu_i\>]\nonumber\\
%\bar{I}_{2}^N&:=-\frac{1}{N^2}\sum_{i,j=1}^N\<v_iu_j, (\Phi_j+\Psi_j)\Psi_i\>.\nonumber
%\end{align*}

%

We recall the following result from probability theory, used in Steps \ref{diffEst3} and \ref{diffEst6} of the above lemma, %which % from classical probability theory
which can be deduced with elementary arguments.
\begin{lemma}\label{lem:pro}
	Let $\{U_{N}\}_{N=1}^{\infty}$ be a non-negative sequence of 1d random variables converging to zero in probability.  Let $\{V_{N}\}_{N=1}^{\infty}$ be a non-negative sequence of random variables with tight laws.  Then the sequence $\{U_{N}V_{N}\}_{N=1}^{\infty}$ converges to zero in probability.
\end{lemma}
%\begin{proof}
%	For each $M,\overline{M}>0$  it holds
%	\begin{equation}
%	P(U_{N}V_{N}>M) \leq P\big(U_{N}>\frac{M}{\overline{M}} \big)+P(V_{N}>\overline{M} ) \nonumber.
%	\end{equation}
%	Given $\epsilon>0$, we may first choose $\overline{M}$ to make the second term small uniformly in $N$ and then send $N \to \infty$ to make the first term small.
%\end{proof}

%%%%%%%%%%%%%%%%%%%%%%%%%%%%%%%%%%%%%%%%%%%%%%%%%%%%%%%%%%%%%%%%  Invariant measure section
%%%%%%%%%%%%%%%%%%%%%%%%%%%%%%%%%%%%%%%%%%%%%%%%%%%%%%%%%%%%%%%%
%%%%%%%%%%%%%%%%%%%%%%%%%%%%%%%%%%%%%%%%%%%%%%%%%%%%%%%%%%%%%%%%
%%%%%%%%%%%%%%%%%%%%%%%%%%%%%%%%%%%%%%%%%%%%%%%%%%%%%%%%%%%%%%%%
%%%%%%%%%%%%%%%%%%%%%%%%%%%%%%%%%%%%%%%%%%%%%%%%%%%%%%%%%%%%%%%%
%%%%%%%%%%%%%%%%%%%%%%%%%%%%%%%%%%%%%%%%%%%%%%%%%%%%%%%%%%%%%%%%
%%%%%%%%%%%%%%%%%%%%%%%%%%%%%%%%%%%%%%%%%%%%%%%%%%%%%%%%%%%%%%%%
%%%%%%%%%%%%%%%%%%%%%%%%%%%%%%%%%%%%%%%%%%%%%%%%%%%%%%%%%%%%%%%%

\section{Invariant measure and observables}\label{sec:inv}

We now study the invariant measure for the equation
\begin{equation}\label{eq2:Psi}
\LL\Psi=-\mathbf{E}[\Psi^2- Z^2]\Psi+\xi,  %\quad \Psi(0)=\psi.
\end{equation}
with $\mathbf{E}[\Psi^2-Z^2]=\mathbf{E}[X^2]+2\mathbf{E}[XZ]$ for $X=\Psi-Z$ and $\xi$ space-time white noise.
Here, since we are only interested in the stationary setting in this section,
we overload the notation in the previous sections and simply write
$Z$ for the  stationary solution to the  linear equation
\begin{equation}\label{eq:Zm}
\LL Z=\xi\;,
\end{equation}
and we consider the decomposition (slightly different from Section \ref{sec:limit})
$X\eqdef \Psi-Z$, so that
\begin{equation}\label{eq2:X}
\LL X=-\mathbf{E}[X^2+2XZ](X+Z),\quad X(0)=\Psi(0)-Z(0).
\end{equation}
For the case that $m=0$ we restrict the solutions $\Psi$ and $Z$ satisfying $\<\Psi,1\>=\<Z,1\>=0$.

%Here \eqref{eq2:Psi} is the same as \eqref{eq:Psi2}.
By Theorem \ref{Th:global} for every initial data
$$\Psi(0)=\psi\in \bC^{-\kappa}$$
with $\E\|\psi\|_{\bC^{-\kappa}}^p\lesssim1$ there exists a unique global solution $\Psi$ to \eqref{eq2:Psi}.
We immediately find that $Z$ is a stationary solution to  \eqref{eq2:Psi}.  This follows since the unique solution to \eqref{eq2:X} starting from zero is identically zero.  Furthermore, we define a semigroup $P_t^*\nu$ to denote the law of $\Psi(t)$ with the initial condition distributed according to a measure $\nu$. By uniqueness of the solutions to \eqref{eq2:Psi}, we have
$P_t^*=P_{t-s}^*P_s^* $ for $ t\geq s\geq0$.
By direct probabilistic calculation we  easily obtain the following result, which implies that  the implicit constant in Lemma \ref{le:ex} is independent of $m$.

\bl\label{le:ex1}
For $\kappa'>\kappa>0$ and $p \geq1$, it holds that
$$
\sup_{m\geq0}\mathbf{E}[\|Z_i\|_{C_T\bC^{-\kappa}}^p]
+\sup_{m\geq0}\mathbf{E}[\| \Wick{Z_iZ_j} \|_{C_T\bC^{-\kappa}}^p]
+\sup_{m\geq0}\mathbf{E}[\| \Wick{Z_i Z_j^2}\|_{C_T\bC^{-\kappa}}^p]
\lesssim 1,
$$
where the proportional constants are independent of $i,j,N$.
\el

\begin{proof} By a standard technique (c.f. \cite{GP17}), it is sufficient to calculate
	\begin{align*}
	\E|\Delta_q Z_i(t)|^2\lesssim \sum_{k\in\mathbb{Z}^2}\int_{-\infty}^t\theta(2^{-q}k)^2 |e^{-2(t-s)(|k|^2+m)}|\dif s \lesssim \sum_{k\in \mathbb{Z}^2} 2^{q \kappa}\frac{1}{|k|^{\kappa}(|k|^2+m)},
	\end{align*}
	where $\Delta_q$ is a Littlewood-Paley block and $\theta$ is the Fourier multiplier associated with $\Delta_q$.
	From here we see the bound is independent of $m$. Other terms can be bounded in a similar way.
\end{proof}

 For $R_N^0$ defined in \eqref{eq:R0} with $Z_i$ stationary, we have the following result.

\bl 
For every $q\geq 1$ it holds that
\begin{align}\label{mom}
\E[(R_N^0)^q]\lesssim1.
\end{align}
\el
\begin{proof}
	Since we will have several similar calculations in the sequel, we first demonstrate
	such calculation in the case $q=1$. We have
	\begin{align*}
	\E\frac{1}{N^2} \sum_{i=1}^{N}  \Big\|\sum_{j=1}^{N}\Lambda^{-s}(\Wick{Z_{j}^{2}Z_{i}})\Big\|^2_{L^2}
	=
	\frac{1}{N^2}\sum_{i,j_1,j_2=1}^N\E \Big\<\Lambda^{-s} \Wick{Z_{j_1}^2 Z_{i}},\Lambda^{-s} \Wick{Z_{j_2}^2 Z_{i}}\Big\>.
	%\\\lesssim&\frac{1}{N^2} \sum_{i=1}^{N}\E
	%	\|\sum_{j\neq i}\Lambda^{-s}(\Wick{Z_{j}^{2}Z_{i}})\|^2_{L^2}+\frac{1}{N^2} \sum_{i=1}^{N}\E
	%	\|\Lambda^{-s}(\Wick{Z_{i}^{3}})\|^2_{L^2}
	%\\=&\frac{1}{N^2} \sum_{i=1}^{N}\sum_{j\neq i}\E\|\Lambda^{-s}(\Wick{Z_{j}^{2}Z_{i}})\|_{L^2}^2+\frac{1}{N^2} \sum_{i=1}^{N}\E
	%	\|\Lambda^{-s}(\Wick{Z_{i}^{3}})\|^2_{L^2}\lesssim 1.
	\end{align*}
	We have 3 summation indices and a factor $1/N^2$.
	The contribution to the sum from the cases  $j_1=i$ or $j_2=i$ or  $j_1=j_2$
	is bounded by a constant in light of Lemma \ref{le:ex1}. If $i,j_1,j_2$ are all different, by independence
	and the fact that Wick products are mean zero, the terms are zero.
	%\hao{I rewrote the $q=1$ case to look more parallel with  $q=2$ below}
	
	For general $q\ge 1$,
	by Gaussian hypercontractivity and the fact that $R_N^0$ is a random variable with finite Wiener chaos decomposition, we have
	\begin{align*}\E[(R_N^0)^q]\lesssim\E[(R_N^0)^2]^{q/2}.\end{align*}
	For the case that $q=2$ we  write it as
	\begin{align*}
	\frac{1}{N^4}\sum_{\substack{i_1,i_2,j_k=1\\ k=1\dots 4}}^N\E \Big\<\Lambda^{-s} \Wick{Z_{j_1}^2 Z_{i_1}},\Lambda^{-s} \Wick{Z_{j_2}^2 Z_{i_1}}\Big\>
	\Big\< \Lambda^{-s} \Wick{Z_{j_3}^2 Z_{i_2}},\Lambda^{-s} \Wick{Z_{j_4}^2 Z_{i_2}}\Big\>.
	\end{align*}
	We have $6$ indices $i_1, i_2, j_k, k=1,...,4$ summing from $1$ to $N$ and an overall factor $1/{N^4}$. Using again Lemma \ref{le:ex1}, we reduce the problem to the cases where five or six of the indices are different.  However, in these two cases, by independence the expectation is zero, so \eqref{mom} follows.
\end{proof}

\subsection{Uniqueness  of the invariant measures}\label{Uniqueness  of invariant measure}
We now turn to the question of uniqueness for the invariant measure of \eqref{eq2:Psi}.
Since the non-linearity in the SPDE \eqref{eq2:Psi} involves the law of the solution, the associated semigroup $P_t^*$ is generally nonlinear
i.e.
$$P_t^*\nu\neq \int (P_t^*\delta_\psi)\nu(\dif \psi),$$
for a non-trivial distribution $\nu$ (see e.g. \cite{W18}).  As a result, its unclear if the general ergodic theory for Markov processes (see e.g. \cite{DZ96}, \cite{HMS11}) can be applied directly in our setting.  Fortunately, \eqref{eq2:X} has a strong damping property in the mean-square sense, which comes to our rescue and allows us to proceed directly by a priori estimates.

\bl\label{le:m1}
There exists $C_0>0$ such that for all
$$m\geq 2C_0(\E \|\Wick{Z_{2}Z_{1}}\|_{\bC^{-s}}^{2}+(\E\|Z_1\|_{\bC^{-s}}^2)^{\frac{1}{1-s}}+1):=m_0,$$
there exists a universal $C$ with the following property: for every solution $\Psi$ to \eqref{eq2:Psi} with $\Psi(0) \in \bC^{-\kappa}$,
\begin{equation}
\sup_{t \geq 1}e^{\frac{mt}{2} }\mathbf{E}\|\Psi(t)-Z(t) \|_{L^2}^2 \leq C.\label{se51}
\end{equation}
\el

\begin{proof}
	The proof relies heavily on several computations performed in Lemma \ref{lem:l2} where we used slightly different notation, so we will write $X_{i}$ instead of $X$ and $Z_{i}$ instead of $Z$ for the remainder of this proof. Revisiting the first step of Lemma \ref{lem:l2} where we established \eqref{s20}, we find that $I^{1}$ defined in \eqref{eq:I} and the first contribution to $I^{2}$ defined in \eqref{eq:I} vanishes in light of $\E(\Wick{Z_{j}^{2}})=0$.  It follows that
	\begin{align}
	&\frac{1}{2}\frac{\dif}{\dif t} \E\|X_{i}\|_{L^{2}}^{2}+\E\|\nabla X_{i}\|_{L^{2}}^{2}+m\E \|X_i\|_{L^2}^2+\|\E X_{i}^{2}\|_{L^{2}}^{2} \nonumber \\
	&=-2 \E \langle X_{i}X_{j},Z_{i}Z_{j} \rangle-3\E\langle  X_{i}X_{j}^{2},Z_{i} \rangle. \nonumber
	\end{align}
	Furthermore, in light of \eqref{s28} and \eqref{se55}, we obtain
	\begin{align}
	\E \langle X_{i}X_{j},Z_{i}Z_{j} \rangle &\lesssim \big (\|\E X_{i}^{2} \|_{L^{1}}^{2}+\E\|\nabla X_{i}\|_{L^{2}}^{2} \| \E X_{j}^{2}  \|_{L^{1}} \big )^{1/2}\big (  \E \|\Wick{Z_{j}Z_{i}}\|_{\bC^{-s}}^{2} \big )^{\frac{1}{2}}\nonumber   \\
	\E \langle X_{i}X_{j}^{2},Z_{i} \rangle &\lesssim \|\E X_{i}^{2}\|_{L^{2} }\big ( \E\|\nabla X_{i}\|_{L^{2}}^{2} \big )^{\frac{s}{2} }\big (\E\|X_{i}\|_{L^{2}}^{2} \big )^{\frac{1-s}{2}}\big (  \E \|Z_{i}\|_{\bC^{-s}}^{2} \big )^{\frac{1}{2}} \nonumber\\
	&+\|\E X_{i}^{2}\|_{L^{2}} \big ( \E\|X_{j}\|_{L^{2}}^{2} \big )^{\frac{1}{2}}\big (  \E \|Z_{i}\|_{\bC^{-s}}^{2} \big )^{\frac{1}{2}} \nonumber.
	\end{align}
	We will use these estimates in two different ways.  On one hand, using Young's inequality with respective exponents $(2,2)$ and $(2,\frac{2}{s},\frac{2}{1-s} )$ followed by  Lemma \ref{le:ex1}, we find that
	\begin{equation}
	\frac{1}{2}\frac{\dif}{\dif t} \E \| X_{i} \|_{L^2}^2+\frac{1}{2}\E\|\nabla X_i\|_{L^2}^2 +m\E \|X_{i}\|_{L^2}^2+ \|\E X_{i}^2\|_{L^2}^2 \lesssim 1.\label{se52}
	\end{equation}
	As a consequence, noting that $ \|\E X_{i}^2\|_{L^2}^2 \geq \big ( \E \| X_{i} \|_{L^2}^2 \big )^{2}$, applying Lemma \ref{lem:co} it holds
	\begin{align}
	\sup_{t>0}(t\wedge1)\mathbf{E}[\|X_{i}(t)\|_{L^2}^2] \lesssim1,\label{bd2:uniX}
	\end{align}
	where the implicit constant is independent of the initial data.  %Furthermore, integrating \eqref{se52} we obtain for $t\geq1$
	%\begin{align}
	%\int_1^t\mathbf{E}\| X_{i}\|_{H^1}^2\dif s \lesssim t \label{bd2:uniX1}.
	%\end{align}
	On the other hand, Young's inequality also yields
	\begin{equation}
	\frac{\dif}{\dif t}\E\|X_{i}\|_{L^{2}}^{2}+m\E\|X_{i}\|_{L^{2}}^{2} \leq C_0(\E \|\Wick{Z_{2}Z_{1}}\|_{\bC^{-s}}^{2}+(\E\|Z_1\|_{\bC^{-s}}^2)^{\frac{1}{1-s}}+1)\mathbf{E}\|X_{i}\|_{L^2}^2. \label{se50}
	\end{equation}
	%Further, by \eqref{bd2:uniX} and \eqref{bd2:uniX1}, there exists a universal $M>0$ such that
	%\begin{equation}
	%\int_{1}^{t}\big(1+ \mathbf{E}\|X_{i}\|_{H^1}^2+\|\mathbf{E}X_{i}^2\|_{L^2}^2   \big)\dif s \leq Mt.\nonumber
	%\end{equation}
	Applying Gronwall's inequality over $[1,t]$ leads to
	\begin{equation}
	e^{(m-\frac{m_0}{2})t}\E \|X_{i}(t)\|_{L^{2}}^{2} \lesssim \E \|X_{i}(1)\|_{L^{2}}^{2} \nonumber,
	\end{equation}
	so choosing $m \geq m_0$, using \eqref{bd2:uniX}, and taking the supremum over $t \geq 1$, we arrive at \eqref{se51}.
\end{proof}
We now apply the above result to show that for sufficiently large mass, the unique invariant measure to \eqref{eq2:Psi} is Gaussian.  To this end, define the $\bC^{-1}$-Wasserstein distance
$$\mathbb{W}_p'(\nu_1,\nu_2):=\inf_{\pi\in\mathscr{C}(\nu_1,\nu_2)}\left(\int\|\phi-\psi\|_{\bC^{-1}}^p\pi(\dif \phi,\dif \psi)\right)^{1/p},$$
where $\mathscr{C}(\nu_1,\nu_2)$ denotes the collection of all couplings of $\nu_1, \nu_2$ satisfying $\int\|\phi\|_{\bC^{-1}}^p\nu_i(\dif \phi)<\infty$ for $i=1,2$.

\bt\label{th:u} For $m_0$ as in Lemma \ref{le:m1} and $m\geq m_0$  the unique invariant measure to \eqref{eq2:Psi} supported on $\bC^{-\kappa}$ is $\mathcal N(0,\frac12(-\Delta+m)^{-1})$, the law of the Gaussian free field.
\et

\begin{proof}
	Recall that $Z$ is a stationary solution to \eqref{eq2:Psi}.  Indeed, by definition, $\Psi=X+Z$, where $X$ solves \eqref{eq2:X}.  However, since $X(0)=0$, the identically zero process is the unique solution to \eqref{eq2:X}.  Hence, the law of $Z$, which we now denote by $\nu$, is invariant under $P_t^*$.  We now claim that for $m \geq m_{0}$, this is the only invariant measure supported on $\bC^{-\kappa}$.  Indeed, let $\nu_1$ be another such measure, then modifying the stochastic basis if needed, we may assume  there exists $\psi \in \bC^{-\kappa}$ on it such that $\psi \thicksim \nu_1$. By similar arguments as in Theorem \ref{Th:global} we may construct a solution $\Psi$ to \eqref{eq2:Psi} with $\Psi(0)=\psi$.  By invariance of $\nu_1$ and $\nu$ and the embedding $L^{2} \hookrightarrow \bC^{-1}$, c.f. Lemma \ref{lem:emb}, it follows that
	\begin{align*}
	&\mathbb{W}_2'(\nu,\nu_1)^2=\mathbb{W}_2'(P_t^*\nu,P_t^*\nu_1)^2\leq\mathbf{E}\|\Psi(t)-Z(t)\|_{\bC^{-1}}^2\lesssim e^{-\frac{mt}{2}},
	\end{align*}
	for $t \geq 1$ by Lemma \ref{le:m1}.  Letting $t\to\infty$ we obtain $\nu=\nu_1$.
\end{proof}

%\hao{I replaced the theorem with this paragraph and remark. We could also move this to the beginning of this section to let the reader know that the (an) invariant measure is just such a simple Gaussian measure.}
%For arbitrary $m\ge 0$, the Gaussian field with covariance $(-\Delta+m)^{-1}$ is an invariant measure for $\Psi$. Indeed, this  Gaussian field is the  invariant measure for $Z$, and if $Z(0)$ is sampled from this measure and $X(0)=0$, then $\mu|_{t=0}= 0$ and thus $\LL X=0$ so $X(t)\equiv 0$ for all $t$.
\begin{remark}
	Note that for the limiting equation $\LL\Psi=-\mu\Psi+\xi$,  if we assume that $\mu$ is simply a constant, it has a Gaussian invariant measure
	$\mathcal N(0,\frac12(-\Delta+m+\mu)^{-1})$.
	Assuming $\Psi \sim \mathcal N(0,\frac12(-\Delta+m+\mu)^{-1})$ and $Z  \sim \mathcal N(0,\frac12(-\Delta+m)^{-1})$,
	the self-consistent condition $\mathbf{E}[\Psi^2- Z^2] = \mu$ then yields
	\begin{equation*}%\label{e:2d-mu-equation}
	\frac12\sum_{k\in \mathbb Z^2} \Big(\frac{1}{|k|^2+m+\mu}-\frac{1}{|k|^2+m}\Big) = \mu
	\end{equation*}
	and for $\mu+m\ge 0$ we only have one solution $\mu=0$, since the LHS is monotonically decreasing in $\mu$.
\end{remark}

\br\label{rem:1} Changing  the renormalization constant in \eqref{eq2:Psi} will alter the mass of the Gaussian invariant measure. For instance, if we %replace $Z$ in \eqref{eq2:Psi}
change the renormalization constant in \eqref{eq2:Psi} to $\E Z_{i,\eps}^2(0,0)$ with  $Z_i$
by  the stationary solution to
$(\p_t-(\Delta-a))Z_i=\xi_i$ with $a>0$, one invariant measure  is  Gaussian $\bar\nu\eqdef \mathcal{N}(0,\frac12(-\Delta+m+\mu_0)^{-1})$ with $\mu_0$ satisfying
\begin{equation*}%\label{e:2d-mu-equation1}
\frac12\sum_{k\in \mathbb Z^2} \Big(\frac{1}{|k|^2+m+\mu_0}-\frac{1}{|k|^2+a}\Big) = \mu_0.
\end{equation*}
Moreover by the same proof of Lemma \ref{le:m1} and Theorem \ref{th:u}, for $m+\mu_0$ large enough, $\bar\nu$ is  the unique invariant measure. Indeed, let $\Psi=\bar{X}+\bar{Z}$ with $\bar{Z}$  the stationary solution to $\LL\bar{Z}=-\mu_0 \bar{Z}+\xi$, then $\bar X$ satisfies the following equation:
\begin{align*}
\LL \bar{X}=-\mu_0\bar{X}-\E[\bar{X}^2+2\bar{X}\bar{Z}](\bar{X}+\bar{Z}),
\end{align*}
which is the same case as \eqref{eq:X2} with $m$ replaced by $m+\mu_0$.
\er

\subsection{Convergence of the invariant measures}\label{sec:con}
As a consequence of Lemma \ref{lem:Yglobal}, the solutions $(\Phi_i)_{1\leq i\leq N}$ to \eqref{eq:Phi2d} form a Markov process on $(\bC^{-\kappa})^{N}$ which, by strong Feller property in \cite{HM18} and irreducibility in \cite{HS19}, will turn out to admit a unique invariant measure, henceforth denoted by $\nu^{N}$.  Our goal in this section is to study the large $N$ behavior of $\nu^{N}$ and show that for sufficiently large mass, as $N \to \infty$, it's marginals are simply products of the Gaussian invariant measure for $\Psi$ identified in Theorem \ref{th:u}.  For this we rely heavily on the computations from Section \ref{sec:Y} for the remainder $Y$, but we leverage these estimates with consequences of stationarity.  To this end, it will be convenient to have a stationary coupling of the linear and non-linear dynamics \eqref{eq:li1} and \eqref{eq:Phi2d} respectively, which is the focus of the following lemma.

\bl\label{lem:zz1} There exists a unique invariant measure $\nu^N$ on $(\bC^{-\kappa})^N$ to \eqref{eq:Phi2d}. Furthermore, there exists a  stationary process $(\Phi_i^N, Z_i)_{1\leq i\leq N}$ such that the components $\Phi_i^N, Z_i$ are stationary  solutions to \eqref{eq:Phi2d} and \eqref{eq:li1}, respectively. Moreover, $\E\|\Phi_i^N(0)-Z_i(0)\|_{H^1}^2\lesssim1$ and for every $q>1$
\begin{align}\label{sszz4}
\E\bigg(\frac{1}{N}\sum_{i=1}^N\|\Phi_i^N(0)-Z_i(0)\|_{L^2}^2\bigg)^q\lesssim1
\end{align}
\el

\begin{proof} In the proof we fix $N$. Let $\Phi_i$ and $Z_i$ be solutions to \eqref{eq:Phi2d} and \eqref{eq:li1}, respectively.  By the general results of \cite[Section 2]{HM18}, it follows that
	$(\Phi_i, Z_i)_{1\leq i\leq N}$ is a Markov process on $(\bC^{-\kappa})^{2N}$, and we denote by $(P_t^N)_{t\geq0}$ the associated Markov semigroup.  To derive the desired structural properties about the limiting measure, we will follow the Krylov-Bogoliubov construction with a specific choice of initial condition that allows to exploit Lemma \ref{Y:L2}.  Namely, we denote by $\Phi_{i}$ the solution to \eqref{eq:Phi2d} starting from the stationary solution $\tilde{Z}_{i}(0)$, so that the process $Y_i=\Phi_i-Z_i$ starts from the origin.  Using Lemma \ref{Y:L2} and Corollary \ref{co:Y} with $y_{j}=0$ together with Lemma \ref{le:ex1} to obtain a uniform bound on $\E R_{N}$ with $R_N$ defined in \eqref{eq:RN} we find %and \cite[Lemma 3.8]{TW18}
	%we have
	%\begin{align*}\sup_{t>0}(t\wedge1)\mathbf{E}[\frac{1}{N}\sum_{i=1}^N\|Y_i(t)\|_{L^2}^2]\lesssim1,\end{align*}
	%where the implicit constant is independent of the initial data and $N$. This implies that
	for every $T\geq 1$
	\begin{align}
	\int_0^T\mathbf{E}\Big(\frac{1}{N}\sum_{i=1}^N\| Y_i(t)\|_{H^1}^2\Big)\dif t & \lesssim T,
	\label{sszz2}
	\\
	\E\int_0^T\bigg(\frac{1}{N}\sum_{i=1}^N\|Y_i(t)\|_{L^2}^2\bigg)^q\dif t & \lesssim T,
	\label{sszz3}
	\end{align}
	where the implicit constant is independent of $T$ and for $m=0$ we used the Poincar\'{e} inequality.
	By \eqref{sszz2} we obtain
	\begin{align*}\int_0^T\mathbf{E}\Big(\frac{1}{N}\sum_{i=1}^N\|\Phi_i(t)\|_{\bC^{-\kappa/2}}^2\Big)
	+\int_0^T\mathbf{E}\Big(\frac{1}{N}\sum_{i=1}^N\|Z_i(t)\|_{\bC^{-\kappa/2}}^2\Big)\lesssim T.
	\end{align*}
	% \scott{The second bound would be $T^{q}$ based on a direct application of the Lemma, this would follow if we used the m when doing Gronwall. }\zhu{we add a corollary to prove it. Please see Corollary \ref{co:Y}.}
	Defining $R_t^N:=\frac{1}{t}\int_0^tP_r^N\dif r$, the above estimates and the compactness of the embedding $\bC^{-\kappa/2} \hookrightarrow \bC^{-\kappa}$ imply the induced laws of $\{R_t^N \}_{t\geq0}$ started from $(\tilde{Z}(0),\tilde{Z}(0))$ are tight on $(\bC^{-\kappa})^{2N}$. Furthermore, by the continuity with respect to initial data, it is easy to check that $(P_t^N)_{t\geq0}$ is  Feller on $(\bC^{-\kappa})^{2N}$.
	By the Krylov-Bogoliubov existence theorem (see \cite[Corollary 3.1.2]{DZ96}) , these laws converge weakly in $(\bC^{-\kappa})^{2N}$ along a subsequence $t_k\to\infty$ to an invariant measure $\pi_N$ for $(P_t^N)_{t\geq0}$. The desired stationary process $(\Phi_i^N, Z_i)_{1\leq i\leq N}$ is defined to be the unique solution to \eqref{eq:Phi2d} and \eqref{eq:li1} obtained by sampling the initial datum $(\phi_{i},z_{i})_{i}$ from $\pi_{N}$. By \eqref{sszz2} we deduce
	\begin{align*}
	\E^{\pi_N}&\|\Phi_i(0)-Z_i(0)\|_{H^1}^2=\E^{\pi_N}\sup_{ \varphi }\<\Phi_i(0)-Z_i(0),\varphi\>^2
	\\
	&=\E\sup_{\varphi}\lim_{T\to\infty}\bigg[\frac1T\int_0^T\<Y_i(t),\varphi\>\dif t\bigg]^2\leq\lim_{T\to\infty}\frac1T\int_0^T\E\|Y_i(t)\|_{H^1}^2\dif t\lesssim1,
	\end{align*}
	where $\sup_{\varphi}$ is over smooth functions $\varphi$ with $\|\varphi\|_{H^{-1}}\leq 1$.
	Similarly using \eqref{sszz3}, \eqref{sszz4} follows.  Finally, we project onto the first component and consider the map $\bar{\Pi}_1:\mathcal{S}'(\mathbb{T}^2)^{2N}\rightarrow \mathcal{S}'(\mathbb{T}^2)^N$ defined through $\bar{\Pi}_1(\Phi,Z)=\Phi$.  Letting $\nu^N=\pi_N\circ \bar{\Pi}_1^{-1}$ yields an invariant measure to \eqref{eq:Phi2d}, and uniqueness follows from the general results of strong Feller property in  \cite[Theorem 3.2]{HM18} and irreducibility in \cite[Theorem 1.4]{HS19}.
\end{proof}

\begin{remark}
	Using a lattice approximation (see e.g. \cite{GH18a}, \cite{HM18a, ZZ18})
	one can show that the measure $\nu^N(\dif\Phi)$ indeed has the form \eqref{e:Phi_i-measure} (with Wick renormalization).
	% \hao{I deleted the measure which was written here and just refer to \eqref{e:Phi_i-measure}}
	%we could formally write$$\nu^N(\dif\Phi)=\frac{1}{C_N}\exp\left(-\int \sum_{j=1}^N|\nabla \Phi_j|^2+m|\Phi_j|^2+\frac{1}{N}(\sum_{j=1}^N\Phi_j^2)^2\dif x\right)\dif \Phi.$$
	%$\Phi=(\Phi_1,\Phi_2,...,\Phi_N)$.
\end{remark}

The next step is to study tightness of the marginal laws of $\nu^{N}$ over $\mathcal{S}'(\mathbb{T}^2)^N$.  To this end, consider the projection $\Pi_i:\mathcal{S}'(\mathbb{T}^2)^N\rightarrow \mathcal{S}'(\mathbb{T}^2)$ defined by $\Pi_i(\Phi)=\Phi_i$ and let $\nu^{N,i}\eqdef \nu^N\circ \Pi_i^{-1}$ be the marginal law of the $i^{th}$ component.
%$\nu^{N,i}$ should converge to the invariant measure of the linear equation.
Furthermore, for $k \leq N$, define the map $\Pi^{(k)}:\mathcal{S}'(\mathbb{T}^2)^N\rightarrow \mathcal{S}'(\mathbb{T}^2)^k$ via $\Pi^{(k)}(\Phi)=(\Phi_i)_{1\leq i\leq k}$ and the let $\nu^{N}_k\eqdef \nu^N\circ (\Pi^{(k)})^{-1}$  be the marginal law of the first $k$ components.
We have the following result: %Now we want to obtain the tightness of $\nu^{N,i}$, which is the invariant measure of $\Phi_i$.

\bt\label{th:con}
$\{\nu^{N,i}\}_{N\geq1}$ form a tight set of probability measures on $\bC^{-\kappa}$ for $\kappa>0$. %Any limit of the subsequence of $\{\nu^{N,i}\}$ is a stationary measure to \eqref{eq:Psi2}.
\et
\begin{proof}
	Let  $(\Phi^N_i, Z_i)_{1 \leq i \leq N}$ be the jointly stationary solutions to \eqref{eq:Phi2d} and \eqref{eq:li1} constructed in Lemma \ref{lem:zz1}.  To prove the result, in light of the compact embedding of $\bC^{-\kappa/2} \hookrightarrow \bC^{-\kappa}$ and the stationarity of $\Phi^{N}_{i}$, it suffices to show that the second moment of $\|\Phi^N_i(0)\|_{\bC^{-\kappa/2}}$ is bounded uniformly in $N$.  To this end,
	let $Y_i^N=\Phi^N_i-Z_i$, which is also stationary and note that
	\begin{align*}
	&\mathbf{E}\|\Phi^N_i(0)\|_{\bC^{-\kappa/2}}^2=\frac{2}{T}\int_{T/2}^T \mathbf{E}\|\Phi^N_i(s)\|_{\bC^{-\kappa/2}}^2\dif s
	\\\leq&\frac{4}{T}\int_{T/2}^T \mathbf{E}\|Z_i(s)\|_{\bC^{-\kappa/2}}^2\dif s+\frac{4}{T}\int_{T/2}^T \mathbf{E}\|Y^N_i(s)\|_{H^1}^2\dif s.
	\end{align*}
	The first term is controlled by Lemma \ref{le:ex1}.  For the second term, symmetry yields $Y^N_{i}$ and $Y^N_j$ are identical in law, which combined with Lemma \ref{Y:L2} implies that
	\begin{align*}
	\frac{2}{T}\int_{T/2}^T \mathbf{E}\|Y^N_i(s)\|_{H^1}^2\dif s
	=\frac{2}{T}\int_{T/2}^T \frac{1}{N}\sum_{i=1}^N\mathbf{E}\|Y^N_i(s)\|_{H^1}^2\dif s
	\leq \frac{C}{T}\mathbf{E}[\int_0^T R_N\dif t]\leq C,\end{align*}
	where we used that by stationarity $\sum_{i=1}^N\mathbf{E}\|Y^N_i(T)\|_{L^2}^2=\sum_{i=1}^N\mathbf{E}\|Y^N_i(T/2)\|_{L^2}^2$, with both being finite in view of Lemma \ref{lem:zz1}.  For $m=0$ we also used the Poincar\'{e} inequality.
	%We could obtain the tightness of stationary solutions $(\Phi_i^N, Z_i, Y_i^N)_i$ on $((\bC^{-\kappa})^\mathbb{N},(\bC^{-\kappa})^\mathbb{N},( H^{1-\kappa})^\mathbb{N})$. Here we impose the product topology on $\bC^{-\kappa})^\mathbb{N}$.
	%We obtain  that there exists a subsequence, which is still denoted by $N$, and  probability measures $(\nu^*,\nu^Z,\nu^X)$ on $((\bC^{-\kappa})^{\mathbb{N}},(\bC^{-\kappa})^\mathbb{N},(H^1)^\mathbb{N})$ such that
	%the laws of $(\Phi_i^N, Z_i, Y_i^N)_i$ converge to $(\nu^*,\nu^Z,\nu^X)$ weakly. By Skorohod representation Theorem we deduce that there exists a probability space $(\tilde{\Omega}, \tilde{\mathcal{F}}, \mathbb{P}) $ and  random variables $(\Phi^N_{i,0}(\tilde{\omega}), Z^N_{i,0}(\tilde{\omega}), Y_{i,0}^N(\tilde{\omega}))$ and  $(\Psi_{i,0}(\tilde{\omega}),Z_{i,0}(\tilde{\omega}),X_{i,0}(\tilde{\omega}))$
	%such that
	%\begin{itemize}
	% \item for every $i$,  $\Phi^N_{i,0}(\tilde{\omega})\rightarrow \Psi_{i,0}(\tilde{\omega})$ in $\bC^{-\kappa}$,
	%$Z^N_{i,0}(\tilde{\omega})\rightarrow Z_{i,0}(\tilde{\omega})$ in $\bC^{-\kappa}$ and
	%$Y^N_{i,0}(\tilde{\omega})\rightarrow X_{i,0}(\tilde{\omega})$ in $H^{1-\kappa}$.
	% \item $(\Phi^N_{i,0}(\tilde{\omega}), Z^N_{i,0}(\tilde{\omega}), Y_{i,0}^N(\tilde{\omega}))_i=^d(\Phi_i^N, Z_i, Y_i^N)_i$.
	%\item  $(\Psi_{i,0}(\tilde{\omega}),Z_{i,0}(\tilde{\omega}),X_{i,0}(\tilde{\omega}))=^d(\nu^*_i,\nu^Z_i,\nu^X_i)$.
	%\end{itemize}
\end{proof}

\br
%It is believed that any limiting measure obtained in Theorem \ref{th:con} is an invariant measure to \eqref{eq:21}. We try to use the convergence of the dynamics i.e. the main result in Section \ref{sec:dif}, to deduce it. However, it is required that
%the initial conditions for each component of  $\Psi_i$ are independent  in Section \ref{sec:dif}, which seems not easy to deduce directly. %of  $X^i$ satisfying  $\E\|\eta_i\|_{L^{p_0}}^{p_0}\lesssim1$ for $p_0>4$,
%%which seems not easy to deduce from the uniform estimate of $Y_i$.
It is reasonable to expect that any limiting measure obtained in Theorem \ref{th:con} is an invariant measure for \eqref{eq:Psi2} assuming only $m\ge 0$. However, this cannot be directly deduced from our main result in Section \ref{sec:dif} because we do not know a-priori that any limiting measure of $\nu^{N}$ is a product measure.
This is problematic because the initial conditions for each component of  $\Psi_i$ are assumed to be independent  in Section \ref{sec:dif}.
Nevertheless, we can prove below that this is indeed true if $m$ is large.
\er %\zhu{we change remark back.}\hao{I slightly reworded this remark. Please have a look.}

In the following we prove the convergence of the measure to the unique invariant measure by using the estimate in Lemma \ref{Y:L3}, which requires $m$ large enough.

Define the $\bC^{-\kappa}$-Wasserstein distance
\begin{align}\label{def:wa}\mathbb{W}_2(\nu_1,\nu_2):=\inf_{\pi\in\mathscr{C}(\nu_1,\nu_2)}\left(\int\|\phi-\psi\|_{\bC^{-\kappa}}^2\pi(\dif \phi,\dif \psi)\right)^{1/2},\end{align}
where $\mathscr{C}(\nu_1,\nu_2)$ denotes the set of all couplings of $\nu_1, \nu_2$ satisfying $\int\|\phi\|_{\bC^{-\kappa}}^2\nu_i(\dif \phi)<\infty$ for $i=1,2$.

\bt\label{th:m1} Let $\nu= \mathcal{N}(0,\frac12(m-\Delta)^{-1})$.
There exist $C_0>0$ such that for all $m\geq m_1$ where
$$m_1\eqdef C_0(\E\|Z_1\|_{\bC^{-s}}^{\frac{2}{2-s} }+\E \|\Wick{Z_{1}^{2}}\|_{\bC^{-s}}^{\frac{2}{2-s} }+E\|\Wick{Z_{2}Z_{1}}\|_{\bC^{-s}}^{2}+1)$$
one has
\begin{equation}
\mathbb{W}_2(\nu^{N,i},\nu) \leq CN^{-\frac{1}{2}}. \label{West}
\end{equation}
Furthermore, $\nu^N_k$ converges to $\nu\times...\times \nu$, as $N\to \infty$.
\et
\begin{proof}
	By Lemma \ref{lem:zz1} we may construct a stationary coupling $(\Phi^N_i, Z_i)$ of $\nu_N$ and $\nu$ whose components satisfy \eqref{eq:21} and \eqref{eq:li1}, respectively.  The stationarity of the joint law of $(\Phi^N_i, Z_i)$ implies that also $Y_i^N=\Phi^N_i-Z_i$ is stationary.  In the following we freely omit the time argument of expectations of stationary quantities.  We now claim that
	\begin{equation}
	\E \|Y_{i}^{N}\|_{H^{1}}^{2} \leq CN^{-1} \label{se8},
	\end{equation}
	which implies \eqref{West} by definition of the Wasserstein metric and the embedding $H^{1} \hookrightarrow \bC^{-\kappa}$ in $d=2$, c.f. Lemma \ref{lem:emb}.  To ease notation, we write $Y_i=Y_i^N$ in the following.   By \eqref{bs16} combined with the stationarity of $(Y_{j})_{j}$ and $(Z_{j})_{j}$, we find
	\begin{align*}
	&\sum_{j=1}^N\E\|\nabla Y_j\|_{L^2}^2+m\sum_{j=1}^N\E\|Y_j\|_{L^2}^2+\frac{1}{N}\E\bigg\|\sum_{i=1}^NY_i^2\bigg\|_{L^2}^2\\
	& \leq C\E R_N^0+\E\bigg(\sum_{j=1}^N\|Y_j \|_{L^2}^2(D_N+D_N^1 )\bigg),
	\end{align*}
	where $R_N^0$ is defined in \eqref{eq:R0} and
	\begin{align*}
	D_N&= C\bigg(\frac{1}{N}\sum_{j=1}^{N} \|Z_j\|_{\bC^{-s}}^{\frac{2}{2-s} }+\frac{1}{N}\sum_{j=1}^{N} \|\Wick{Z_{j}^{2}}\|_{\bC^{-s}}^{\frac{2}{2-s} }+1\bigg).\nonumber\\
	D_{N}^{1}&= C\bigg (\frac{1}{N^{2}}\sum_{i,j=1}^{N}\|\Wick{Z_{j}Z_{i}}\|_{\bC^{-s}}^{2}  \bigg ).\nonumber
	\end{align*}
	%\scott{For $D_{N}^{1}$, this comes from $R_{N}^{2}$ and the power there is $\frac{1}{2-s}$.  But the function $z \mapsto z^{\frac{1}{2-s}}$ is concave for $s \in (0,1)$.  Perhaps you mean to first do Young's and use the $+1$,  but then we should get for example $\|\Wick{Z_{j}Z_{i}}\|_{B_{\infty,\infty}^{-s}}^{ \frac{4}{2-s} }$ } \zhu{we change it as above.}
	Setting $A\eqdef\E D_N$ and $A_1\eqdef\E\|\Wick{Z_{2}Z_{1}}\|_{\bC^{-s}}^{2}$ we may re-center $D_{N}$ and $D_{N}^{1}$ above and divide by $N$ to obtain
	\begin{align*}
	&\frac1N\sum_{j=1}^N\E\|\nabla Y_j(t)\|_{L^2}^2+(m-A-A_1)\frac1N\sum_{j=1}^N\E\|Y_j(t)\|_{L^2}^2+\frac{1}{N^2}\E\bigg\|\sum_{i=1}^NY_i^2(t)\bigg\|_{L^2}^2
	\\&\leq C\frac1N\E R_N^0+\frac1N\E\bigg(\sum_{j=1}^N\|Y_j(t)\|_{L^2}^2(|D_N-A|+|D_N^1-A_1|)\bigg)
	\\&\leq C\frac1N\E R_N^0+\frac{1}{2}\frac1{N^2}\E\bigg(\sum_{j=1}^N\|Y_j \|_{L^2}^2\bigg)^2+\E|D_N-A|^2+\E|D_N^{1}-A_1|^2.
	\end{align*}
	For $m\geq A+A_1+1$, using that $Y_{i}$ and $Y_{j}$ are equal in law, we obtain
	\begin{align}
	\E\|Y_i\|_{H^{1}}^2&\leq \frac{1}{N}\sum_{j=1}^N\E\|\nabla Y_j(t)\|_{L^2}^2+ (m-A-A_{1})\frac1N\sum_{j=1}^N\E\|Y_j\|_{L^2}^2\nonumber \\
	&\leq C\frac1N\E R_N^0+\E|D_N-A|^2+\E|D_N^1-A_1|^2.\label{sszz1}
	\end{align}
	Using independence, we find
	\begin{equation}
	\E|D_N(t)-A|^2 \leq \frac{1}{N}\text{Var} \bigg (\|Z_1\|_{\bC^{-s}}^{\frac{2}{2-s} }+\|\Wick{Z_{1}^{2}}\|_{\bC^{-s}}^{\frac{2}{2-s} }   \bigg ) \leq \frac{C}{N}. \label{se6}
	\end{equation}
	To estimate $D_{N}^{1}$, we write  $M_{i,j}=\|\Wick{Z_{j}Z_{i}}\|_{\bC^{-s}}^{2}-A_1$ for $i\neq j$ and $M_{i,i}=\|\Wick{Z_{i}^2}\|_{\bC^{-s}}^{2}-A_1$ and have
	\begin{align}
	\E \bigg ( \frac{1}{N^{2}}& \sum_{i,j=1}^{N}M_{i,j}  \bigg )^{2}  \leq \E\bigg ( \frac{1}{N^{2}}\sum_{i=1}^N\sum_{j\neq i}M_{i,j}+\frac{1}{N^{2}}\sum_{i=1}^{N}M_{i,i}  \bigg )^{2} \nonumber \\
	& \leq \frac{2}{N^4}\sum_{i_1\neq j_1,i\neq j}\E  ( M_{i,j}M_{i_1,j_1} ) +\frac{2}{N^{2}}\E(M_{1,1}^{2})
	\lesssim \frac{1}{N}+\frac{2}{N^{2}}\E(M_{1,1}^{2} )\lesssim\frac{1}{N},\nonumber
	\end{align}
	where we used that for the case that $(i,j,i_1,j_1)$ are different,  $\E  ( M_{i,j}M_{i_1,j_1} ) =\E M_{i,j}\E M_{i_1,j_1}=0$.

	\iffalse of  random variables with entries in $L^{2}(\Omega)$ such that the off-diagonal entries are mean zero and identically distributed and for each $i\neq j\neq i_1\neq j_1$, $M_{i,j}$ is independent of $M_{i_1,j_1}$, \zhu{we make change here} it holds
	\begin{equation}
	\E \bigg ( \frac{1}{N}\sum_{i,j=1}^{N}M_{i,j}  \bigg )^{2} \leq \frac{N-1}{N^{2}}\E(M_{1,2}^{2} ) +\frac{1}{N^{3}}\E(M_{1,1}^{2} ). \label{se5}
	\end{equation}
	Indeed, it suffices to divide the $RHS$ into on-diagonal and off-diagonal entries, use Jensen's inequality, and then independence to obtain
	\begin{align}
	\E \bigg ( \frac{1}{N^{2}}\sum_{i,j=1}^{N}M_{i,j}  \bigg )^{2} &\leq \E\bigg ( \frac{1}{N^{2}}\sum_{i=1}^N\sum_{j\neq i}M_{i,j}+\frac{1}{N^{2}}\sum_{i=1}^{N}M_{i,i}  \bigg )^{2} \nonumber \\
	& \leq \frac{2}{N}\sum_{i=1}^{N}\E \bigg ( \frac{1}{N}\sum_{j \neq i} M_{i,j} \bigg )^{2} +\frac{2}{N^{3}}\E(M_{1,1}^{2}) \nonumber \\
	&\leq \frac{2(N-1)}{N^{2}}\E(M_{1,2}^{2} )+  \frac{2(N-1)}{N^{2}}\E(M_{1,2}M_{1,3} )+\frac{2}{N^{3}}\E(M_{1,1}^{2} ).\nonumber\\
	&\leq \frac{C(N-1)}{N^{2}}\E(M_{1,2}^{2} )+\frac{2}{N^{3}}\E(M_{1,1}^{2} ).\nonumber
	\end{align}
	Applying \eqref{se5} we obtain\fi Then we have
	\begin{align}
	\E \big | D_{N}^{1}-A_{1} \big |^{2} \lesssim \frac1N\label{se7}.
	\end{align}
	Inserting the estimates \eqref{se6}, and \eqref{se7} into \eqref{sszz1} and using \eqref{mom}, we obtain \eqref{se8}, completing the proof.
	\end{proof}
	
\br\label{rem:lambda}
Instead of assuming $m$ large, one could  alternatively consider arbitrary $m>0 $ and assume small nonlinearity. Namely, we could consider a nonlinearity    $-\frac{\lambda}{N}\sum_{j=1}^N \Wick{\Phi_j^2\Phi_i}$
instead of that of \eqref{eq:Phi2d},
and $-\lambda\mathbf{E}[\Psi^2- Z^2]\Psi$ instead of that of \eqref{eq2:Psi},
for $\lambda>0$. By tracing the proofs of Lemma~\ref{le:m1} and Theorem~\ref{th:m1},
we can easily see that given any $m>0$, there exists a constant $\lambda_0>0$, so that the statements of Lemma~\ref{le:m1} and Theorem~\ref{th:m1} hold
for any $\lambda \in (0,\lambda_0)$.
\er

\br\label{rem:2} Following Remark~\ref{rem:1},
with a change of renormalization constant therein,
%For $m+\mu_0$ large enough, $\bar\nu\eqdef\mathcal{N}(0,(-\Delta+m+\mu_0)^{-1})$ is  the unique invariant measure to \eqref{eq2:Psi}. In this case,
we can write $\Phi_i=\bar{Y}_i+\bar{Z}_i$ with $\bar{Z}_i$  the stationary solution to $\LL\bar{Z}_i=-\mu_0 \bar{Z}_i+\xi_i$. Then $\bar Y_i$ satisfies
\begin{align*}
\LL \bar{Y}_i=-\mu_0\bar{Y}_i-\frac{1}{N}\sum_{j=1}^N(\bar Y_j^2\bar Y_i+\bar{Y}_j^2\bar{Z}_i+2\bar{Y_j}\bar{Y}_i\bar{Z}_j
+2\bar{Y}_j\Wick{\bar{Z_i}\bar{Z}_j}+\Wick{\bar{Z}_j^2}\bar{Y}_i+\Wick{\bar{Z}_i\bar{Z}_j^2})
-\frac{2\mu_0}{N}(\bar{Y}_i+\bar{Z}_i),
\end{align*}
which is the same case as \eqref{eq:22} with $m$ replaced by $m+\mu_0$ and an extra term $\frac{2\mu_0}{N}(\bar{Y}_i+\bar{Z}_i)$. Here the Wick product of  $\bar{Z}_j$ is defined similarly as in section \ref{sec:re}. By the same proof of  Theorem \ref{th:m1} we know for $m+\mu_0$ large enough, $\nu^{N,i}$ (renormalized as in Remark \ref{rem:1}) converges to $\bar\nu$ and the other results in Theorem \ref{th:m1} also hold in this case.
\er

\section{Observables and their nontriviality}
\label{sec:non}

%We consider the stationary setting, namely,
%suppose $\Phi=(\Phi_i)_{1\leq i\leq N}\thicksim \nu^N$
%(where $\nu^N$ is as in \eqref{e:Phi_i-measure}. %\eqref{e:nuN1d}
%but with $m$ therein replaced by $m-C_{\wk}$).
%As in Section \ref{sec:label} 

\subsection{Observables}\label{sec:label}

In quantum field theories with symmetries, quantities that are
invariant under action of the symmetry group are of particular interest;
examples of such quantities in the SPDE setting include gauge invariant observables
e.g. \cite[Section~2.4]{Shen2018Abelian}. The model we study here exhibits $O(N)$ rotation symmetry
and formally, functions of the squared ``norm'' $\sum_i \Phi_i^2$ are
quantities that are $O(N)$ invariant.  Of course, such observables need to be suitably renormalized
to be well-defined and suitably scaled by factors of $N$ to have nontrivial limit as $N\to \infty$.

In this section we study the following two observables:
\begin{equ}[obob]
	\frac{1}{N^{1/2}}\sum_{i=1}^N \Wick{\Phi_i^2}\;,
	\qquad
	\frac{1}{N} \Wick{\Big( \sum_{i=1}^N \Phi_i^2 \Big)^2} \;,
\end{equ}
with $\Phi=(\Phi_i)_{1\leq i\leq N}\thicksim \nu^N$ for the invariant measure $\nu^N$ to \eqref{eq:Phi2d} given in Lemma \ref{lem:zz1}. In this section we omit the superscript  $N$ for simplicity.
These are defined as follows.
By Lemma \ref{lem:zz1}  we  decompose $\Phi_i=Y_i+Z_i$ with $(Y_i, Z_i)$ stationary.
With this we define
\begin{align}\label{ob1}
\frac{1}{\sqrt{N}}\sum_{i=1}^N \Wick{\Phi_i^2}
& \eqdef \frac{1}{\sqrt{N}}\sum_{i=1}^N(Y_i^2+2Y_iZ_i+\Wick{Z_i^2}),
\\
\frac{1}{N} \Wick{\Big( \sum_{i=1}^N \Phi_i^2 \Big)^2}
& \eqdef
\frac{1}{N}\sum_{i,j=1}^N \Big(Y_i^2Y_j^2+4Y_i^2Y_jZ_j+2Y_i^2\Wick{Z_j^2}\label{*}\\&
\qquad+\Wick{Z_i^2Z_j^2}+4Y_i\Wick{Z_iZ_j^2}+4Y_iY_j\Wick{Z_iZ_j}\Big).\label{ob2}
\end{align}
Here the Wick products are canonically defined as in \eqref{e:wick-tilde} with $a_\varepsilon=\E[Z_{i,\varepsilon}^2(0,0)]$, in particular
\begin{equ}[e:Zi2Zj2]
	\Wick{Z_i^2Z_j^2} =
	\begin{cases}
		\lim\limits_{\varepsilon\to0}(Z_{i,\varepsilon}^4-6a_\varepsilon Z_{i,\varepsilon}^2+3a_\varepsilon^2)   & (i=j)\\
		\lim\limits_{\varepsilon\to0}(Z_{i,\varepsilon}^2-a_\varepsilon)(Z_{j,\varepsilon}^2-a_\varepsilon) & (i\neq j).
	\end{cases}
\end{equ}

\br\label{rem:ob}
One could also define \eqref{obob} in $L^p(\nu^N)$ directly without using the decomposition $\Phi_i=Y_i+Z_i$.
In fact,
by similar argument as in  \cite[Section 8.6]{MR887102} or \cite{MR0489552}, one can show that $\nu^N$ is absolutely continuous with respect to the  corresponding Gaussian free field $\tilde\nu$  with a density in $L^p(\tilde\nu)$ for $p\in (1,\infty)$.
Since  \eqref{obob} with each $\Phi_i$ replaced by $Z_i$ can be defined via $L^p(\tilde{\nu})$ limit of mollification,
using %\cite[Corollary 3.4]{DPD03}
argument along the line of \cite[Lemma 3.6]{RZZ17} we know that
\eqref{obob} can be also defined as $L^p(\nu^N)$ limit of mollification (essentially H\"older inequality), and they
have the same law as the right hand side of \eqref{ob1} and \eqref{*}\eqref{ob2}.
\er
%\zhu{we add definition and remark for observable.}\hao{I slightly reworded the remark}

In this section we also consider $Y_i, Z_i$ as stationary process with $Z_i$ as the stationary solution of \eqref{eq:Zm} and $Y_i$ as the  solution of \eqref{eq:22}.

\bl\label{th:m2} There exists an $m_{0}$ such that for $m \geq m_{0}$ and $q\geq 1$
\begin{align}
&\E \bigg [ \bigg (\sum_{i=1}^N\|Y_i\|_{L^2}^2 \bigg)^q \bigg ]+\E \bigg[\bigg(\sum_{i=1}^N\|Y_i\|_{L^2}^2+1\bigg)^{q}\bigg(\sum_{i=1}^N\|\nabla Y_i\|_{L^2}^2\bigg) \bigg] \lesssim 1,\label{se19} \\
&\E \bigg[\bigg(\sum_{i=1}^N\|Y_i\|_{L^2}^2+1\bigg)^{q}\bigg\|\sum_{i=1}^NY_i^{2} \bigg\|_{L^2}^2\bigg]\lesssim 1,\label{se16}
\end{align}
where the implicit constant is independent of $N$.

\el
\begin{proof}
	First we observe that \eqref{se16} may be quickly deduced from \eqref{se19} with the help of the inequality
	\begin{equation}
	\bigg\|\sum_{i=1}^NY_i^2\bigg\|_{L^2}^2 \lesssim\bigg(\sum_{i=1}^N\|Y_i\|_{H^{1}}^{2}\bigg)\bigg(\sum_{i=1}^N\|Y_i\|_{L^2}^2\bigg). \label{se15}
	\end{equation}
	To obtain \eqref{se15}, note first that
	\begin{equation}
	\bigg\|\sum_{i=1}^NY_i^2\bigg\|_{L^2}^2=\sum_{i,j=1}^{N}\|Y_{i}Y_{j}\|_{L^{2}}^{2}\nonumber.
	\end{equation}
	Furthermore, by H\"{o}lder's inequality, \eqref{diff2}, and Young's inequality
	\begin{align}
	\|Y_{i}Y_{j}\|_{L^{2}}^{2} \leq \|Y_{i}\|_{L^{4}}^{2} \|Y_{j}\|_{L^{4}}^{2}&\lesssim \|Y_{i}\|_{H^{1}}\|Y_{j}\|_{L^{2}}\|Y_{j}\|_{H^{1}}\|Y_{i}\|_{L^{2}}\nonumber \\
	&\lesssim \|Y_{i}\|_{H^{1}}^{2}\|Y_{j}\|_{L^{2}}^{2}+\|Y_{i}\|_{H^{1}}^{2}\|Y_{j}\|_{L^{2}}^{2} \nonumber.
	\end{align}
	Summing both sides over $i,j$ and using symmetry with respect to the roles of $i$ and $j$, we obtain \eqref{se15}.  The remainder of the proof is devoted to \eqref{se19}.

	To shorten the expressions that follow, we introduce the quantities $F \eqdef \sum_{i=1}^N\|\nabla Y_i\|_{L^2}^2+\frac{1}{N}\|\sum_{i=1}^N Y_i^2\|_{L^2}^2$ and $U \eqdef \sum_{i=1}^N\|Y_i\|_{L^2}^2$.  Note that  $F$ and $U$ are stationary, so we will freely omit the time argument below.  Our starting point is the key inequality \eqref{bs16}, which may be recast in terms of $U$ and $F$ as
	\begin{equation}
	\frac{\dif}{\dif t}U+F+mU \leq CR_{N}^{0}+C\big (D_{N}+D_{N}^{1} \big )U \nonumber.
	\end{equation}
	Muliplying the above by $U^{q-1}$ we find that for $q\geq1$ it holds
	\begin{equation}
	\frac{1}{q}\frac{\dif}{\dif t}U^{q}+U^{q-1}F+mU^{q} \leq CR_{N}^{0}U^{q-1}+C\big (D_{N}+D_{N}^{1} \big )U^{q} \nonumber.
	\end{equation}
	As in the proof of Theorem \ref{th:m1}, we now define $A \eqdef \E(D_{N} )$ and $A_{1}\eqdef\E\|\Wick{Z_1Z_2}\|_{\bC^{-s}}^2$.  Subtract the mean from $D_{N}+D_{N}^{1}$ and take expectation on both sides to find
	\begin{align}
	&\E \big [U^{q-1}F \big ]+(m-A-A_{1})\E \big [U^{q} \big ] \nonumber \\
	&\leq C\E \big [R_{N}^{0}U^{q-1} \big ]+C \E \big [\big (D_{N}+D_{N}^{1}-A-A_{1} \big )U^{q} \big ].  \nonumber \\
	& \leq C\|R_{N}^{0}\|_{L^{q}(\Omega)} \big ( \E U^{q} \big )^{\frac{q-1}{q} } +C\| D_{N}-A+D_{N}^{1}-A_{1}\|_{L^{q+1}(\Omega)} \big ( \E U^{q+1} \big )^{\frac{q}{q+1} } \nonumber \\
	& \leq C\big ( \E U^{q} \big )^{\frac{q-1}{q} }+CN^{-\frac{1}{2}}\big ( \E U^{q+1} \big )^{\frac{q}{q+1} } \nonumber,
	\end{align}
	where we used $\E U^q(t)=\E U^q(0)$ in the first inequality and we used a Gaussian hypercontractivity upgrade of \eqref{se6} and \eqref{se7} in the last line.  Using Young's inequality, we may absorb the first term to the left and obtain
	\begin{equation}
	\E \big [U^{q-1}F \big ]+(m-A-A_{1}-1)\E \big [U^{q} \big ]  \leq C+CN^{-\frac{1}{2}}\big ( \E U^{q+1} \big )^{\frac{q}{q+1} }\label{se20}.
	\end{equation}
	The strategy now is to first use the dissipative quantity on the LHS of \eqref{se20} to obtain $\E(U^{q}) \leq CN^{\frac{q-1}{2}}$, and then use the massive term on the LHS of \eqref{se20} to iteratively decrease the power of $N$ and eventually arrive at $\E(U^{q}) \leq C$.  Once this is established, plugging the bound back into \eqref{se20} completes the proof.\\
	
	Indeed, first observe that $F \geq N^{-1}U^{2}$ so that $\E(U^{q-1}F) \geq N^{-1}\E(U^{q+1})$.  Hence, Young's inequality with exponents $(q+1,\frac{q+1}{q})$ leads to  $\E(U^{q}) \leq CN^{\frac{q-1}{2}}$.  Defining $A_q \eqdef\E  U^q$ and discarding the dissipative term, \eqref{se20} implies
	\begin{align}\label{szz4}
	A_q\lesssim A_{q+1}^{\frac{q}{q+1}}N^{-1/2}+1.
	\end{align}
	We have $A_q\lesssim N^{\frac{q-1}{2}}$, which gives
	\begin{align*}
	A_q\lesssim N^{\frac{q^2}{2(q+1)}-\frac{1}{2}}+1.
	\end{align*}
	%\scott{This remainder of the proof I only checked on an example so far, but it seems reasonable. Very nice.} \zhu{you could check this by iteration. i.e.
	%$$A_q\lesssim (N^{a_{m,q+1}})^{\frac{q}{q+1}}N^{-\frac12}+1\lesssim N^{a_{m+1,q}}+1$$. }
	Substituting into \eqref{szz4} and use induction  we have for $n\geq 1$
	\begin{align}\label{sz2}
	A_q\lesssim N^{a_{n,q}}+1,
	\end{align}
	with $a_{n,q}=\frac{q(q+n-1)}{2(q+n)}-\frac{q}{2}(\sum_{k=1}^{n-1}\frac{1}{q+k})-\frac12$. Here $\sum_{k=1}^0=0$. In fact, we could check \eqref{sz2}
	by $$A_q\lesssim (N^{a_{n,q+1}})^{\frac{q}{q+1}}N^{-\frac12}+1\lesssim N^{a_{n+1,q}}+1.$$
	For fixed $q\geq 1$ we could always find $n$ large enough such that $a_{n,q}<0$, which implies that
	$A_q\lesssim1$ and the result follows.
\end{proof}

\begin{theorem}\label{th:o1}
	Let $\Phi=(\Phi_i)_{1\leq i\leq N}\thicksim \nu^N$ and $m$ be given as in Lemma \ref{th:m2}, then the laws of $\frac{1}{\sqrt{N}} \sum_{i=1}^{N}\Wick{\Phi_{i}^{2}}$ are tight on $B^{-2\kappa}_{2,2}$, and the laws of $\frac{1}{N} \Wick{\Big( \sum_{i=1}^N \Phi_i^2 \Big)^2}$
	are tight on $B^{-3\kappa}_{1,1}$.
	%	
	%	
	%	the following hold:
	%	\begin{enumerate}
	%		\item The laws of $N^{-\frac{1}{2}} \sum_{i=1}^{N}\Wick{\Phi_{i}^{2}}$ are tight on $B^{-2\kappa}_{2,2}$ and the subseqential limits can be identified with those of
	%		\begin{equation}
	%		N^{-\frac{1}{2}}\sum_{i=1}^{N} \big ( \Wick{Z_{i}^{2}}+2Y_{i}Z_{i} \big ). \label{se35}
	%		\end{equation}
	%		
	%		\item The laws of $\frac{1}{N} \Wick{\Big( \sum_{i=1}^N \Phi_i^2 \Big)^2}$
	%		%$N^{-1} \sum_{i,j=1}^{N}\Wick{\Phi_{i}^{2}\Phi_{j}^{2} }$
	%		are tight on $B^{-3\kappa}_{1,1}$ and the subseqential limits can be identified with those of
	%		\begin{equation}
	%		N^{-1}\sum_{i,j=1}^{N} \big ( \Wick{Z_{i}^{2}Z_{j}^{2} }+Y_{i}\Wick{Z_{i}Z_{j}^{2}}+Y_{i}Y_{j}\Wick{Z_{i}Z_{j}} \big ). \label{se36}
	%		\end{equation}
	%		
	%	\end{enumerate}	
\end{theorem}

\begin{proof}
	%	Recall the definition:
	%	\begin{align}\frac{1}{\sqrt{N}}\sum_{i=1}^N:\Phi_i^2:
	%	=\frac{1}{\sqrt{N}}\sum_{i=1}^NY_i^2+\frac{2}{\sqrt{N}}\sum_{i=1}^NY_iZ_i+\frac{1}{\sqrt{N}}\sum_{i=1}^N \Wick{Z_i^2}.\label{se18}
	%	\end{align}
	Note that the first term $\frac{1}{\sqrt{N}}\sum_{i=1}^N Y_i^2$ on RHS of  \eqref{ob1} converges to zero in $L^{2}(\Omega;L^{2})$ as an immediate consequence of \eqref{se16};
	so we can actually prove a  stronger result than stated, namely the subseqential limits can be identified with those of the last two terms in \eqref{ob1}.
	We will now show that
	the other two quantities induce tight laws on  $B^{-3\kappa}_{2,2}$, which implies the first part of the theorem. The second sum in \eqref{ob1} can be estimated using Lemma \ref{lem:multi} and Lemma \ref{lem:emb} to find for $s \in (\kappa,2\kappa)$
	\begin{align*}
	&\bigg \|\frac{1}{\sqrt{N}}\sum_{i=1}^NY_iZ_i \bigg \|_{B^{-s}_{2,2}}\lesssim \frac{1}{\sqrt{N}}\sum_{i=1}^N\|Y_i\|_{B^{s}_{2,2}}\|Z_i\|_{\bC^{-\kappa}}\\
	\lesssim& \sum_{i=1}^N\|Y_i\|_{B^{s}_{2,2}}^2+ \frac{1}{N}\sum_{i=1}^N\|Z_i\|_{\bC^{-\kappa}}^2\lesssim \sum_{i=1}^N\|Y_i\|_{H^1}^{2s}\|Y_i\|_{L^2}^{2(1-s)}+ \frac{1}{N}\sum_{i=1}^N\|Z_i\|_{\bC^{-\kappa}}^2\\
	\lesssim&\sum_{i=1}^N\|Y_i\|_{H^1}^{2}+\sum_{i=1}^N\|Y_i\|_{L^2}^{2}+ \frac{1}{N}\sum_{i=1}^N\|Z_i\|_{\bC^{-\kappa}}^2,
	\end{align*}
	which is bounded in expectation by a constant using Lemma \ref{th:m2} for $q=1$ and Lemma \ref{le:ex1}.
	%	\begin{align*}
	%	\E   \bigg \|\frac{1}{\sqrt{N}}\sum_{i=1}^NY_iZ_i  \bigg \|_{B^{-\kappa}_{2,2}} \lesssim1.
	%	\end{align*}
	For the third sum in \eqref{ob1} we use independence to find for $s\in (\kappa,2\kappa)$
	\begin{align*}
	\E\   \bigg \|\frac{1}{\sqrt{N}}\sum_{i=1}^N \Wick{Z_i^2}  \bigg \|_{B^{-s}_{2,2}}^2
	=\E \frac{1}{N} \bigg \langle \Lambda^{-s}\sum_{i=1}^N \Wick{Z_i^2},\Lambda^{-s}\sum_{i=1}^N \Wick{Z_i^2} \bigg \rangle
	=\E \| \Wick{Z_i^2}\|_{B^{-s}_{2,2}}^2\lesssim1.
	\end{align*}
	By the triangle inequality and the embedding $L^{2}\hookrightarrow B^{-\kappa}_{2,2}$
	we find that $\frac{1}{\sqrt N} \sum_{i=1}^N \Wick{\Phi_i^2}$ is bounded in $L^{1}(\Omega; B^{-s}_{2,2})$.  In light of the
	compact embedding $B^{-s}_{2,2}\subset B^{-2\kappa}_{2,2}$, the tightness claim follows.
	%	Furthermore, since the first contribution to \eqref{se18} converges strongly to zero, only contributions from \eqref{se35} can contribute to the limiting law.
	
	For the second observable, we will also show a stronger result: the subseqential limits can be identified with those of the last three terms in \eqref{ob2}.
	%	\begin{align}
	%	\frac{1}{N} \sum_{i,j=1}^{N} \Wick{\Phi_{i}^{2}\Phi_{j}^{2} }&=\frac{1}{N}\sum_{i,j=1}^{N} \big ( Y_{i}^{2}Y_{j}^{2}+4Y_i^2Y_jZ_j+2Y_i^2  \Wick{Z_j^2} \big )\label{se31}\\
	%	&+\frac{1}{N}\sum_{i,j=1}^{N} \big ( \Wick{Z_{i}^{2}Z_{j}^{2} }+4Y_{i}\Wick{Z_{i}Z_{j}^{2}}+4Y_{i}Y_{j}\Wick{Z_{i}Z_{j}} \big ) \label{se30}.
	%	\end{align}
	We start with the first 3 terms in \eqref{*}, which will be shown to converge to zero in $L^{1}(\Omega; B^{-2s}_{1,1})$ for $s>\kappa$.   For the first term of \eqref{*} we use Lemma \ref{th:m2} to obtain
	\begin{align*}
	\E \bigg \|\frac{1}{N}\sum_{i,j=1}^NY_i^2Y_j^2 \bigg \|_{L^1} =\frac{1}{N}\E \bigg \| \sum_{i=1}^NY_i^2 \bigg \|_{L^{2}}^{2} \lesssim \frac{1}{N},
	\end{align*}
	so this term converges to zero in $L^{1}(\Omega;L^{1})$.  For the second term of \eqref{*}, using \eqref{ssz1} of Lemma \ref{Y:L2} with $\varphi Z_{j}$ in place of $Z_{j}$ we obtain
	\begin{align*}
	&\sup_{\|\varphi\|_{\bC^{2s}} \leq 1} \bigg |\frac{1}{N} \sum_{i,j=1}^N \langle Y_i^2Y_jZ_j,\varphi \rangle \bigg |\ \\
	&\lesssim \frac{1}{N}\bigg \| \sum_{i=1}^{N}Y_{i}^{2} \bigg \|_{L^{2}} \bigg ( \sum_{j=1}^{N}\|Y_{j}\|^{2 }_{L^{2}}   \bigg )^{\frac{1-s}{2} }\bigg ( \sum_{j=1}^{N} \|\nabla Y_{j}\|_{L^{2}}^{2} \bigg )^{\frac{s}{2}}   \bigg (  \sum_{j=1}^{N}\|Z_{j}\|_{\bC^{-s}}^{2}   \bigg )^{1/2}\nonumber \\
	&+ \frac{1}{N}\bigg \| \sum_{i=1}^{N}Y_{i}^{2} \bigg \|_{L^{2}}\bigg ( \sum_{j=1}^{N}\|Y_{j}\|^{2 }_{L^{2}}   \bigg )^{\frac{1}{2} } \bigg (  \sum_{j=1}^{N}\|Z_{j}\|_{\bC^{-s}}^{2}   \bigg )^{1/2} \nonumber.
	\end{align*}
	Hence, using Young's inequality we find for $\delta>0$ small enough
	\begin{align}
	\bigg \|\frac{1}{N} \sum_{i,j=1}^N  Y_i^2Y_jZ_j  \bigg \|_{B^{-2s}_{1,1}} &\leq N^{-\delta}\bigg \| \sum_{i=1}^{N}Y_{i}^{2} \bigg \|_{L^{2}}^{2}+N^{-\delta}\bigg (\sum_{j=1}^{N}\|\nabla Y_{j}\|_{L^{2}}^{2} \bigg )\nonumber  \\
	&+N^{-\delta}\bigg (\sum_{j=1}^{N}\|Y_{j}\|_{L^{2}}^{2}\bigg ) \bigg ( \frac{1}{N}\sum_{j=1}^{N}\|Z_{j}\|_{\bC^{-s}}^{2}  +1 \bigg )^{\frac{1}{1-s}}  \nonumber.
	\end{align}
	Both terms above converge to zero in $L^{1}(\Omega)$ as a consequence of \eqref{se19}, \eqref{se16} and Lemma \ref{le:ex1}.  For the third term in \eqref{*}, a calculation similar to \eqref{s9} using     		Lemma  \ref{lem:multi} with $Y_{i}$ replaced by $\varphi Y_{i}$ with $\varphi \in \bC^{2s}$  yields
	\begin{align}
	\sup_{\|\varphi\|_{\bC^{2s}} \leq 1} & \bigg |\frac{1}{N}\sum_{i,j=1}^{N} \langle Y_{i}^{2}\Wick{Z_{j}^{2}},\varphi \rangle \bigg |
	\lesssim \frac{1}{N}\sum_{i=1}^{N} \|\Lambda^s(Y_{i}^2)\|_{L^{2}} \bigg\|\sum_{j=1}^N\Lambda^{-s}(\Wick{Z_{j}^{2}}) \bigg \|_{L^2}\nonumber \\
	&\lesssim \bigg (\sum_{i=1}^{N}\|\nabla Y_{i}\|^{1+s}_{L^{2}} \|Y_{i}\|^{1-s}_{L^{2}}+\|Y_{i}^{2}\|_{L^{2}}  \bigg )\bigg \|\frac{1}{N}\sum_{j=1}^N\Lambda^{-s}(\Wick{Z_{j}^{2}}) \bigg \|_{L^2}\nonumber \\
	&\lesssim \bigg (\sum_{i=1}^{N}\| Y_{i}\|^{2}_{H^1}\bigg)^{\frac{1+s}{2}}\Big(\sum_{i=1}^N\|Y_{i}\|_{L^{2}}^{2}  \Big )^{\frac{1-s}{2}}\bigg \|\frac{1}{N}\sum_{j=1}^N\Lambda^{-s}(\Wick{Z_{j}^{2}}) \bigg \|_{L^2},  \label{zz4}
	\end{align}
	where we used \eqref{e:Lambda-prod} and Lemma \ref{lem:interpolation} to have
	$$\|\Lambda^s(Y_{i}^2)\|_{L^{2}}\lesssim \|\Lambda^s Y_i\|_{L^4}\|Y_i\|_{L^4}\lesssim\|\nabla Y_{i}\|^{1+s}_{L^{2}} \|Y_{i}\|^{1-s}_{L^{2}}+\|Y_{i}\|_{L^{4}}^2 .$$
	The first part of the product in \eqref{zz4} is bounded in $L^{1}(\Omega)$ by \eqref{se19}.  For the second part of the product, we use independence to obtain
	\begin{align*}\E \bigg\|\frac{1}{N} \sum_{j=1}^N\Lambda^{-s}(\Wick{Z_{j}^{2}}) \bigg\|_{L^2}^2
	\lesssim \frac{1}{N^2}\sum_{j=1}^N\E \|\Lambda^{-s}(\Wick{Z_{j}^{2}})\|_{L^2}^2\lesssim \frac{1}{N},\end{align*}
	so together  we find
	$\E\| \frac{1}{N}\sum_{i,j}  Y_{i}^{2} \Wick{Z_{j}^{2}}  \|_{B^{-2s}_{1,1}} $ converges to $0$.

	We now turn to terms in  \eqref{ob2} and derive suitable moment bounds.  For the first of these terms we have
	\begin{align*}
	\E \bigg\|\frac{1}{N}\sum_{i,j=1}^N \Wick{Z_i^2Z_j^2} \bigg \|_{B^{-\kappa}_{2,2}}^2
	& =\E \frac{1}{N^2} \bigg\langle \Lambda^{-\kappa}\sum_{i,j=1}^N \Wick{Z_i^2Z_j^2},\Lambda^{-\kappa}\sum_{i,j=1}^N \Wick{Z_i^2Z_j^2} \bigg \rangle
	\\ & \lesssim\E \| \Wick{ Z_1^2Z_2^2}\|_{B^{-\kappa}_{2,2}}^2+\frac{1}{N}\E \|\Wick{Z_1^4}\|_{B^{-\kappa}_{2,2}}^2\lesssim1.
	\end{align*}
	For the next term, using \eqref{ns7} with $\varphi Y_{i}$ in place of $Y_{i}$ we find
	\begin{align*}
	&\sup_{\|\varphi\|_{\bC^{2s}} \leq 1}\bigg | \frac{1}{N}\sum_{i,j=1}^N \langle Y_i \Wick{Z_iZ_j^2},\varphi \rangle \bigg | \nonumber \\
	&\lesssim \bigg(\sum_{i=1}^{N}\|\Lambda^sY_{i}\|^2_{L^{2}}\bigg)^{1/2}\bigg(\frac{1}{N^2} \sum_{i=1}^{N} \bigg\| \sum_{j=1}^{N}\Lambda^{-s}(\Wick{Z_{j}^{2}Z_{i}}) \bigg \|^2_{L^2}\bigg)^{1/2} \nonumber\\
	&\lesssim  \bigg(\sum_{i=1}^{N}\|Y_{i}\|^2_{H^{1}}\bigg)^{s/2}
	\bigg(\sum_{i=1}^{N}\|Y_{i}\|^2_{L^{2}}\bigg)^{(1-s)/2}\bigg(\frac{1}{N^2} \sum_{i=1}^{N}\bigg \|\sum_{j=1}^{N}\Lambda^{-s}(\Wick{Z_{j}^{2}Z_{i}})\bigg \|^2_{L^2}\bigg)^{1/2}.
	\end{align*}
	Using \eqref{mom} and Lemma \ref{th:m2} we deduce for some $p$ satisfying $sp<2$
	\begin{align*}
	\E \bigg\|\frac{1}{N}\sum_{i,j=1}^NY_i,\Wick{Z_iZ_j^2}\bigg\|_{B^{-2s}_{1,1}}^p\lesssim1.
	\end{align*}
	For the last term, we argue similarly to \eqref{s6} but $\Wick{Z_iZ_j}$ replace by $\varphi \Wick{Z_iZ_j}$ for $\varphi \in \bC^{s}$ to deduce boundedness in $L^{1}(\Omega;B^{-2s}_{1,1})$.
	
	Combining the above observations with the triangle inequality, we find that the second observable %$\frac{1}{N} \sum_{i,j}\Wick{\Phi_{i}^{2}\Phi_{j}^{2} }$
	is uniformly bounded in probability as a $B^{-2s}_{1,1}$ valued random variable.  By compactness   	of the embedding of $B^{-2s}_{1,1}$ into $B^{-{2s}-\delta}_{1,1}$ for $\delta>0$, we obtain the result.
	%Since the contributions from \eqref{se31} converge strongly to zero, only the contributions of \eqref{se36} contribute in the limit, completing the proof.
\end{proof}

\subsection{$L^p$-estimate}\label{se:lp}
In light of Lemma \ref{th:m2} and the Sobolev embedding theorem, it follows that for each component $i$, $\E \|Y_{i}\|_{L^{p}}^{2} \lesssim 1$, $p>1$.  Our goal now in this subsection is to upgrade from the second moment to higher moments of the $L^{p}$ norm.  We do so by revisiting the energy estimates for the PDE \eqref{eq:22}.  Since we work with a fixed component rather than an aggregate quantity, these bounds come with a price: the estimate is no longer uniform in $N$.  Nonetheless, the power of $N$ that appears is ultimately small enough for a successful application of the estimate in Lemma \ref{lem:zmm}, en route to Theorem \ref{theo:nontrivial2}.  

\bl\label{th:lp} Let $m$ as in Lemma \ref{th:m2} and $p>2$.  For each component $i$ it holds 
\begin{align*}
\E \|Y_i\|_{L^p}^p+\E\|Y_i^{p-2}|\nabla Y_i|^2\|_{L^1}+\frac1N\sum_{j=1}^N\E\|Y_i^pY_j^2\|_{L^1}\lesssim N^{p/2},
\end{align*}
where the implicit constant is independent of $N$ and $i$.   
\el
\begin{proof}
	\newcounter{GlobLp} % proofstep = 0
	\refstepcounter{GlobLp} % increases value by 1
	Fix a component $i$.  Given $p> 2$, let $s>0$ be a small number to be selected (depending on $p$ ) in the final step of the proof. % such that
	%$sp<\frac{1}{2}$ and $\frac{2}{p}+s<1$.
	%We omit the subscripts $i$ for simplicity of notation,
	%namely we write $X$ for $X_i$, $Z$ for $Z_i$, and $\bX$, $\bZ$ for their independent copies.
	We will perform an $L^p$ estimate: integrating \eqref{eq:22} against $|Y_i|^{p-2}Y_i$ we obtain
	\begin{align}\label{eq:LpPhi}
	&\frac{1}{p}\frac{\dif}{\dif t} \|Y_i\|_{L^{p}}^{p}
	+(p-1) \||Y_i|^{p-2}|\nabla Y_i|^2\|_{L^1}
	+\frac1N\sum_{j=1}^N\||Y_i|^p Y_j^2\|_{L^1}+m\|Y_i\|_{L^p}^p\nonumber
	\\=&-2 \Big\< \frac1N\sum_{j=1}^N Y_jZ_j,|Y_i|^p\Big\>
	-\Big\<\frac1N\sum_{j=1}^NY_j^2|Y_i|^{p-2}Y_i,Z_i\Big\>
	-2 \Big\< \frac1N\sum_{j=1}^NY_j\Wick{Z_jZ_i},|Y_i|^{p-2}Y_i \Big\> \nonumber
	\\
	&\qquad +\Big\<\frac1N\sum_{j=1}^N\Wick{Z_j^2},|Y_i|^p \Big\>+\Big\< \frac1N\sum_{j=1}^N\Wick{Z_iZ_j^2},Y_i|Y_i|^{p-2}\Big\>
	=: \sum_{k=1}^5I_k.
	\end{align}
	Define $D \eqdef \|Y_i^{p-2}|\nabla Y_i|^2\|_{L^1}$ and for each $j$, let $A_j \eqdef \|Y_i^pY_j^2\|_{L^1}$.

	{\sc Step} \arabic{GlobLp} \label{GlobLp1} \refstepcounter{GlobLp} (Estimate of $I_1$)
	In this step we show that
	\begin{align}\label{est:I1}
	I_1
	&	\leq
	\frac{1}{10}\frac1N\sum_{j=1}^NA_j+ \frac{1}{10} D+C\|Y_i\|_{L^p}^{p} F,
	\end{align}
	where
	$$F\eqdef
	\frac1N\sum_{j=1}^N\|Y_j\|_{H^1}^{2s} \|Z_j\|_{\bC^{-s}}^{2}
	%+\frac1N\sum_{j=1}^N\|Y_j\|_{H^1}^{\frac{2s}{1+s}}\|Z_j\|_{\bC^{-s}}^{\frac{2}{1+s}}
	+\frac1N\sum_{j=1}^N\|Z_j\|_{\bC^{-s}}^{2/(1-s)}+1.
	$$

	In the following we prove \eqref{est:I1}. By Lemma~\ref{lem:dual+MW} we have
	\begin{align*}
	I_1
	\lesssim 
	\frac1N\sum_{j=1}^N\Big[\|Y_j |Y_i|^p\|_{L^1}\|Z_j\|_{\bC^{-s}}\Big]+\frac1N\sum_{j=1}^N\Big[ \|Y_j |Y_i|^p\|_{L^1}^{1-s}\|\nabla (Y_j |Y_i|^p)\|_{L^1}^{s}\|Z_j\|_{\bC^{-s}}\Big]
	=: I_1^{(1)}+I_1^{(2)}.% \label{e:5.4I1}
	\end{align*}
	To estimate $I_1^{(1)}$, notice that the Cauchy-Schwarz inequality yields
	\begin{equation}\label{e:Y_jY_i^p}
	\|Y_j |Y_i|^p\|_{L^1} \leq \|Y_j |Y_i|^\frac{p}{2}\|_{L^2} \||Y_{i}|^{\frac p2}\|_{L^2}=A_j^{1/2}\|Y_i\|_{L^p}^{\frac p 2 }.
	\end{equation}
	Combining this with Young's inequality, we obtain
	\begin{align*}
	I_1^{(1)} %\le \E \Big[\|\bX X^p\|_{L^1} & \|\bZ\|_{\bC^{-s}}\Big]
	\leq\frac1N\sum_{j=1}^N\left[ A_j^{1/2} \|Y_i\|_{L^p}^{\frac p 2 }\|Z_j\|_{\bC^{-s}}\right]
	\le \frac{1}{10}\frac1N\sum_{j=1}^N A_j+ C\|Y_i\|_{L^p}^{p} \frac1N\sum_{j=1}^N\|Z_j\|_{\bC^{-s}}^2.
	\end{align*}

	To estimate $I_1^{(2)}$, note first that $\nabla(Y_{j}|Y_{i}|^{p})=\nabla Y_{j} \big (|Y_{i}|^{\frac p 2} \big )^{2}+2Y_{j}|Y_{i}|^{\frac{p}{2}}\nabla |Y_{i}|^{\frac{p}{2}}$.  Hence, using H\"older's inequality
	followed by Gagliardo-Nirenberg \eqref{e:Gagliardo} with $(s,q,r,\alpha)=(0,4,2,\frac12)$,
	\begin{align*}
	\|\nabla (Y_j |Y_i|^p)\|_{L^1}
	& \le
	\|\nabla  Y_j\|_{L^{2}}\||Y_i|^{\frac{p}{2}}\|_{L^4}^2
	+2\|Y_{j}|Y_i|^{\frac{p}{2}} \|_{L^{2}}\|\nabla |Y_i|^{\frac{p}{2}}\|_{L^{2}} \\
	& \lesssim
	\||Y_i|^{\frac{p}{2}}\|_{H^1} \||Y_i|^{\frac{p}{2}}\|_{L^2} \|Y_j\|_{H^1}
	+ \sqrt{A_jD}.
	\end{align*}
	Since $\||Y_i|^{\frac{p}{2}}\|_{H^1} \lesssim D^{\frac12}+ \||Y_i|^{\frac{p}{2}}\|_{L^2}$,	using \eqref{e:Y_jY_i^p} again and $\||Y_i|^{\frac{p}{2}}\|_{L^2}=\||Y_{i}|^{p}\|_{L^1}^{1/2}$ 
	yields
	\begin{align*}
	I_1^{(2)}
	&	\lesssim
	\frac1N\sum_{j=1}^N\Big[A_j^{\frac{1-s}{2}}\||Y_i|^p\|_{L^1}^{\frac{1-s}{2}}
	\Big(
	D^{\frac{s}{2}} \||Y_i|^p\|_{L^1}^{\frac{s}{2}} \|Y_j\|_{H^1}^s
	+ \||Y_i|^p\|_{L^1}^{s}  \|Y_j\|_{H^1}^s
	+A_j^{\frac{s}{2}}D^{\frac{s}{2}}\Big)
	\|Z_j\|_{\bC^{-s}}\Big]
	\\
	&\le
	\frac{1}{10}\frac1N\sum_{j=1}^NA_j+ \frac{1}{10} D
	\\&+C\|Y_i\|_{L^p}^{p}
	\Big(
	\frac1N\sum_{j=1}^N\|Y_j\|_{H^1}^{2s} \|Z_j\|_{\bC^{-s}}^{2}
	+\frac1N\sum_{j=1}^N\|Y_j\|_{H^1}^{\frac{2s}{1+s}}\|Z_j\|_{\bC^{-s}}^{\frac{2}{1+s}}
	+\frac1N\sum_{j=1}^N\|Z_j\|_{\bC^{-s}}^{2/(1-s)}\Big)
	\\&\leq\frac{1}{10}\frac1N\sum_{j=1}^NA_j+ \frac{1}{10} D+C\|Y_i\|_{L^p}^{p}F,
	\end{align*}
	where the second inequality follows from three applications of Young's inequality.  We view the summand in first term as $A_{j}^{\frac{1-s}{2}}D^{\frac s 2}(\|Y_{i} \|_{L^{p}}^{\frac p 2} \|Y_{j}\|_{H^{1}}^{s}\|Z_{j}\|_{\bC^{-s}} )$ and use exponents $(\frac{2}{1-s},\frac{2}{s},2)$.  We view the second summand as $A_{j}^{\frac{1-s}{2}}(\|Y_{i} \|_{L^{p}}^{\frac{p(1+s)}{2}}\|Y_{j}\|_{H^{1}}^{s}\|Z_{j}\|_{\bC^{-s}})$ and we use exponents $(\frac{2}{1-s},\frac{2}{1+s})$.  Finally, the third summand is viewed as $A_{j}^{\frac{1}{2}}D^{\frac s 2}(\|Y_{i} \|_{L^{p}}^{\frac{p(1-s) }{2}}\|Z_{j}\|_{\bC^{-s}})$ and we use exponents $(2,\frac{2}{s},\frac{2}{1-s} )$.
	
	%\begin{align*}
	%F_1\lesssim& \frac1N\Big(\sum_{j=1}^N\|Y_j\|_{H^1}^{2}\Big)^s\Big(\sum_{j=1}^N \|Z_j\|_{\bC^{-s}}^{\frac2{1-s}}\Big)^{1-s}.
	%\\\lesssim&\frac1N\Big(\sum_{j=1}^N\|Y_j\|_{H^1}^{2}\Big)+\frac1N\Big(\sum_{j=1}^N \|Z_j\|_{\bC^{-s}}^{\frac2{1-s}}\Big)^{1-s}
	%\end{align*}

	\medskip
	
	{\sc Step} \arabic{GlobLp} \label{GlobLp} \refstepcounter{GlobLp} (Estimates for $I_{2}$)
	In this step, we show that
	\begin{align}\label{est:I2}
	I_2
	\le \frac1{10} \bigg\|\frac1N\sum_{j=1}^NY_j^2Y_i^{p} \bigg\|_{L^1}+\frac1{10} D+C\|Y_{i}\|_{L^{p}}^{p}+C F_1,
	\end{align}
	where
	$$F_1\eqdef\Big\|\frac1N\sum_{j=1}^NY_j^2\Big\|_{L^1}\|Z_i\|_{\bC^{-s}}^p+\Big\|\frac1N\sum_{j=1}^NY_j^2\Big\|_{L^1}\Big(\frac1N\sum_{j=1}^N\|Y_j\|_{H^1}^2\Big)^{\frac{sp}{1-s}}
	\|Z_i\|_{\bC^{-s}}^{\frac{p}{1-s}}+1.$$
	%	For the second term on the right hand side of \eqref{eq:LpPhi} we  have
	To prove \eqref{est:I2}, by Lemma~\ref{lem:dual+MW}  one has
	\begin{align*}
	I_2
	&=\frac1N\sum_{j=1}^N\< (Y_j^2Y_i|Y_i|^{p-2} ),Z_i \>
	\\
	&\lesssim
	\Big(\Big\|\frac1N\sum_{j=1}^N Y_j^2Y_i|Y_i|^{p-2}\Big\|_{L^1}
	+\Big\|\frac1N\sum_{j=1}^N\nabla(Y_j^2Y_i|Y_i|^{p-2} ) \Big\|_{L^1}^s
	\Big\|\frac1N\sum_{j=1}^N Y_j^2Y_i|Y_i|^{p-2} \Big\|_{L^1}^{1-s}\Big)\|Z_i\|_{\bC^{-s}}.
	\end{align*}
	%\hao{Writing $Y_j^2Y_i^{p-1}=Y_j^{\frac{2}{p}} (Y_j^{\frac{2p-2}{p}} Y_i^{p-1})$ and by H\"older (for sum) and Jensen inequality we have} 
	%\zhu{We just use H\"older's inequality for intergral, right?}\hao{I still think you're using H\"older for sum (otherwise why sum is inside the $L^1$ norm on RHS?) In any case I agree with this bound; we could  just leave $Y_j^2Y_i^{p-1}=Y_j^{\frac{2}{p}} (Y_j^{\frac{2p-2}{p}} Y_i^{p-1})$} 
	Using the triangle inequality, writing $Y_j^2 |Y_i|^{p-1}=(Y_{j}^{2} |Y_{i}|^{p} )^{\frac{p-1}{p}}(Y_{j}^{2})^{\frac{1}{p}}$ and using Holder's inequality, 
	\begin{align*}
	\Big\|\frac1N\sum_{j=1}^N  Y_j^2Y_i|Y_i|^{p-2}\Big\|_{L^1}\lesssim\Big\|\frac1N\sum_{j=1}^NY_j^2 |Y_i|^{p}\Big\|_{L^1}^{\frac{p-1}p}
	\Big\|\frac1N\sum_{j=1}^NY_j^2\Big\|_{L^1}^{\frac{1}p}.
	\end{align*}
	Another application of Holder's inequality gives
	\begin{align}\label{eq:a}
	\|\nabla(Y_i|Y_i|^{p-2}) \|_{L^{\frac{p}{p-1}}} \lesssim \||\nabla Y_{i}| Y_i^{\frac {p-2} 2 } \|_{L^{2}} \||Y_{i} |^{\frac{p-2}{2} } \|_{L^{\frac{2p}{p-2} } }
	=D^{1/2}\||Y_i|^p\|_{L^{1}}^{\frac{p-2}{2p}}.
	\end{align}
	For $r>2$, using Holder's inequality and the Sobolev embedding $H^1\subset L^{\frac{2r}{r-2}}$   we have 
	%\hao{I don't quite understand how you got $\|Y_j\|_{H^1}^2\|Y_i^{p-1}\|_{L^r}$,  what is $r$, and then how did you bound $\|Y_i^{p-1}\|_{L^r}$ by the $\|Y_i^{p/2}\|_{H^1}^{2(1-\frac{1}{p})}$ in the 2nd line?} \zhu{Here we use \eqref{e:Lambda-prod} and for the second term we first have $\frac1N\sum_{j=1}^N \|\nabla Y_j\|_{L^{2}}\|Y_j\|_{L^{r'}}\|Y_i^{p-1}\|_{L^r}$ for $r>2$ and $\frac1{r'}+\frac1r=\frac12$. So we can choose $r$ for every $r>2$ and use Sobolev embedding $H^1\subset L^{r'}$ to have the following last term in the first line. We also use Sobolev embedding to get from the first line to the second line. 
	%}\hao{From 1st to 2nd line, Sobolev embedding only gives $\|Y_i^{p-1}\|_{H^1}$, and then what did you do?} \zhu{We write $\|Y_i^{p-1}\|_{L^r}$ as $\|Y_i^{p/2}\|_{L^q}^{2(1-\frac1p)}$, which can be bounded by $\|Y_i^{p/2}\|_{H^1}^{2(1-\frac1p)}$ for $q=\frac{2r(p-1)}{p}$.}
	\begin{align*}
	\Big\| \frac1N\sum_{j=1}^N   \nabla(Y_j^2Y_i|Y_i|^{p-2} ) \Big\|_{L^1}
	&\lesssim
	\frac1N\sum_{j=1}^N  \|\nabla(Y_i|Y_i|^{p-2}) \|_{L^{\frac{p}{p-1}}} \|Y_j^2\|_{L^p}
	+ \|\nabla Y_j\|_{L^2} \|Y_{j}\|_{L^{\frac{2r}{r-2}}} \|Y_{i}|Y_{i}|^{p-2} \|_{L^r}
	%\\	&\lesssim
	%\frac1N\sum_{j=1}^N  \|\nabla(Y_i|Y_i|^{p-2}) \|_{L^{\frac{p}{p-1}}} \|Y_j^2\|_{L^p}
	%+\frac1N\sum_{j=1}^N \|Y_j\|_{H^1}^2 \|Y_i^{p-1}\|_{L^r}
	\\
	&\lesssim
	\frac1N\sum_{j=1}^N 
	D^{1/2}\||Y_i|^p\|_{L^{1}}^{\frac{p-2}{2p}}\|Y_j^2\|_{L^p}
	+\||Y_i|^{\frac p 2}\|_{H^1}^{2(1-\frac{1}{p})}\frac1N\sum_{j=1}^N\|Y_j\|_{H^1}^2
	\\
	&\lesssim \Big(\frac1N\sum_{j=1}^N\|Y_j\|_{H^1}^2\Big)
	\Big(\||Y_i|^{p/2}\|_{H^1}^{2(1-\frac{1}{p})}+D^{1/2}\||Y_i|^p\|_{L^{1}}^{\frac{p-2}{2p}}\Big),
	\end{align*}
	%where in the first inequality we used \eqref{e:Lambda-prod}
	%which gives as the second term $ \|\nabla Y_j\|_{L^{2}}\|Y_j\|_{L^{r'}}\|Y_i^{p-1}\|_{L^r}$ with $\frac1{r'}+\frac1r=\frac12$, followed by Sobolev embedding $H^1\subset L^{r'}$ 
	where in the second inequality we used that $\|Y_{i}|Y_{i}|^{p-2} \|_{L^r}=\||Y_i|^{\frac p 2}\|_{L^q}^{2(1-\frac1p)}$ for $q=\frac{2r(p-1)}{p}$ and again the Sobolev embedding theorem.
	%	for $r>2$. For the first term we have
	%	\begin{align}\label{eq:a}
	%	\|\nabla(Y_i|Y_i|^{p-2})\|_{L^{\frac{p}{p-1}}}
	%	\lesssim D^{1/2}\|Y_i^p\|_{L^{1}}^{\frac{p-2}{2p}}.
	%	\end{align}
	%	For the second term we have
	%	\begin{align*}
	%	\frac1N\sum_{j=1}^N\|Y_j\|_{H^1}^2\|Y_i^{p-1}\|_{L^r}\lesssim \|Y_i^{p/2}\|_{H^1}^{2(1-\frac{1}{p})}\frac1N\sum_{j=1}^N\|Y_j\|_{H^1}^2.
	%	\end{align*}
	%	Then we deduce
	%	\begin{align*}
	%	\Big\|\frac1N\sum_{j=1}^N\nabla[(Y_j^2)Y_i^{p-1}]\Big\|_{L^1}\lesssim\frac1N\sum_{j=1}^N\|Y_j\|_{H^1}^2
	%	(\|Y_i^{p/2}\|_{H^1}^{2(1-\frac{1}{p})}+D^{1/2}\|Y_i^p\|_{L^{1}}^{\frac{p-2}{2p}}),
	%	\end{align*}
	Combining the above estimates we deduce
	\begin{align*}
	I_2
	&\lesssim 
	\Big\|\frac1N\sum_{j=1}^NY_j^2|Y_i|^{p} \Big\|_{L^1}^{\frac{p-1}p}
	\Big\|\frac1N\sum_{j=1}^NY_j^2 \Big\|_{L^1}^{\frac{1}p} 
	\|Z_i\|_{\bC^{-s}}
	\\
	& +\Big\|\frac1N\sum_{j=1}^NY_j^2|Y_i|^{p}\Big\|_{L^1}^{\frac{(p-1)(1-s)}p}
	\Big\|\frac1N\sum_{j=1}^NY_j^2\Big\|_{L^1}^{\frac{1-s}p}
	\Big(\frac1N\sum_{j=1}^N\|Y_j\|_{H^1}^2\Big)^s
	\Big(\||Y_i|^{ \frac p 2}\|_{H^1}^{2(1-\frac{1}{p})}+D^{1/2}\||Y_i|^p\|_{L^{1}}^{\frac{p-2}{2p}}\Big)^s   \|Z_i\|_{\bC^{-s}}
	\\
	&\lesssim
	\frac1{10} \Big\|\frac1N\sum_{j=1}^NY_j^2|Y_i|^{p} \Big\|_{L^1}
	+\frac1{10} D+\Big\|\frac1N\sum_{j=1}^NY_j^2\Big\|_{L^1}\|Z_i\|_{\bC^{-s}}^{ p }+\Big\|\frac1N\sum_{j=1}^NY_j^2\Big\|_{L^1}\Big(\frac1N\sum_{j=1}^N\|Y_j\|_{H^1}^2\Big)^{\frac{sp}{1-s}}
	\|Z_i\|_{\bC^{-s}}^{\frac{p}{1-s}}.\\
	&+C\|Y_{i}\|_{L^{p}}^{p}+1, 
	%\\:=&\frac1{10} \|\frac1N\sum_{j=1}^NY_j^2Y_i^{p}\|_{L^1}+\frac1{10} D+F_1.
	\end{align*}
	where we applied Cauchy's inequality with exponents 
	$(\frac{p}{p-1},p)$ for the first term and $(\frac{p}{(p-1)(1-s)},\frac{p}{1-s},\frac{1}{s})$ for the second term.  
	
	\medskip
	
	{\sc Step} \arabic{GlobLp} \label{GlobLp2} \refstepcounter{GlobLp} (Estimate of $I_3$-$I_5$)
	
	In this step, we show that
	\begin{align}\label{est:I3}
	\sum_{k=3}^5I_k\leq\frac1{10} D%+\frac1{10}\||Y_i|^{\frac p2}\|_{H^1}^2
	+\frac1{10}\frac1N\sum_{j=1}^N\|Y_j^2|Y_i|^p\|_{L^1}+C\|Y_i\|_{L^p}^pF_2+CF_3,
	\end{align}
	where
	\begin{align*}
	F_2 & \eqdef
	\Big\|\frac1N\sum_{j=1}^N\Wick{Z_j^2}\Big\|_{\bC^{-s}}^{\frac1{1-s/2}}+1, \\
	F_3 & \eqdef 
	\Big(\frac1N\sum_{j=1}^N\|\Wick{Z_i Z_j}\|_{\bC^{-s}}^2\Big)^{\frac{p}2}+\Big(\frac1N\sum_{j=1}^N\|\Wick{Z_i Z_j}\|_{\bC^{-s}}^2\Big)^{\frac{p}{2(1-s)}}\Big(\frac1N\sum_{j=1}^N\| Y_j\|_{H^1}^2\Big)^{\frac{ps}{2(1-s)}}
	\\&\qquad+\Big\|\frac1N\sum_{j=1}^N\Wick{Z_iZ_j^2}\Big\|_{\bC^{-s}}^p+1.
	%\\&+\Big(\Big\|\frac1N\sum_{j=1}^N\Wick{Z_j^2}\Big\|_{\bC^{-s}}^{\frac1{1-s/2}}+1\Big),
	\end{align*}
	
	To prove \eqref{est:I3}, we apply Lemma \ref{lem:dual+MW} to find
	\begin{align*}
	I_3&= \frac1N\sum_{j=1}^N\< Y_j Y_i|Y_i|^{p-2},\Wick{Z_i Z_j}\>
	\\
	& \lesssim
	\frac1N\sum_{j=1}^N \Big(\|Y_j Y_i|Y_i|^{p-2}\|_{L^1}+\|Y_j Y_i|Y_i|^{p-2}\|_{L^1}^{1-s}\|\nabla(Y_j Y_i|Y_i |^{p-2})\|_{L^1}^s\Big)\|\Wick{Z_i Z_j}\|_{\bC^{-s}}
	.\end{align*}
	By the Cauchy-Schwartz inequality, 
	\begin{align*}
	\|Y_j Y_i|Y_i|^{p-2}\|_{L^1}=\|(Y_{j}|Y_{i}|^{\frac p 2}) |Y_{i}|^{\frac p2-1} \|_{L^{1}}\lesssim \|Y_{j}^{2}|Y_{i}|^{p}\|_{L^{1}}^{\frac 12 } \||Y_{i}|^{p-2} \|_{L^{1}}^{\frac12}  ,
	\end{align*}
	and by Holder's inequality and the Sobolev embedding theorem
	\begin{align*}
	\|\nabla(Y_j Y_i|Y_i |^{p-2})\|_{L^1}
	&\lesssim 
	\|\nabla Y_j\|_{L^2}\||Y_i|^{p-1}\|_{L^2}
	+\|\nabla Y_i|Y_i|^{p-2}\|_{L^{\frac{p}{p-1}}}\|Y_j\|_{L^p}
	\\
	&\lesssim
	\|\nabla Y_j\|_{L^2}\|| Y_i |^{\frac p2}\|_{L^{\frac{4(p-1)}{p } }}^{2(1-\frac{1}{p})}
	+D^{1/2}\||Y_i|^p\|_{L^{1}}^{\frac{p-2}{2p}}\|Y_j\|_{L^p},
	\\
	&\lesssim
	\|\nabla Y_j\|_{L^2}\||Y_i|^{\frac p2}\|_{H^1}^{2(1-\frac{1}{p})}
	+D^{1/2}\||Y_i|^p\|_{L^{1}}^{\frac{p-2}{2p}}\|Y_j\|_{L^p},
	\end{align*}
	where we used \eqref{eq:a}.
	Combining these observations and applying the Cauchy-Schwartz inequality for the summation we find
	\begin{align*}
	I_3
	&\lesssim
	\Big(\frac1N\sum_{j=1}^N\|Y_j^2|Y_i|^p\|_{L^1}\Big)^{1/2}\| |Y_{i}|^{p-2} \|_{L^{1}}^{\frac12}\Big(\frac1N\sum_{j=1}^N\|\Wick{Z_i Z_j}\|_{\bC^{-s}}^2\Big)^{1/2}  \\&
	+\frac1N\sum_{j=1}^N \|\nabla Y_j\|_{L^2}^s\||Y_i|^{\frac p 2}\|_{H^1}^{2s(1-\frac{1}{p})}\|Y_j^2|Y_i|^p\|_{L^1}^{\frac{1-s}2}\||Y_i|^{p-2}\|_{L^1}^\frac{1-s}{2}\|\Wick{Z_i Z_j}\|_{\bC^{-s}}
	\\&
	+\frac1N\sum_{j=1}^N D^{\frac{s}2}\||Y_i|^{p}\|_{L^1}^{\frac{s (p-2)}{2p}}\|Y_j\|_{L^p}^s\|Y_j^2|Y_i|^p\|_{L^1}^{\frac{1-s}2}\||Y_i|^{p-2}\|_{L^1}^\frac{1-s}{2}\|\Wick{Z_i Z_j}\|_{\bC^{-s}}
	:=\sum_{k=1}^3I_{3k}.
	\end{align*}
	To estimate $I_{31}$, we use Young's inequality with exponents $(p,2,\frac{2p}{p-2})$ together with the embedding $L^{p} \hookrightarrow L^{p-2}$, and this leads to the $1$ in $F_{2}$ and first term in $F_{3}$.  To estimate $I_{32}$,  we use H\"older's inequality for the summation with exponents $(2,\frac{2}{s},\frac{2}{1-s})$ to find
	\begin{align*}
	I_{32}& \lesssim
	\||Y_i|^{\frac p 2}\|_{H^1}^{2s(1-\frac{1}{p})}\||Y_i|^{p-2}\|_{L^1}^\frac{1-s}{2}
	\\
	&\qquad\times \Big(\frac1N\sum_{j=1}^N\|\nabla Y_j\|_{L^2}^2\Big)^{s/2} \Big(\frac1N\sum_{j=1}^N\|Y_j^2|Y_i|^p\|_{L^1}\Big)^{\frac{1-s}2}\Big(\frac1N\sum_{j=1}^N\|\Wick{Z_i Z_j}\|_{\bC^{-s}}^2\Big)^{1/2}
	\\
	&\le
	\frac1{40}\||Y_i|^{\frac p 2 }\|_{H^1}^2+\frac1{30}\frac1N\sum_{j=1}^N\|Y_j^2Y_i^p\|_{L^1}+1
	\\&\qquad +\|Y_i \|_{L^p}^{p-2}\Big(\frac1N\sum_{j=1}^N\|\nabla Y_j\|_{L^2}^2\Big)^{s/(1-s)} \Big(\frac1N\sum_{j=1}^N\|\Wick{Z_i Z_j}\|_{\bC^{-s}}^2\Big)^{1/(1-s)},
	\end{align*}
	where we used Young's inequality with exponents $ (\frac{p}{(p-1) s},\frac{p}{s}, \frac{2}{1-s},\frac{2}{1-s} )$ in the last step.  The estimate for $I_{33}$ uses Holder's inequality with the same exponents, followed by the Sobolev embedding theorem to yield
	\begin{align*}
	I_{33}  
	&\lesssim
	D^{\frac{s}2}\||Y_i|^{p}\|_{L^1}^{\frac{s (p-2)}{2p}}
	\||Y_i|^{p-2}\|_{L^1}^\frac{1-s}{2}
	\\
	&\qquad\times
	\Big(\frac1N\sum_{j=1}^N\|Y_j\|_{H^1}^2\Big)^{s/2}\Big(\frac1N\sum_{j=1}^N\|Y_j^2|Y_i|^p\|_{L^1}\Big)^{\frac{1-s}2}\Big(\frac1N\sum_{j=1}^N\|\Wick{Z_i Z_j}\|_{\bC^{-s}}^2\Big)^{1/2}
	\\
	&\le \frac1{40} D+\frac1{30}\frac1N\sum_{j=1}^N\|Y_j^2Y_i^p\|_{L^1}+\|Y_i \|_{L^p}^{p-2 }
	\\&\qquad +\|Y_i\|_{L^p}^{p-2}
	\Big(\frac1N\sum_{j=1}^N\|Y_j\|_{H^1}^2\Big)^{s/(1-s)}\Big(\frac1N\sum_{j=1}^N\|\Wick{Z_i Z_j}\|_{\bC^{-s}}^2\Big)^{1/(1-s)},
	\end{align*}
	where we used Young's inequality with exponents $ (\frac{2}{s},\frac2s, \frac{2}{1-s},\frac{2}{1-s} )$ in the last step.  Combining the above estimates, \eqref{est:I3} follows for $I_3$ by Young's inequality with exponents $(\frac{p}{p-2},\frac{p}{2})$.

	We now turn to $I_4$ and note that $\|\nabla |Y_i| ^p\|_{L^1}\lesssim \|\nabla |Y_i|^{\frac p 2 }\|_{L^2}\||Y_i|^{\frac p 2 }\|_{L^2}\lesssim D^{1/2}\||Y_i|^{p}\|_{L^1}^{1/2}$, so that
	\begin{align*}
	I_4 =\Big\<\frac1N\sum_{j=1}^N\Wick{Z_j^2} &  ,|Y_i|^p \Big\>
	\lesssim
	\Big\|\frac1N\sum_{j=1}^N\Wick{Z_j^2}\Big\|_{\bC^{-s}}
	\Big(\||Y_i|^p\|_{L^1}+\||Y_i|^p\|_{L^1}^{1-s}\|\nabla |Y_i|^p\|_{L^1}^s\Big)
	\\
	&\lesssim \Big\|\frac1N\sum_{j=1}^N\Wick{Z_j^2}\Big\|_{\bC^{-s}}
	\Big(\|Y_i\|_{L^p}^{p}+\|Y_i\|_{L^p}^{p(1-s/2) }D^{s/2}\Big)
	\\ 
	& \leq \frac1{40}D+C \Big(\Big\|\frac1N\sum_{j=1}^N\Wick{Z_j^2}\Big\|_{\bC^{-s}}^{\frac1{1-s/2}}+1\Big)\|Y_i\|_{L^p}^{p},
	\end{align*}
	by Young's inequality with exponents $(\frac{2}{s},\frac{2}{2-s})$, which implies \eqref{est:I3} for $I_4$.
	Finally, we estimate $I_5$ and note
	\begin{align*}
	I_5&=\Big\<\frac1N\sum_{j=1}^N\Wick{Z_iZ_j^2},|Y_i|^{p-2}Y_i \Big\>
	\\
	& \lesssim \Big\|\frac1N\sum_{j=1}^N\Wick{Z_iZ_j^2}\Big\|_{\bC^{-s}}\Big(\||Y_i|^{p-1}\|_{L^1}+\||Y_i|^{p-1}\|_{L^1}^{1-s}\|\nabla (|Y_i|^{p-2}Y_i)\|_{L^1}^s\Big)
	\\
	& \lesssim \Big\|\frac1N\sum_{j=1}^N\Wick{Z_iZ_j^2}\Big\|_{\bC^{-s}}\Big(\|Y_i\|_{L^{p-1} }^{p-1 }+\|Y_i\|_{L^{p-1} }^{(p-1)(1-s) } D^{\frac s2} \|Y_i\|_{L^{p-2}}^{\frac{(p-2)s}2}\Big),
	\end{align*}
	which gives \eqref{est:I3} for $I_5$ by Young's inequality.
	
	{\sc Step} \arabic{GlobLp} \label{GlobLp} \refstepcounter{GlobLp} (Conclusion)
	We now insert the inequalities \eqref{est:I1}, \eqref{est:I2}, and \eqref{est:I3} into the RHS of \eqref{eq:LpPhi} to obtain
	\begin{align}
	&\frac{1}{p}\frac{\dif}{\dif t} \|Y_i\|_{L^{p}}^{p}
	+\frac{1}{2} \||Y_i|^{p-2}|\nabla Y_i|^2\|_{L^1}
	+\frac{1}{2N}\sum_{j=1}^N\||Y_i|^p Y_j^2\|_{L^1}+m\|Y_i\|_{L^p}^p \nonumber\\
	&\leq C(F+F_{2}) \|Y_{i}\|_{L^{p}}^{p}+CF_{1}+CF_{3}.\label{zm1}
	\end{align}
	First note that by Holder's inequality and Young's inequality, it holds
	\begin{align}
	\|Y_i\|_{L^p}^{p}(F+F_{2})
	\le
	\|Y_i\|_{L^{p+2}}^{p} (F+F_{2})
	\le
	\big (N^{-\frac{p}{p+2} }\|Y_i\|_{L^{p+2}}^{p} \big )\big (N^{\frac{p}{p+2} }(F+F_{2}) \big ) \nonumber\\
	\le
	\frac{1}{4N}\|Y_{i}\|_{L^{p+2}}^{p+2}+C N^{\frac{p}{2}}(F+F_{2}) ^{\frac{p+2}{2}},\label{est:I11}
	\end{align}
	and by choosing $s$ sufficiently small depending on $p$, we may apply Lemma \ref{th:m2} and Lemma \ref{le:ex} to obtain  
	\begin{align*}
	\E (F+F_{2})^{\frac{p+2}{2}}+\E (F_{1}+F_{3}) \lesssim1.
	\end{align*}
	
	By a similar argument as in the proof of Lemma \ref{lem:zz1}, we first obtain $\E \|Y_i(0)\|_{L^p}^p\lesssim C(N)$. In fact, we  choose the solution $\tilde\Phi_i$ to equation \eqref{eq:Phi2d} starting from the stationary solution $\tilde{Z}_i(0)$, so that the process $\tilde Y_i=\tilde\Phi_i-Z_i$ starts from the origin. Using \eqref{zm1} and \eqref{sszz2}, \eqref{sszz3}, Lemma \ref{le:ex} we find for $T\geq 1$
	$$\int_0^T\E\|\tilde Y_i(t)\|_{L^p}^p\lesssim T,$$
	 which implies that $\E \|Y_i(0)\|_{L^p}^p\lesssim C(N)$ by similar argument as in the proof of  Lemma \ref{lem:zz1}.

	 Taking expectation on both sides of \eqref{zm1} and using stationarity of $(Y_j)_j$, we find
	\begin{align}
	&m\E\|Y_i\|_{L^p}^p+\frac{1}{2} \E\||Y_i|^{p-2}|\nabla Y_i|^2\|_{L^1}
	+\frac{1}{4N}\sum_{j=1}^N \E\||Y_i|^p Y_j^2\|_{L^1} \nonumber\\
	&\leq N^{\frac{p}{2}}\E(F+F_{2}) ^{\frac{p+2}{2}}+C\E(F_{1}+F_{3}),	\nonumber
	\end{align}
	which completes the proof.	
\end{proof}

\subsection{Correlations of observables}\label{se:non}

Now we turn to study the statistical property
of the limiting observable, namely, we show that the limiting observables have {\it nontrivial} laws,
in the sense that although $\Phi_i$ converges to the (trivial) stationary solution $Z_i$ (and $\Wick{\Phi_i^2} \to\Wick{Z_i^2}$ as $N\to \infty$ for each $i$),
the observables do not converge to the ones with $\Phi_i$ replaced by $Z_i$. We then write  for shorthand
\begin{equ}[e:-Cw-d2]
	\PPhi^2 \eqdef \sum_{i=1}^N \Phi_i^2,
	\quad
	\Wick{\PPhi^2} \eqdef \sum_{i=1}^N  \Wick{\Phi_i^2},
	\quad
	\mathbf{Z}^2 \eqdef \sum_{i=1}^N Z_i^2,
	\quad
	\Wick{\mathbf{Z}^2} \eqdef \sum_{i=1}^N  \Wick{Z_i^2}
\end{equ}
%Consider the observables
%$$
%\frac{1}{\sqrt{N}} \Wick{\PPhi^2}
%\qquad
%\frac{1}{N} \Wick{(\PPhi^2)^2}\eqdef \frac{1}{N} \Wick{\Big( \sum_{i=1}^N \Phi_i^2 \Big)^2}
%$$
%as in \eqref{ob1}, \eqref{ob2}. 
The two observables in \eqref{obob} can be then written as 
$\frac{1}{\sqrt{N}} \Wick{\PPhi^2}$ and
$\frac{1}{N} \Wick{(\PPhi^2)^2}$.
We are in the same setting as in Section \ref{sec:label}, i.e.  we decompose $\Phi_i=Y_i+Z_i$ with $(Y_i, Z_i)$ stationary and we also consider $Y_i, Z_i$ as stationary process with $Z_i$ as the stationary solution of \eqref{eq:Zm} and $Y_i$ as the  solution of \eqref{eq:22}.

To state such ``nontriviality'' results, we  consider the correlation function
$$
G_N(x-z) = \E \Big[\frac{1}{\sqrt{N}} \Wick{\PPhi^2}(x)\frac{1}{\sqrt{N}} \Wick{\PPhi^2}(z)\Big].
$$
%\zhu{we may change the defintion as below. Do you agree? In appendix we first try to modify the integration by parts formula. }\hao{I would slightly prefer a cleaner notation like $\Wick{\PPhi^2} (\rho_x^\eps)$ because it ``almost look like'' $\Wick{\PPhi^2} (x)$} \zhu{I make some changes. but how about $\< \Wick{ (\PPhi^2)^2},\rho_x^\eps\>$? Do you want to write it as $ \Wick{ (\PPhi^2)^2}(\rho_x^\eps)$, which semms a little bit long? 
%} 
This precisely means the following.
For a smooth function $g$ and a distribution $F$ (such as $\Wick{\PPhi^2}$),  we write $F(g)\eqdef\langle F,g\rangle$.
Then for a smooth test function $f$ the above correlation function $G_N$ is understood as
$$
G_N(f) \eqdef \lim_{\eps\to0} \int\E[\frac{1}{\sqrt{N}} \Wick{\PPhi^2}(\rho_x^\eps)\frac{1}{\sqrt{N}} \Wick{\PPhi^2}(\rho_z^\eps)]f(x-z)\dif z.
$$
%\zhu{We could take a limit in $\bC^{-\kappa}$. It is unclear to me why $G_N$ is a function not a distribution even we decomposing $\Phi=Y+Z$. We can only understand this limit  $\lim_{\eps\to0}\E\<YZ,\rho_x^\eps\>$? in $\bC^{-\kappa}$}% But I think for fix $x$ and $z$ we could prove the right hand $\E [\frac{1}{\sqrt{N}} \langle\Wick{\PPhi^2},\rho_x^\eps\rangle \frac{1}{\sqrt{N}} \langle\Wick{\PPhi^2},\rho_z^\eps\rangle]$ is uniformly bounded w.r.t. $\eps$ by decomposing $\Phi=Y+Z$. For example,
%\begin{align*}\lim_{\eps\to0}\E [ \langle\Wick{\Phi^2},\rho_x^\eps\rangle  ]
%=\E[Y^2(x)]+2\lim_{\eps\to0}\E\<YZ,\rho_x^\eps\>.
%\end{align*}
%Here $\E\<YZ,\rho_x^\eps\>$ is uniformly bounded w.r.t. $x, \eps$? Then we could always find the limit by choosing subsequence, right?
for some mollifier $\rho$ with $\rho_x^\eps(z)=\eps^{-2}\rho(\frac{z-x}\eps)$ and $\rho_x^\eps\to \delta_x$ as $\eps\to0$.
Here by translation invariance of $\nu^N$, $G_N(f)$ does not depend on $x$.  We define the Fourier transform of $G_N$ as $\widehat{G}_N(k)=G_N(e_{-k})$ with $k\in \mZ^2$, with $\{e_k\}$ the Fourier basis.
For comparison,
we first note that
$$
%\lim_{N\to \infty}
\lim_{\eps\to0}\E [\frac{1}{\sqrt N}\Wick{\mathbf Z^2}(\rho_x^\eps) \frac{1}{\sqrt N} \Wick{\mathbf Z^2}(\rho_z^\eps)]
=
2 C(x-z)^2
$$
for any $N$ and $x,z\in \mathbb T^2$, %\scott{A reference for the identity above in the case $N=1$ would be useful if easily available somewhere in the literature.  Should we state also the analogous result for $(Z^{2})^{2}$? },
where $C=\frac12(m-\Delta)^{-1}$, which follows from definition of $\Wick{\mathbf Z^2}$ and Wick's theorem.
Also, $\E  \Wick{(\mathbf{Z}^2)^2}=0$ for any $N$.

We denote $\widehat f$ the Fourier transform of a function $f$.  

\begin{theorem}\label{theo:nontrivial2}
	 Let $m$ as in Lemma \ref{th:m2}. It holds that
	\begin{equs}
		\lim_{N\to \infty} \widehat{G_N} & = 2\widehat{C^2} /(1+2\widehat{C^2}) ,
		\label{e:bubble-chain2}
		\\
		\lim_{N\to \infty}\lim_{\eps\to0}
		\E \frac{1}{N} \<\Wick{ (\PPhi^2)^2} ,\rho_x^\eps\> &=
		-4  \sum_{k\in \mathbb {Z}^2} \widehat{C^2}(k)^2 /(1+2\widehat{C^2}(k)).
	\end{equs}
\end{theorem}
In particular, in view of the discussion above the theorem,
as $N\to \infty$
the limiting law of
$\frac{1}{\sqrt{N}} \Wick{\PPhi^2} $
and $\frac{1}{N} \Wick{(\PPhi^2)^2} $
are different from that of
$\frac{1}{\sqrt{N}} \Wick{\mathbf{Z}^2}$ and
$\frac{1}{N} \Wick{(\mathbf{Z}^2)^2}$.

\begin{proof} %\zhu{The following result should be understood in the sense of approximation and also the result from Appendix \ref{sec:A2}. maybe convolution approximation? Do we need to change everything depending on $\eps$.}\hao{My understanding is, for instance, although $Z_1(x)$ does not make sense, $\E(Z_1(x)Z_1(y))$ is smooth as long as $x\neq y$.
	%I think here $G_N(x-z)$ is smooth function except for diagonal (just like $C(x-z)$), which isn't a priori obvious but could be seen from our proof.
	% and the integrals in \eqref{e:IBP1d2}, \eqref{e:IBP2d2} are integrable. I even think the expectations on the RHS of $R_N$ and $Q_N$ are smooth as long as $z_1,z_2,x$ are different; for instance we can probably see this is true by decomposing into $Z+Y$? So basically my opinion is, is it possible to just write a remark after this proof to mention that these expectations are smooth functions? or a remark which says the expressions in this proof can be understood as limits of some approximation such as convolution?} \zhu{Our question is how to see these functions are smooth except the diagnoal?}
	%\zhu{We rewrite the integration by parts formula in appendix by choosing test functions. But the  limit in \eqref{e:IBP2d2} is still unclear to me, since we don't know $G_N$ is  a function? Our previous result only implies that $G_N$ is a distribution.  Maybe we use Kupi's work? It's also not clear to me why $R_N, Q_N$ are functions. As we prove the result for fourier transform of them, there's no problem here.}
	Integration by parts formula gives us the following
	identities (see  Appendix~\ref{sec:D2})
	\begin{equ}[e:IBP1d2]
		-\frac{1}{4N} \lim_{\eps\to0}\E \< \Wick{ (\PPhi^2)^2},\rho_x^\eps\>
		=\frac{(N+2)}{2N} \sum_{k\in\mZ^2} \widehat{C^2}(k) \widehat{G_N}(k)  + R_N,
	\end{equ}
	\begin{equ}[e:IBP2d2]
		(\frac12+\frac{N+2}{N}\widehat{C^2}) \widehat{G_N}  = \widehat{C C_N} + \widehat{Q_N}/N
	\end{equ}
	with
	\begin{align*}
	C_N(f)&=\lim_{\eps\to0} \int \E  [ \Phi_1(\rho_x^\eps) \Phi_1(\rho_z^\eps)]f(x-z)\dif z,
	\\R_N &= -\frac{1}{N^2} \lim_{\eps\to0}\int C(x-z_1) C(x-z_2) \E \Big[ \Wick{\Phi_1 \PPhi^2} \!(\rho_{z_1}^\eps)
	\Wick{\Phi_1 \PPhi^2} \!(\rho_{z_2}^\eps)
	\Wick{\PPhi^2}\! (\rho_x^\eps)\Big] \dif z_1\dif z_2,
	\\
	Q_N &=-2\lim_{\eps\to0}\int  C(x-y)C(x-z)
	\E \Big[ \Wick{\Phi_1 \PPhi^2} \!(\rho_y^\eps) \,\Phi_1 \!(\rho_z^\eps) \Big]\,\dif y
	\\
	& + \lim_{\eps\to0}\frac{2}{N} \int  C(x-z_1) C(x-z_2)
	%\\ &\qquad  \times
	\E \Big[
	\Wick{\Phi_1 \PPhi^2} \!(\rho_{z_1}^\eps)
	\Wick{\Phi_1 \PPhi^2}\! (\rho_{z_2}^\eps)
	\Wick{\PPhi^2}\!(\rho_z^\eps) \Big] \dif z_1\dif z_2
	\\
	& \eqdef Q^N_1+Q^N_2.
	\end{align*}
	We first use \eqref{e:IBP2d2} and we know $\widehat{C_N}\to \widehat{C}$ as $N\to \infty$ by using Lemma \ref{th:m2} and 
	\begin{align*}\widehat{C_N}(k)
	=&c\E \int Y_1(x)Y_1(z)e_{-k}(x-z)\dif z\dif x+c\E\int Y_1(x)\<Z_1,e_{-k}(x-\cdot)\>\dif x
	\\&+c\E\int Y_1(z)\<Z_1,e_{-k}(\cdot-z)\>\dif z+\widehat{C}(k),\end{align*}
	for constant $c$. %\scott{Might be good to explain this one.  For fixed $\epsilon$ we could use Theorem 5.11 (the part where we say "furthermore..",).  Are we also commuting limits in $\epsilon$ and $N$ here? }. 
	Since $C$ is positive definite, $\widehat C$ is a positive function, and so is $\widehat{C^2}$; this allows us to divide both sides by $\frac12+\widehat{C^2}$.
	In Lemmas \ref{le:I} and \ref{lem:q22} below
	we show that  $\widehat{Q_N}(k)/N$ vanishes  in the limit for every $k$.
	We therefore obtain \eqref{e:bubble-chain2}.
%	In particular,
%	$\lim_{N\to \infty} \widehat{G_N}  \neq 2\widehat{C^2}$,
%	showing that
%	the random fields $\frac{1}{\sqrt{N}} \Wick{\PPhi^2}  $
%	and
%	$\frac{1}{\sqrt{N}} \Wick{\mathbf{Z}^2}$ have different limiting laws.
	
	By Lemma \ref{lem:q22},  $R_N$ converges weakly to $0$ as $N\to \infty$.  However, $R_N$ is independent of $x$ as a consequence of spatial translation invariance, so we also obtain pointwise convergence of $R_{N}$ to zero.
	 By Lemma \ref{lem:CN},  %translation invariance we know for some positive constant $c>0$
	%\begin{align*}
	%\widehat{CC_N}(k)=&\lim_{\eps\to0} \E  \int C(x-z)  \Phi_1(\rho_x^\eps) \Phi_1(\rho_z^\eps)e_k(x-z)\dif z
	%\\=&c\lim_{\eps\to0} \E  \int C(x-z)  \Phi_1(\rho_x^\eps) \Phi_1(\rho_z^\eps)e_k(x-z)\dif z\dif x
	%\\=&c\E    \<\Phi_1 (m-\Delta)^{-1}(\Phi_1e_k),e_k\>,
	%\end{align*}
	$\widehat{CC_N}(k)$ is uniformly bounded by $c(m+|k|^2)^{-1+\kappa}$ for $c>0$ and small $\kappa>0$. By  dominated convergence theorem we have 
	$$\frac{(N+2)}{2N} \sum_{k\in\mZ^2} \widehat{C^2}(k) \frac{\widehat{CC_N}(k)}{(\frac12+\frac{N+2}N\widehat{C^2}(k))}\to   \sum_{k\in\mathbb Z^2} \frac{\widehat{C^2}(k)\cdot \widehat{C^2}(k)} {1+2\widehat{C^2}(k)},\quad \mbox{as} \quad N\to \infty.$$
	By Lemmas \ref{le:I} and \ref{lem:q22} and  dominated convergence theorem we find 
	$$\frac{(N+2)}{2N} \sum_{k\in\mZ^2} \widehat{C^2}(k) \frac{\widehat{Q_N}(k)}{N(\frac12+\frac{N+2}N\widehat{C^2}(k))}\to 0,\quad \mbox{as} \quad N\to \infty,$$
	which combined with  Lemma  \ref{lem:q22} and \eqref{e:IBP1d2}, \eqref{e:IBP2d2} implies that %\zhu{We add the above convergence since we need to take sum w.r.t. $k$.}
	\begin{align*}
	\lim_{N\to \infty}
	\E \frac{1}{4N} \<\Wick{ (\PPhi^2)^2} ,\rho_x^\eps\>
	%=&- \int C(x-z)^2 \lim_{N\to \infty} G_N(x-z)\,\dif z
	=&
	-  \sum_{k\in\mathbb Z^2} \widehat{C^2}(k)\cdot \widehat{C^2}(k) /(1+2\widehat{C^2}(k))
	\end{align*}
	where the sum is over integers (i.e. Fourier variables). 
	This is non-zero, showing that
	the limiting law of  $ \frac1N \Wick{ (\PPhi^2)^2}  $ is different from
	that of $\frac{1}{N} \Wick{(\mathbf{Z}^2)^2}$.
\end{proof}

We will use Lemma   \ref{th:m2}  and \ref{th:lp} to control  the remainder terms from integration by parts formula. In fact, all the remainder terms will be controlled by the following terms:
\begin{align*}
\E A_1^{\ell_1}A_2^{\ell_2}A_3^{\ell_3},
\end{align*}
with $\ell_i\geq0,$ where  (for $s>3\kappa>0$ small enough)
\begin{align*}
A_1:=& \|Y_1\|_{L^2}+\|Z_1\|_{\bC^{-\frac\kappa 3}},\\
A_2:=& \bigg\|\sum_{i=1}^NY_i^2\bigg \|_{L^{1}}+\sum_{i=1}^N\|Y_i\|_{H^{\kappa}}\|Z_i\|_{\bC^{-\frac\kappa2}}+\bigg\|\sum_{i=1}^N\Wick{Z_i^2} \bigg \|_{H^{-\frac\kappa2}},\\
A_3:=&\bigg\|Y_1\sum_{i=1}^NY_i^2 \bigg \|_{L^{1+s}}+\bigg \|\Lambda^{-s}(Z_1\sum_{i=1}^NY_i^2) \bigg \|_{L^{1+s}}
\\&+\bigg \|\Lambda^{-s}(Y_1\sum_{i=1}^NY_iZ_i) \bigg \|_{L^{1+s}}+\bigg \|\Lambda^{-s}(Z_1\sum_{i=1}^NY_iZ_i) \bigg \|_{L^{1+s}}
\\&+\bigg \|\Lambda^{-s}(Y_1\sum_{i=1}^N\Wick{Z_i^2}) \bigg \|_{L^{1+s}}+\bigg \|\sum_{i=1}^N\Wick{Z_1Z_i^2} \bigg \|_{H^{-s}}
:=\sum_{i=1}^6\bar{A}_{3i}.
\end{align*}
%In fact, we have
%\begin{align*}
%\<\Phi_1,e_k\>\lesssim &|k|^s A_1,\quad
%\<\Wick{\PPhi^2},e_k\>\lesssim |k|^s  A_2,\quad
%\<\Wick{\Phi_1\PPhi^2},e_k\>\lesssim |k|^s  A_3.
%\end{align*}

\bl\label{lem:zmm} For $m$ as in Lemma \ref{th:m2} and for each $\ell_i\geq0$ with $\ell_2\frac\kappa2+3\ell_3 s<1$ it holds
\begin{align*}
\E A_1^{\ell_1} A_2^{\ell_2} A_3^{\ell_3} \lesssim N^{(\ell_2+\ell_3)/2},
\end{align*}
where the implicit constant is independent of $N$.
\el
\begin{proof} By Lemma \ref{th:m2} and Lemma \ref{le:ex}, it follows that
	$\E A_1^{\ell_1}\lesssim1$ for every $\ell_1\geq0$.
	Using the interpolation Lemma \ref{lem:interpolation} followed by  H\"older's inequality with exponents $(\frac{2}{\kappa},\frac{2}{1-\kappa},2)$ we find
	\begin{align*}
	A_2\lesssim &\sum_{i=1}^N\|Y_i\|_{L^2}^2+\Big(\sum_{i=1}^N\|Y_i\|_{L^2}^2\Big)^{\frac{1-\kappa}{2} }\Big(\sum_{i=1}^N\|Y_i\|_{H^1}^2\Big)^{ \frac\kappa2 }
	\Big(\sum_{i=1}^N\|Z_i\|_{\bC^{-\frac\kappa2}}^2\Big)^{\frac12}+\Big\|\sum_{i=1}^N \Wick{Z_i^2}\Big\|_{H^{-\frac\kappa2}}
	\\:=&A_{21}+A_{22}+A_{23}.
	\end{align*}
	By Lemma \ref{th:m2} and Lemma \ref{le:ex}, it follows that for all $\ell_2\geq0$ with $\frac\kappa2\ell_2<1$  
	%\hao{Are we also using Lemma~\ref{le:ex} and/or \eqref{eq:Ui}?  (or something like proof of \eqref{mom} ? I don't know. Should the first inequality below be $A_{21}$? } )\zhu{We use Lemma \ref{th:m2} and for $A_{23}$ we used \eqref{eq:Ui}}
	$$\E A_{21}^{\ell_2}\lesssim 1,\quad\E A_{22}^{\ell_2}+ \E A_{23}^{\ell_2}\lesssim N^{\frac{\ell_2}{2}},$$
	where we used \eqref{eq:Ui} and Gaussian hypercontractivity for the bound of $A_{23}$. %give $N^{1/2}$.
	It remains to consider $A_3$. Starting with $\bar{A}_{31}$, let $p \in (1,2)$ and $q>1$ satisfy $\frac1q+\frac1p=\frac1{1+s}$, then use H\"older's inequality and interpolate to obtain
	\begin{align*}
	\bar{A}_{31}\lesssim \|Y_1\|_{L^q}\Big\|\sum_{i=1}^NY_i^2\Big\|_{L^{p}}
	\lesssim \|Y_1\|_{L^q}\Big\|\sum_{i=1}^NY_i^2\Big\|_{L^{2}}^{2-\frac2p}\Big\|\sum_{i=1}^NY_i^2\Big\|_{L^{1}}^{\frac{2}{p}-1}:=A_{31}.
	\end{align*}
	Given $\ell_3\geq 0$, choosing $p$ close enough to $1$ and $q>\ell_3$ to ensure $\ell_{3}(2-\frac{2}{p})\frac{q}{q-\ell_3}<2$,  we may use Lemma \ref{th:lp} and Lemma \ref{th:m2} and H\"older's inequality to obtain the bound
	$$\E A_{31}^{\ell_3}\lesssim N^{\frac{\ell_3}2}.$$
	%\zhu{We change this estimate compared with the estimate before, since we want the exponent for $\Big\|\sum_{i=1}^NY_i^2\Big\|_{L^{2}}$ small enough.} $A_{31}$ gives $N^{1/2}$.
	The next three terms can be estimated by using Lemma \ref{lem:multi} and Lemma \ref{lem:interpolation}.  Specifically, we use that for $s \in (0,1)$ it holds $$\|\Lambda^{-s}(fg)\|_{L^{1+s}} \lesssim \|f \|_{B^{s}_{1+s,1}}\|g\|_{\bC^{-s+\kappa}}\lesssim (\|f\|_{H^{2s}}\wedge \|\Lambda^{3s}f\|_{L^1} )\|g\|_{\bC^{-s+\kappa}},$$ 
	for $s>3\kappa>0$ small enough.
	 %\scott{I did a calculation by duality and the previous inequality seems true, and it seems this is what you used for $\overline{A}_{34}$.  So why is there $H^{2s}$ appearing in the bounds for $\overline{A}_{32}$ and $\overline{A}_{33}$? } \zhu{I make changes above and below.}
	 %\zhu{In fact we choose $f=Y_iY_1$ and we use $\|\Lambda^{-s}(fg)\|_{L^{1+s}} \lesssim \|\Lambda^sf \|_{L^{1+s+\kappa}}\|g\|_{\bC^{-s}}$ and we have $\|\Lambda^s(Y_iY_1)\|_{L^{1+s+\kappa}}$, which can be bounded by $\|Y_i\|_{H^{2s}}\|Y_1\|_{H^{2s}}$.}
	\begin{align*}
	\bar{A}_{32}
	&\lesssim \|Z_1\|_{\bC^{-s+\kappa }}\Big\|\Lambda^{3s}\sum_{i=1}^NY_i^2\Big\|_{L^1}
	\lesssim 
	\|Z_1\|_{\bC^{-s}}\Big(\sum_{i=1}^N\|Y_i\|_{H^1}^2\Big)^{3s}\Big(\sum_{i=1}^N\|Y_i\|_{L^2}^2\Big)^{1-3s}:=A_{32},
	\\
	\bar{A}_{33}
	&\lesssim 
	\sum_{i=1}^N\|Z_i\|_{\bC^{-s+\kappa }}\| Y_i\|_{H^{3s}}\|Y_1\|_{H^{3s}}
	\\
	&\lesssim \Big(\sum_{i=1}^N\|Z_i\|_{\bC^{-s+\kappa }}^2\Big)^{1/2}\|Y_1\|_{H^1}^{3s}\|Y_1\|_{L^2}^{1-3s}\Big(\sum_{i=1}^N\|Y_i\|_{H^1}^2\Big)^{3s/2}
	\Big(\sum_{i=1}^N\|Y_i\|_{L^2}^2\Big)^{\frac{1-3s}{2}}:=A_{33},
	\\
	\bar{A}_{34}
	&\lesssim \sum_{i=1}^N\|\Wick{Z_1Z_i}\|_{\bC^{-s+\kappa }}\|Y_i\|_{H^{2s}}
	\\
	&\lesssim \Big(\sum_{i=1}^N\|\Wick{Z_1Z_i}\|_{\bC^{-s+\kappa}}^2\Big)^{1/2}
	\Big(\sum_{i=1}^N\|Y_i\|_{H^1}^2\Big)^{s}\Big(\sum_{i=1}^N\|Y_i\|_{L^2}^2\Big)^{\frac{1-2s}2}:=A_{34},
	\end{align*}
	and
	%$A_{34}$ gives $N^{1/2}$.
	\begin{align*}
	\bar{A}_{35}\lesssim
	\Big\|\sum_{i=1}^N\Wick{Z_i^2}Y_1\Big\|_{B^{-s}_{1+s,1}}
	\lesssim
	\Big\|\sum_{i=1}^N\Wick{Z_i^2}\Big\|_{H^{-s+\kappa}}
	\|Y_1\|_{H^{3s}}\lesssim \Big\|\sum_{i=1}^N\Wick{Z_i^2}\Big\|_{H^{-s+\kappa}}
	\|Y_1\|_{H^1}^{3s}\|Y_1\|_{L^2}^{1-3s}:=A_{35}.
	\end{align*}
	By Lemma \ref{th:m2} we deduce that  for $3\ell_3s<1$
	$$\E A_{3i}^{\ell_3}\lesssim N^{\frac{\ell_3}2},$$
	with $i=2,\dots,5$. The last term is given as
	\begin{align*}
	A_{36}:=\Big\|\sum_{i=1}^N\Wick{Z_1Z_i^2}\Big\|_{H^{-s}},
	\end{align*}
	which by similar argument as in the proof of \eqref{mom}  implies that
	$\E A_{36}^{\ell_3}\lesssim N^{\frac{\ell_3}2},$
	%	As a result, by Lemma \ref{th:m2} and Lemma \ref{th:lp} and
	Combining the above estimates and H\"older's inequality we obtain
	\begin{align*}
	\E A_1^{\ell_1}A_2^{\ell_2}A_3^{\ell_3}\lesssim \sum_{(i_1,i_2,i_3)}\E A_{1i_1}^{\ell_1}A_{2i_2}^{\ell_2}A_{3i_3}^{\ell_3}\lesssim N^{(\ell_2+\ell_3)/2}.
	\end{align*}
\end{proof}

In the following proof we use $c$ to denote  positive constant, which may change from line to line.

\bl\label{lem:CN}
It holds that
$$|\widehat{CC_N}(k)|\lesssim (1+|k|^2)^{-1+\kappa},$$
for $\kappa>0$, where the proportional constant is independent of $k$. 
\el
\begin{proof}
By translation invariance we know for some positive constant $c>0$ 
\begin{align*}
&\widehat{CC_N}(k)=\lim_{\eps\to0} \E  \int C(x-z)  \Phi_1(\rho_x^\eps) \Phi_1(\rho_z^\eps)e_{-k}(x-z)\dif z
\\=&c\lim_{\eps\to0} \E  \int C(x-z)  \Phi_1(\rho_x^\eps) \Phi_1(\rho_z^\eps)e_{-k}(x-z)\dif z\dif x
=c\E    \<\Phi_1 (m-\Delta)^{-1}(\Phi_1e_k),e_k\>
\\=&c\E \sum_{k_1\in\mZ^2} \widehat{\Phi_1}(k_1) (m+|k-k_1|^2)^{-1}\widehat{\Phi_1}(-k_1)
\\\lesssim& \Big[\E \Big(\sum_{k_1\in\mZ^2} |\widehat{\Phi_1}(k_1)|^2 (m+|k-k_1|^2)^{-1}\Big)\Big]^{\frac12}\Big[\E\Big( \sum_{k_1\in\mZ^2} |\widehat{\Phi_1}(-k_1)|^2 (m+|k-k_1|^2)^{-1}\Big)\Big]^{\frac12}
\\\eqdef& I_1^{\frac12}I_2^\frac12.
\end{align*}
$I_1$ is  bounded by 
$$\E \sum_{k_1\in\mZ^2} |\widehat{Y_1}(k_1)|^2 (m+|k-k_1|^2)^{-1}+\E \sum_{k_1\in\mZ^2} |\widehat{Z_1}(k_1)|^2 (m+|k-k_1|^2)^{-1}.$$
Using Lemma \ref{th:m2} we know
$$\E|\widehat{Y_1}(k_1)|^2\lesssim (1+|k_1|)^{-2}\E\|Y_1\|_{H^1}^2\lesssim (1+|k_1|)^{-2}.$$ Using translation invariance we find
$$\E \sum_{k_1\in\mZ^2} |\widehat{Z_1}(k_1)|^2 (m+|k-k_1|^2)^{-1}=c\widehat{C^2}(k).$$ 
Now the desired bound for $I_1$ follows from Lemma \ref{lem:sum}. Similarly we deduce the required bound for $I_2$ and the result follows.
\end{proof}

\bl\label{le:I}It holds that for every $k\in \mZ^2$
$$|\widehat{Q_1^N}(k)|\lesssim N^{1/2}(1+|k|)^{-\kappa/2},$$
for every $0<\kappa<\frac16$, where the proportional constant is independent of $N$ and $k$. 
\el

\begin{proof} We use translation invariance property to 
	write the Fourier transform of  $Q_1^N$ as
	%\begin{align*}
	%\sum_{i=1}^N\E\Big[(m-\Delta)^{-1} (\Wick{\Phi_1\Phi_i^2})\cdot (m-\Delta)^{-1}\Phi_1\Big],
	%\end{align*}
	%\end{align*}
	\begin{align*}
	\widehat{Q_1^N}(-k)=&\int Q_1^Ne_k (x-z)\dif z =c\int Q_1^Ne_k (x-z)\dif x\dif z
	\\
	=&c\E\Big[\Big\<(m-\Delta)^{-1}(\Wick{\Phi_1\Phi_i^2})\cdot (m-\Delta)^{-1}[\Phi_1\cdot e_{-k}],e_{-k}\Big\>\Big],
	\end{align*}
 	which by Lemma \ref{lem:multi} can be bounded by 
	$$(1+|k|)^{-\kappa}\E\Big[\Big\|(m-\Delta)^{-1}\Big((Y_1+Z_1)\sum_{i=1}^N(Y_i^2+2Y_iZ_i+:Z_i^2:)\Big)\Big\|_{\bC^{\kappa}} \|(m-\Delta)^{-1}[(Y_1+Z_1)e_{-k}]\|_{\bC^\kappa}\Big],$$
	where we used $\|e_k\|_{B^{-\kappa}_{1,1}}\lesssim (1+|k|)^{-\kappa}$, which can be checked by direct calculation.
	 By Besov embedding Lemma \ref{lem:emb} and Schauder estimate  we know that for $\frac12>s>3\kappa>0$ small enough
	\begin{align}\label{eq:s}
	\|(m-\Delta)^{-1}f\|_{\bC^\kappa}\lesssim \|f\|_{L^{1+s}}\wedge \|\Lambda^{-s}f\|_{L^{1+s}}\wedge \|f\|_{H^{-s}}\wedge \|f\|_{\bC^{-\kappa/3}}.\end{align}
	\eqref{eq:s} combined with Lemma \ref{lem:multi}
	implies that the above term %containing $Y_1\sum_{i=1}^NY_i^2$
	could be controlled by
	$\E[A_{3}A_{1}]\|e_k\|_{\bC^{\kappa/2}}$, which by  Lemma \ref{lem:zmm} can be bounded by $N^{\frac12}(1+|k|)^{-\kappa/2}$ %\scott{Here, do you mean to have $+\frac{\kappa}{2}$ from the $\|e_k\|_{\bC^{\kappa/2}}$?  } \zhu{yes.}.

\end{proof}

\bl\label{lem:q22} For every $k\in\mZ^2$
$$|\widehat{Q_2^N}(k)|\lesssim N^{1/2}(1+|k|)^{-\frac\kappa2},\quad |\widehat{R^N}(k)|\lesssim N^{-1/2}(1+|k|)^{\frac\kappa2},$$
for $0<\kappa<\frac2{37}$, where the proportional constant is independent of $N$ and $k$. 
\el

\begin{proof} %It is sufficient to bound $|\widehat{Q_2^N}(k)|$ and $N|\widehat{R^N}(k)|$ satisfies exactly the same bound. 
	We write the Fourier transform of $Q_2^N$ as
	\begin{align*}
	\widehat{Q_2^N}(-k)&=\int Q_2^N e_k(x-z)\dif x=c\int Q_2^N e_k(x-z)\dif x\dif z\\
	&=\frac{c}{N}\E\Big\<\Big(\sum_{i=1}^N(m-\Delta)^{-1}\Phi_1\Phi_i^2\Big)^2,e_{-k}\Big\>
	\Big\<\sum_{i=1}^N\Wick{\Phi_i^2},e_{k}\Big\>
	,
	\end{align*}
	which can be bounded by 
	\begin{align*}
	&\frac{c}{N}\E\Big[\Big\|(m-\Delta)^{-1}\Big((Y_1+Z_1)\sum_{i=1}^N(Y_i^2+2Y_iZ_i+:Z_i^2:)\Big)\Big\|_{\bC^\kappa}^2\|e_{-k}\|_{B^{-\kappa}_{1,1}}
	\\&\qquad\times \Big|\Big\<\Big(\sum_{i=1}^N(Y_i^2+2Y_iZ_i+:Z_i^2:)\Big),e_{-k}\Big\>\Big|\Big].
	\end{align*}
	Using \eqref{eq:s} for the first line and Lemma \ref{lem:multi} for the second line and $\|e_k\|_{B^{-\kappa}_{1,1}}\lesssim (1+|k|)^{-\kappa}$, the above term %containing $Y_1\sum_{i=1}^NY_i^2$
	can be estimated by
	$\frac{1}{N}\E[A_{3}^2A_{2}](|k|+1)^{-\kappa}\|e_{-k}\|_{B^{\frac\kappa2}_{2,1}}$ for $s>3\kappa, 6s+\frac\kappa2<1$, which by  Lemma \ref{lem:zmm} can be bounded by $N^{\frac12}(|k|+1)^{-\frac\kappa2}$.
	
	Similarly we write
	\begin{align*}
	\widehat{R^N}(k)&
	=\frac{c}{N^2}\E\Big\<\Big(\sum_{i=1}^N(m-\Delta)^{-1}\Phi_1\Phi_i^2\bigg)^2\Big(\sum_{i=1}^N\Wick{\Phi_i^2}\Big),e_{k}\Big\>
	,
	\end{align*}
	which by Lemma \ref{lem:multi} can be bounded by 
	\begin{align*}
	&\frac{c(1+|k|)^{\frac\kappa2}}{N^2}\E\Big\|(m-\Delta)^{-1}\Big((Y_1+Z_1)\sum_{i=1}^N(Y_i^2+2Y_iZ_i+:Z_i^2:)\Big)\Big\|_{\bC^\kappa}^2A_2,
	\end{align*}
	for $s>3\kappa>0$ and $6s+\frac\kappa2<1$. By \eqref{eq:s} % Besov embedding Lemma \ref{lem:emb} we know 
	%\begin{align}\label{eq:s1}
	%\|(m-\Delta)^{-1}f\|_{\bC^\kappa}\lesssim \|f\|_{L^{1+s}}\wedge \|\Lambda^{-s}f\|_{L^{1+s}}\wedge \|f\|_{H^{-s}}\wedge \|f\|_{\bC^{-\kappa}},\end{align}
	we know that the above term %containing $Y_1\sum_{i=1}^NY_i^2$
	could be controlled by
	$\frac{1}{N^2}\E[A_{3}^2A_{2}](1+|k|)^{\frac\kappa2}$, which by  Lemma \ref{lem:zmm} can be bounded by $N^{-\frac12}(1+|k|)^{\frac\kappa2}$.
\end{proof}

\def\wk{\mbox{\tiny \textsc{w}}}
 \appendix
  \renewcommand{\appendixname}{Appendix~\Alph{section}}
  \renewcommand{\theequation}{A.\arabic{equation}}
%\section{Appendix}

\section{Notations and Besov spaces}\label{s:not}
Throughout the paper, we use the notation $a\lesssim b$ if there exists a constant $c>0$ such that $a\leq cb$, and we write $a\simeq b$ if $a\lesssim b$ and $b\lesssim a$. Given a Banach space $E$ with a norm $\|\cdot\|_E$ and $T>0$, we write $C_TE=C([0,T];E)$ for the space of continuous functions from $[0,T]$ to $E$, equipped with the supremum norm $\|f\|_{C_TE}=\sup_{t\in[0,T]}\|f(t)\|_{E}$.  For $p\in [1,\infty]$ we write $L^p_TE=L^p([0,T];E)$ for the space of $L^p$-integrable functions from $[0,T]$ to $E$, equipped with the usual $L^p$-norm.
Let $\mathcal{S}'$ be the space of distributions on $\mathbb{T}^d$. We use $(\Delta_{i})_{i\geq -1}$ to denote the Littlewood--Paley blocks for a dyadic partition of unity.
Besov spaces on the torus with general indices $\alpha\in \R$, $p,q\in[1,\infty]$ are defined as
the completion of $C^\infty$ with respect to the norm
$$
\|u\|_{B^\alpha_{p,q}}:=
\Big(\sum_{j\geq-1}(2^{j\alpha}\|\Delta_ju\|_{L^p}^q)\Big)^{1/q},
$$
and the H\"{o}lder-Besov space $\bC^\alpha$ is given by $\bC^\alpha=B^\alpha_{\infty,\infty}$.  We will often write $\|\cdot\|_{\bC^\alpha}$ instead of $\|\cdot\|_{B^\alpha_{\infty,\infty}}$.

Set $\Lambda= (1-\Delta)^{\frac{1}{2}}$. For $s\geq0$, $p\in [1,+\infty]$ we use $H^{s}_p$ to denote the subspace of $L^p$, consisting of all  $f$   which can be written in the form $f=\Lambda^{-s}g, g\in L^p$ and the $H^{s}_p$ norm of $f$ is defined to be the $L^p$ norm of $g$, i.e. $\|f\|_{H^{s}_p}:=\|\Lambda^s f\|_{L^p}$. For $s<0$, $p\in (1,\infty)$, $H^s_p$ is the dual space of $H^{-s}_q$ with $\frac{1}{p}+\frac{1}{q}=1$.  Set $H^s:=H^s_2$.

The following embedding results will  be frequently used (e.g.  \cite{Tri78}).

\bl\label{lem:emb} (i) Let $1\leq p_1\leq p_2\leq\infty$ and $1\leq q_1\leq q_2\leq\infty$, and let $\alpha\in\mathbb{R}$. Then $B^\alpha_{p_1,q_1} \subset B^{\alpha-d(1/p_1-1/p_2)}_{p_2,q_2}$. (cf. \cite[Lemma~A.2]{GIP15})

(ii) Let $s\in \R$, $1<p<\infty$, $\epsilon>0$. Then $H^s_2=B^s_{2,2}$, and
$B^s_{p,1}\subset H^{s}_p\subset B^{s}_{p,\infty}\subset B^{s-\epsilon}_{p,1}$. (cf. \cite[Theorem 4.6.1]{Tri78})
% \hao{reference for $H^{s+\epsilon}_p\subset B^{s}_{p,1}$?}

%(iii) $B^\alpha_{p_1,q_1}\subset B^\beta_{p_2,q_2}$ if $\beta\leq \alpha, p_2\leq p_1$ and $q_2\geq q_1$. \hao{reference? do we ever need this?} \scott{Step 4 of Theorem 5.1 uses the embedding of $H^{-s}$ into $H^{-1}$ }

(iii) Let $1\leq p_1\leq p_2<\infty$ and let $\alpha\in\mathbb{R}$. Then $H^\alpha_{p_1} \subset H^{\alpha-d(1/p_1-1/p_2)}_{p_2}$.

Here  $\subset$ means  continuous and dense embedding.
\el

We recall the following interpolation inequality and  multiplicative inequality for the elements in $H^s_p$:

\bl\label{lem:interpolation}
(i)  Suppose that $s\in (0,1)$ and $p\in (1,\infty)$. Then for $f\in H^1_p$
$$\|f\|_{H^s_p}\lesssim \|f\|_{L^p}^{1-s}\|f\|_{H^1_p}^s.$$
(cf. \cite[Theorem 4.3.1]{Tri78})

(ii) Suppose that $s>0$ and $p\in [1,\infty)$. It holds that
\begin{equation}\label{e:Lambda-prod}
\|\Lambda^s(fg)\|_{L^p}\lesssim\|f\|_{L^{p_1}}\|\Lambda^sg\|_{L^{p_2}}+\|g\|_{L^{p_3}}
\|\Lambda^sf\|_{L^{p_4}},
\end{equation}
with $p_i\in (1,\infty], i=1,...,4$ such that
$$\frac{1}{p}=\frac{1}{p_1}+\frac{1}{p_2}=\frac{1}{p_3}+\frac{1}{p_4}.$$
(cf. see \cite[Theorem~1]{MR3200091})

(iii) (Gagliardo-Nirenberg inequality) For $s\in [0,1)$, $\alpha\in (0,1)$, $r\geq1$,
\begin{equ}[e:Gagliardo]
	\|u\|_{H^s_q}\lesssim \|u\|_{H^1}^\alpha\|u\|_{L^r}^{1-\alpha}
\end{equ}
with
$\frac{1}{q}=\frac{s}{d}+\alpha(\frac12-\frac1d)+\frac{1-\alpha}{r}$.

%(iii) Suppose that $s>0$ and $p\in (1,\infty)$. It holds that
%$$\|\Lambda^s(fg)-f\Lambda^sg\|_{L^p}\lesssim\|\nabla f\|_{L^{p_1}}\|g\|_{H^{s-1}_{p_2}}+\|g\|_{L^{p_4}}
%\|\Lambda^sf\|_{L^{p_3}},$$
%with $p_i\in (1,\infty], i=1,...,4$ such that
%$$\frac{1}{p}=\frac{1}{p_1}+\frac{1}{p_2}=\frac{1}{p_3}+\frac{1}{p_4}.$$
\el

\bl\label{lem:multi}
(i) Let $\alpha,\beta\in\mathbb{R}$ and $p, p_1, p_2, q\in [1,\infty]$ be such that $\frac{1}{p}=\frac{1}{p_1}+\frac{1}{p_2}$.
The bilinear map $(u, v)\mapsto uv$
extends to a continuous map from ${B}^\alpha_{p_1,q}\times {B}^\beta_{p_2,q}$ to ${B}^{\alpha\wedge\beta}_{p,q}$  if $\alpha+\beta>0$. (cf. \cite[Corollary~2]{MW17})

(ii) (Duality.) Let $\alpha\in (0,1)$, $p,q\in[1,\infty]$, $p'$ and $q'$ be their conjugate exponents, respectively. Then the mapping  $(u, v)\mapsto \<u,v\>=\int uv \dif x$  extends to a continuous bilinear form on $B^\alpha_{p,q}\times B^{-\alpha}_{p',q'}$, and one has $|\<u,v\>|\lesssim \|u\|_{B^\alpha_{p,q}}\|v\|_{B^{-\alpha}_{p',q'}}$ (cf.  \cite[Proposition~7]{MW17}).
\el

%In the sequel, whenever $\<f,g\>$ is well-defined we will, with a slight abuse of notation, just write it as $\int fg\,\dif x$, since in the context it will be clear that one of $f,g$ is distribution.\hao{I added this notation remark.} \scott{Probably a good idea: btw, derivative of the energy identity is not quite rigorous for a similar reason }
We recall the following  smoothing effect of the heat flow $S_t=e^{t(\Delta-m)}$, $m\geq0$ (e.g. \cite[Lemma~A.7]{GIP15}, \cite[Proposition~5]{MW17}).
\vskip.10in
\bl\label{lem:heat}  Let $u\in B^{\alpha}_{p,q}$ for some $\alpha\in \mathbb{R}, p,q\in [1,\infty]$. Then for every $\delta\geq0$ and $t\in [0,T]$
$$\|S_tu\|_{B^{\alpha+\delta}_{p,q}}\lesssim t^{-\delta/2}\|u\|_{B^{\alpha}_{p,q}},$$
where the proportionality constant is  independent of $t$.
\el

%\hao{should also include this in a suitable place:}\zhu{It is obvious}

\begin{lemma}\label{lem:dual+MW}
	For $s \in (0,1)$
	\begin{equation*}
	|\<g,f\>| \lesssim \big (\|\nabla g \|_{L^{1}}^{s}\|g\|_{L^{1}}^{1-s}+ \|g\|_{L^{1}} \big )\|f\|_{\bC^{-s}}  .
	\end{equation*}
\end{lemma}
\begin{proof}
	This follows from Lemma~\ref{lem:multi} which states that $\<g,f\>$ is a continuous bilinear form on
	$B^s_{1,1}\times \bC^{-s} $,
	together with \cite[Proposition~8]{MW17} which states that
	$
	\|g\|_{B^s_{1,1}}
	\lesssim \|\nabla g \|_{L^{1}}^{s}\|g\|_{L^{1}}^{1-s}+ \|g\|_{L^{1}} $.
\end{proof}

We also recall the following comparison test result, which has been proved in \cite[Lemma 3.8]{TW18}.

\bl\label{lem:co} Let $f:[0,T]\to [0,\infty)$ differentiable such that for every $t\in [0,T]$
$$\frac{\dif f}{\dif t}+c_1f^2\leq c_2.$$
Then for $t>0$
$$f(t)\leq \Big (t^{-1}\frac{2}{c_1}\Big)\vee\Big(\frac{2c_2}{c_1}\Big)^{\frac{1}{2}}.$$
\el

We recall the following result for sum  from \cite[Lemma 3.10]{ZZ15}.

\bl\label{lem:sum}
Let $0<l,r<d$, $l+r-d>0$. Then it holds
$$\sum_{k_1\in\mZ^d}(1+|k_1|)^{-l}(1+|k-k_1|)^{-r}\lesssim (1+|k|)^{d-l-r}.$$
\el

\section{Proof of Lemma \ref{lem:Yglobal}}\label{sec:A1}
\begin{proof} For initial value $y_i\in \bC^{\beta}(\mathbb{T}^2), \beta\in (1,2)$ we could use similar argument as in \cite[Theorem 6.1]{MW17} to obtain global solutions $(Y_i)$ to \eqref{eq:22} with each $Y_i\in C_T\bC^\beta$. In fact, we use mild solutions and fixed point argument to obtain unique local solutions. Furthermore, for fixed $N$ we  obtain a global in time $L^p$-estimate, $p>1$, which gives the required global solutions.

Moreover, for general initial data $y_i\in L^2$, we consider smooth approximation $(y_i^\varepsilon)$ to initial data $y_i$. For $(y_i^\varepsilon)$ we construct  solutions $Y_i^\varepsilon\in C_T\bC^\beta$ by the above argument. For $Y_i^\varepsilon$ we could do the uniform estimate as in Lemma \ref{Y:L2} and obtain
\begin{align*}
&\frac{1}{N}\sup_{t\in[0,T]}\sum_{j=1}^{N}\|Y_{j}^\varepsilon\|_{L^{2} }^2+\frac{1}{N}\sum_{j=1}^{N}\| \nabla  Y_{j}^\varepsilon\|_{L^2(0,T;L^2) }^2+\bigg \| \frac{1}{N}\sum_{i=1}^{N}(Y_{i}^\varepsilon)^{2}\bigg \|_{L^2(0,T;L^2) }^{2} \leq C,
\end{align*}
where $C$ is independent of $\varepsilon$. By standard compactness argument we deduce that there exist a sequence $\{\varepsilon_k\}$ and
$Y_i\in L^\infty_TL^2\cap L^2_TH^1\cap L^4_TL^4$ such that $Y^{\varepsilon_k}_i\to Y$ in $L^2_TH^\delta\cap C_TH^{-1}$, $\delta<1$. Furthermore, by similar argument as in the proof of \cite[Theorem 4.3]{RYZ18} we obtain $Y_i\in C_TL^2\cap L^4_TL^4\cap L^2_TH^1$. For the uniqueness part we could do similar estimate as $I_1^N$ and $I_2^N$ for the difference $v_i$ in Section \ref{sec:dif}. From the estimates \eqref{diff1}, \eqref{est55} and \eqref{eq:zz} in Section \ref{sec:dif} the regularity for $Y_i$ is enough for the uniqueness.

\end{proof}

\section{Consequences of Dyson-Schwinger equations}\label{sec:D2}
\renewcommand{\theequation}{C.\arabic{equation}}
Dyson-Schwinger equations are relations between correlation functions of different orders.
Here we derive the identities  \eqref{e:IBP1d2} and \eqref{e:IBP2d2} using Dyson-Schwinger equations;
these are essentially in \cite{MR578040} (Eq~(7)(8) therein), but since we're in slightly different setting, we give some details here to be self-contained. They are  consequences of integration by parts formula (e.g. \cite[Theorem~6.7]{GH18a}
for the $\Phi^4$ model),
In the case of the N-component $\Phi^4$ model (i.e. linear sigma model),
for a fixed $N$, 
$\Phi\thicksim \nu^N$ and writing 
$ \PPhi^2 \eqdef \sum_{i=1}^N \Phi_i^2$ as shorthand,
it is easy to  derive the following integration by parts formula:
% in the direction $\rho_{1,x}^\eps=(\rho_x^\eps,0,\dots,0)$
$$
\E \Big(D_{\rho_{1,x}^\eps} F (\Phi) \Big)
= 2\E \Big( \<\Phi_1,(m - \Delta_x ) \rho_x^\eps\> F(\Phi)\Big)
+ \frac{2}{N}  \E \Big( F(\Phi) \<\Wick{\Phi_1 \PPhi^2},\rho_x^\eps\> \Big)
$$
where $D_{\rho_{1,x}^\eps}F(\Phi)$ denotes the Fr\'echet derivative along $\rho_{1,x}^\eps \eqdef (\rho_x^\eps,0,\dots,0)$ (namely varying only $\Phi_1$ in the direction $\rho_x^\eps$).
In terms of  Green's function $C(x-y) =\frac12 (m - \Delta )^{-1}(x-y)$ we can also write it as 
\begin{equation}\label{e:IBP-Cd2}
\int C(x-z ) \E \Big(D_{\rho_{1,z}^\eps}F(\Phi)\Big) \dif z
= \E \Big( \<\Phi_1,\rho_x^\eps\> F(\Phi)\Big)
+ \frac{2}{N}   \int C(x-z )
\E \Big( F(\Phi) \<\Wick{\Phi_1 \PPhi^2},\rho_z^\eps\> \Big) \dif z
\end{equation}
%Here we will assume $d=1$ and the Wick products are understood as Sec~\ref{sec:d1}.

%The proof of \eqref{e:IBP-Cd2} is standard by using  Gaussian integration by parts.
Here we apply \eqref{e:IBP-Cd2} to prove  \eqref{e:IBP1d2}.
%Recall that  $C_{\wk}=C(0)$.
Taking
$F(\Phi) = \<\Wick{\Phi_1 \PPhi^2},\rho_x^\eps\>$ one has  
\begin{equs}
	\lim_{\eps\to0}\frac{2}{N}
	\int & C(x-z) \E \Big(\<\Wick{\Phi_1 \PPhi^2},\rho_x^\eps\>\<\Wick{\Phi_1 \PPhi^2},\rho_z^\eps\>\Big) \dif z
	\\
	&=
	\lim_{\eps\to0}\Big[\int \frac{N+2}NC(x-z)\E\Big(\<\Wick{ \PPhi^2}\rho_z^\eps,\rho_x^\eps\>  \Big)\dif z
	-\E \Big(\<\Phi_1,\rho_x^\eps\>\<\Wick{\Phi_1 \PPhi^2},\rho_x^\eps\> \Big)\Big]
	\\&=-\lim_{\eps\to0}\frac{1}{N}
	\E  \<\Wick{ (\PPhi^2)^2} ,\rho_x^\eps\>
\end{equs}
using the definition of Wick products and the symmetry (i.e. exchangeability of $(\Phi_i)_i$) and  taking  limit in $\bC^{-\kappa}$. 
\footnote{
For instance, in the first step, recalling the precise definition of
$\Wick{\Phi_1 \PPhi^2}$,
the derivative $D_{\rho_{1,z}^\eps}F(\Phi)$
gives $3\Wick{\Phi_1^2}\rho_z^\eps+\sum_{i=2}^N \Wick{\Phi_i^2}\rho_z^\eps$
which by exchangeability can be rewritten as
$\frac{N+2}{N} \Wick{\PPhi^2}\rho_z^\eps$ inside expectation.
The last step follows similarly; or it could be viewed as an $N$ dimensional generalization of the well-known relation
	$H_{n+1}(x)=xH_n(x) -H_n'(x)$ for Hermite polynomials $H_n$.}

Next, taking $F= \<\Wick{\PPhi^2},\rho_x^\eps\> \<\Wick{\Phi_1\PPhi^2},\rho_z^\eps\>$ one has
\begin{equs}
	&\lim_{\eps\to0}\frac{2}{N} \int C(x-z_1) \E  \Big(\<\Wick{\PPhi^2},\rho_x^\eps\> \<\Wick{\Phi_1\PPhi^2},\rho_{z_1}^\eps\>\<\Wick{\Phi_1\PPhi^2},\rho_z^\eps\>\Big)\dif z_1
	\\
	&=\lim_{\eps\to0}\frac{N+2}{N}\int C(x-z_1)\E\Big( \<\Wick{\PPhi^2},\rho_x^\eps\>
	 \< \Wick{\PPhi^2},\rho_{z_1}^\eps\rho_{z}^\eps\> \Big)\dif z_1
	 \\
	&\qquad+ \lim_{\eps\to0}\Big[2\int C(x-z_1)\E\Big( \<\Phi_1\rho_{z_1}^\eps,\rho_x^\eps\>  \<\Wick{\Phi_1\PPhi^2},\rho_z^\eps\> \Big)\dif z_1
	 - \E\Big(\< \Phi_1,\rho_x^\eps\>  \<\Wick{\PPhi^2},\rho_x^\eps\> \<\Wick{\Phi_1\PPhi^2},\rho_z^\eps\>\Big)\Big]
	\\
	&=\lim_{\eps\to0}\frac{N+2}{N} C(x-z) \E [\<\Wick{\PPhi^2},\rho_x^\eps\> \<\Wick{\PPhi^2},\rho_z^\eps\>]
	- \lim_{\eps\to0}\E \Big(\<\Wick{\Phi_1 \PPhi^2},\rho_x^\eps\> \<\Wick{\Phi_1\PPhi^2},\rho_z^\eps\>\Big)
\end{equs}
where  the limit is understood in $\bC^{-\kappa}$ and  we again used symmetry %, for instance
%we replaced $\frac{\delta  \Wick{\Phi_1\PPhi^2(z)}}{\delta \Phi_1(z)} $ by $\frac{N+2}{N} \Wick{\PPhi^2(z)}$
 under expectation. %, as well as a simple relation %(noting that $\Wick{\Phi^3}=\Phi^3-3C_{\wk} \Phi$)
%$$
%2C_{\wk} \Phi_1(x) - \Phi_1(x) \Wick{\PPhi(x)^2} = - \Wick{\Phi_1(x) \PPhi^2(x)}
%$$
From the above two correlation identities,
we cancel out the 6th order correlation term and then obtain
 \eqref{e:IBP1d2}. Note that we also use Fourier transform to have
$$\lim_{\eps\to0}\int \frac{N+2}{N^2} C(x-z)^2 \E [\<\Wick{\PPhi^2},\rho_x^\eps\> \<\Wick{\PPhi^2},\rho_z^\eps\>]\dif z=\sum_{k\in\mZ^2} \frac{N+2}{N} \widehat{C^2}(k) \widehat{G_N}(k).$$

Next, take $F(\Phi)= \int C(x-y) \<\Wick{\Phi_1\PPhi^2},\rho_y^\eps\>\< \Wick{\PPhi^2},\rho_z^\eps\>\dif y$.
We have, by \eqref{e:IBP-Cd2},
\begin{equs}
	{} &\frac{2}{N} \int C(x-y_1)C(x-y_2) \E  \Big( \<\Wick{\Phi_1\PPhi^2},\rho_{y_1}^\eps\>\<\Wick{\Phi_1\PPhi^2},\rho_{y_2}^\eps\> \<\Wick{\PPhi^2},\rho_z^\eps\> \Big)\dif y_1\dif y_2
	\\
	&=2 \E\Big( \<\Phi_1,C^\eps_x\rho_z^\eps\> \<\Wick{\Phi_1\PPhi^2},C^\eps_x\> \Big)
 +\frac{N+2}N \E\Big( \<\Wick{\PPhi^2},\rho_z^\eps\> \<\Wick{\PPhi^2},(C_x^{\eps})^2\> \Big)
	\\
	&\qquad - \int C(x-y) \E\Big(\< \Phi_1,\rho_x^\eps\> \<\Wick{\Phi_1\PPhi^2},\rho_y^\eps\>\< \Wick{\PPhi^2},\rho_z^\eps\> \Big)\dif y,
\end{equs}
with $C^\eps_x(\cdot)\eqdef (C*\rho_0^\eps)(x-\cdot)$.
Note that as $\eps\to0$, the limit of the LHS is precisely $Q_2^N$ and the limit of the first term on the RHS is just $-Q_1^N$.
The limit of the Fourier transform of the second term on the RHS equals
$$
(N+2) \widehat{C^2}\widehat{G_N}.
$$
To deal with the last term above, taking $F(\Phi)=\<\Phi_1,\rho_x^\eps\>\<\Wick{\PPhi^2},\rho_z^\eps\> $ and applying  \eqref{e:IBP-Cd2}, one has
\begin{equs}
&2	\int  C(x-y) \E\Big( \<\Phi_1,\rho_x^\eps\> \<\Wick{\Phi_1\PPhi^2},\rho_y^\eps\> \<\Wick{\PPhi^2},\rho_z^\eps\> \Big)\dif y
	\\
	& = N \int C(x-y)\<\rho_x^\eps,\rho_y^\eps\> \dif y\E \<\Wick{\PPhi^2},\rho_z^\eps\> + 2N \int C(x-z_1) \E \Big(\<\Phi_1,\rho_x^\eps\>\<\Phi_1,\rho_{z_1}^\eps \rho_z^\eps\>  \Big)\dif z_1
	\\
	&\qquad - N\E[ \<\Phi_1,\rho_x^\eps\>^2 \<\Wick{\PPhi^2},\rho_z^\eps\> ].
\end{equs}
Letting $\eps\to0$ we deduce the limit of the Fourier transform of the RHS is
\begin{align*}
- N\widehat{G_N}
+  2N \widehat{CC_N}.\end{align*}
We then obtain \eqref{e:IBP2d2}.

%\hao{Note that actually $2C (x-z) \E [\Phi_1 (x) \Phi_1(z)]$ here and $2 C(x-z)^2$ in \eqref{e:IBP2} are different, and  \eqref{e:IBP2} does not have the coefficient $\frac{N+2}{N}$. However it doesn't matter in the limit.}\zhu{we could follow Hao's calculation.}

\section{Proof of Step \ref{diffEst5} in the proof of Theorem \ref{th:conv-v}}\label{sec:E2}
\renewcommand{\theequation}{D.\arabic{equation}}

We write $\bar{I}^N$ as $\sum_{i=1}^3\bar{I}^N_{i}$:
\begin{align*}
\bar{I}^N_1&\eqdef-\frac{1}{N}\sum_{i,j=1}^N \Big[\<Y_j^2,v_iu_i\>+2\<Y_jY_iu_j,v_i \>\Big],
\\
\bar{I}^N_{21}&\eqdef-\frac{1}{N}\sum_{i,j=1}^N 2\<Y_jv_i,\Wick{Z_i^NZ_j^N}-\Wick{Z_iZ_j}\>,\\
\bar{I}^N_{22}&\eqdef-\frac{1}{N}\sum_{i,j=1}^N \<Y_iv_i,\Wick{Z_j^{N,2}}-\Wick{Z_j^2}\>\\
\bar{I}^N_3&\eqdef-\frac{1}{N}\sum_{i,j=1}^N
\<v_i,\Wick{Z_i^NZ_j^{N,2}}-\Wick{Z_iZ_j^2}\>.
\end{align*}

In the following we estimate each term and show that for $\delta>0$ small 
\begin{equation}\aligned\label{eq:er}
|\bar{I}^N|\lesssim& \delta \Big(\sum_{i=1}^{N} \|\nabla v_{i}\|_{L^{2}}^{2}+\frac{1}{N}\sum_{i,j=1}^{N}\|Y_{j}v_{i}\|_{L^{2}}^{2}  \Big )+\sum_{i=1}^N\| v_i\|_{L^2}^2
\\&+\Big(\sum_{i=1}^N\|u_i\|_{L^\infty}^2\Big)\Big(\frac{1}{N}\sum_{j=1}^{N}\|Y_{j}\|_{L^{2}}^{2}  \Big )+\tilde{R}_N, \endaligned
\end{equation}
with %$\E \int_0^T R_N\dif t\lesssim 1$ and
\begin{align*}
\tilde{R}_N\eqdef& \frac{1}{N}\sum_{i,j=1}^N\|\Wick{Z_i^NZ_j^{N,2}}-\Wick{Z_iZ_j^2}\|_{\bC^{-s}}^2\\&+ \Big (\frac{1}{N}\sum_{j=1}^{N} \| Y_{j}\|_{H^1}^{2}\Big)^{s}\mathfrak{Z}_N+\mathfrak{Z}_N+\Big (\frac{1}{N}\sum_{i=1}^{N} \| Y_{i}\|_{H^1}^{2}\Big)^{s}\Big (\frac{1}{N}\sum_{i=1}^{N} \| Y_{i}\|_{L^{2}}^{2}\Big)^{1-s}\bar{\mathfrak{Z}}_N,
\end{align*} 
with $\bar{\mathfrak{Z}}_N$ and ${\mathfrak{Z}}_N$ introduced in \eqref{e:z} below.

For $\bar{I}^N_1$ we use Young's inequality to have 
\begin{align*}|\bar{I}_{1}^N|
\lesssim& \delta\frac{1}{N}\sum_{i,j=1}^N\int Y_j^2v_i^2+\frac{1}{N}\sum_{i,j=1}^N\int Y_j^2u_i^2
\\\lesssim&\delta\frac{1}{N}\sum_{i,j=1}^N\int Y_j^2v_i^2+\Big(\frac{1}{N}\sum_{j=1}^N\|Y_j\|_{L^2}^2\Big)\Big(\sum_{i=1}^N\|u_i\|_{L^\infty}^2\Big),\end{align*}
which gives the first contribution to \eqref{eq:er}.

For $\bar{I}^N_3$ we use Lemma \ref{lem:multi}, Lemma \ref{lem:interpolation} and Young's inequality to have
\begin{align*}|\bar{I}_{3}^N|
&\lesssim \frac{1}{N}\sum_{i,j=1}^N\|v_i\|_{H^{2s}}
\|\Wick{Z_i^NZ_j^{N,2}}-\Wick{Z_iZ_j^2}\|_{\bC^{-s}}
\\&\lesssim \delta\sum_{i=1}^N\|\nabla v_i\|_{L^2}^2+ \sum_{i=1}^N\| v_i\|_{L^2}^2
+\frac{1}{N}\sum_{i,j=1}^N\|\Wick{Z_i^NZ_j^{N,2}}-\Wick{Z_iZ_j^2}\|_{\bC^{-s}}^2,\end{align*}
%\scott{I geuss we also need $-s+\kappa$ here also?} \zhu{Since $v_i$ is $2s$, $\bC^{-s}$ is enough.}
which gives the second contribution to \eqref{eq:er}.

For $\bar{I}^N_{2k}$ we set 
\begin{align}\label{e:z}
\mathfrak{Z}_N\eqdef&\frac{1}{N}\sum_{i,j=1}^N\|\Wick{Z_i^NZ_j^N}-\Wick{Z_iZ_j}\|_{\bC^{-s}}^2,\qquad
\bar{\mathfrak{Z}}_N\eqdef\sum_{j=1}^N\|\Wick{Z_j^{N,2}}-\Wick{Z_j^2}\|_{\bC^{-s}}^2.
\end{align}
We use Lemma \ref{lem:dual+MW} and H\"older's inequality to have
\begin{align*}
|\bar{I}^N_{21}|\lesssim&\frac{1}{N}\sum_{i,j=1}^N (\|\nabla (Y_jv_i)\|_{L^1}^s\|Y_jv_i\|_{L^1}^{1-s}+\|Y_jv_i\|_{L^1})\|\Wick{Z_i^NZ_j^N}-\Wick{Z_iZ_j}\|_{\bC^{-s}}
\\\lesssim&\delta
\Big(\frac{1}{N}\sum_{j=1}^N\|v_iY_j\|_{L^2}^2\Big)
+\mathfrak{Z}_N
 \nonumber \\
&+ \Big (\sum_{i=1}^{N} \|v_{i}\|_{L^{2}}^{2}  \Big )^{s/2} \Big (\frac{1}{N}\sum_{j=1}^{N} \|\nabla Y_{j}\|_{L^{2}}^{2}\Big)^{s/2}\Big(\frac{1}{N}\sum_{i,j=1}^{N}\|v_iY_j\|_{L^2}^2  \Big )^{\frac{1-s}2}\mathfrak{Z}_N^{1/2}
\\
&+ \Big (\sum_{i=1}^{N} \|\nabla v_{i}\|_{L^{2}}^{2}  \Big )^{s/2} \Big (\frac{1}{N}\sum_{j=1}^{N} \| Y_{j}\|_{L^{2}}^{2}\Big)^{s/2}\Big(\frac{1}{N}\sum_{i,j=1}^{N}\|v_iY_j\|_{L^2}^2  \Big )^{\frac{1-s}2}\mathfrak{Z}_N^{1/2},
\end{align*}
which by Young's inequality gives 
\begin{align*}
|\bar{I}^N_{21}|\lesssim&
\delta\Big(\frac{1}{N}\sum_{i,j=1}^N\|v_iY_j\|_{L^2}^2\Big)
+\delta(\sum_{i=1}^{N} \|\nabla v_{i}\|_{L^{2}}^{2}  \Big )
\nonumber \\
&+ \Big (\sum_{i=1}^{N} \|v_{i}\|_{L^{2}}^{2}  \Big )+ \Big (\frac{1}{N}\sum_{j=1}^{N} \| Y_{j}\|_{H^1}^{2}\Big)^{s}\mathfrak{Z}_N+\mathfrak{Z}_N
.
\end{align*}
Similarly we deduce
\begin{align*}
|\bar{I}^N_{22}|\lesssim&\frac{1}{N}\sum_{i,j=1}^N (\|\nabla (Y_iv_i)\|_{L^1}^s\|Y_iv_i\|_{L^1}^{1-s}+\|Y_iv_i\|_{L^1})\|\Wick{Z_j^{N,2}}-\Wick{Z_j^{2}}\|_{\bC^{-s}}
\\\lesssim&
\Big(\sum_{i=1}^N\|v_i\|_{L^2}^2\Big)
+\bar{\mathfrak{Z}}_N\Big(\frac{1}{N}\sum_{i=1}^N\|Y_i\|_{L^2}^2\Big)
\nonumber \\
&+ \Big (\sum_{i=1}^{N} \|v_{i}\|_{L^{2}}^{2}  \Big )^{1/2} \Big (\frac{1}{N}\sum_{i=1}^{N} \|\nabla Y_{i}\|_{L^{2}}^{2}\Big)^{s/2}\Big(\frac{1}{N}\sum_{i=1}^{N}\|Y_i\|_{L^2}^2  \Big )^{\frac{1-s}2}\bar{\mathfrak{Z}}_N^{1/2}
\\
&+ \Big (\sum_{i=1}^{N} \|\nabla v_{i}\|_{L^{2}}^{2}  \Big )^{s/2} \Big (\frac{1}{N}\sum_{i=1}^{N} \| Y_{i}\|_{L^{2}}^{2}\Big)^{1/2}\Big(\sum_{i=1}^{N}\|v_i\|_{L^2}^2  \Big )^{\frac{1-s}2}\bar{\mathfrak{Z}}_N^{1/2},
\end{align*}
which implies 
\begin{align*}
|\bar{I}^N_{22}|\lesssim&
\delta(\sum_{i=1}^{N} \|\nabla v_{i}\|_{L^{2}}^{2}  \Big )
+ \Big (\sum_{i=1}^{N} \|v_{i}\|_{L^{2}}^{2}  \Big )+ \Big (\frac{1}{N}\sum_{i=1}^{N} \| Y_{i}\|_{H^1}^{2}\Big)^{s}\Big (\frac{1}{N}\sum_{i=1}^{N} \| Y_{i}\|_{L^{2}}^{2}\Big)^{1-s}\bar{\mathfrak{Z}}_N
.
\end{align*}
Thus we deduce \eqref{eq:er}. By Assumption \ref{a:main} and Lemma \ref{le:ex} it is easy to find $\frac1N\|\tilde{R}_N\|_{L^1_T}\to0$ in $L^1(\Omega)$. The result follows by similar argument as Step \ref{diffEst4} and Step \ref{diffEst6}.

\bibliographystyle{alphaabbr}
\bibliography{Reference}

\end{document}